\documentclass[english]{article}
\usepackage[T1]{fontenc}
\usepackage[utf8]{inputenc}
\usepackage{babel}
\usepackage{float}
\usepackage{units}
\usepackage{dsfont}
\usepackage{amsmath}
\usepackage{amsthm}
\usepackage{amssymb}
\usepackage{graphicx}
\usepackage[authoryear,round]{natbib}
\usepackage[pdfusetitle,
 bookmarks=true,bookmarksnumbered=false,bookmarksopen=false,
 breaklinks=false,pdfborder={0 0 1},backref=false,colorlinks=false]
 {hyperref}

\makeatletter

\newcommand{\lyxmathsym}[1]{\ifmmode\begingroup\def\b@ld{bold}
  \text{\ifx\math@version\b@ld\bfseries\fi#1}\endgroup\else#1\fi}

\floatstyle{ruled}
\newfloat{algorithm}{tbp}{loa}
\providecommand{\algorithmname}{Algorithm}
\floatname{algorithm}{\protect\algorithmname}

\theoremstyle{plain}
\newtheorem{assumption}{\protect\assumptionname}
\theoremstyle{remark}
\newtheorem{rem}{\protect\remarkname}
\theoremstyle{plain}
\newtheorem{thm}{\protect\theoremname}
\theoremstyle{plain}
\newtheorem{lem}{\protect\lemmaname}

\@ifundefined{date}{}{\date{}}
\usepackage[paperwidth=8.5in, paperheight=11in, margin=1in]{geometry}

\makeatother

\providecommand{\assumptionname}{Assumption}
\providecommand{\lemmaname}{Lemma}
\providecommand{\remarkname}{Remark}
\providecommand{\theoremname}{Theorem}

\begin{document}
\global\long\def\A{\mathfrak{A}}%
\global\long\def\E{\mathbb{E}}%
\global\long\def\F{\mathcal{F}}%
\global\long\def\G{\mathfrak{G}}%
\global\long\def\H{\mathrm{H}}%
\global\long\def\I{\mathrm{I}}%
\global\long\def\N{\mathbb{N}}%
\global\long\def\O{\mathcal{O}}%
\global\long\def\R{\mathbb{R}}%
\global\long\def\S{\mathbb{S}}%
\global\long\def\V{\mathbb{V}}%
\global\long\def\W{\mathbb{W}}%
\global\long\def\X{\mathbb{X}}%
\global\long\def\rb{\mathrm{b}}%
\global\long\def\rc{\mathrm{c}}%
\global\long\def\ru{\mathrm{u}}%
\global\long\def\d{\mathrm{d}}%
\global\long\def\bd{\mathbf{d}}%
\global\long\def\be{\mathbf{e}}%
\global\long\def\bg{\mathbf{g}}%
\global\long\def\bff{\mathbf{f}}%
\global\long\def\bu{\mathbf{u}}%
\global\long\def\bv{\mathbf{v}}%
\global\long\def\bw{\mathbf{w}}%
\global\long\def\bx{\mathbf{x}}%
\global\long\def\by{\mathbf{y}}%
\global\long\def\bz{\mathbf{z}}%
\global\long\def\p{\mathfrak{p}}%
\global\long\def\bzero{\mathbf{0}}%
\global\long\def\1{\mathds{1}}%
\global\long\def\defeq{\triangleq}%
\global\long\def\argmin{\mathrm{argmin}}%
\global\long\def\dis{\mathbb{D}}%
\global\long\def\lar{\mathfrak{l}}%
\global\long\def\sma{\mathfrak{s}}%
\global\long\def\tr{\mathrm{Tr}}%
\global\long\def\clip{\mathrm{clip}}%
\global\long\def\eff{\mathrm{eff}}%
\global\long\def\type{\mathrm{type}}%
\global\long\def\cvx{\mathrm{cvx}}%
\global\long\def\str{\mathrm{str}}%
\global\long\def\with{\mathrm{w}\text{/}}%
\global\long\def\without{\mathrm{w}\text{/}\mathrm{o}}%
\global\long\def\poly{\mathrm{poly}}%
\global\long\def\mydots{\cdots}%
\global\long\def\cm{\tau}%
\global\long\def\hres{I}%
\global\long\def\hc{A}%
\global\long\def\eres{J}%
\global\long\def\ec{B}%
\global\long\def\sgn{\mathrm{sgn}}%
\global\long\def\TV{\mathrm{TV}}%
\global\long\def\KL{\mathrm{KL}}%

\title{Clipped Gradient Methods for Nonsmooth Convex Optimization under Heavy-Tailed
Noise: A Refined Analysis\thanks{A preliminary conference version is accepted at ICLR 2026. Compared
to the conference version, we include the formal statements of lower
bounds and their proofs.}}
\author{Zijian Liu\thanks{Stern School of Business, New York University, zl3067@stern.nyu.edu.}}
\maketitle
\begin{abstract}
Optimization under heavy-tailed noise has become popular recently,
since it better fits many modern machine learning tasks, as captured
by empirical observations. Concretely, instead of a finite second
moment on gradient noise, a bounded $\mathfrak{p}$-th moment where
$\mathfrak{p}\in(1,2]$ has been recognized to be more realistic (say
being upper bounded by $\sigma_{\mathfrak{l}}^{\mathfrak{p}}$ for
some $\sigma_{\mathfrak{l}}\geq0$). A simple yet effective operation,
gradient clipping, is known to handle this new challenge successfully.
Specifically, Clipped Stochastic Gradient Descent (Clipped SGD) guarantees
a high-probability rate $\mathcal{O}(\sigma_{\mathfrak{l}}\ln(1/\delta)T^{\frac{1}{\mathfrak{p}}-1})$
(resp. $\mathcal{O}(\sigma_{\mathfrak{l}}^{2}\ln^{2}(1/\delta)T^{\frac{2}{\mathfrak{p}}-2})$)
for nonsmooth convex (resp. strongly convex) problems, where $\delta\in(0,1]$
is the failure probability and $T\in\mathbb{N}$ is the time horizon.
In this work, we provide a refined analysis for Clipped SGD and offer
two rates, $\mathcal{O}(\sigma_{\mathfrak{l}}d_{\mathrm{eff}}^{-\frac{1}{2\mathfrak{p}}}\ln^{1-\frac{1}{\mathfrak{p}}}(1/\delta)T^{\frac{1}{\mathfrak{p}}-1})$
and $\mathcal{O}(\sigma_{\mathfrak{l}}^{2}d_{\mathrm{eff}}^{-\frac{1}{\mathfrak{p}}}\ln^{2-\frac{2}{\mathfrak{p}}}(1/\delta)T^{\frac{2}{\mathfrak{p}}-2})$,
faster than the aforementioned best results, where $d_{\mathrm{eff}}\geq1$
is a quantity we call the \textit{generalized effective dimension}.
Our analysis improves upon the existing approach in two respects:
better utilization of Freedman's inequality and finer bounds for clipping
error under heavy-tailed noise. In addition, we extend the refined
analysis to convergence in expectation and obtain new rates that break
the known lower bounds. Lastly, to complement the study, we establish
new lower bounds for both high-probability and in-expectation convergence.
Notably, the in-expectation lower bounds match our new upper bounds,
indicating the optimality of our refined analysis for convergence
in expectation.
\end{abstract}

\section{Introduction\label{sec:introduction}}

In first-order methods for stochastic optimization, one can only
query an unbiased though noisy gradient and then implement a gradient
descent step, which is known as Stochastic Gradient Descent (SGD)
\citep{10.1214/aoms/1177729586}. Under the widely assumed finite
variance condition, i.e., the gradient noise\footnote{This refers to the difference between the stochastic estimate and
the true gradient.} has a finite second moment, the in-expectation convergence of SGD
has been substantially studied \citep{doi:10.1137/16M1080173,lan2020first}.

However, many recent empirical observations suggest that the finite
variance assumption might be too strong and could be violated in different
tasks \citep{pmlr-v97-simsekli19a,NEURIPS2020_b05b57f6,NEURIPS2020_f3f27a32,pmlr-v139-garg21b,pmlr-v139-gurbuzbalaban21a,pmlr-v139-hodgkinson21a,pmlr-v238-battash24a}.
Instead, a bounded $\p$-th moment condition where $\p\in\left(1,2\right]$
(say with an upper bound $\sigma_{\lar}^{\p}$ for some $\sigma_{\lar}\geq0$)
better fits modern machine learning, which is named heavy-tailed noise.
Facing this new challenge, SGD has been proved to exhibit undesirable
behaviors \citep{NEURIPS2020_b05b57f6,pmlr-v202-sadiev23a}. Therefore,
an algorithmic change is necessary. A simple yet effective operation,
gradient clipping, is known to handle this harder situation successfully
with both favorable practical performance and provable theoretical
guarantees (see, e.g., \citet{pmlr-v28-pascanu13,NEURIPS2020_b05b57f6}).
The clipping mechanism replaces the stochastic gradient $\bg_{t}$
in every iterate of SGD with its truncated counterpart $\clip_{\cm_{t}}(\bg_{t})$,
resulting in a method known as Clipped SGD, where $\cm_{t}$ is called
the clipping threshold and $\clip_{\cm}(\bg)\defeq\min\left\{ 1,\nicefrac{\cm}{\left\Vert \bg\right\Vert }\right\} \bg$
is the clipping function.

Specifically, for nonsmooth convex (resp. strongly convex) optimization,
Clipped SGD achieves a high-probability rate $\O(\sigma_{\lar}\ln(1/\delta)T^{\frac{1}{\p}-1})$\footnote{When stating rates in this section, we only keep the dominant term
when $T\to\infty$ and $\delta\to0$ for simplicity.} (resp. $\O(\sigma_{\lar}^{2}\ln^{2}(1/\delta)T^{\frac{2}{\p}-2})$)
\citep{liu2023stochastic}, where $\delta\in\left(0,1\right]$ is
the failure probability and $T\in\N$ is the time horizon. These two
results seem to be optimal as they match the existing in-expectation
lower bounds \citep{nemirovskij1983problem,pmlr-v178-vural22a,NEURIPS2020_b05b57f6},
if viewing the $\poly(\ln(1/\delta))$ term as a constant. However,
a recent advance \citep{NEURIPS2024_10bf9689} established a better
rate $\O(\sigma_{\lar}d_{\eff}^{-\frac{1}{4}}\sqrt{\ln(\ln(T)/\delta)/T})$
for general convex problems when $\p=2$, where $1\leq d_{\eff}\leq d$
is known as the \textit{effective dimension} (also named \textit{intrinsic
dimension} \citep{MAL-048}) and $d$ is the true dimension. This
reveals that the in-expectation lower bound does not necessarily apply
to the term containing $\poly(\ln(1/\delta))$. More importantly,
such a result hints that a general improvement may exist for all $\p\in\left(1,2\right]$.

This work confirms that a general improvement does exist by providing
a refined analysis for Clipped SGD. Concretely, we offer two faster
rates, $\O(\sigma_{\lar}d_{\eff}^{-\frac{1}{2\p}}\ln^{1-\frac{1}{\p}}(1/\delta)T^{\frac{1}{\p}-1})$
for general convex problems with a known $T$ and $\O(\sigma_{\lar}^{2}d_{\eff}^{-\frac{1}{\p}}\ln^{2-\frac{2}{\p}}(1/\delta)T^{\frac{2}{\p}-2})$
for strongly convex problems with an unknown $T$, improved upon the
aforementioned best results, where $1\leq d_{\eff}\leq\O(d)$ is a
quantity that we call the \textit{generalized effective dimension}\footnote{We use the same notation to denote the effective dimension and the
generalized version proposed by us, since our new quantity can recover
the previous one when $\p=2$. See discussion after (\ref{eq:main-d-eff})
for details.}. Moreover, we devise an algorithmic variant of Clipped SGD named
Stabilized Clipped SGD that achieves the same rate\footnote{To clarify, ``the same rate'' refers to the same lower-order term.
The full bound is slightly different.} for convex objectives listed above in an anytime fashion, i.e., no
extra $\poly(\ln T)$ factor even without $T$.

We highlight that our analysis improves upon the existing approach
in two respects: 1. We observe a better way to apply Freedman's inequality
when analyzing Clipped SGD, which leads to a provably tighter concentration.
Remarkably, our approach is fairly simple in contrast to the previous
complex iterative refinement strategy \citep{NEURIPS2024_10bf9689}.
2. We establish finer bounds for clipping error under heavy-tailed
noise, which is another essential ingredient in the analysis for Clipped
SGD when the noise has a heavy tail. We believe both of these new
insights could be of independent interest and potentially useful for
future research.

Furthermore, equipped with the new finer bounds for clipping error,
we extend the analysis to in-expectation convergence and obtain two
new rates, $\O(\sigma_{\lar}d_{\eff}^{-\frac{2-\p}{2\p}}T^{\frac{1}{\p}-1})$
for general convex objectives and $\O(\sigma_{\lar}^{2}d_{\eff}^{-\frac{2-\p}{\p}}T^{\frac{2}{\p}-2})$
for strongly convex problems. Notably, once $\p<2$, these two rates
are both faster by a $\poly(1/d_{\eff})$ factor than the known optimal
lower bounds $\Omega(\sigma_{\lar}T^{\frac{1}{\p}-1})$ and $\Omega(\sigma_{\lar}^{2}T^{\frac{2}{\p}-2})$
in the corresponding setting \citep{nemirovskij1983problem,pmlr-v178-vural22a,NEURIPS2020_b05b57f6}.

Lastly, to complement the study, we establish new lower bounds for
both high-probability and in-expectation convergence. Notably, the
in-expectation lower bounds match our new upper bounds, indicating
the optimality of our refined analysis for convergence in expectation.

\subsection{Related Work\label{subsec:related}}

We review the literature that studies nonsmooth (strongly) convex
optimization under heavy-tailed noise. For other different settings,
e.g., smooth (strongly) convex or smooth/nonsmooth nonconvex problems
under heavy-tailed noise, the interested reader could refer to, for
example, \citet{nazin2019algorithms,pmlr-v125-davis20a,NEURIPS2020_abd1c782,pmlr-v139-mai21a,NEURIPS2021_26901deb,NEURIPS2021_9cdf2656,pmlr-v151-tsai22a,Holland_2022,doi:10.1137/21M145896X,pmlr-v202-sadiev23a,pmlr-v195-liu23c,NEURIPS2023_4c454d34,pmlr-v238-puchkin24a,pmlr-v235-gorbunov24a,pmlr-v235-liu24bo,pmlr-v258-armacki25a,pmlr-v258-hubler25a,liu2025nonconvex,JMLR:v26:24-1991},
for recent progress.

\textbf{High-probability rates.} If $\p=2$, \citet{gorbunov2024high}
proves the first $\O(\sigma_{\lar}\sqrt{\ln(T/\delta)/T})$ (resp.
$\O(\sigma_{\lar}^{2}\ln(T/\delta)/T)$) high-probability rate for
nonsmooth convex (resp. strongly convex) problems under standard assumptions.
If additionally assuming a bounded domain, an improved rate $\O(\sigma_{\lar}\sqrt{\ln(1/\delta)/T})$
for convex objectives is obtained by \citet{doi:10.1137/22M1536558}.
Still for convex problems, \citet{NEURIPS2024_10bf9689} recently
gives the first refined bound $\O(\sigma_{\lar}d_{\eff}^{-\frac{1}{4}}\sqrt{\ln(\ln(T)/\delta)/T})$
but additionally requiring $T\geq\Omega(\ln(\ln d))$, where $d_{\eff}$
(resp. $d$) is the effective (resp. true) dimension, satisfying $1\leq d_{\eff}\leq d$.
For general $\p\in\left(1,2\right]$, \citet{NEURIPS2022_349956de}
studies the harder online convex optimization, whose result implies
a rate $\O(\sigma_{\lar}\poly(\ln(T/\delta))T^{\frac{1}{\p}-1})$
for heavy-tailed convex optimization. Later on, \citet{liu2023stochastic}
establishes two bounds, $\O(\sigma_{\lar}\ln(1/\delta)T^{\frac{1}{\p}-1})$
and $\O(\sigma_{\lar}^{2}\ln^{2}(1/\delta)T^{\frac{2}{\p}-2})$, for
convex and strongly convex problems, respectively. These two rates
are the best-known results for general $\p\in\left(1,2\right]$ and
have been recognized as optimal since they match the in-expectation
lower bounds (see below), if viewing the $\poly(\ln(1/\delta))$ term
as a constant.

\textbf{In-expectation rates. }Note that the in-expectation rates
for $\p=2$ are not worth much attention as they are standard results
\citep{doi:10.1137/16M1080173,lan2020first}. As for general $\p\in\left(1,2\right]$,
many existing works prove the rates $\O(\sigma_{\lar}T^{\frac{1}{\p}-1})$
and $\O(\sigma_{\lar}^{2}T^{\frac{2}{\p}-2})$ \citep{NEURIPS2020_b05b57f6,pmlr-v178-vural22a,liu2023stochastic,liu2024revisiting,parletta2025improvedanalysisclippedstochastic,fatkhullin2025can,liu2025online}.

\textbf{Lower bounds. }The high-probability lower bounds are not fully
explored in the literature. To the best of our knowledge, there are
only few results for the general convex case and no lower bounds for
the strongly convex case. Therefore, the following discussion is only
for convex problems. For $\p=2$, \citet{pmlr-v247-carmon24a} shows
a lower bound $\Omega(\sigma_{\lar}\sqrt{\ln(1/\delta)/T})$. However,
it is only proved for $d=1$ (or at most $d=4$). As such, it cannot
reveal useful information for the case that $d$ should also be viewed
as a parameter (if more accurately, $d_{\eff}$). In other words,
it does not contradict our new refined upper bound. For general $\p\in\left(1,2\right]$,
\citet{5394945} is the only work that we are aware of. However, as
far as we can check, only the time horizon $T$ is in the right order
of $\Omega(T^{\frac{1}{\p}-1})$. For other parameters, they are either
hidden or not tight.

Next, we summarize the in-expectation lower bounds. For convex problems,
it is known that any first-order method cannot do better than $\Omega(\sigma_{\lar}T^{\frac{1}{\p}-1})$
\citep{nemirovskij1983problem,pmlr-v178-vural22a}. If strong convexity
additionally holds, \citet{NEURIPS2020_b05b57f6} establishes the
lower bound $\Omega(\sigma_{\lar}^{2}T^{\frac{2}{\p}-2})$.

\section{Preliminary\label{sec:preliminary}}

\textbf{Notation.} $\N$ is the set of natural numbers (excluding
$0$). We denote by $\left[T\right]\triangleq\left\{ 1,\mydots,T\right\} ,\forall T\in\N$.
$\left\langle \cdot,\cdot\right\rangle $ represents the standard
Euclidean inner product. $\left\Vert \bx\right\Vert $ is the Euclidean
norm of the vector $\bx$ and $\left\Vert \mathbf{X}\right\Vert $
is the operator norm of the matrix $\mathbf{X}$. $\tr(\mathbf{X})$
is the trace of a square matrix $\mathbf{X}$. $\S^{d-1}$ stands
for the unit sphere in $\R^{d}$. Given a convex function $h:\R^{d}\to\R$,
$\nabla h(\bx)$ denotes an arbitrary element in $\partial h(\bx)$
where $\partial h(\bx)$ is the subgradient set of $h$ at $\bx$.
$\sgn(x)$ is the sign function with $\sgn(0)=0$.

We study the composite optimization problem in the form of
\[
\inf_{\bx\in\X}F(\bx)\defeq f(\bx)+r(\bx),
\]
where $\X\subseteq\R^{d}$ is a nonempty closed convex set. Our analysis
relies on the following assumptions.
\begin{assumption}
\label{assu:minimizer}There exists $\bx_{\star}\in\X$ such that
$F_{\star}\defeq F(\bx_{\star})=\inf_{\bx\in\X}F(\bx)$.
\end{assumption}
\begin{assumption}
\label{assu:obj}Both $f:\R^{d}\to\R$ and $r:\R^{d}\to\R$ are convex.
In addition, $r$ is $\mu$-strongly convex on $\X$ for some $\mu\geq0$,
i.e., $r(\bx)\geq r(\by)+\left\langle \nabla r(\by),\bx-\by\right\rangle +\frac{\mu}{2}\left\Vert \bx-\by\right\Vert ^{2},\forall\bx,\by\in\X$.
\end{assumption}
\begin{assumption}
\label{assu:lip}$f$ is $G$-Lipschitz on $\X$, i.e., $\left\Vert \nabla f(\bx)\right\Vert \leq G,\forall\bx\in\X$.
\end{assumption}
The above assumptions are standard in the literature \citep{doi:10.1137/16M1080173,nesterov2018lectures,lan2020first}.
Next, we consider a fine-grained heavy-tailed noise assumption, the
key to obtaining refined convergence for Clipped SGD.
\begin{assumption}
\label{assu:oracle}There exists a function $\bg:\X\times\Xi\to\R^{d}$
and a probability distribution $\dis$ on $\Xi$ such that $\E_{\xi\sim\dis}\left[\bg(\bx,\xi)\right]=\nabla f(\bx),\forall\bx\in\X$.
In addition, for some $\p\in\left(1,2\right]$, we have
\begin{eqnarray*}
\E_{\xi\sim\dis}\left[\left|\left\langle \be,\bg(\bx,\xi)-\nabla f(\bx)\right\rangle \right|^{\p}\right]\leq\sigma_{\sma}^{\p}, & \E_{\xi\sim\dis}\left[\left\Vert \bg(\bx,\xi)-\nabla f(\bx)\right\Vert ^{\p}\right]\leq\sigma_{\lar}^{\p}, & \forall\bx\in\X,\be\in\S^{d-1},
\end{eqnarray*}
where $\sigma_{\sma}$ and $\sigma_{\lar}$ are two constants satisfying
$0\leq\sigma_{\sma}\leq\sigma_{\lar}\leq\sqrt{\pi d/2}\sigma_{\sma}.$
\end{assumption}
\begin{rem}
In the remaining paper, if the context is clear, we drop the subscript
$\xi\sim\dis$ in $\E_{\xi\sim\dis}$ to ease the notation. Moreover,
$\bd(\bx,\xi)\defeq\bg(\bx,\xi)-\nabla f(\bx)$ denotes the error
in estimating the gradient.
\end{rem}
\begin{rem}
\label{rem:implicit}It is noteworthy that Assumption \ref{assu:oracle}
actually implicitly exists in prior works for heavy-tailed stochastic
optimization, since Cauchy-Schwarz inequality gives us
\[
\E\left[\left|\left\langle \be,\bd(\bx,\xi)\right\rangle \right|^{\p}\right]\leq\E\left[\left\Vert \be\right\Vert ^{\p}\left\Vert \bd(\bx,\xi)\right\Vert ^{\p}\right]=\E\left[\left\Vert \bd(\bx,\xi)\right\Vert ^{\p}\right],\forall\bx\in\X,\be\in\S^{d-1}.
\]
In other words, once the condition $\E\left[\left\Vert \bd(\bx,\xi)\right\Vert ^{\p}\right]\leq\sigma_{\lar}^{\p},\forall\bx\in\X$
is assumed like in prior works, there must exist a real number $0\leq\sigma_{\sma}\leq\sigma_{\lar}$
such that $\E\left[\left|\left\langle \be,\bd(\bx,\xi)\right\rangle \right|^{\p}\right]\leq\sigma_{\sma}^{\p},\forall\bx\in\X,\be\in\S^{d-1}$.
\end{rem}
\begin{rem}
The reason we can assume $\sigma_{\lar}\leq\sqrt{\pi d/2}\sigma_{\sma}$
is that $\E\left[\left\Vert \bd(\bx,\xi)\right\Vert ^{\p}\right]\leq(\pi d/2)^{\frac{\p}{2}}\sigma_{\sma}^{\p}$
holds provided $\E\left[\left|\left\langle \be,\bd(\bx,\xi)\right\rangle \right|^{\p}\right]\leq\sigma_{\sma}^{\p},\forall\be\in\S^{d-1}$,
due to Lemma 4.1 in \citet{pmlr-v178-cherapanamjeri22a}.
\end{rem}
Now we define the following quantity named \textit{generalized effective
dimension} (where we use the convention $0=0/0$),
\begin{equation}
d_{\eff}\defeq\sigma_{\lar}^{2}/\sigma_{\sma}^{2}\in\left\{ 0\right\} \cup\left[1,\pi d/2\right]=\O(d),\label{eq:main-d-eff}
\end{equation}
in which $d_{\eff}=0$ if and only if $\sigma_{\lar}=\sigma_{\sma}=0$,
i.e., the noiseless case. As discussed later, this definition recovers
the effective dimension used in \citet{NEURIPS2024_10bf9689} when
$\p=2$.

To better understand Assumption \ref{assu:oracle}, we first take
$\p=2$. Note that a finite second moment of $\bd(\bx,\xi)$ implies
the covariance matrix $\Sigma(\bx)\defeq\E\left[\bd(\bx,\xi)\bd^{\top}(\bx,\xi)\right]\in\R^{d\times d}$
is well defined. As such, we can interpret $\sigma_{\lar}$ and $\sigma_{\sma}$
as $\sigma_{\lar}^{2}=\sup_{\bx\in\X}\tr(\Sigma(\bx))$ and $\sigma_{\sma}^{2}=\sup_{\bx\in\X}\left\Vert \Sigma(\bx)\right\Vert $.
In particular, if $\Sigma(\bx)\preceq\Sigma,\forall\bx\in\X$ holds
for some positive semidefinite $\Sigma$ as assumed in \citet{NEURIPS2024_10bf9689},
then one can directly take $\sigma_{\lar}^{2}=\tr(\Sigma)$ and $\sigma_{\sma}^{2}=\left\Vert \Sigma\right\Vert $,
which also recovers the effective dimension defined as $\nicefrac{\tr(\Sigma)}{\left\Vert \Sigma\right\Vert }$
in \citet{NEURIPS2024_10bf9689}.

For general $\p\in\left(1,2\right]$, as discussed in Remark \ref{rem:implicit},
one can view Assumption \ref{assu:oracle} as a finer version of the
classical heavy-tailed noise condition, the latter omits the existence
of $\sigma_{\sma}$. Therefore, Assumption \ref{assu:oracle} describes
the behavior of noise more precisely. Such refinement was only introduced
to the classical mean estimation problem \citep{pmlr-v178-cherapanamjeri22a}
as far as we know, and hence is new to the optimization literature.
In Appendix \ref{sec:d-eff}, we provide more discussions on how large
$d_{\eff}$ can be across different settings.

\section{Clipped Stochastic Gradient Descent\label{sec:algorithm}}

\begin{algorithm}[H]
\caption{\label{alg:clipped-SGD}Clipped Stochastic Gradient Descent (Clipped
SGD)}

\textbf{Input:} initial point $\bx_{1}\in\X$, stepsize $\eta_{t}>0$,
clipping threshold $\cm_{t}>0$

\textbf{for} $t=1$ \textbf{to} $T$ \textbf{do}

$\quad$$\bg_{t}^{\rc}=\clip_{\cm_{t}}(\bg_{t})$ where $\bg_{t}=\bg(\bx_{t},\xi_{t})$
and $\xi_{t}\sim\dis$ is sampled independently from the history

$\quad$$\bx_{t+1}=\argmin_{\bx\in\X}r(\bx)+\left\langle \bg_{t}^{\rc},\bx\right\rangle +\frac{\left\Vert \bx-\bx_{t}\right\Vert ^{2}}{2\eta_{t}}$

\textbf{end for}
\end{algorithm}

We present the main method studied in this work, Clipped Stochastic
Gradient Descent (Clipped SGD), in Algorithm \ref{alg:clipped-SGD}.
Strictly speaking, the algorithm should be called Proximal Clipped
SGD as it contains a proximal update step. However, we drop the word
``Proximal'' for simplicity. We remark that Clipped SGD with a proximal
step has not been fully studied yet and is different from the Prox-Clipped-SGD-Shift
method introduced in \citet{pmlr-v235-gorbunov24a}, the only work
considering composite optimization under heavy-tailed noise that we
are aware of.

In comparison to the classical Proximal SGD, Algorithm \ref{alg:clipped-SGD}
only contains an extra clipping operation on the stochastic gradient.
As pointed out in prior works (e.g., \citet{pmlr-v202-sadiev23a}),
the additional clipping step is the key to proving the high-probability
convergence.

\section{Refined High-Probability Rates\label{sec:hp-rates}}

In this section, we will establish refined high-probability convergence
results for Clipped SGD. To simplify the notation in the upcoming
theorems, we denote by $D\defeq\left\Vert \bx_{\star}-\bx_{1}\right\Vert $
the distance between the optimal solution and the initial point. Moreover,
given $\delta\in\left(0,1\right]$, we introduce the quantity
\begin{equation}
\cm_{\star}\defeq\left(\min\left\{ \frac{\sigma_{\sma}\sigma_{\lar}^{\p-1}}{\ln\frac{3}{\delta}},\frac{\sigma_{\sma}^{2}}{\sigma_{\lar}^{2-\p}\1\left[\p<2\right]}\right\} \right)^{\frac{1}{\p}},\label{eq:main-hp-tau-star}
\end{equation}
which is an important value used in the clipping threshold. Recall
that $d_{\eff}=\sigma_{\lar}^{2}/\sigma_{\sma}^{2}$ , then $\cm_{\star}$
can be equivalently written into
\begin{eqnarray}
\cm_{\star}=\sigma_{\lar}/\varphi_{\star}^{1/\p} & \text{where} & \varphi_{\star}\defeq\max\left\{ \sqrt{d_{\eff}}\ln\frac{3}{\delta},d_{\eff}\1\left[\p<2\right]\right\} .\label{eq:main-hp-varphi-star}
\end{eqnarray}

\subsection{General Convex Case}

We start from the general convex case (i.e., $\mu=0$ in Assumption
\ref{assu:obj}). $\bar{\bx}_{T+1}^{\cvx}\defeq\frac{1}{T}\sum_{t=1}^{T}\bx_{t+1}$
in the following denotes the average iterate after $T$ steps. To
clarify, $T$ is assumed to be known in advance in this subsection.
Though Clipped SGD can provably handle an unknown time horizon $T$,
it is well-known to incur extra $\poly(\ln T)$ factors \citep{liu2023stochastic}.
To deal with this issue, we propose a variant of Clipped SGD named
Stabilized Clipped SGD in Appendix \ref{sec:stabilized}, which incorporates
the stabilization trick introduced by \citet{JMLR:v23:21-1027}. As
an example, Theorem \ref{thm:cvx-hp-dep-t} in Appendix \ref{sec:theorems}
shows that Stabilized Clipped SGD converges at an almost identical
rate to Theorem \ref{thm:main-cvx-hp-dep-T} below, but in an anytime
fashion without incurring any $\poly(\ln T)$ factor.
\begin{thm}
\label{thm:main-cvx-hp-dep-T}Under Assumptions \ref{assu:minimizer},
\ref{assu:obj} (with $\mu=0$), \ref{assu:lip} and \ref{assu:oracle},
for any $T\in\N$ and $\delta\in\left(0,1\right]$, setting $\eta_{t}=\eta_{\star},\cm_{t}=\max\left\{ 2G,\cm_{\star}T^{\frac{1}{\p}}\right\} ,\forall t\in\left[T\right]$
where $\eta_{\star}$ is a properly picked stepsize (explicated in
Theorem \ref{thm:cvx-hp-dep-T}), then Clipped SGD (Algorithm \ref{alg:clipped-SGD})
guarantees that with probability at least $1-\delta$, $F(\bar{\bx}_{T+1}^{\cvx})-F_{\star}$
converges at the rate of 
\[
\O\left(\frac{(\varphi+\ln\frac{3}{\delta})GD}{T}+\frac{GD}{\sqrt{T}}+\frac{(\sigma_{\sma}^{\frac{2}{\p}-1}\sigma_{\lar}^{2-\frac{2}{\p}}+\sigma_{\sma}^{\frac{1}{\p}}\sigma_{\lar}^{1-\frac{1}{\p}}\ln^{1-\frac{1}{\p}}\frac{3}{\delta})D}{T^{1-\frac{1}{\p}}}\right),
\]
where $\varphi\leq\varphi_{\star}$ is a constant (explicated in Theorem
\ref{thm:cvx-hp-dep-T}) and equals $\varphi_{\star}$ when $T=\Omega\left(\frac{G^{\p}}{\sigma_{\lar}^{\p}}\varphi_{\star}\right)$.
\end{thm}
To better understand Theorem \ref{thm:main-cvx-hp-dep-T}, we first
consider a special case of $\p=2$ (i.e., the classical finite variance
condition) and obtain a rate at most $\O\left(\frac{(\sqrt{d_{\eff}}+1)\ln(\frac{1}{\delta})GD}{T}+\frac{(G+\sigma_{\lar}+\sqrt{\sigma_{\sma}\sigma_{\lar}\ln(\frac{1}{\delta})})D}{\sqrt{T}}\right)$.
In comparison, the previous best high-probability bound in the finite
variance setting proved by \citet{NEURIPS2024_10bf9689} is $\O\left(C_{T}+\frac{(\sqrt{d_{\eff}}+\frac{G}{\sigma_{\sma}})\ln(\frac{\ln T}{\delta})GD}{T}+\frac{(G+\sigma_{\lar}+\sqrt{\sigma_{\sma}(\sigma_{\lar}+G)\ln(\frac{\ln T}{\delta})})D}{\sqrt{T}}\right)$,
but under an extra requirement $T\geq\Omega(\ln(\ln d))$, where $C_{T}$
is a term in the order of $\O(T^{-\frac{3}{2}})$ but will blow up
to $+\infty$ when the variance approaches $0$. As one can see, even
in this special case, our result immediately improves upon \citet{NEURIPS2024_10bf9689}
in the following three folds: 1. Our theory works for any time horizon
$T\in\N$. 2. Our bound is strictly better than theirs by shaving
off many redundant terms. Especially, the dependence on $\delta$
is only $\ln(1/\delta)$ in contrast to their $\ln((\ln T)/\delta)$.
3. Our rate will not blow up when $\sigma_{\lar}\to0$ (equivalently,
$\sigma_{\sma}\to0$) and instead recover the standard $\O(GD/\sqrt{T})$
result for deterministic nonsmooth convex optimization \citep{nesterov2018lectures}.

Next, the prior best result for $\p\in\left(1,2\right]$ is $\O\left(\frac{GD\ln\frac{1}{\delta}}{\sqrt{T}}+\frac{\sigma_{\lar}D\ln\frac{1}{\delta}}{T^{1-\frac{1}{\p}}}\right)$
\citep{liu2023stochastic}, whose dominant term is $\O(\sigma_{\lar}D\ln(1/\delta)T^{\frac{1}{\p}-1})$
as $T$ becomes larger. In comparison, using $d_{\eff}=\sigma_{\lar}^{2}/\sigma_{\sma}^{2}$,
the lower-order term in Theorem \ref{thm:main-cvx-hp-dep-T} can be
written as $\O(\sigma_{\lar}D(d_{\eff}^{\frac{1}{2}-\frac{1}{\p}}+d_{\eff}^{-\frac{1}{2\p}}\ln^{1-\frac{1}{\p}}(1/\delta))T^{\frac{1}{\p}-1})$.
Therefore, Theorem \ref{thm:main-cvx-hp-dep-T} improves upon \citet{liu2023stochastic}
for large $T$ by a factor of 
\begin{equation}
\rho\defeq\Theta\left(\frac{d_{\eff}^{\frac{1}{2}-\frac{1}{\p}}+d_{\eff}^{-\frac{1}{2\p}}\ln^{1-\frac{1}{\p}}\frac{1}{\delta}}{\ln\frac{1}{\delta}}\right)=\Theta\left(\frac{1}{d_{\eff}^{\frac{2-\p}{2\p}}\ln\frac{1}{\delta}}+\frac{1}{d_{\eff}^{\frac{1}{2\p}}\ln^{\frac{1}{\p}}\frac{1}{\delta}}\right).\label{eq:main-rho-star}
\end{equation}

\begin{rem}
Especially, when $d_{\eff}=\Omega(d)$, $\rho$ could be in the order
of $\Theta(\poly(1/d,1/\ln(1/\delta)))$. We provide an example in
Appendix \ref{sec:d-eff} showing that $d_{\eff}=\Omega(d)$ is attainable.
\end{rem}
For general $T\in\N$, note that $\O(GD\ln(1/\delta)/T+\text{last two terms})$
in Theorem \ref{thm:main-cvx-hp-dep-T} are always smaller than the
rate of \citet{liu2023stochastic} due to $\sigma_{\sma}\leq\sigma_{\lar}$.
Therefore, we only need to pay attention to the redundant term $\O(\varphi GD/T)$.
Observe that a critical time could be $T_{\star}=\Theta(\varphi_{\star}^{2})=\Theta(d_{\eff}\ln^{2}(1/\delta)+d_{\eff}^{2}\1\left[\p<2\right])$\footnote{\label{fn:critical}Actually, any $T_{\star}$ that makes $\O(\varphi GD/T)$
in Theorem \ref{thm:main-cvx-hp-dep-T} smaller than the sum of the
terms left is enough. Hence, it is possible to find a smaller critical
time. We keep this one here due to its clear expression.}. Once $T\geq T_{\star}$, we can ignore $\O(\varphi GD/T)$ as it
is at most $\O(GD/\sqrt{T})$ now. It is currently unknown whether
the term $\O(\varphi GD/T)$ is inevitable or can be removed to obtain
a better bound than \citet{liu2023stochastic} for any $T\in\N$.
We remark that similar additional terms also appear in the refined
rate for $\p=2$ by \citet{NEURIPS2024_10bf9689} as discussed before.

\subsection{Strongly Convex Case}

We now move to the strongly convex case (i.e., $\mu>0$ in Assumption
\ref{assu:obj}). $\bar{\bx}_{T+1}^{\str}\defeq\frac{\sum_{t=1}^{T}(t+4)(t+5)\bx_{t+1}}{\sum_{t=1}^{T}(t+4)(t+5)}$
in the following denotes the weighted average iterate after $T$ steps.
Unlike the general convex case, we do not need to know $T$ in advance
to remove the extra $\poly(\ln T)$ factor.
\begin{thm}
\label{thm:main-str-hp-dep}Under Assumptions \ref{assu:minimizer},
\ref{assu:obj} (with $\mu>0$), \ref{assu:lip} and \ref{assu:oracle},
for any $T\in\N$ and $\delta\in\left(0,1\right]$, setting $\eta_{t}=\frac{6}{\mu t},\cm_{t}=\max\left\{ 2G,\cm_{\star}t^{\frac{1}{\p}}\right\} ,\forall t\in\left[T\right]$,
then Clipped SGD (Algorithm \ref{alg:clipped-SGD}) guarantees that
with probability at least $1-\delta$, both $F(\bar{\bx}_{T+1}^{\str})-F_{\star}$
and $\mu\left\Vert \bx_{T+1}-\bx_{\star}\right\Vert ^{2}$ converge
at the rate of
\[
\O\left(\frac{\mu D^{2}}{T^{3}}+\frac{(\varphi^{2}+\ln^{2}\frac{3}{\delta})G^{2}}{\mu T^{2}}+\frac{G^{2}}{\mu T}+\frac{\sigma_{\sma}^{\frac{4}{\p}-2}\sigma_{\lar}^{4-\frac{4}{\p}}+\sigma_{\sma}^{\frac{2}{\p}}\sigma_{\lar}^{2-\frac{2}{\p}}\ln^{2-\frac{2}{\p}}\frac{3}{\delta}}{\mu T^{2-\frac{2}{\p}}}\right),
\]
where $\varphi\le\varphi_{\star}$ is the same constant as in Theorem
\ref{thm:main-cvx-hp-dep-T} and equals $\varphi_{\star}$ when $T=\Omega\left(\frac{G^{\p}}{\sigma_{\lar}^{\p}}\varphi_{\star}\right)$.
\end{thm}
\begin{rem}
\label{rem:reduction}The problem studied in prior works (e.g., \citet{liu2023stochastic,gorbunov2024high})
considers strongly convex and Lipschitz $f$ with $r=0$, which seems
different from our assumption of strongly convex $r$. However, a
simple reduction can convert their instance to fit our setting. Moreover,
the first term $\O(\mu D^{2}/T^{3})$ in Theorem \ref{thm:main-str-hp-dep}
can also be omitted in that case (as we will do so in the following
discussion). We refer the interested reader to Appendix \ref{sec:reduction}
for the reduction and why the term $\O(\mu D^{2}/T^{3})$ can be ignored.
\end{rem}
To save space, we only compare with the rate $\O\left(\frac{G^{2}\ln^{2}\frac{1}{\delta}}{\mu T}+\frac{(\sigma_{\lar}^{2}+\sigma_{\lar}^{\p}G^{2-\p})\ln^{2}\frac{1}{\delta}}{\mu T^{2-\frac{2}{\p}}}\right)$
\citep{liu2023stochastic} for general $\p\in\left(1,2\right]$. For
the special case $\p=2$, the rate of \citet{liu2023stochastic} is
almost identical to the bound of \citet{gorbunov2024high}; moreover,
as far as we know, no improved result like \citet{NEURIPS2024_10bf9689}
has been obtained to give a better bound for the term containing $\poly(\ln(1/\delta))$.
Similar to the discussion after Theorem \ref{thm:main-cvx-hp-dep-T},
one can find that for large $T$, the improvement over \citet{liu2023stochastic}
is at least by a factor of
\[
\rho^{2}\overset{(\ref{eq:main-rho-star})}{=}\Theta\left(\frac{1}{d_{\eff}^{\frac{2-\p}{\p}}\ln^{2}\frac{1}{\delta}}+\frac{1}{d_{\eff}^{\frac{1}{\p}}\ln^{\frac{2}{\p}}\frac{1}{\delta}}\right)=\Theta\left(\poly\left(\frac{1}{d_{\eff}},\frac{1}{\ln\frac{1}{\delta}}\right)\right).
\]
For general $T\in\N$, every term in Theorem \ref{thm:main-str-hp-dep}
is still better except for $\O(\varphi^{2}G^{2}/(\mu T^{2}))$. However,
this extra term has no effect once $T\geq T_{\star}=\Theta(\varphi_{\star}^{2})=\Theta(d_{\eff}\ln^{2}(1/\delta)+d_{\eff}^{2}\1\left[\p<2\right])$,
the same critical time for Theorem \ref{thm:main-cvx-hp-dep-T} (a
similar discussion to Footnote \ref{fn:critical} also applies here),
since it is at most $\O(G^{2}/(\mu T))$ now, being dominated by other
terms. Same as before, it is unclear whether this redundant term $\O(\varphi^{2}G^{2}/(\mu T^{2}))$
can be shaved off to conclude a faster rate for any $T\in\N$ or not.
We leave it as future work and look forward to it being addressed.

\section{Proof Sketch and New Insights\label{sec:sketch}}

In this section, we sketch the proof of Theorem \ref{thm:main-cvx-hp-dep-T}
as an example and introduce our new insights in the analysis. To
start with, given $T\in\N$ and suppose $\eta_{t}=\eta,\cm_{t}=\cm,\forall t\in\left[T\right]$
for simplicity, we have the following inequality for Clipped SGD (see
Lemma \ref{lem:cvx-hp-anytime} in Appendix \ref{sec:analysis}),
which holds almost surely without any restriction on $\cm$,
\begin{align}
 & F(\bar{\bx}_{T+1}^{\cvx})-F_{\star}\leq\frac{D^{2}}{\eta T}+\frac{2\hres_{T}^{\cvx}}{T},\enskip\text{where }\hres_{T}^{\cvx}\text{ is a residual term in the order of}\nonumber \\
 & \hres_{T}^{\cvx}=\O\left(\eta\left(\underbrace{\max_{t\in\left[T\right]}\left(\sum_{s=1}^{t}\left\langle \bd_{s}^{\ru},\by_{s}\right\rangle \right)^{2}}_{\mathrm{I}}+\underbrace{\sum_{t=1}^{T}\left\Vert \bd_{t}^{\ru}\right\Vert ^{2}}_{\mathrm{II}}+\underbrace{\left(\sum_{t=1}^{T}\left\Vert \bd_{t}^{\rb}\right\Vert \right)^{2}}_{\mathrm{III}}+G^{2}T\right)\right),\label{eq:main-I}
\end{align}
in which $\bd_{t}^{\ru}\defeq\bg_{t}^{\rc}-\E_{t-1}\left[\bg_{t}^{\rc}\right]$
and $\bd_{t}^{\rb}\defeq\E_{t-1}\left[\bg_{t}^{\rc}\right]-\nabla f(\bx_{t})$
respectively denote the unbiased and biased part in the clipping error,
where $\E_{t}\left[\cdot\right]\defeq\E\left[\cdot\mid\F_{t}\right]$
for $\F_{t}\defeq\sigma(\xi_{1},\mydots,\xi_{t})$ being the natural
filtration, and $\by_{t}$ is some predictable vector (i.e., $\by_{t}\in\F_{t-1}$)
satisfying $\left\Vert \by_{t}\right\Vert \leq1$ almost surely.

The term $\eta G^{2}T$ in $\hres_{T}^{\cvx}$ is standard. Hence,
the left task is to bound terms $\mathrm{I}$, $\mathrm{II}$ and,
$\mathrm{III}$ in high probability. In particular, for $\mathrm{I}$
and $\mathrm{III}$, we will move beyond the existing approach via
a refined analysis. To formalize the difference, we borrow the following
bounds for clipping error commonly used in the literature (see, e.g.,
\citet{pmlr-v202-sadiev23a,liu2023stochastic,NEURIPS2023_4c454d34}):
\begin{eqnarray}
\left\Vert \bd_{t}^{\ru}\right\Vert \leq\O(\cm), & \E_{t-1}\left[\left\Vert \bd_{t}^{\ru}\right\Vert ^{2}\right]\overset{\text{if }\cm\geq2G}{\leq}\O(\sigma_{\lar}^{\p}\cm^{2-\p}), & \left\Vert \bd_{t}^{\rb}\right\Vert \overset{\text{if }\cm\geq2G}{\leq}\O(\sigma_{\lar}^{\p}\cm^{1-\p}).\label{eq:main-clip-old}
\end{eqnarray}

\textbf{Term $\mathrm{I}$.} Note that $X_{t}\defeq\left\langle \bd_{t}^{\ru},\by_{t}\right\rangle $
is a martingale difference sequence (MDS), then Freedman's inequality
(Lemma \ref{lem:Freedman} in Appendix \ref{sec:analysis}) implies
with probability at least $1-\delta$, $\sqrt{\mathrm{I}}\leq\O(\max_{t\in\left[T\right]}\left|X_{t}\right|\ln(1/\delta)+\sqrt{\sum_{t=1}^{T}\E_{t-1}\left[X_{t}^{2}\right]\ln(1/\delta)})$
(this inequality is for illustration, not entirely rigorous in math).
To the best of our knowledge, prior works studying Clipped SGD under
heavy-tailed noise always bound similar terms in the following manner
\begin{eqnarray*}
\left|X_{t}\right|\overset{\left\Vert \by_{t}\right\Vert \leq1}{\leq}\left\Vert \bd_{t}^{\ru}\right\Vert \overset{(\ref{eq:main-clip-old})}{\leq}\O(\cm) & \text{and} & \E_{t-1}\left[X_{t}^{2}\right]\overset{\left\Vert \by_{t}\right\Vert \leq1}{\leq}\E_{t-1}\left[\left\Vert \bd_{t}^{\ru}\right\Vert ^{2}\right]\overset{(\ref{eq:main-clip-old})}{\leq}\O(\sigma_{\lar}^{\p}\cm^{2-\p}).
\end{eqnarray*}
However, a critical observation is that the above-described widely
adopted way is very likely to be loose, as the conditional variance
can be better controlled by
\[
\E_{t-1}\left[X_{t}^{2}\right]=\by_{t}^{\top}\E_{t-1}\left[\bd_{t}^{\ru}(\bd_{t}^{\ru})^{\top}\right]\by_{t}\overset{\left\Vert \by_{t}\right\Vert \leq1}{\leq}\left\Vert \E_{t-1}\left[\bd_{t}^{\ru}(\bd_{t}^{\ru})^{\top}\right]\right\Vert .
\]
Note that $\left\Vert \E_{t-1}\left[\bd_{t}^{\ru}(\bd_{t}^{\ru})^{\top}\right]\right\Vert $
is at most $\E_{t-1}\left[\left\Vert \bd_{t}^{\ru}\right\Vert ^{2}\right]$
but could be much smaller. Inspired by this, we develop a new bound
for $\left\Vert \E_{t-1}\left[\bd_{t}^{\ru}(\bd_{t}^{\ru})^{\top}\right]\right\Vert $
in Lemma \ref{lem:main-clip-ineq}. Consequently, this better utilization
of Freedman's inequality concludes a tighter high-probability bound
for term $\mathrm{I}$.

Actually, this simple but effective idea has been implicitly used
in \citet{NEURIPS2024_10bf9689} when $\p=2$. However, their proof
eventually falls complex due to an argument they call the iterative
refinement strategy, which not only imposes extra undesired factors
like $\ln((\ln T)/\delta)$ in their final bound but also leads to
an additional requirement $T\geq\Omega(\ln(\ln d))$ in their theory.
Our analysis indicates that such a complication is unnecessary, instead,
one can keep it simple.

\textbf{Term $\mathrm{II}$.} For this term, we follow the same way
employed in many previous works (e.g., \citet{NEURIPS2021_26901deb,NEURIPS2022_349956de}),
i.e., let $X_{t}\defeq\left\Vert \bd_{t}^{\ru}\right\Vert ^{2}-\E_{t-1}\left[\left\Vert \bd_{t}^{\ru}\right\Vert ^{2}\right]$
and decompose $\sum_{t=1}^{T}\left\Vert \bd_{t}^{\ru}\right\Vert ^{2}\overset{(\ref{eq:main-clip-old})}{\leq}\O(\sum_{t=1}^{T}X_{t}+\sigma_{\lar}^{\p}\cm^{2-\p}T)$
then use Freedman's inequality to bound $\sum_{t=1}^{T}X_{t}$.
\begin{rem}
Although the above analysis follows the literature, we still obtain
a refined inequality for $\E_{t-1}\left[\left\Vert \bd_{t}^{\ru}\right\Vert ^{2}\right]$
in Lemma \ref{lem:main-clip-ineq}, in the sense of dropping the condition
$\cm\geq2G$ required in (\ref{eq:main-clip-old}).
\end{rem}
\textbf{Term $\mathrm{III}$.} Estimating the clipping error $\left\Vert \bd_{t}^{\rb}\right\Vert $
is another key ingredient when analyzing Clipped SGD. As far as we
know, all existing works apply the inequality $\left\Vert \bd_{t}^{\rb}\right\Vert \leq\O(\sigma_{\lar}^{\p}\cm^{1-\p})$
in (\ref{eq:main-clip-old}). However, we show that this important
inequality still has room for improvement. In other words, it is in
fact not tight, as revealed by our finer bounds in Lemma \ref{lem:main-clip-ineq}.
Thus, our result is more refined.

From the above discussion, in addition to better utilization of Freedman's
inequality, the improvement heavily relies on finer bounds for clipping
error under heavy-tailed noise, which we give in the following Lemma
\ref{lem:main-clip-ineq}.
\begin{lem}
\label{lem:main-clip-ineq}Under Assumptions \ref{assu:lip} and \ref{assu:oracle},
and assuming $\cm_{t}=\cm>0$, there are:
\begin{align*}
 & \left\Vert \bd_{t}^{\ru}\right\Vert \leq\O(\cm),\enskip\left\Vert \E_{t-1}\left[\bd_{t}^{\ru}(\bd_{t}^{\ru})^{\top}\right]\right\Vert \overset{\text{if }\cm\geq2G}{\leq}\O(\sigma_{\sma}^{\p}\cm^{2-\p}+\sigma_{\lar}^{\p}G^{2}\cm^{-\p}),\\
 & \E_{t-1}\left[\left\Vert \bd_{t}^{\ru}\right\Vert ^{2}\right]\leq\O(\sigma_{\lar}^{\p}\cm^{2-\p}),\enskip\left\Vert \bd_{t}^{\rb}\right\Vert \overset{\text{if }\cm\geq2G}{\leq}\O(\sigma_{\sma}\sigma_{\lar}^{\p-1}\cm^{1-\p}+\sigma_{\lar}^{\p}G\cm^{-\p}).
\end{align*}
\end{lem}
\begin{rem}
We highlight that Theorem \ref{thm:clip} in Appendix \ref{sec:clip}
provides a further generalization of clipping error bounds under heavy-tailed
noise not limited to clipped gradient methods (even without the requirement
in the form of $\cm\geq2G$), which could be potentially useful for
future research.
\end{rem}
Except for the standard bound $\left\Vert \bd_{t}^{\ru}\right\Vert \leq\O(\cm)$,
the other three inequalities in Lemma \ref{lem:main-clip-ineq} are
either new or improve over the existing results. 1. The bound on $\left\Vert \E_{t-1}\left[\bd_{t}^{\ru}(\bd_{t}^{\ru})^{\top}\right]\right\Vert $
is new in the heavy-tailed setting. Importantly, observe that $\O(\sigma_{\sma}^{\p}\cm^{2-\p}+\sigma_{\lar}^{\p}G^{2}\cm^{-\p})\leq\O(\sigma_{\lar}^{\p}\cm^{2-\p})$
due to $\sigma_{\sma}\leq\sigma_{\lar}$ and $\cm\geq2G$, which thereby
leads to a tighter high-probability bound for term $\mathrm{I}$ in
combination with our better application of Freedman's inequality (see
the paragraph before starting with \textbf{Term $\mathrm{I}$.}).
2. For term $\E_{t-1}\left[\left\Vert \bd_{t}^{\ru}\right\Vert ^{2}\right]$,
in contrast to (\ref{eq:main-clip-old}), Lemma \ref{lem:main-clip-ineq}
removes the condition $\cm\geq2G$. Moreover, the hidden constant
in our lemma is actually slightly better. 3. As mentioned above (see
the paragraph before starting with \textbf{Term $\mathrm{III}$.}),
the bound of $\left\Vert \bd_{t}^{\rb}\right\Vert $ is another key
to obtaining a refined result. Precisely, we note that the new bound
$\O(\sigma_{\sma}\sigma_{\lar}^{\p-1}\cm^{1-\p}+\sigma_{\lar}^{\p}G\cm^{-\p})$
improves upon $\O(\sigma_{\lar}^{\p}\cm^{1-\p})$ in (\ref{eq:main-clip-old})
because of $\sigma_{\sma}\leq\sigma_{\lar}$ and $\cm\geq2G$. Therefore,
Lemma \ref{lem:main-clip-ineq} guarantees a better control for term
$\mathrm{III}$.

Combining all the new insights mentioned, we can finally prove Theorem
\ref{thm:main-cvx-hp-dep-T}. As one can imagine, the analysis sketched
above is essentially more refined than previous works, since we apply
tighter bounds for the two central parts in analyzing Clipped SGD,
i.e., concentration inequalities and estimation of clipping error.
To confirm this claim, we discuss how to recover the existing rate
through our finer analysis, the details of which are deferred to Appendix
\ref{sec:theorems}.

Lastly, we mention that Theorem \ref{thm:main-str-hp-dep} for strongly
convex problems is also inspired by the above two new insights. The
full proofs of both Theorems \ref{thm:main-cvx-hp-dep-T} and \ref{thm:main-str-hp-dep}
can be found in Appendix \ref{sec:theorems}.

\section{Extension to Faster In-Expectation Convergence\label{sec:ex-rates}}

In this section, we show that Lemma \ref{lem:main-clip-ineq} presented
before can also lead to faster in-expectation convergence for Clipped
SGD, further highlighting the value of refined clipping error bounds.
Proofs of both theorems given below can be found in Appendix \ref{sec:theorems}.

This time, we consider a new quantity $\widetilde{\cm}_{\star}\defeq\sigma_{\sma}^{\frac{2}{\p}}/(\sigma_{\lar}^{\frac{2}{\p}-1}\1\left[\p<2\right])$
for the clipping threshold. Recall that $d_{\eff}=\sigma_{\lar}^{2}/\sigma_{\sma}^{2}$
, then $\widetilde{\cm}_{\star}$ can be equivalently written into
\begin{eqnarray}
\widetilde{\cm}_{\star}=\sigma_{\lar}/\widetilde{\varphi}_{\star}^{1/\p} & \text{where} & \widetilde{\varphi}_{\star}\defeq d_{\eff}\1\left[\p<2\right].\label{eq:main-ex-varphi-star}
\end{eqnarray}

\begin{rem}
When $\p=2$, $\widetilde{\varphi}_{\star}=0\Rightarrow\widetilde{\cm}_{\star}=+\infty$,
i.e., no clipping operation is required. This matches the well-known
fact that SGD provably converges in expectation under the finite variance
condition.
\end{rem}

\subsection{General Convex Case}
\begin{thm}
\label{thm:main-cvx-ex-dep-T}Under Assumptions \ref{assu:minimizer},
\ref{assu:obj} (with $\mu=0$), \ref{assu:lip} and \ref{assu:oracle},
for any $T\in\N$, setting $\eta_{t}=\eta_{\star},\cm_{t}=\max\left\{ 2G,\widetilde{\cm}_{\star}T^{\frac{1}{\p}}\right\} ,\forall t\in\left[T\right]$
where $\eta_{\star}$ is a properly picked stepsize (explicated in
Theorem \ref{thm:cvx-ex-dep-T}), then Clipped SGD (Algorithm \ref{alg:clipped-SGD})
guarantees that $\E\left[F(\bar{\bx}_{T+1}^{\cvx})-F_{\star}\right]$
converges at the rate of 
\[
\O\left(\frac{\widetilde{\varphi}GD}{T}+\frac{GD}{\sqrt{T}}+\frac{\sigma_{\sma}^{\frac{2}{\p}-1}\sigma_{\lar}^{2-\frac{2}{\p}}D}{T^{1-\frac{1}{\p}}}\right),
\]
where $\widetilde{\varphi}\leq\widetilde{\varphi}_{\star}$ is a constant
(explicated in Theorem \ref{thm:cvx-ex-dep-T}) and equals $\widetilde{\varphi}_{\star}$
when $T=\Omega\left(\frac{G^{\p}}{\sigma_{\lar}^{\p}}\widetilde{\varphi}_{\star}\right)$.
\end{thm}
Theorem \ref{thm:main-cvx-ex-dep-T} gives a better lower-order term
$\O(\sigma_{\lar}d_{\eff}^{\frac{1}{2}-\frac{1}{\p}}DT^{\frac{1}{\p}-1})$
(recall $d_{\eff}=\sigma_{\lar}^{2}/\sigma_{\sma}^{2}$) than the
existing lower bound $\Omega(\sigma_{\lar}DT^{\frac{1}{\p}-1})$ \citep{nemirovskij1983problem,pmlr-v178-vural22a}
by a factor of $\Theta(1/d_{\eff}^{\frac{2-\p}{2\p}})$, a strict
improvement being polynomial in $1/d_{\eff}$, if $\p\in\left(1,2\right)$.
For the case of an unknown $T$, the interested reader could refer
to Theorem \ref{thm:cvx-ex-dep-t} in Appendix \ref{sec:theorems}.

\subsection{Strongly Convex Case}
\begin{thm}
\label{thm:main-str-ex-dep}Under Assumptions \ref{assu:minimizer},
\ref{assu:obj} (with $\mu>0$), \ref{assu:lip} and \ref{assu:oracle},
for any $T\in\N$, setting $\eta_{t}=\frac{6}{\mu t},\cm_{t}=\max\left\{ 2G,\widetilde{\cm}_{\star}t^{\frac{1}{\p}}\right\} ,\forall t\in\left[T\right]$,
then Clipped SGD (Algorithm \ref{alg:clipped-SGD}) guarantees that
both $\E\left[F(\bar{\bx}_{T+1}^{\str})-F_{\star}\right]$ and $\mu\E\left[\left\Vert \bx_{T+1}-\bx_{\star}\right\Vert ^{2}\right]$
converge at the rate of
\[
\O\left(\frac{\mu D^{2}}{T^{3}}+\frac{\widetilde{\varphi}^{2}G^{2}}{\mu T^{2}}+\frac{G^{2}}{\mu T}+\frac{\sigma_{\sma}^{\frac{4}{\p}-2}\sigma_{\lar}^{4-\frac{4}{\p}}}{\mu T^{2-\frac{2}{\p}}}\right),
\]
where $\widetilde{\varphi}\leq\widetilde{\varphi}_{\star}$ is the
same constant as in Theorem \ref{thm:main-cvx-ex-dep-T} and equals
$\widetilde{\varphi}_{\star}$ when $T=\Omega\left(\frac{G^{\p}}{\sigma_{\lar}^{\p}}\widetilde{\varphi}_{\star}\right)$,
\end{thm}
Theorem \ref{thm:main-str-ex-dep} provides a faster rate $\O(\sigma_{\lar}^{2}d_{\eff}^{1-\frac{2}{\p}}T^{\frac{2}{\p}-2})$
than the known lower bound $\Omega(\sigma_{\lar}^{2}T^{\frac{2}{\p}-2})$
\citep{NEURIPS2020_b05b57f6} by a factor of $\Theta(1/d_{\eff}^{\frac{2-\p}{\p}})$.
This is again a strict improvement once $\p<2$, and could be in the
order of $\poly(1/d)$ if $d_{\eff}=\Omega(d)$.

\section{Lower Bounds\label{sec:informal-lb}}

To complement the study, we provide new high-probability and in-expectation
lower bounds for both $\mu=0$ and $\mu>0$. We employ information-theoretic
methods to establish these new lower bounds, following the existing
literature \citep{5394945,6142067,NIPS2013_2812e5cf,pmlr-v178-vural22a,pmlr-v247-carmon24a,ma2024high}.
For complete proofs, the interested reader could refer to Appendix
\ref{sec:formal-lb}.
\begin{rem}
One may wonder why our upper bounds can beat the existing lower bounds,
and also where the difference between our new lower bounds and the
prior ones lies. The key is our fine-grained Assumption \ref{assu:oracle}.
Roughly speaking, the existing lower bounds are proved for the following
oracle class (we slightly abuse the notation by still using $\bg$
to denote the stochastic gradient oracle),
\[
\G_{\sigma_{\lar}}^{\p}=\left\{ \bg:\R^{d}\times\mathfrak{f}\to\R^{d}:\substack{\E\left[\bg(\bx,f)\mid\bx,f\right]=\nabla f(\bx)\in\partial f(\bx)\\
\E\left[\left\Vert \bg(\bx,f)-\nabla f(\bx)\right\Vert ^{\p}\mid\bx,f\right]\leq\sigma_{\lar}^{\p}
}
,\forall\bx\in\R^{d},f\in\mathfrak{f}\right\} ,
\]
where $\p\in\left(1,2\right]$ and $\sigma_{\lar}\geq0$ are two parameters
and $\mathfrak{f}$ is the function class that we are interested in
(e.g., the family of $G$-Lipschitz convex functions). In contrast,
the oracle class we study is parameterized by one more parameter $\sigma_{\sma}\in\left[\sigma_{\lar}/\sqrt{\pi d/2},\sigma_{\lar}\right]$
as follows,
\[
\G_{\sigma_{\sma},\sigma_{\lar}}^{\p}\defeq\left\{ \bg:\R^{d}\times\mathfrak{f}\to\R^{d}:\substack{\E\left[\bg(\bx,f)\mid\bx,f\right]=\nabla f(\bx)\in\partial f(\bx)\\
\E\left[\left|\left\langle \be,\bg(\bx,f)-\nabla f(\bx)\right\rangle \right|^{\p}\mid\bx,f\right]\leq\sigma_{\sma}^{\p},\forall\be\in\S^{d-1}\\
\E\left[\left\Vert \bg(\bx,f)-\nabla f(\bx)\right\Vert ^{\p}\mid\bx,f\right]\leq\sigma_{\lar}^{\p}
}
,\forall\bx\in\R^{d},f\in\mathfrak{f}\right\} .
\]
Note that there is $\mathfrak{\G}_{\sigma_{\sma},\sigma_{\lar}}^{\p}\subseteq\mathfrak{\G}_{\sigma_{\lar}}^{\p}$,
implying the lower bound proved for $\mathfrak{\G}_{\sigma_{\lar}}^{\p}$
could be loose for $\mathfrak{\G}_{\sigma_{\sma},\sigma_{\lar}}^{\p}$.
Therefore, our upper bounds can surpass the existing lower bounds,
and our new lower bounds are established for the fine-grained oracle
class $\mathfrak{\G}_{\sigma_{\sma},\sigma_{\lar}}^{\p}$.
\end{rem}

\subsection{High-Probability Lower Bounds}
\begin{thm}[Informal version of Theorem \ref{thm:cvx-hp-lb}]
\label{thm:main-cvx-hp-lb}Under Assumptions \ref{assu:minimizer},
\ref{assu:obj} (with $\mu=0$), \ref{assu:lip} and \ref{assu:oracle},
assuming $d\geq d_{\eff}\geq1$ and $\delta\in\left(0,\frac{1}{10}\right)$,
any algorithm converges at least at the rate of $\Omega\left(\frac{(\sigma_{\sma}^{\frac{2}{\p}-1}\sigma_{\lar}^{2-\frac{2}{\p}}+\sigma_{\sma}\ln^{1-\frac{1}{\p}}\frac{1}{\delta})D}{T^{1-\frac{1}{\p}}}\right)$
with probability at least $\delta$ when $T$ is large enough.
\end{thm}
\begin{thm}[Informal version of Theorem \ref{thm:str-hp-lb}]
\label{thm:main-str-hp-lb}Under Assumptions \ref{assu:minimizer},
\ref{assu:obj} (with $\mu>0$), \ref{assu:lip} and \ref{assu:oracle},
assuming $d\geq d_{\eff}\geq1$ and $\delta\in\left(0,\frac{1}{10}\right)$,
any algorithm converges at least at the rate of $\Omega\left(\frac{\sigma_{\sma}^{\frac{4}{\p}-2}\sigma_{\lar}^{4-\frac{4}{\p}}+\sigma_{\sma}^{2}\ln^{2-\frac{2}{\p}}\frac{1}{\delta}}{\mu T^{2-\frac{2}{\p}}}\right)$
with probability at least $\delta$ when $T$ is large enough.
\end{thm}
Compared to our upper bounds in high probability, i.e., Theorems \ref{thm:main-cvx-hp-dep-T}
($\mu=0$) and \ref{thm:main-str-hp-dep} ($\mu>0$), there are still
differences between the terms that contain the $\poly(\ln(1/\delta))$
factor. Closing this important gap is an interesting task, which we
leave for future work.

\subsection{In-Expectation Lower Bounds}
\begin{thm}[Informal version of Theorem \ref{thm:cvx-ex-lb}]
\label{thm:main-cvx-ex-lb}Under Assumptions \ref{assu:minimizer},
\ref{assu:obj} (with $\mu=0$), \ref{assu:lip} and \ref{assu:oracle},
assuming $d\geq d_{\eff}\geq1$, any algorithm converges at least
at the rate of $\Omega\left(\frac{\sigma_{\sma}^{\frac{2}{\p}-1}\sigma_{\lar}^{2-\frac{2}{\p}}D}{T^{1-\frac{1}{\p}}}\right)$
in expectation when $T$ is large enough.
\end{thm}
\begin{thm}[Informal version of Theorem \ref{thm:str-ex-lb}]
\label{thm:main-str-ex-lb}Under Assumptions \ref{assu:minimizer},
\ref{assu:obj} (with $\mu>0$), \ref{assu:lip} and \ref{assu:oracle},
assuming $d\geq d_{\eff}\geq1$, any algorithm converges at least
at the rate of $\Omega\left(\frac{\sigma_{\sma}^{\frac{4}{\p}-2}\sigma_{\lar}^{4-\frac{4}{\p}}}{\mu T^{2-\frac{2}{\p}}}\right)$
in expectation when $T$ is large enough.
\end{thm}
For in-expectation convergence, the above lower bounds match our new
upper bounds, i.e., Theorems \ref{thm:main-cvx-ex-dep-T} ($\mu=0$)
and \ref{thm:main-str-ex-dep} ($\mu>0$), indicating the optimality
of our refined analysis for convergence in expectation.

\section{Conclusion and Future Work\label{sec:conclusion}}

In this work, we provide a refined analysis of Clipped SGD and obtain
faster high-probability rates than the previously best-known bounds.
The improvement is achieved by better utilization of Freedman's inequality
and finer bounds for clipping error under heavy-tailed noise. Moreover,
we extend the analysis to in-expectation convergence and show new
rates that break the existing lower bounds. To complement the study,
we establish new lower bounds for both high-probability and in-expectation
convergence. Notably, the in-expectation upper and lower bounds match
each other, indicating the optimality of our refined analysis for
convergence in expectation.

There are still some directions worth exploring in the future, which
we list below:

\textbf{The extra term.} Each of our refined rates has a higher-order
term related to $d_{\eff}$ (e.g., $\O(\varphi GD/T)$ in Theorem
\ref{thm:main-cvx-hp-dep-T} and $\O(\varphi^{2}G^{2}/(\mu T^{2}))$
in Theorem \ref{thm:main-str-hp-dep}). Although it is negligible
when $T$ is large, proving/disproving it can be removed for any $T\in\N$
could be an interesting task.

\textbf{Gaps in high-probability bounds.} As discussed in Section
\ref{sec:informal-lb}, there are still gaps between high-probability
upper and lower bounds for both convex and strongly convex cases.
Closing them is an important direction for the future.

\textbf{Other optimization problems. }We remark that our two new insights
are not limited to nonsmooth convex problems. Instead, they are general
concepts/results. Therefore, we believe that it is possible to apply
them to other optimization problems under heavy-tailed noise (e.g.,
smooth (strongly) convex/nonconvex problems) and obtain improved upper
bounds faster than existing ones.

\bibliographystyle{plainnat}
\bibliography{ref}

\begin{thebibliography}{62}
\providecommand{\natexlab}[1]{#1}
\providecommand{\url}[1]{\texttt{#1}}
\expandafter\ifx\csname urlstyle\endcsname\relax
  \providecommand{\doi}[1]{doi: #1}\else
  \providecommand{\doi}{doi: \begingroup \urlstyle{rm}\Url}\fi

\bibitem[Agarwal et~al.(2012)Agarwal, Bartlett, Ravikumar, and Wainwright]{6142067}
Alekh Agarwal, Peter~L. Bartlett, Pradeep Ravikumar, and Martin~J. Wainwright.
\newblock Information-theoretic lower bounds on the oracle complexity of stochastic convex optimization.
\newblock \emph{IEEE Transactions on Information Theory}, 58\penalty0 (5):\penalty0 3235--3249, 2012.
\newblock \doi{10.1109/TIT.2011.2182178}.

\bibitem[Armacki et~al.(2025)Armacki, Yu, Sharma, Joshi, Bajovic, Jakovetic, and Kar]{pmlr-v258-armacki25a}
Aleksandar Armacki, Shuhua Yu, Pranay Sharma, Gauri Joshi, Dragana Bajovic, Dusan Jakovetic, and Soummya Kar.
\newblock High-probability convergence bounds for online nonlinear stochastic gradient descent under heavy-tailed noise.
\newblock In Yingzhen Li, Stephan Mandt, Shipra Agrawal, and Emtiyaz Khan, editors, \emph{Proceedings of The 28th International Conference on Artificial Intelligence and Statistics}, volume 258 of \emph{Proceedings of Machine Learning Research}, pages 1774--1782. PMLR, 03--05 May 2025.
\newblock URL \url{https://proceedings.mlr.press/v258/armacki25a.html}.

\bibitem[Battash et~al.(2024)Battash, Wolf, and Lindenbaum]{pmlr-v238-battash24a}
Barak Battash, Lior Wolf, and Ofir Lindenbaum.
\newblock Revisiting the noise model of stochastic gradient descent.
\newblock In Sanjoy Dasgupta, Stephan Mandt, and Yingzhen Li, editors, \emph{Proceedings of The 27th International Conference on Artificial Intelligence and Statistics}, volume 238 of \emph{Proceedings of Machine Learning Research}, pages 4780--4788. PMLR, 02--04 May 2024.
\newblock URL \url{https://proceedings.mlr.press/v238/battash24a.html}.

\bibitem[Beck and Teboulle(2003)]{BECK2003167}
Amir Beck and Marc Teboulle.
\newblock Mirror descent and nonlinear projected subgradient methods for convex optimization.
\newblock \emph{Operations Research Letters}, 31\penalty0 (3):\penalty0 167--175, 2003.
\newblock ISSN 0167-6377.
\newblock \doi{https://doi.org/10.1016/S0167-6377(02)00231-6}.
\newblock URL \url{https://www.sciencedirect.com/science/article/pii/S0167637702002316}.

\bibitem[Bottou et~al.(2018)Bottou, Curtis, and Nocedal]{doi:10.1137/16M1080173}
L\'{e}on Bottou, Frank~E. Curtis, and Jorge Nocedal.
\newblock Optimization methods for large-scale machine learning.
\newblock \emph{SIAM Review}, 60\penalty0 (2):\penalty0 223--311, 2018.
\newblock \doi{10.1137/16M1080173}.
\newblock URL \url{https://doi.org/10.1137/16M1080173}.

\bibitem[Bretagnolle and Huber(1979)]{Bretagnolle1979}
J.~Bretagnolle and C.~Huber.
\newblock Estimation des densit{\'e}s: risque minimax.
\newblock \emph{Zeitschrift f{\"u}r Wahrscheinlichkeitstheorie und Verwandte Gebiete}, 47\penalty0 (2):\penalty0 119--137, 1979.
\newblock ISSN 1432-2064.
\newblock \doi{10.1007/BF00535278}.
\newblock URL \url{https://doi.org/10.1007/BF00535278}.

\bibitem[Burkholder(1973)]{10.1214/aop/1176997023}
D.~L. Burkholder.
\newblock {Distribution Function Inequalities for Martingales}.
\newblock \emph{The Annals of Probability}, 1\penalty0 (1):\penalty0 19 -- 42, 1973.
\newblock \doi{10.1214/aop/1176997023}.
\newblock URL \url{https://doi.org/10.1214/aop/1176997023}.

\bibitem[Carmon and Hinder(2024)]{pmlr-v247-carmon24a}
Yair Carmon and Oliver Hinder.
\newblock The price of adaptivity in stochastic convex optimization.
\newblock In Shipra Agrawal and Aaron Roth, editors, \emph{Proceedings of Thirty Seventh Conference on Learning Theory}, volume 247 of \emph{Proceedings of Machine Learning Research}, pages 772--774. PMLR, 30 Jun--03 Jul 2024.
\newblock URL \url{https://proceedings.mlr.press/v247/carmon24a.html}.

\bibitem[Cherapanamjeri et~al.(2022)Cherapanamjeri, Tripuraneni, Bartlett, and Jordan]{pmlr-v178-cherapanamjeri22a}
Yeshwanth Cherapanamjeri, Nilesh Tripuraneni, Peter Bartlett, and Michael Jordan.
\newblock Optimal mean estimation without a variance.
\newblock In Po-Ling Loh and Maxim Raginsky, editors, \emph{Proceedings of Thirty Fifth Conference on Learning Theory}, volume 178 of \emph{Proceedings of Machine Learning Research}, pages 356--357. PMLR, 02--05 Jul 2022.
\newblock URL \url{https://proceedings.mlr.press/v178/cherapanamjeri22a.html}.

\bibitem[Cutkosky and Mehta(2021)]{NEURIPS2021_26901deb}
Ashok Cutkosky and Harsh Mehta.
\newblock High-probability bounds for non-convex stochastic optimization with heavy tails.
\newblock In M.~Ranzato, A.~Beygelzimer, Y.~Dauphin, P.S. Liang, and J.~Wortman Vaughan, editors, \emph{Advances in Neural Information Processing Systems}, volume~34, pages 4883--4895. Curran Associates, Inc., 2021.
\newblock URL \url{https://proceedings.neurips.cc/paper_files/paper/2021/file/26901debb30ea03f0aa833c9de6b81e9-Paper.pdf}.

\bibitem[Das et~al.(2024)Das, Nagaraj, Pal, Suggala, and Varshney]{NEURIPS2024_10bf9689}
Aniket Das, Dheeraj Nagaraj, Soumyabrata Pal, Arun~Sai Suggala, and Prateek Varshney.
\newblock Near-optimal streaming heavy-tailed statistical estimation with clipped sgd.
\newblock In A.~Globerson, L.~Mackey, D.~Belgrave, A.~Fan, U.~Paquet, J.~Tomczak, and C.~Zhang, editors, \emph{Advances in Neural Information Processing Systems}, volume~37, pages 8834--8900. Curran Associates, Inc., 2024.
\newblock URL \url{https://proceedings.neurips.cc/paper_files/paper/2024/file/10bf96894abaf4c293b205709a98fc74-Paper-Conference.pdf}.

\bibitem[Davis and Drusvyatskiy(2020)]{pmlr-v125-davis20a}
Damek Davis and Dmitriy Drusvyatskiy.
\newblock High probability guarantees for stochastic convex optimization.
\newblock In Jacob Abernethy and Shivani Agarwal, editors, \emph{Proceedings of Thirty Third Conference on Learning Theory}, volume 125 of \emph{Proceedings of Machine Learning Research}, pages 1411--1427. PMLR, 09--12 Jul 2020.
\newblock URL \url{https://proceedings.mlr.press/v125/davis20a.html}.

\bibitem[Duchi et~al.(2013)Duchi, Jordan, and McMahan]{NIPS2013_2812e5cf}
John Duchi, Michael~I Jordan, and Brendan McMahan.
\newblock Estimation, optimization, and parallelism when data is sparse.
\newblock In C.J. Burges, L.~Bottou, M.~Welling, Z.~Ghahramani, and K.Q. Weinberger, editors, \emph{Advances in Neural Information Processing Systems}, volume~26. Curran Associates, Inc., 2013.
\newblock URL \url{https://proceedings.neurips.cc/paper_files/paper/2013/file/2812e5cf6d8f21d69c91dddeefb792a7-Paper.pdf}.

\bibitem[Durrett(2019)]{Durrett_2019}
Rick Durrett.
\newblock \emph{Probability: Theory and Examples}.
\newblock Cambridge Series in Statistical and Probabilistic Mathematics. Cambridge University Press, 5 edition, 2019.

\bibitem[Fang et~al.(2022)Fang, Harvey, Portella, and Friedlander]{JMLR:v23:21-1027}
Huang Fang, Nicholas J.~A. Harvey, Victor~S. Portella, and Michael~P. Friedlander.
\newblock Online mirror descent and dual averaging: Keeping pace in the dynamic case.
\newblock \emph{Journal of Machine Learning Research}, 23\penalty0 (121):\penalty0 1--38, 2022.
\newblock URL \url{http://jmlr.org/papers/v23/21-1027.html}.

\bibitem[Fatkhullin et~al.(2025)Fatkhullin, H{\"u}bler, and Lan]{fatkhullin2025can}
Ilyas Fatkhullin, Florian H{\"u}bler, and Guanghui Lan.
\newblock Can sgd handle heavy-tailed noise?
\newblock \emph{arXiv preprint arXiv:2508.04860}, 2025.

\bibitem[Freedman(1975)]{10.1214/aop/1176996452}
David~A. Freedman.
\newblock {On Tail Probabilities for Martingales}.
\newblock \emph{The Annals of Probability}, 3\penalty0 (1):\penalty0 100 -- 118, 1975.
\newblock \doi{10.1214/aop/1176996452}.
\newblock URL \url{https://doi.org/10.1214/aop/1176996452}.

\bibitem[Garg et~al.(2021)Garg, Zhanson, Parisotto, Prasad, Kolter, Lipton, Balakrishnan, Salakhutdinov, and Ravikumar]{pmlr-v139-garg21b}
Saurabh Garg, Joshua Zhanson, Emilio Parisotto, Adarsh Prasad, Zico Kolter, Zachary Lipton, Sivaraman Balakrishnan, Ruslan Salakhutdinov, and Pradeep Ravikumar.
\newblock On proximal policy optimization's heavy-tailed gradients.
\newblock In Marina Meila and Tong Zhang, editors, \emph{Proceedings of the 38th International Conference on Machine Learning}, volume 139 of \emph{Proceedings of Machine Learning Research}, pages 3610--3619. PMLR, 18--24 Jul 2021.
\newblock URL \url{https://proceedings.mlr.press/v139/garg21b.html}.

\bibitem[Gilbert(1952)]{6773017}
E.~N. Gilbert.
\newblock A comparison of signalling alphabets.
\newblock \emph{The Bell System Technical Journal}, 31\penalty0 (3):\penalty0 504--522, 1952.
\newblock \doi{10.1002/j.1538-7305.1952.tb01393.x}.

\bibitem[Gorbunov et~al.(2020)Gorbunov, Danilova, and Gasnikov]{NEURIPS2020_abd1c782}
Eduard Gorbunov, Marina Danilova, and Alexander Gasnikov.
\newblock Stochastic optimization with heavy-tailed noise via accelerated gradient clipping.
\newblock In H.~Larochelle, M.~Ranzato, R.~Hadsell, M.F. Balcan, and H.~Lin, editors, \emph{Advances in Neural Information Processing Systems}, volume~33, pages 15042--15053. Curran Associates, Inc., 2020.
\newblock URL \url{https://proceedings.neurips.cc/paper_files/paper/2020/file/abd1c782880cc59759f4112fda0b8f98-Paper.pdf}.

\bibitem[Gorbunov et~al.(2024{\natexlab{a}})Gorbunov, Danilova, Shibaev, Dvurechensky, and Gasnikov]{gorbunov2024high}
Eduard Gorbunov, Marina Danilova, Innokentiy Shibaev, Pavel Dvurechensky, and Alexander Gasnikov.
\newblock High-probability complexity bounds for non-smooth stochastic convex optimization with heavy-tailed noise.
\newblock \emph{Journal of Optimization Theory and Applications}, pages 1--60, 2024{\natexlab{a}}.

\bibitem[Gorbunov et~al.(2024{\natexlab{b}})Gorbunov, Sadiev, Danilova, Horv\'{a}th, Gidel, Dvurechensky, Gasnikov, and Richt\'{a}rik]{pmlr-v235-gorbunov24a}
Eduard Gorbunov, Abdurakhmon Sadiev, Marina Danilova, Samuel Horv\'{a}th, Gauthier Gidel, Pavel Dvurechensky, Alexander Gasnikov, and Peter Richt\'{a}rik.
\newblock High-probability convergence for composite and distributed stochastic minimization and variational inequalities with heavy-tailed noise.
\newblock In Ruslan Salakhutdinov, Zico Kolter, Katherine Heller, Adrian Weller, Nuria Oliver, Jonathan Scarlett, and Felix Berkenkamp, editors, \emph{Proceedings of the 41st International Conference on Machine Learning}, volume 235 of \emph{Proceedings of Machine Learning Research}, pages 15951--16070. PMLR, 21--27 Jul 2024{\natexlab{b}}.
\newblock URL \url{https://proceedings.mlr.press/v235/gorbunov24a.html}.

\bibitem[Gurbuzbalaban et~al.(2021)Gurbuzbalaban, Simsekli, and Zhu]{pmlr-v139-gurbuzbalaban21a}
Mert Gurbuzbalaban, Umut Simsekli, and Lingjiong Zhu.
\newblock The heavy-tail phenomenon in sgd.
\newblock In Marina Meila and Tong Zhang, editors, \emph{Proceedings of the 38th International Conference on Machine Learning}, volume 139 of \emph{Proceedings of Machine Learning Research}, pages 3964--3975. PMLR, 18--24 Jul 2021.
\newblock URL \url{https://proceedings.mlr.press/v139/gurbuzbalaban21a.html}.

\bibitem[Hodgkinson and Mahoney(2021)]{pmlr-v139-hodgkinson21a}
Liam Hodgkinson and Michael Mahoney.
\newblock Multiplicative noise and heavy tails in stochastic optimization.
\newblock In Marina Meila and Tong Zhang, editors, \emph{Proceedings of the 38th International Conference on Machine Learning}, volume 139 of \emph{Proceedings of Machine Learning Research}, pages 4262--4274. PMLR, 18--24 Jul 2021.
\newblock URL \url{https://proceedings.mlr.press/v139/hodgkinson21a.html}.

\bibitem[Holland(2022)]{Holland_2022}
Matthew~J. Holland.
\newblock Anytime guarantees under heavy-tailed data.
\newblock \emph{Proceedings of the AAAI Conference on Artificial Intelligence}, 36\penalty0 (6):\penalty0 6918--6925, Jun. 2022.
\newblock \doi{10.1609/aaai.v36i6.20649}.
\newblock URL \url{https://ojs.aaai.org/index.php/AAAI/article/view/20649}.

\bibitem[H{\"u}bler et~al.(2025)H{\"u}bler, Fatkhullin, and He]{pmlr-v258-hubler25a}
Florian H{\"u}bler, Ilyas Fatkhullin, and Niao He.
\newblock From gradient clipping to normalization for heavy tailed sgd.
\newblock In Yingzhen Li, Stephan Mandt, Shipra Agrawal, and Emtiyaz Khan, editors, \emph{Proceedings of The 28th International Conference on Artificial Intelligence and Statistics}, volume 258 of \emph{Proceedings of Machine Learning Research}, pages 2413--2421. PMLR, 03--05 May 2025.
\newblock URL \url{https://proceedings.mlr.press/v258/hubler25a.html}.

\bibitem[Ivgi et~al.(2023)Ivgi, Hinder, and Carmon]{pmlr-v202-ivgi23a}
Maor Ivgi, Oliver Hinder, and Yair Carmon.
\newblock {D}o{G} is {SGD}'s best friend: A parameter-free dynamic step size schedule.
\newblock In Andreas Krause, Emma Brunskill, Kyunghyun Cho, Barbara Engelhardt, Sivan Sabato, and Jonathan Scarlett, editors, \emph{Proceedings of the 40th International Conference on Machine Learning}, volume 202 of \emph{Proceedings of Machine Learning Research}, pages 14465--14499. PMLR, 23--29 Jul 2023.
\newblock URL \url{https://proceedings.mlr.press/v202/ivgi23a.html}.

\bibitem[Jakoveti{\'c} et~al.(2023)Jakoveti{\'c}, Bajovi{\'c}, Sahu, Kar, Milo{\v s}evi{\'c}, and Stamenkovi{\'c}]{doi:10.1137/21M145896X}
Du{\v s}an Jakoveti{\'c}, Dragana Bajovi{\'c}, Anit~Kumar Sahu, Soummya Kar, Nemanja Milo{\v s}evi{\'c}, and Du{\v s}an Stamenkovi{\'c}.
\newblock Nonlinear gradient mappings and stochastic optimization: A general framework with applications to heavy-tail noise.
\newblock \emph{SIAM Journal on Optimization}, 33\penalty0 (2):\penalty0 394--423, 2023.
\newblock \doi{10.1137/21M145896X}.
\newblock URL \url{https://doi.org/10.1137/21M145896X}.

\bibitem[Lan(2020)]{lan2020first}
Guanghui Lan.
\newblock \emph{First-order and stochastic optimization methods for machine learning}.
\newblock Springer, 2020.

\bibitem[Liu et~al.(2024)Liu, Wang, and Zhang]{pmlr-v235-liu24bo}
Langqi Liu, Yibo Wang, and Lijun Zhang.
\newblock High-probability bound for non-smooth non-convex stochastic optimization with heavy tails.
\newblock In Ruslan Salakhutdinov, Zico Kolter, Katherine Heller, Adrian Weller, Nuria Oliver, Jonathan Scarlett, and Felix Berkenkamp, editors, \emph{Proceedings of the 41st International Conference on Machine Learning}, volume 235 of \emph{Proceedings of Machine Learning Research}, pages 32122--32138. PMLR, 21--27 Jul 2024.
\newblock URL \url{https://proceedings.mlr.press/v235/liu24bo.html}.

\bibitem[Liu(2025)]{liu2025online}
Zijian Liu.
\newblock Online convex optimization with heavy tails: Old algorithms, new regrets, and applications.
\newblock \emph{arXiv preprint arXiv:2508.07473}, 2025.

\bibitem[Liu and Zhou(2023)]{liu2023stochastic}
Zijian Liu and Zhengyuan Zhou.
\newblock Stochastic nonsmooth convex optimization with heavy-tailed noises: High-probability bound, in-expectation rate and initial distance adaptation.
\newblock \emph{arXiv preprint arXiv:2303.12277}, 2023.

\bibitem[Liu and Zhou(2024)]{liu2024revisiting}
Zijian Liu and Zhengyuan Zhou.
\newblock Revisiting the last-iterate convergence of stochastic gradient methods.
\newblock In \emph{The Twelfth International Conference on Learning Representations}, 2024.
\newblock URL \url{https://openreview.net/forum?id=xxaEhwC1I4}.

\bibitem[Liu and Zhou(2025)]{liu2025nonconvex}
Zijian Liu and Zhengyuan Zhou.
\newblock Nonconvex stochastic optimization under heavy-tailed noises: Optimal convergence without gradient clipping.
\newblock In \emph{The Thirteenth International Conference on Learning Representations}, 2025.
\newblock URL \url{https://openreview.net/forum?id=NKotdPUc3L}.

\bibitem[Liu et~al.(2023)Liu, Zhang, and Zhou]{pmlr-v195-liu23c}
Zijian Liu, Jiawei Zhang, and Zhengyuan Zhou.
\newblock Breaking the lower bound with (little) structure: Acceleration in non-convex stochastic optimization with heavy-tailed noise.
\newblock In Gergely Neu and Lorenzo Rosasco, editors, \emph{Proceedings of Thirty Sixth Conference on Learning Theory}, volume 195 of \emph{Proceedings of Machine Learning Research}, pages 2266--2290. PMLR, 12--15 Jul 2023.
\newblock URL \url{https://proceedings.mlr.press/v195/liu23c.html}.

\bibitem[Ma et~al.(2024)Ma, Verchand, and Samworth]{ma2024high}
Tianyi Ma, Kabir~A Verchand, and Richard~J Samworth.
\newblock High-probability minimax lower bounds.
\newblock \emph{arXiv preprint arXiv:2406.13447}, 2024.

\bibitem[Mai and Johansson(2021)]{pmlr-v139-mai21a}
Vien~V. Mai and Mikael Johansson.
\newblock Stability and convergence of stochastic gradient clipping: Beyond lipschitz continuity and smoothness.
\newblock In Marina Meila and Tong Zhang, editors, \emph{Proceedings of the 38th International Conference on Machine Learning}, volume 139 of \emph{Proceedings of Machine Learning Research}, pages 7325--7335. PMLR, 18--24 Jul 2021.
\newblock URL \url{https://proceedings.mlr.press/v139/mai21a.html}.

\bibitem[Nazin et~al.(2019)Nazin, Nemirovsky, Tsybakov, and Juditsky]{nazin2019algorithms}
Alexander~V Nazin, Arkadi~S Nemirovsky, Alexandre~B Tsybakov, and Anatoli~B Juditsky.
\newblock Algorithms of robust stochastic optimization based on mirror descent method.
\newblock \emph{Automation and Remote Control}, 80\penalty0 (9):\penalty0 1607--1627, 2019.

\bibitem[Nemirovski and Yudin(1983)]{nemirovskij1983problem}
Arkadi Nemirovski and David Yudin.
\newblock Problem complexity and method efficiency in optimization.
\newblock \emph{Wiley-Interscience}, 1983.

\bibitem[Nesterov et~al.(2018)]{nesterov2018lectures}
Yurii Nesterov et~al.
\newblock \emph{Lectures on convex optimization}, volume 137.
\newblock Springer, 2018.

\bibitem[Nguyen et~al.(2023)Nguyen, Nguyen, Ene, and Nguyen]{NEURIPS2023_4c454d34}
Ta~Duy Nguyen, Thien~H Nguyen, Alina Ene, and Huy Nguyen.
\newblock Improved convergence in high probability of clipped gradient methods with heavy tailed noise.
\newblock In A.~Oh, T.~Naumann, A.~Globerson, K.~Saenko, M.~Hardt, and S.~Levine, editors, \emph{Advances in Neural Information Processing Systems}, volume~36, pages 24191--24222. Curran Associates, Inc., 2023.
\newblock URL \url{https://proceedings.neurips.cc/paper_files/paper/2023/file/4c454d34f3a4c8d6b4ca85a918e5d7ba-Paper-Conference.pdf}.

\bibitem[Nolan(2020)]{nolan2020univariate}
John~P Nolan.
\newblock \emph{Univariate stable distributions}.
\newblock Springer, 2020.

\bibitem[Parletta et~al.(2024)Parletta, Paudice, Pontil, and Salzo]{doi:10.1137/22M1536558}
Daniela~Angela Parletta, Andrea Paudice, Massimiliano Pontil, and Saverio Salzo.
\newblock High probability bounds for stochastic subgradient schemes with heavy tailed noise.
\newblock \emph{SIAM Journal on Mathematics of Data Science}, 6\penalty0 (4):\penalty0 953--977, 2024.
\newblock \doi{10.1137/22M1536558}.
\newblock URL \url{https://doi.org/10.1137/22M1536558}.

\bibitem[Parletta et~al.(2025)Parletta, Paudice, and Salzo]{parletta2025improvedanalysisclippedstochastic}
Daniela~Angela Parletta, Andrea Paudice, and Saverio Salzo.
\newblock An improved analysis of the clipped stochastic subgradient method under heavy-tailed noise, 2025.
\newblock URL \url{https://arxiv.org/abs/2410.00573}.

\bibitem[Pascanu et~al.(2013)Pascanu, Mikolov, and Bengio]{pmlr-v28-pascanu13}
Razvan Pascanu, Tomas Mikolov, and Yoshua Bengio.
\newblock On the difficulty of training recurrent neural networks.
\newblock In Sanjoy Dasgupta and David McAllester, editors, \emph{Proceedings of the 30th International Conference on Machine Learning}, volume~28 of \emph{Proceedings of Machine Learning Research}, pages 1310--1318, Atlanta, Georgia, USA, 17--19 Jun 2013. PMLR.
\newblock URL \url{https://proceedings.mlr.press/v28/pascanu13.html}.

\bibitem[Puchkin et~al.(2024)Puchkin, Gorbunov, Kutuzov, and Gasnikov]{pmlr-v238-puchkin24a}
Nikita Puchkin, Eduard Gorbunov, Nickolay Kutuzov, and Alexander Gasnikov.
\newblock Breaking the heavy-tailed noise barrier in stochastic optimization problems.
\newblock In Sanjoy Dasgupta, Stephan Mandt, and Yingzhen Li, editors, \emph{Proceedings of The 27th International Conference on Artificial Intelligence and Statistics}, volume 238 of \emph{Proceedings of Machine Learning Research}, pages 856--864. PMLR, 02--04 May 2024.
\newblock URL \url{https://proceedings.mlr.press/v238/puchkin24a.html}.

\bibitem[Raginsky and Rakhlin(2009)]{5394945}
Maxim Raginsky and Alexander Rakhlin.
\newblock Information complexity of black-box convex optimization: A new look via feedback information theory.
\newblock In \emph{2009 47th Annual Allerton Conference on Communication, Control, and Computing (Allerton)}, pages 803--810, 2009.
\newblock \doi{10.1109/ALLERTON.2009.5394945}.

\bibitem[Robbins and Monro(1951)]{10.1214/aoms/1177729586}
Herbert Robbins and Sutton Monro.
\newblock {A Stochastic Approximation Method}.
\newblock \emph{The Annals of Mathematical Statistics}, 22\penalty0 (3):\penalty0 400 -- 407, 1951.
\newblock \doi{10.1214/aoms/1177729586}.
\newblock URL \url{https://doi.org/10.1214/aoms/1177729586}.

\bibitem[Sadiev et~al.(2023)Sadiev, Danilova, Gorbunov, Horv\'{a}th, Gidel, Dvurechensky, Gasnikov, and Richt\'{a}rik]{pmlr-v202-sadiev23a}
Abdurakhmon Sadiev, Marina Danilova, Eduard Gorbunov, Samuel Horv\'{a}th, Gauthier Gidel, Pavel Dvurechensky, Alexander Gasnikov, and Peter Richt\'{a}rik.
\newblock High-probability bounds for stochastic optimization and variational inequalities: the case of unbounded variance.
\newblock In Andreas Krause, Emma Brunskill, Kyunghyun Cho, Barbara Engelhardt, Sivan Sabato, and Jonathan Scarlett, editors, \emph{Proceedings of the 40th International Conference on Machine Learning}, volume 202 of \emph{Proceedings of Machine Learning Research}, pages 29563--29648. PMLR, 23--29 Jul 2023.
\newblock URL \url{https://proceedings.mlr.press/v202/sadiev23a.html}.

\bibitem[Samorodnitsky and Taqqu(1994)]{samorodnitsky1994stable}
Gennady Samorodnitsky and Murad~S Taqqu.
\newblock \emph{Stable non-Gaussian random processes: stochastic models with infinite variance}, volume~1.
\newblock CRC press, 1994.

\bibitem[Simsekli et~al.(2019)Simsekli, Sagun, and Gurbuzbalaban]{pmlr-v97-simsekli19a}
Umut Simsekli, Levent Sagun, and Mert Gurbuzbalaban.
\newblock A tail-index analysis of stochastic gradient noise in deep neural networks.
\newblock In Kamalika Chaudhuri and Ruslan Salakhutdinov, editors, \emph{Proceedings of the 36th International Conference on Machine Learning}, volume~97 of \emph{Proceedings of Machine Learning Research}, pages 5827--5837. PMLR, 09--15 Jun 2019.
\newblock URL \url{https://proceedings.mlr.press/v97/simsekli19a.html}.

\bibitem[Sun et~al.(2025)Sun, Liu, and Yuan]{JMLR:v26:24-1991}
Tao Sun, Xinwang Liu, and Kun Yuan.
\newblock Revisiting gradient normalization and clipping for nonconvex sgd under heavy-tailed noise: Necessity, sufficiency, and acceleration.
\newblock \emph{Journal of Machine Learning Research}, 26\penalty0 (237):\penalty0 1--42, 2025.
\newblock URL \url{http://jmlr.org/papers/v26/24-1991.html}.

\bibitem[Tropp(2015)]{MAL-048}
Joel~A. Tropp.
\newblock An introduction to matrix concentration inequalities.
\newblock \emph{Foundations and Trends® in Machine Learning}, 8\penalty0 (1-2):\penalty0 1--230, 2015.
\newblock ISSN 1935-8237.
\newblock \doi{10.1561/2200000048}.
\newblock URL \url{http://dx.doi.org/10.1561/2200000048}.

\bibitem[Tsai et~al.(2022)Tsai, Prasad, Balakrishnan, and Ravikumar]{pmlr-v151-tsai22a}
Che-Ping Tsai, Adarsh Prasad, Sivaraman Balakrishnan, and Pradeep Ravikumar.
\newblock Heavy-tailed streaming statistical estimation.
\newblock In Gustau Camps-Valls, Francisco J.~R. Ruiz, and Isabel Valera, editors, \emph{Proceedings of The 25th International Conference on Artificial Intelligence and Statistics}, volume 151 of \emph{Proceedings of Machine Learning Research}, pages 1251--1282. PMLR, 28--30 Mar 2022.
\newblock URL \url{https://proceedings.mlr.press/v151/tsai22a.html}.

\bibitem[Varshamov(1957)]{varshamov1957evaluation}
Rom~Rubenovich Varshamov.
\newblock The evaluation of signals in codes with correction of errors.
\newblock In \emph{Doklady Akademii Nauk}, volume 117, pages 739--741. Russian Academy of Sciences, 1957.

\bibitem[Vural et~al.(2022)Vural, Yu, Balasubramanian, Volgushev, and Erdogdu]{pmlr-v178-vural22a}
Nuri~Mert Vural, Lu~Yu, Krishna Balasubramanian, Stanislav Volgushev, and Murat~A Erdogdu.
\newblock Mirror descent strikes again: Optimal stochastic convex optimization under infinite noise variance.
\newblock In Po-Ling Loh and Maxim Raginsky, editors, \emph{Proceedings of Thirty Fifth Conference on Learning Theory}, volume 178 of \emph{Proceedings of Machine Learning Research}, pages 65--102. PMLR, 02--05 Jul 2022.
\newblock URL \url{https://proceedings.mlr.press/v178/vural22a.html}.

\bibitem[Wang et~al.(2021)Wang, Gurbuzbalaban, Zhu, Simsekli, and Erdogdu]{NEURIPS2021_9cdf2656}
Hongjian Wang, Mert Gurbuzbalaban, Lingjiong Zhu, Umut Simsekli, and Murat~A Erdogdu.
\newblock Convergence rates of stochastic gradient descent under infinite noise variance.
\newblock In M.~Ranzato, A.~Beygelzimer, Y.~Dauphin, P.S. Liang, and J.~Wortman Vaughan, editors, \emph{Advances in Neural Information Processing Systems}, volume~34, pages 18866--18877. Curran Associates, Inc., 2021.
\newblock URL \url{https://proceedings.neurips.cc/paper_files/paper/2021/file/9cdf26568d166bc6793ef8da5afa0846-Paper.pdf}.

\bibitem[Warmuth et~al.(1997)Warmuth, Jagota, et~al.]{warmuth1997continuous}
Manfred~K Warmuth, Arun~K Jagota, et~al.
\newblock Continuous and discrete-time nonlinear gradient descent: Relative loss bounds and convergence.
\newblock In \emph{Electronic proceedings of the 5th International Symposium on Artificial Intelligence and Mathematics}, volume 326. Citeseer, 1997.

\bibitem[Zhang et~al.(2020)Zhang, Karimireddy, Veit, Kim, Reddi, Kumar, and Sra]{NEURIPS2020_b05b57f6}
Jingzhao Zhang, Sai~Praneeth Karimireddy, Andreas Veit, Seungyeon Kim, Sashank Reddi, Sanjiv Kumar, and Suvrit Sra.
\newblock Why are adaptive methods good for attention models?
\newblock In H.~Larochelle, M.~Ranzato, R.~Hadsell, M.F. Balcan, and H.~Lin, editors, \emph{Advances in Neural Information Processing Systems}, volume~33, pages 15383--15393. Curran Associates, Inc., 2020.
\newblock URL \url{https://proceedings.neurips.cc/paper_files/paper/2020/file/b05b57f6add810d3b7490866d74c0053-Paper.pdf}.

\bibitem[Zhang and Cutkosky(2022)]{NEURIPS2022_349956de}
Jiujia Zhang and Ashok Cutkosky.
\newblock Parameter-free regret in high probability with heavy tails.
\newblock In S.~Koyejo, S.~Mohamed, A.~Agarwal, D.~Belgrave, K.~Cho, and A.~Oh, editors, \emph{Advances in Neural Information Processing Systems}, volume~35, pages 8000--8012. Curran Associates, Inc., 2022.
\newblock URL \url{https://proceedings.neurips.cc/paper_files/paper/2022/file/349956dee974cfdcbbb2d06afad5dd4a-Paper-Conference.pdf}.

\bibitem[Zhou et~al.(2020)Zhou, Feng, Ma, Xiong, Hoi, and E]{NEURIPS2020_f3f27a32}
Pan Zhou, Jiashi Feng, Chao Ma, Caiming Xiong, Steven Chu~Hong Hoi, and Weinan E.
\newblock Towards theoretically understanding why sgd generalizes better than adam in deep learning.
\newblock In H.~Larochelle, M.~Ranzato, R.~Hadsell, M.F. Balcan, and H.~Lin, editors, \emph{Advances in Neural Information Processing Systems}, volume~33, pages 21285--21296. Curran Associates, Inc., 2020.
\newblock URL \url{https://proceedings.neurips.cc/paper_files/paper/2020/file/f3f27a324736617f20abbf2ffd806f6d-Paper.pdf}.

\bibitem[Zolotarev(1986)]{zolotarev1986one}
Vladimir~M Zolotarev.
\newblock \emph{One-dimensional stable distributions}, volume~65.
\newblock American Mathematical Soc., 1986.

\end{thebibliography}

\clearpage

\appendix

\section{Lower Bounds on $d_{\protect\eff}$\label{sec:d-eff}}

This section provides lower bounds on $d_{\eff}$ for the additive
noise model, i.e., $\bg(\bx,\xi)=\nabla f(\bx)+\xi$. Given $i\in\left[d\right]$,
$\xi_{i}$ denotes the $i$-th coordinate of $\xi$, and $\sigma_{i}\defeq\left(\E\left[\left|\xi_{i}\right|^{\p}\right]\right)^{\frac{1}{\p}}$
is the $\p$-th moment of $\xi_{i}$. Additionally, $\iota:\left[d\right]\to\left[d\right]$
is the permutation that makes $\sigma_{i}$ in a nonincreasing order,
i.e., $\sigma_{\iota_{1}}\geq\sigma_{\iota_{2}}\geq\mydots\geq\sigma_{\iota_{d-1}}\ge\sigma_{\iota_{d}}$.

\subsection{Independent Coordinates}

In this subsection, we assume that all of $\xi_{i}$ are mutually
independent.
\begin{itemize}
\item For any $j\in\left[d\right]$, we can lower bound 
\[
\left\Vert \xi\right\Vert ^{\p}=\left(\sum_{i=1}^{d}\xi_{i}^{2}\right)^{\frac{\p}{2}}\geq\left(\sum_{i=1}^{j}\xi_{\iota_{i}}^{2}\right)^{\frac{\p}{2}}\geq j^{\frac{\p}{2}-1}\sum_{i=1}^{j}\left|\xi_{\iota_{i}}\right|^{\p},
\]
where the last step is by the concavity of $x^{\frac{\p}{2}}$, which
implies $\sigma_{\lar}^{\p}=\E\left[\left\Vert \xi\right\Vert ^{\p}\right]\geq j^{\frac{\p}{2}-1}\sum_{i=1}^{j}\sigma_{\iota_{i}}^{\p}.$
Therefore, we can find
\begin{equation}
\sigma_{\lar}^{\p}\ge\max_{j\in\left[d\right]}j^{\frac{\p}{2}-1}\sum_{i=1}^{j}\sigma_{\iota_{i}}^{\p}.\label{eq:sigma-l-general}
\end{equation}
\item For any $\be\in\S^{d-1}$, we write $\be=\sum_{i=1}^{d}\lambda_{i}\be_{i}$
where $\sum_{i=1}^{d}\lambda_{i}^{2}=1$ and $\be_{i}$ denotes the
all-zero vector except for the $i$-th coordinate, which is one. Therefore,
we have
\[
\E\left[\left|\left\langle \be,\xi\right\rangle \right|^{\p}\right]=\E\left[\left|\sum_{i=1}^{d}\lambda_{i}\xi_{i}\right|^{\p}\right]\overset{(a)}{\leq}2^{2-\p}\sum_{i=1}^{d}\E\left[\left|\lambda_{i}\xi_{i}\right|^{\p}\right]=2^{2-\p}\sum_{i=1}^{d}\left|\lambda_{i}\right|^{\p}\sigma_{i}^{\p}\overset{(b)}{\leq}2^{2-\p}\left(\sum_{i=1}^{d}\sigma_{i}^{\frac{2\p}{2-\p}}\right)^{1-\frac{\p}{2}},
\]
where $(a)$ holds by $\left|a+b\right|^{\p}\leq\left|a\right|^{\p}+\p\left|a\right|^{\p-1}\sgn(a)b+2^{2-\p}\left|b\right|^{\p}$
(see Proposition 18 of \citet{pmlr-v178-vural22a}) and the mutual
independence of $\xi_{i}$, and $(b)$ is due to 
\[
\sum_{i=1}^{d}\left|\lambda_{i}\right|^{\p}\sigma_{i}^{\p}\leq\left(\sum_{i=1}^{d}\lambda_{i}^{2}\right)^{\frac{\p}{2}}\left(\sum_{i=1}^{d}\sigma_{i}^{\frac{2\p}{2-\p}}\right)^{1-\frac{\p}{2}}=\left(\sum_{i=1}^{d}\sigma_{i}^{\frac{2\p}{2-\p}}\right)^{1-\frac{\p}{2}}.
\]
Hence, we know
\begin{equation}
\sigma_{\sma}^{\p}=\sup_{\be\in\S^{d-1}}\E\left[\left|\left\langle \be,\xi\right\rangle \right|^{\p}\right]\leq2^{2-\p}\left(\sum_{i=1}^{d}\sigma_{i}^{\frac{2\p}{2-\p}}\right)^{1-\frac{\p}{2}}.\label{eq:sigma-s-general}
\end{equation}
\end{itemize}
As such, we can lower bound
\begin{equation}
d_{\eff}=\frac{\sigma_{\lar}^{2}}{\sigma_{\sma}^{2}}\overset{(\ref{eq:sigma-l-general}),(\ref{eq:sigma-s-general})}{\geq}\frac{\max_{j\in\left[d\right]}j^{1-\frac{2}{\p}}\left(\sum_{i=1}^{j}\sigma_{\iota_{i}}^{\p}\right)^{\frac{2}{\p}}}{2^{\frac{4}{\p}-2}\left(\sum_{i=1}^{d}\sigma_{i}^{\frac{2\p}{2-\p}}\right)^{\frac{2}{\p}-1}}.\label{eq:d-eff-general}
\end{equation}
Though (\ref{eq:d-eff-general}) does not directly give a lower bound
for $d_{\eff}$ expressed in terms of $d$, it has already provided
some useful information. For example, when $\sigma_{i}$ are all in
the same order, (\ref{eq:d-eff-general}) implies that $d_{\eff}=\Omega\left(d^{2-\frac{2}{\p}}\right)$.

\subsection{I.I.D. Coordinates}

In this subsection, we further assume that all $\xi_{i}$ are i.i.d.
and then lower bound $d_{\eff}$ by $d$. Since all coordinates are
identically distributed now, we write $\sigma_{i}=\sigma,\forall i\in\left[d\right]$
for some $\sigma\geq0$ in the following.

\subsubsection{A General $\Omega(d^{2-\frac{2}{\protect\p}})$ Bound}

We invoke (\ref{eq:d-eff-general}) and plug in $\sigma_{i}=\sigma$
to obtain
\begin{equation}
d_{\eff}\overset{(\ref{eq:d-eff-general})}{\geq}\frac{\max_{j\in\left[d\right]}j\sigma^{2}}{2^{\frac{4}{\p}-2}d^{\frac{2}{\p}-1}\sigma^{2}}=\frac{d^{2-\frac{2}{\p}}}{2^{\frac{4}{\p}-2}}.\label{eq:d-eff-iid}
\end{equation}
When $\p=2$, the above bound recovers the fact that $d_{\eff}=d$
for $\xi$ with i.i.d. coordinates.

\subsubsection{A Special $\Omega(d)$ Bound}

Now we consider a special kind of noise. Suppose $\p\in\left(1,2\right)$
and all $\xi_{i}$ have the characteristic function
\[
\E\left[\exp\left(\mathrm{i}t\xi_{i}\right)\right]=\exp\left(-\gamma^{\alpha}\left|t\right|^{\alpha}\left(1-\mathrm{i}\beta\tan\left(\frac{\pi\alpha}{2}\right)\sgn(t)\right)\right),\forall t\in\R,
\]
where $\alpha=\p+\epsilon$ for $\epsilon\in\left(0,2-\p\right]$,
$\beta\in\left[-1,1\right]$, and $\gamma\geq0$. Such a distribution
is known as $\alpha$-stable distribution satisfying that $\E\left[\xi_{i}\right]=0$,
$\sigma<\infty$, and $\sum_{i=1}^{d}\xi_{i}$ equals to $d^{\frac{1}{\alpha}}\xi_{1}$
in distribution \citep{zolotarev1986one,samorodnitsky1994stable,nolan2020univariate}.
This suggests that we can lower bound $\sigma_{\lar}^{\p}$ in another
way,
\begin{equation}
\sigma_{\lar}^{\p}=\E\left[\left\Vert \xi\right\Vert ^{\p}\right]=\E\left[\left(\sum_{i=1}^{d}\xi_{i}^{2}\right)^{\frac{\p}{2}}\right]\geq\frac{\E\left[\left|\sum_{i=1}^{d}\xi_{i}\right|^{\p}\right]}{18^{\p}\p^{\p}\left(\frac{\p}{\p-1}\right)^{\frac{\p}{2}}}=\frac{\E\left[\left|d^{\frac{1}{\alpha}}\xi_{1}\right|^{\p}\right]}{18^{\p}\p^{\p}\left(\frac{\p}{\p-1}\right)^{\frac{\p}{2}}}=\frac{d^{\frac{\p}{\p+\epsilon}}\sigma^{\p}}{18^{\p}\p^{\p}\left(\frac{\p}{\p-1}\right)^{\frac{\p}{2}}},\label{eq:sigma-l-stable}
\end{equation}
where the inequality is due to \citet{10.1214/aop/1176997023}. Therefore,
in this special case, we have
\begin{equation}
d_{\eff}=\frac{\sigma_{\lar}^{2}}{\sigma_{\sma}^{2}}\overset{(\ref{eq:sigma-s-general}),(\ref{eq:sigma-l-stable})}{\geq}\frac{d^{\frac{2}{\p+\epsilon}}\sigma^{2}\big/18^{2}\frac{\p^{3}}{\p-1}}{2^{\frac{4}{\p}-2}d^{\frac{2}{\p}-1}\sigma^{2}}=\frac{(\p-1)d^{1-\frac{2\epsilon}{\p(\p+\epsilon)}}}{\p^{3}3^{4}2^{\frac{4}{\p}}}.\label{eq:d-eff-stable}
\end{equation}
In particular, for any $0<\epsilon\leq\min\left\{ \frac{\p}{2\ln d-1},2-\p\right\} $
(assume $d\geq2$ here, since the case $d=1$ is trivial), 
\begin{equation}
\frac{2\epsilon}{\p(\p+\epsilon)}\leq\frac{1}{\p\ln d}\Rightarrow d^{\frac{2\epsilon}{\p(\p+\epsilon)}}\leq e^{\frac{1}{\p}}\Rightarrow d_{\eff}\overset{(\ref{eq:d-eff-stable})}{\geq}\frac{(\p-1)d}{\p^{3}3^{4}2^{\frac{4}{\p}}e^{\frac{1}{\p}}}=\Omega(d).\label{eq:d-eff-stable-final}
\end{equation}

\section{Reduction for Strongly Convex Problems\label{sec:reduction}}

We provide the reduction mentioned in Remark \ref{rem:reduction}.
Recall that existing works assume $f$ being $\mu$-strongly convex
and $G$-Lipschitz with a minimizer $\bx_{\star}$ on $\X$. Now we
consider the following problem instance to fit into our problem structure
\[
F(\bx)=\underbrace{f(\bx)-\frac{\mu}{2}\left\Vert \bx-\by\right\Vert ^{2}}_{\defeq\bar{f}(\bx)}+\underbrace{\frac{\mu}{2}\left\Vert \bx-\by\right\Vert ^{2}}_{\defeq r(\bx)}=f(\bx),
\]
where $\by$ can be any known point in $\X$. For example, one can
set $\by=\bx_{1}$ to be the initial point. Next, we show that $F$
fulfills all assumptions in Section \ref{sec:preliminary}.
\begin{itemize}
\item $F$ on $\X$ has the same optimal solution $\bx_{\star}$ as $f$
and hence satisfies Assumption \ref{assu:minimizer}.
\item Note that $\bar{f}$ is convex (since $f$ is $\mu$-strongly convex)
and $r$ is $\mu$-strongly convex, which fits Assumption \ref{assu:obj}.
\item Moreover, because $f$ is $\mu$-strongly convex and $G$-Lipschitz
with a minimizer $\bx_{\star}\in\X$, a well-known fact is that $\X$
has to be bounded, since for any $\bx\in\X$,
\begin{align}
\frac{\mu}{2}\left\Vert \bx-\bx_{\star}\right\Vert ^{2} & \leq f(\bx)-f(\bx_{\star})-\left\langle \nabla f(\bx_{\star}),\bx-\bx_{\star}\right\rangle \leq f(\bx)-f(\bx_{\star})\nonumber \\
 & \leq\left\langle \nabla f(\bx),\bx-\bx_{\star}\right\rangle \leq\left\Vert \nabla f(\bx)\right\Vert \left\Vert \bx-\bx_{\star}\right\Vert \leq G\left\Vert \bx-\bx_{\star}\right\Vert \nonumber \\
\Rightarrow\left\Vert \bx-\bx_{\star}\right\Vert  & \leq\frac{2G}{\mu}.\label{eq:compact}
\end{align}
Then we can calculate $\nabla\bar{f}(\bx)=\nabla f(\bx)-\mu(\bx-\by),\forall\bx\in\X$
and find $\left\Vert \nabla\bar{f}(\bx)\right\Vert \leq\left\Vert \nabla f(\bx)\right\Vert +\mu\left\Vert \bx-\bx_{\star}\right\Vert +\mu\left\Vert \by-\bx_{\star}\right\Vert \overset{(\ref{eq:compact})}{\leq}5G,\forall\bx\in\X$,
meaning that Assumption \ref{assu:lip} holds under the parameter
$5G$.
\item In addition, suppose we have a first-order oracle $\bg(\bx,\xi)$
for $\nabla f$ satisfying Assumption \ref{assu:oracle}. Then $\bar{\bg}(\bx,\xi)\defeq\bg(\bx,\xi)-\mu(\bx-\by)$
is a first-order oracle for $\bar{f}$ satisfying Assumption \ref{assu:oracle}
with the same parameters $\p$, $\sigma_{\sma}$ and $\sigma_{\lar}$.
\end{itemize}
Therefore, any instance in existing works can be transferred to fit
our problem structure. Moreover, for such an instance, we have $D=\left\Vert \bx_{1}-\bx_{\star}\right\Vert \overset{(\ref{eq:compact})}{\leq}\frac{2G}{\mu}$,
implying that the first term $\O\left(\frac{\mu D^{2}}{T^{3}}\right)$
in Theorem \ref{thm:main-str-hp-dep} is at most $\O\left(\frac{G^{2}}{\mu T^{3}}\right)$,
which can be further bounded by the third term $\O\left(\frac{G^{2}}{\mu T}\right)$.
So $\O\left(\frac{\mu D^{2}}{T^{3}}\right)$ in Theorem \ref{thm:main-str-hp-dep}
can be omitted if compared with prior works.
\begin{rem}
The above reduction does not hold in the reverse direction. This is
because, as one can see, the domain $\X$ in prior works has to be
bounded (due to (\ref{eq:compact})), which is however not necessary
under our problem structure. For example, $\X$ in our problem can
take $\R^{d}$, which cannot be true for previous works in contrast.
In other words, the problem studied in our paper is strictly more
general.
\end{rem}

\section{Stabilized Clipped Stochastic Gradient Descent\label{sec:stabilized}}

\begin{algorithm}[H]
\caption{\label{alg:stabilized-clipped-SGD}Stabilized Clipped Stochastic Gradient
Descent (Stabilized Clipped SGD)}

\textbf{Input:} initial point $\bx_{1}\in\X$, stepsize $\eta_{t}>0$,
clipping threshold $\cm_{t}>0$

\textbf{for} $t=1$ \textbf{to} $T$ \textbf{do}

$\quad$$\bg_{t}^{\rc}=\clip_{\cm_{t}}(\bg_{t})$ where $\bg_{t}=\bg(\bx_{t},\xi_{t})$
and $\xi_{t}\sim\dis$ is sampled independently from the history

$\quad$$\bx_{t+1}=\argmin_{\bx\in\X}r(\bx)+\left\langle \bg_{t}^{\rc},\bx\right\rangle +\frac{\left\Vert \bx-\bx_{t}\right\Vert ^{2}}{2\eta_{t}}+\frac{(\eta_{t}/\eta_{t+1}-1)\left\Vert \bx-\bx_{1}\right\Vert ^{2}}{2\eta_{t}}$

\textbf{end for}
\end{algorithm}

In this section, we propose Stabilized Clipped Stochastic Gradient
Descent (Stabilized Clipped SGD) in Algorithm \ref{alg:stabilized-clipped-SGD},
an algorithmic variant of Clipped SGD to deal with the undesired $\poly(\ln T)$
factor appearing in the anytime convergence rate of Clipped SGD for
general convex functions.

Compared to Clipped SGD, the only difference is an extra $\frac{(\eta_{t}/\eta_{t+1}-1)\left\Vert \bx-\bx_{1}\right\Vert ^{2}}{2\eta_{t}}$
term injected into the update rule, which is borrowed from the dual
stabilization technique introduced by \citet{JMLR:v23:21-1027}. The
stabilization trick was originally introduced to make Online Mirror
Descent \citep{nemirovskij1983problem,warmuth1997continuous,BECK2003167}
achieve an anytime optimal $\O(\sqrt{T})$ regret on unbounded domains
without knowing $T$. For how it works and the intuition behind this
mechanism, we kindly refer the reader to \citet{JMLR:v23:21-1027}
for details. Inspired by its anytime optimality, we incorporate it
with Clipped SGD here and will show that this stabilized modification
also works well under heavy-tailed noise. Precisely, assuming all
problem-dependent parameters are known but not $T$, we prove in Theorem
\ref{thm:cvx-hp-dep-t} that Stabilized Clipped SGD converges at an
anytime rate almost identical (though slightly different) to the bound
for Clipped SGD given in Theorem \ref{thm:main-cvx-hp-dep-T} that
requires a known $T$ in contrast.

Lastly, we remark that when the stepsize $\eta_{t}$ is constant,
Stabilized Clipped SGD and Clipped SGD degenerate to the same algorithm.
Therefore, Theorems \ref{thm:main-cvx-hp-dep-T} and \ref{thm:main-cvx-ex-dep-T}
can directly apply to Stabilized Clipped SGD as well. For the same
reason and also to save space, we will only analyze Stabilized Clipped
SGD when studying general convex functions.

\section{Finer Bounds for Clipping Error under Heavy-Tailed Noise\label{sec:clip}}

In this section, we study the clipping error under heavy-tailed noise,
whose finer bounds are critical in the analysis. Moreover, instead
of limiting to clipped gradient methods, we will study a more general
setting as in the following Theorem \ref{thm:clip}, which may benefit
broader research. In Appendix \ref{sec:analysis}, we apply this general
result to prove clipping error bounds specialized for clipped gradient
methods in Lemma \ref{lem:clip-ineq}, which is the full statement
of Lemma \ref{lem:main-clip-ineq}.
\begin{thm}
\label{thm:clip}Given a $\sigma$-algebra $\F$ and two random vectors
$\bg,\bff\in\R^{d}$, suppose they satisfy $\E\left[\bg\mid\F\right]=\bff$
and, for some $\p\in\left(1,2\right]$ and two constants $\sigma_{\sma},\sigma_{\lar}\geq0$,
\begin{eqnarray}
\E\left[\left\Vert \bg-\bff\right\Vert ^{\p}\mid\F\right]\leq\sigma_{\lar}^{\p}, & \E\left[\left|\left\langle \be,\bg-\bff\right\rangle \right|^{\p}\mid\F\right]\leq\sigma_{\sma}^{\p}, & \forall\be\in\S^{d-1}.\label{eq:clip-assu}
\end{eqnarray}
Moreover, we assume there exists another random vector $\bar{\bg}\in\R^{d}$
that is independent of $\bg$ conditional on $\F$ and satisfies that
$\bar{\bg}\mid\F$ equals $\bg\mid\F$ in distribution. For any $0<\cm\in\F$,
let $\bg^{\rc}\defeq\clip_{\cm}(\bg)=\min\left\{ 1,\frac{\cm}{\left\Vert \bg\right\Vert }\right\} \bg$,
$\bd^{\ru}\defeq\bg^{\rc}-\E\left[\bg^{\rc}\mid\F\right]$, $\bd^{\rb}\defeq\E\left[\bg^{\rc}\mid\F\right]-\bff$,
and $\chi(\alpha)\defeq\1\left[(1-\alpha)\cm\geq\left\Vert \bff\right\Vert \right],\forall\alpha\in\left[0,1\right)$,
then there are:
\begin{enumerate}
\item \label{enu:clip-1}$\left\Vert \bd^{\ru}\right\Vert \leq2\cm$.
\item \label{enu:clip-2}$\E\left[\left\Vert \bd^{\ru}\right\Vert ^{2}\mid\F\right]\leq4\sigma_{\lar}^{\p}\cm^{2-\p}$.
\item \label{enu:clip-3}$\left\Vert \E\left[\bd^{\ru}\left(\bd^{\ru}\right)^{\top}\mid\F\right]\right\Vert \leq4\sigma_{\sma}^{\p}\cm^{2-\p}+4\left\Vert \bff\right\Vert ^{2}$.
\item \label{enu:clip-4}$\left\Vert \E\left[\bd^{\ru}\left(\bd^{\ru}\right)^{\top}\mid\F\right]\right\Vert \chi(\alpha)\leq4\sigma_{\sma}^{\p}\cm^{2-\p}+4\alpha^{1-\p}\sigma_{\lar}^{\p}\left\Vert \bff\right\Vert ^{2}\cm^{-\p}$.
\item \label{enu:clip-5}$\left\Vert \bd^{\rb}\right\Vert \leq\sqrt{2}\left(\sigma_{\lar}^{\p-1}+\left\Vert \bff\right\Vert ^{\p-1}\right)\sigma_{\sma}\cm^{1-\p}+2\left(\sigma_{\lar}^{\p}+\left\Vert \bff\right\Vert ^{\p}\right)\left\Vert \bff\right\Vert \cm^{-\p}$.
\item \label{enu:clip-6}$\left\Vert \bd^{\rb}\right\Vert \chi(\alpha)\leq\sigma_{\sma}\sigma_{\lar}^{\p-1}\cm^{1-\p}+\alpha^{1-\p}\sigma_{\lar}^{\p}\left\Vert \bff\right\Vert \cm^{-\p}$.
\end{enumerate}
\end{thm}
Before proving Theorem \ref{thm:clip}, we discuss one point here.
As one can see, we require the existence of a random vector $\bar{\bg}\in\R^{d}$
satisfying a certain condition. This technical assumption is mild
as it can hold automatically in many cases. For example, if $\F$
is the trivial sigma algebra, then we can set $\bar{\bg}$ as an independent
copy of $\bg$. For clipped gradient methods under Assumption \ref{assu:oracle},
suppose $\F=\F_{t-1}$, $\bg=\bg(\bx_{t},\xi_{t})$ and $\bff=\nabla f(\bx_{t})$,
then we can set $\bar{\bg}=\bg(\bx_{t},\xi_{t+1})$, where we recall
$\F_{t-1}=\sigma(\xi_{1},\mydots,\xi_{t-1})$ and $\xi_{1}$ to $\xi_{t+1}$
are sampled from $\dis$ independently.

\begin{proof}
Inspired by \citet{NEURIPS2024_10bf9689}, we denote by $h\defeq\min\left\{ 1,\frac{\cm}{\left\Vert \bg\right\Vert }\right\} \in\left[0,1\right]$.
Under this notation, we have
\begin{equation}
\bg^{\rc}=\clip_{\cm}(\bg)=h\bg.\label{eq:clip-representation}
\end{equation}
We first give two useful properties of $h$.
\begin{itemize}
\item For any $q\geq0$, we have
\[
1-h\leq\frac{\left\Vert \bg\right\Vert ^{q}}{\left\Vert \bg\right\Vert ^{q}}\1\left[\left\Vert \bg\right\Vert \geq\cm\right]\leq\frac{\left\Vert \bg\right\Vert ^{q}}{\cm^{q}}\1\left[\left\Vert \bg\right\Vert \geq\cm\right]\leq\frac{\left\Vert \bg\right\Vert ^{q}}{\cm^{q}},
\]
which implies
\begin{equation}
1-h\leq\inf_{q\geq0}\frac{\left\Vert \bg\right\Vert ^{q}}{\cm^{q}}.\label{eq:clip-h-unknown}
\end{equation}
\item We can also observe
\begin{align*}
1-h & =\frac{\left\Vert \bg\right\Vert -\cm}{\left\Vert \bg\right\Vert }\1\left[\left\Vert \bg\right\Vert \geq\cm\right]\leq\frac{\left\Vert \bg\right\Vert -\cm}{\cm}\1\left[\left\Vert \bg\right\Vert \geq\cm\right]\\
 & \leq\frac{\left\Vert \bg-\bff\right\Vert +\left\Vert \bff\right\Vert -\cm}{\cm}\1\left[\left\Vert \bg\right\Vert \geq\cm\right],
\end{align*}
which implies
\begin{align}
(1-h)\chi(\alpha) & \leq\frac{\left\Vert \bg-\bff\right\Vert +\left\Vert \bff\right\Vert -\cm}{\cm}\1\left[\left\Vert \bg\right\Vert \geq\cm\geq\frac{\left\Vert \bff\right\Vert }{1-\alpha}\right]\nonumber \\
 & \leq\frac{\left\Vert \bg-\bff\right\Vert }{\cm}\1\left[\left\Vert \bg\right\Vert \geq\cm\geq\frac{\left\Vert \bff\right\Vert }{1-\alpha}\right]\leq\inf_{q\geq1}\frac{\left\Vert \bg-\bff\right\Vert ^{q}}{\alpha^{q-1}\cm^{q}}\chi(\alpha),\label{eq:clip-h-known}
\end{align}
where the last step is by noticing that the event $\left\{ \left\Vert \bg\right\Vert \geq\cm\geq\frac{\left\Vert \bff\right\Vert }{1-\alpha}\right\} $
implies the event $\left\{ \cm\geq\frac{\left\Vert \bff\right\Vert }{1-\alpha},\left\Vert \bg-\bff\right\Vert \geq\alpha\cm\right\} $,
thereby leading to, for any $q\geq1$,
\begin{align*}
\frac{\left\Vert \bg-\bff\right\Vert }{\cm}\1\left[\left\Vert \bg\right\Vert \geq\cm\geq\frac{\left\Vert \bff\right\Vert }{1-\alpha}\right] & \leq\frac{\left\Vert \bg-\bff\right\Vert }{\cm}\1\left[\cm\geq\frac{\left\Vert \bff\right\Vert }{1-\alpha},\left\Vert \bg-\bff\right\Vert \geq\alpha\cm\right]\\
 & \leq\frac{\left\Vert \bg-\bff\right\Vert ^{q}}{\alpha^{q-1}\cm^{q}}\1\left[\cm\geq\frac{\left\Vert \bff\right\Vert }{1-\alpha},\left\Vert \bg-\bff\right\Vert \geq\alpha\cm\right]\\
 & \leq\frac{\left\Vert \bg-\bff\right\Vert ^{q}}{\alpha^{q-1}\cm^{q}}\chi(\alpha).
\end{align*}
\end{itemize}
For $\bar{\bg}$, we use $\bar{\bg}^{\rc}$ to denote the clipped
version of $\bar{\bg}$ under the same clipping threshold $\cm$,
i.e., $\bar{\bg}^{\rc}\defeq\clip_{\cm}(\bar{\bg})=\min\left\{ 1,\frac{\cm}{\left\Vert \bar{\bg}\right\Vert }\right\} \bar{\bg}$.
By our assumption on $\bar{\bg}$, the following results hold
\begin{align}
\E\left[\bg^{\rc}\mid\F\right] & =\E\left[\bar{\bg}^{\rc}\mid\F\right]=\E\left[\bar{\bg}^{\rc}\mid\F,\bg\right],\label{eq:clip-bar-g-independence}\\
\E\left[\left\Vert \bg-\bff\right\Vert ^{\p}\mid\F\right] & =\E\left[\left\Vert \bar{\bg}-\bff\right\Vert ^{\p}\mid\F\right]\leq\sigma_{\lar}^{\p}.\label{eq:clip-bar-g-moment}
\end{align}

We first prove inequalities for $\bd^{\ru}$.
\begin{itemize}
\item Inequality \ref{enu:clip-1}. Note that $\left\Vert \bg^{\rc}\right\Vert \leq\cm$,
implying $\left\Vert \bd^{\ru}\right\Vert =\left\Vert \bg^{\rc}-\E\left[\bg^{\rc}\mid\F\right]\right\Vert \leq2\cm$.
\item Inequality \ref{enu:clip-2}. We observe that
\begin{align}
\E\left[\left\Vert \bd^{\ru}\right\Vert ^{2}\mid\F\right] & =\E\left[\left\Vert \bg^{\rc}-\E\left[\bg^{\rc}\mid\F\right]\right\Vert ^{2}\mid\F\right]\overset{(\ref{eq:clip-bar-g-independence})}{=}\E\left[\left\Vert \E\left[\bg^{\rc}-\bar{\bg}^{\rc}\mid\F,\bg\right]\right\Vert ^{2}\mid\F\right]\nonumber \\
 & \overset{(a)}{\leq}\E\left[\left\Vert \bg^{\rc}-\bar{\bg}^{\rc}\right\Vert ^{2}\mid\F\right]\leq(2\cm)^{2-\p}\E\left[\left\Vert \bg^{\rc}-\bar{\bg}^{\rc}\right\Vert ^{\p}\mid\F\right]\label{eq:clip-1}\\
 & \overset{(b)}{\leq}(2\cm)^{2-\p}\E\left[\left\Vert \bg-\bar{\bg}\right\Vert ^{\p}\mid\F\right]\overset{(c)}{\leq}4\sigma_{\lar}^{\p}\cm^{2-\p},\nonumber 
\end{align}
where $(a)$ is by the convexity of $\left\Vert \cdot\right\Vert ^{2}$
and the tower property, $(b)$ holds because $\clip_{\cm}$ is a nonexpansive
mapping, and $(c)$ follows by when $\p>1$
\[
\left\Vert \bg-\bar{\bg}\right\Vert ^{\p}\leq2^{\p-1}\left(\left\Vert \bg-\bff\right\Vert ^{\p}+\left\Vert \bar{\bg}-\bff\right\Vert ^{\p}\right)\Rightarrow\E\left[\left\Vert \bg-\bar{\bg}\right\Vert ^{\p}\mid\F\right]\overset{(\ref{eq:clip-assu}),(\ref{eq:clip-bar-g-moment})}{\leq}2^{\p}\sigma_{\lar}^{\p}.
\]
\end{itemize}
The third and fourth inequalities are more technical. Let $\be\in\S^{d-1}$
be a unit vector, we know
\begin{equation}
\be^{\top}\E\left[\bd^{\ru}\left(\bd^{\ru}\right)^{\top}\mid\F\right]\be=\E\left[\left|\left\langle \be,\bd^{\ru}\right\rangle \right|^{2}\mid\F\right]=\E\left[\left|\left\langle \be,\bg^{\rc}-\E\left[\bg^{\rc}\mid\F\right]\right\rangle \right|^{2}\mid\F\right].\label{eq:clip-2}
\end{equation}
We will bound this term in two approaches.

On the one hand, we have
\begin{align}
 & \E\left[\left|\left\langle \be,\bg^{\rc}-\E\left[\bg^{\rc}\mid\F\right]\right\rangle \right|^{2}\mid\F\right]\nonumber \\
= & \E\left[\left|\left\langle \be,\bg^{\rc}-\bff\right\rangle \right|^{2}\mid\F\right]-\E\left[\left|\left\langle \be,\bff-\E\left[\bg^{\rc}\mid\F\right]\right\rangle \right|^{2}\mid\F\right]\nonumber \\
\leq & \E\left[\left|\left\langle \be,\bg^{\rc}-\bff\right\rangle \right|^{2}\mid\F\right]\overset{(\ref{eq:clip-representation})}{=}\E\left[\left|\left\langle \be,h\bg-\bff\right\rangle \right|^{2}\mid\F\right]\nonumber \\
= & \E\left[\left|h\left\langle \be,\bg-\bff\right\rangle -(1-h)\left\langle \be,\bff\right\rangle \right|^{2}\mid\F\right]\leq\E\left[h\left|\left\langle \be,\bg-\bff\right\rangle \right|^{2}+(1-h)\left|\left\langle \be,\bff\right\rangle \right|^{2}\mid\F\right]\nonumber \\
\leq & \E\left[h\left|\left\langle \be,\bg-\bff\right\rangle \right|^{2}+(1-h)\left\Vert \bff\right\Vert ^{2}\mid\F\right],\label{eq:clip-3}
\end{align}
where the last step is by $\left|\left\langle \be,\bff\right\rangle \right|\leq\left\Vert \be\right\Vert \left\Vert \bff\right\Vert =\left\Vert \bff\right\Vert $
and $1-h\geq0$. By Cauchy-Schwarz inequality again,
\begin{align*}
\left|\left\langle \be,\bg-\bff\right\rangle \right|^{2-\p} & \leq\left\Vert \bg-\bff\right\Vert ^{2-\p}\leq\left(\left\Vert \bg\right\Vert +\left\Vert \bff\right\Vert \right)^{2-\p}\overset{\p>1}{\leq}\left\Vert \bg\right\Vert ^{2-\p}+\left\Vert \bff\right\Vert ^{2-\p}\\
\Rightarrow h\left|\left\langle \be,\bg-\bff\right\rangle \right|^{2-\p} & \leq h^{\p-1}\left\Vert h\bg\right\Vert ^{2-\p}+h\left\Vert \bff\right\Vert ^{2-\p}\overset{h\leq1<\p}{\leq}\left\Vert h\bg\right\Vert ^{2-\p}+\left\Vert \bff\right\Vert ^{2-\p}\\
 & \overset{(\ref{eq:clip-representation})}{=}\left\Vert \bg^{\rc}\right\Vert ^{2-\p}+\left\Vert \bff\right\Vert ^{2-\p}\leq\cm^{2-\p}+\left\Vert \bff\right\Vert ^{2-\p},
\end{align*}
which implies
\begin{equation}
\E\left[h\left|\left\langle \be,\bg-\bff\right\rangle \right|^{2}\mid\F\right]\leq\left(\cm^{2-\p}+\left\Vert \bff\right\Vert ^{2-\p}\right)\E\left[\left|\left\langle \be,\bg-\bff\right\rangle \right|^{\p}\mid\F\right]\overset{(\ref{eq:clip-assu})}{\leq}\sigma_{\sma}^{\p}\cm^{2-\p}+\sigma_{\sma}^{\p}\left\Vert \bff\right\Vert ^{2-\p}.\label{eq:clip-4}
\end{equation}
Combine (\ref{eq:clip-2}), (\ref{eq:clip-3}) and (\ref{eq:clip-4})
to obtain for any unit vector $\be\in\S^{d-1}$,
\begin{align}
\be^{\top}\E\left[\bd^{\ru}\left(\bd^{\ru}\right)^{\top}\mid\F\right]\be & \leq\sigma_{\sma}^{\p}\cm^{2-\p}+\sigma_{\sma}^{\p}\left\Vert \bff\right\Vert ^{2-\p}+\left\Vert \bff\right\Vert ^{2}\E\left[1-h\mid\F\right]\nonumber \\
\Rightarrow\left\Vert \E\left[\bd^{\ru}\left(\bd^{\ru}\right)^{\top}\mid\F\right]\right\Vert  & \leq\sigma_{\sma}^{\p}\cm^{2-\p}+\sigma_{\sma}^{\p}\left\Vert \bff\right\Vert ^{2-\p}+\left\Vert \bff\right\Vert ^{2}\E\left[1-h\mid\F\right].\label{eq:clip-5}
\end{align}

On the other hand, we can follow a similar way of proving (\ref{eq:clip-1})
to show
\begin{align}
\E\left[\left|\left\langle \be,\bg^{\rc}-\E\left[\bg^{\rc}\mid\F\right]\right\rangle \right|^{2}\mid\F\right] & \leq(2\cm)^{2-\p}\E\left[\left|\left\langle \be,\bg^{\rc}-\bar{\bg}^{\rc}\right\rangle \right|^{\p}\mid\F\right]\nonumber \\
 & \leq4\cm^{2-\p}\E\left[\left|\left\langle \be,\bg^{\rc}-\bff\right\rangle \right|^{\p}\mid\F\right].\label{eq:clip-6}
\end{align}
Similar to (\ref{eq:clip-3}), there is
\begin{align}
\E\left[\left|\left\langle \be,\bg^{\rc}-\bff\right\rangle \right|^{\p}\mid\F\right] & \leq\E\left[h\left|\left\langle \be,\bg-\bff\right\rangle \right|^{\p}+(1-h)\left\Vert \bff\right\Vert ^{\p}\mid\F\right]\nonumber \\
 & \overset{h\leq1}{\leq}\E\left[\left|\left\langle \be,\bg-\bff\right\rangle \right|^{\p}+(1-h)\left\Vert \bff\right\Vert ^{\p}\mid\F\right]\nonumber \\
 & \overset{(\ref{eq:clip-assu})}{\leq}\sigma_{\sma}^{\p}+\left\Vert \bff\right\Vert ^{\p}\E\left[1-h\mid\F\right].\label{eq:clip-7}
\end{align}
Combine (\ref{eq:clip-2}), (\ref{eq:clip-6}) and (\ref{eq:clip-7})
to obtain for any unit vector $\be\in\S^{d-1}$,
\begin{align}
\be^{\top}\E\left[\bd^{\ru}\left(\bd^{\ru}\right)^{\top}\mid\F\right]\be & \leq4\sigma_{\sma}^{\p}\cm^{2-\p}+4\cm^{2-\p}\left\Vert \bff\right\Vert ^{\p}\E\left[1-h\mid\F\right]\nonumber \\
\Rightarrow\left\Vert \E\left[\bd^{\ru}\left(\bd^{\ru}\right)^{\top}\mid\F\right]\right\Vert  & \leq4\sigma_{\sma}^{\p}\cm^{2-\p}+4\cm^{2-\p}\left\Vert \bff\right\Vert ^{\p}\E\left[1-h\mid\F\right].\label{eq:clip-8}
\end{align}

Recall by our definition $\chi(0)=\1\left[\cm\geq\left\Vert \bff\right\Vert \right]$,
we then denote by $\bar{\chi}(0)\defeq1-\chi(0)=\1\left[\cm<\left\Vert \bff\right\Vert \right]$.
Therefore,
\begin{align*}
\left\Vert \E\left[\bd^{\ru}\left(\bd^{\ru}\right)^{\top}\mid\F\right]\right\Vert \chi(0) & \overset{(\ref{eq:clip-5})}{\leq}\left(\sigma_{\sma}^{\p}\cm^{2-\p}+\sigma_{\sma}^{\p}\left\Vert \bff\right\Vert ^{2-\p}+\left\Vert \bff\right\Vert ^{2}\E\left[1-h\mid\F\right]\right)\chi(0)\\
 & \overset{\p\leq2}{\leq}\left(2\sigma_{\sma}^{\p}\cm^{2-\p}+\left\Vert \bff\right\Vert ^{2}\E\left[1-h\mid\F\right]\right)\chi(0),
\end{align*}
and
\begin{align*}
\left\Vert \E\left[\bd^{\ru}\left(\bd^{\ru}\right)^{\top}\mid\F\right]\right\Vert \bar{\chi}(0) & \overset{(\ref{eq:clip-8})}{\leq}\left(4\sigma_{\sma}^{\p}\cm^{2-\p}+4\cm^{2-\p}\left\Vert \bff\right\Vert ^{\p}\E\left[1-h\mid\F\right]\right)\bar{\chi}(0)\\
 & \overset{\p\leq2}{\leq}\left(4\sigma_{\sma}^{\p}\cm^{2-\p}+4\left\Vert \bff\right\Vert ^{2}\E\left[1-h\mid\F\right]\right)\bar{\chi}(0),
\end{align*}
which together imply
\begin{equation}
\left\Vert \E\left[\bd^{\ru}\left(\bd^{\ru}\right)^{\top}\mid\F\right]\right\Vert \leq4\sigma_{\sma}^{\p}\cm^{2-\p}+4\left\Vert \bff\right\Vert ^{2}\E\left[1-h\mid\F\right].\label{eq:clip-9}
\end{equation}

Now we are ready to prove inequalities \ref{enu:clip-3} and \ref{enu:clip-4}.
\begin{itemize}
\item Inequality \ref{enu:clip-3}. We use (\ref{eq:clip-9}) to know
\[
\left\Vert \E\left[\bd^{\ru}\left(\bd^{\ru}\right)^{\top}\mid\F\right]\right\Vert \leq4\sigma_{\sma}^{\p}\cm^{2-\p}+4\left\Vert \bff\right\Vert ^{2}\E\left[1-h\mid\F\right]\leq4\sigma_{\sma}^{\p}\cm^{2-\p}+4\left\Vert \bff\right\Vert ^{2}.
\]
\item Inequality \ref{enu:clip-4}. By (\ref{eq:clip-h-known}), we have
\begin{align*}
(1-h)\chi(\alpha) & \leq\frac{\left\Vert \bg-\bff\right\Vert ^{\p}}{\alpha^{\p-1}\cm^{\p}}\chi(\alpha)\\
\Rightarrow\E\left[1-h\mid\F\right]\chi(\alpha) & \overset{\chi(\alpha)\in\F}{=}\E\left[(1-h)\chi(\alpha)\mid\F\right]\leq\E\left[\frac{\left\Vert \bg-\bff\right\Vert ^{\p}}{\alpha^{\p-1}\cm^{\p}}\chi(\alpha)\mid\F\right]\\
 & \overset{(\ref{eq:clip-assu})}{\leq}\frac{\sigma_{\lar}^{\p}\chi(\alpha)}{\alpha^{\p-1}\cm^{\p}}\leq\frac{\sigma_{\lar}^{\p}}{\alpha^{\p-1}\cm^{\p}}.
\end{align*}
Now we use (\ref{eq:clip-9}) to know
\begin{align*}
\left\Vert \E\left[\bd^{\ru}\left(\bd^{\ru}\right)^{\top}\mid\F\right]\right\Vert \chi(\alpha) & \leq4\sigma_{\sma}^{\p}\cm^{2-\p}\chi(\alpha)+4\left\Vert \bff\right\Vert ^{2}\E\left[1-h\mid\F\right]\chi(\alpha)\\
 & \leq4\sigma_{\sma}^{\p}\cm^{2-\p}+4\alpha^{1-\p}\sigma_{\lar}^{\p}\left\Vert \bff\right\Vert ^{2}\cm^{-\p}.
\end{align*}
\end{itemize}
Finally, we prove the last two inequalities related to $\bd^{\rb}$.
Still let $\be$ represent a unit vector in $\R^{d}$, then by the
definition of $\bd^{\rb}$,
\begin{align}
\left\langle \be,\bd^{\rb}\right\rangle  & \overset{(\ref{eq:clip-representation})}{=}\left\langle \be,\E\left[h\bg\mid\F\right]-\bff\right\rangle =\E\left[(h-1)\left\langle \be,\bg\right\rangle \mid\F\right]\nonumber \\
 & =\E\left[(h-1)\left\langle \be,\bg-\bff\right\rangle \mid\F\right]-\left\langle \be,\bff\right\rangle \E\left[1-h\mid\F\right]\nonumber \\
 & \overset{(d)}{\leq}\E\left[(1-h)\left|\left\langle \be,\bg-\bff\right\rangle \right|\mid\F\right]+\left\Vert \bff\right\Vert \E\left[1-h\mid\F\right]\nonumber \\
 & \overset{(e)}{\leq}\left(\E\left[(1-h)^{\frac{\p}{\p-1}}\mid\F\right]\right)^{1-\frac{1}{\p}}\sigma_{\sma}+\left\Vert \bff\right\Vert \E\left[1-h\mid\F\right]\nonumber \\
\Rightarrow\left\Vert \bd^{\rb}\right\Vert  & \leq\left(\E\left[(1-h)^{\frac{\p}{\p-1}}\mid\F\right]\right)^{1-\frac{1}{\p}}\sigma_{\sma}+\left\Vert \bff\right\Vert \E\left[1-h\mid\F\right],\label{eq:clip-10}
\end{align}
where $(d)$ is by $h\leq1$ and $-\left\langle \be,\bff\right\rangle \leq\left\Vert \be\right\Vert \left\Vert \bff\right\Vert =\left\Vert \bff\right\Vert $,
and $(e)$ is by H\"{o}lder's inequality and (\ref{eq:clip-assu}).
\begin{itemize}
\item Inequality \ref{enu:clip-5}. Noticing that $\frac{\p}{\p-1}\geq1$
and $1-h\leq1$, we then have
\[
(1-h)^{\frac{\p}{\p-1}}\leq1-h\overset{(\ref{eq:clip-h-unknown})}{\leq}\frac{\left\Vert \bg\right\Vert ^{\p}}{\cm^{\p}}\leq\frac{2^{\p-1}\left(\left\Vert \bg-\bff\right\Vert ^{\p}+\left\Vert \bff\right\Vert ^{\p}\right)}{\cm^{\p}},
\]
which implies
\[
\E\left[(1-h)^{\frac{\p}{\p-1}}\mid\F\right]\leq\E\left[1-h\mid\F\right]\overset{(\ref{eq:clip-assu})}{\leq}\frac{2^{\p-1}\left(\sigma_{\lar}^{\p}+\left\Vert \bff\right\Vert ^{\p}\right)}{\cm^{\p}}.
\]
Combine (\ref{eq:clip-10}) and the above inequality to have
\begin{align*}
\left\Vert \bd^{\rb}\right\Vert  & \leq\left(\frac{2^{\p-1}\left(\sigma_{\lar}^{\p}+\left\Vert \bff\right\Vert ^{\p}\right)}{\cm^{\p}}\right)^{1-\frac{1}{\p}}\sigma_{\sma}+\left\Vert \bff\right\Vert \frac{2^{\p-1}\left(\sigma_{\lar}^{\p}+\left\Vert \bff\right\Vert ^{\p}\right)}{\cm^{\p}}\\
 & \overset{\p\leq2}{\leq}\sqrt{2}\left(\sigma_{\lar}^{\p-1}+\left\Vert \bff\right\Vert ^{\p-1}\right)\sigma_{\sma}\cm^{1-\p}+2\left(\sigma_{\lar}^{\p}+\left\Vert \bff\right\Vert ^{\p}\right)\left\Vert \bff\right\Vert \cm^{-\p}.
\end{align*}
\item Inequality \ref{enu:clip-6}. Recall that $\chi(\alpha)\in\left\{ 0,1\right\} \in\F$,
which implies
\begin{align*}
\left\Vert \bd^{\rb}\right\Vert \chi(\alpha) & \overset{(\ref{eq:clip-10})}{\leq}\left(\E\left[\left((1-h)\chi(\alpha)\right)^{\frac{\p}{\p-1}}\mid\F\right]\right)^{1-\frac{1}{\p}}\sigma_{\sma}+\left\Vert \bff\right\Vert \E\left[(1-h)\chi(\alpha)\mid\F\right]\\
 & \overset{\frac{1}{\p-1}\geq1}{\leq}\left(\E\left[\left((1-h)\chi(\alpha)\right)^{\p}\mid\F\right]\right)^{1-\frac{1}{\p}}\sigma_{\sma}+\left\Vert \bff\right\Vert \E\left[(1-h)\chi(\alpha)\mid\F\right]\\
 & \overset{(\ref{eq:clip-h-known})}{\leq}\left(\E\left[\frac{\left\Vert \bg-\bff\right\Vert ^{\p}}{\cm^{\p}}\chi(\alpha)\mid\F\right]\right)^{1-\frac{1}{\p}}\sigma_{\sma}+\left\Vert \bff\right\Vert \E\left[\frac{\left\Vert \bg-\bff\right\Vert ^{\p}}{\alpha^{\p-1}\cm^{\p}}\chi(\alpha)\mid\F\right]\\
 & \overset{(\ref{eq:clip-assu})}{\leq}\left(\sigma_{\sma}\sigma_{\lar}^{\p-1}\cm^{1-\p}+\alpha^{1-\p}\sigma_{\lar}^{\p}\left\Vert \bff\right\Vert \cm^{-\p}\right)\chi(\alpha)\\
 & \leq\sigma_{\sma}\sigma_{\lar}^{\p-1}\cm^{1-\p}+\alpha^{1-\p}\sigma_{\lar}^{\p}\left\Vert \bff\right\Vert \cm^{-\p}.
\end{align*}
\end{itemize}
\end{proof}

\section{Full Theorems for Upper Bounds and Proofs\label{sec:theorems}}

In this section, we provide the full description of each theorem given
in the main paper with the proof. Besides, we also present new anytime
convergence of Stabilized Clipped SGD. All intermediate results used
in the analysis are deferred to be proved in Appendix \ref{sec:analysis}.

Before starting, we recall that $D=\left\Vert \bx_{\star}-\bx_{1}\right\Vert $
denotes the distance between the optimal solution and the initial
point.

For high-probability convergence, as proposed in (\ref{eq:main-hp-tau-star}),
a repeatedly used quantity in the clipping threshold is
\begin{equation}
\cm_{\star}=\left(\min\left\{ \frac{\sigma_{\sma}\sigma_{\lar}^{\p-1}}{\ln\frac{3}{\delta}},\frac{\sigma_{\sma}^{2}}{\sigma_{\lar}^{2-\p}\1\left[\p<2\right]}\right\} \right)^{\frac{1}{\p}},\label{eq:hp-tau-star}
\end{equation}
where $\delta\in\left(0,1\right]$ is the failure probability, $\p\in\left(1,2\right]$
and $0\leq\sigma_{\sma}\leq\sigma_{\lar}$ are introduced in Assumption
\ref{assu:oracle}. Another useful value mentioned before in (\ref{eq:main-hp-varphi-star})
is
\begin{equation}
\varphi_{\star}=\max\left\{ \sqrt{d_{\eff}}\ln\frac{3}{\delta},d_{\eff}\1\left[\p<2\right]\right\} ,\label{eq:hp-varphi-star}
\end{equation}
where $d_{\eff}=\sigma_{\lar}^{2}/\sigma_{\sma}^{2}$ is called generalized
effective dimension defined in (\ref{eq:main-d-eff}) satisfying
\begin{equation}
d_{\eff}\in\left\{ 0\right\} \cup\left[1,\pi d/2\right],\label{eq:d-eff}
\end{equation}
in which $d_{\eff}=0$ if and only if $\sigma_{\lar}=\sigma_{\sma}=0$,
i.e., the noiseless case. Lastly, it is noteworthy that the following
equation always holds
\begin{equation}
\varphi_{\star}=\frac{\sigma_{\lar}^{\p}}{\cm_{\star}^{\p}}.\label{eq:hp-varphi-tau-equation}
\end{equation}

For in-expectation convergence, we will consider a larger quantity
in the clipping threshold as mentioned in Section \ref{sec:ex-rates}:
\begin{equation}
\widetilde{\cm}_{\star}=\frac{\sigma_{\sma}^{\frac{2}{\p}}}{\sigma_{\lar}^{\frac{2}{\p}-1}\1\left[\p<2\right]}.\label{eq:ex-tau-star}
\end{equation}
We also recall
\begin{equation}
\widetilde{\varphi}_{\star}=d_{\eff}\1\left[\p<2\right].\label{eq:ex-varphi-star}
\end{equation}
Note that there is
\begin{equation}
\widetilde{\varphi}_{\star}=\frac{\sigma_{\lar}^{\p}}{\widetilde{\cm}_{\star}^{\p}}.\label{eq:ex-varphi-tau-equation}
\end{equation}

\subsection{General Convex Case}

We provide different convergence rates for general convex objectives.
Recall that $\bar{\bx}_{T+1}^{\cvx}$ stands for the average iterate
after $T$ steps, i.e.,
\begin{equation}
\bar{\bx}_{T+1}^{\cvx}=\frac{1}{T}\sum_{t=1}^{T}\bx_{t+1}.\label{eq:cvx-avg-x}
\end{equation}
Moreover, note that Clipped SGD and Stabilized Clipped SGD are the
same when the stepsize is constant, as mentioned in Appendix \ref{sec:stabilized}.
Hence, everything in this subsection is proved based on the analysis
for Stabilized Clipped SGD.

\subsubsection{High-Probability Convergence}

\textbf{Known $T$.} We begin with the situation where the time horizon
$T$ is known in advance. Theorem \ref{thm:cvx-hp-dep-T} below shows
the refined high-probability rate for Clipped SGD.
\begin{thm}[Full statement of Theorem \ref{thm:main-cvx-hp-dep-T}]
\label{thm:cvx-hp-dep-T}Under Assumptions \ref{assu:minimizer},
\ref{assu:obj} (with $\mu=0$), \ref{assu:lip} and \ref{assu:oracle},
for any $T\in\N$ and $\delta\in\left(0,1\right]$, setting $\eta_{t}=\eta_{\star},\cm_{t}=\max\left\{ \frac{G}{1-\alpha},\cm_{\star}T^{\frac{1}{\p}}\right\} ,\forall t\in\left[T\right]$
where $\alpha=1/2$,
\begin{equation}
\eta_{\star}=\min\left\{ \frac{D/G}{\varphi+\ln\frac{3}{\delta}},\frac{D/G}{\sqrt{T}},\frac{D}{\left(\sigma_{\sma}^{\frac{2}{\p}-1}\sigma_{\lar}^{2-\frac{2}{\p}}+\sigma_{\sma}^{\frac{1}{\p}}\sigma_{\lar}^{1-\frac{1}{\p}}\ln^{1-\frac{1}{\p}}\frac{3}{\delta}\right)T^{\frac{1}{\p}}}\right\} ,\label{eq:cvx-hp-dep-T-eta-star}
\end{equation}
and $\varphi\leq\varphi_{\star}$ is a constant defined in (\ref{eq:varphi})
and equals $\varphi_{\star}$ when $T=\Omega\left(\frac{G^{\p}}{\sigma_{\lar}^{\p}}\varphi_{\star}\right)$,
then Clipped SGD (Algorithm \ref{alg:clipped-SGD}) guarantees that
with probability at least $1-\delta$, $F(\bar{\bx}_{T+1}^{\cvx})-F_{\star}$
converges at the rate of
\[
\O\left(\frac{\left(\varphi+\ln\frac{3}{\delta}\right)GD}{T}+\frac{GD}{\sqrt{T}}+\frac{\left(\sigma_{\sma}^{\frac{2}{\p}-1}\sigma_{\lar}^{2-\frac{2}{\p}}+\sigma_{\sma}^{\frac{1}{\p}}\sigma_{\lar}^{1-\frac{1}{\p}}\ln^{1-\frac{1}{\p}}\frac{3}{\delta}\right)D}{T^{1-\frac{1}{\p}}}\right).
\]
\end{thm}
\begin{rem}
\label{rem:alpha-dep}There are two points we want to emphasize:

First, the choice $\alpha=1/2$ is not essential and can be changed
to any $\alpha\in\left(0,1\right)$, only resulting in a different
hidden constant in the $\O$ notation. In the proof, we try to keep
$\alpha$ until the very last step. Moreover, we would like to mention
that a small $\alpha$ may lead to better practical performance as
suggested in Remark 2 of \citet{parletta2025improvedanalysisclippedstochastic}.

Second, these rates are presented while assuming the knowledge of
all problem-dependent parameters, as ubiquitously done in the optimization
literature. However, not all problem-dependent parameters are necessary
if one only wants to ensure convergence. For example, in the above
Theorem \ref{thm:cvx-hp-dep-T}, taking $\eta_{t}=\min\left\{ \frac{\lambda}{G\sqrt{T}},\frac{\lambda}{\cm_{t}}\right\} ,\cm_{t}=\max\left\{ 2G,\cm T^{\frac{1}{\p}}\right\} ,\forall t\in\left[T\right]$
where $\lambda,\cm>0$ (like Theorem 3 in \citet{liu2023stochastic})
is sufficient to prove that Clipped SGD converges. Therefore, when
proving these theorems, we also try to keep a general version of the
stepsize scheduling and the clipping threshold until the very last
step.
\end{rem}
\begin{proof}
First, a constant stepsize fulfills the requirement of Lemma \ref{lem:cvx-hp-anytime}.
In addition, our choices of $\eta_{t}$ and $\cm_{t}$ also satisfy
Conditions \ref{enu:cvx-dep-res-1} and \ref{enu:cvx-dep-res-2} (with
$\alpha=1/2$) in Lemma \ref{lem:cvx-dep-res}. Therefore, given $T\in\N$
and $\delta\in\left(0,1\right]$, Lemmas \ref{lem:cvx-hp-anytime}
and \ref{lem:cvx-dep-res} together yield with probability at least
$1-\delta$,
\begin{align}
\frac{\left\Vert \bx_{\star}-\bx_{T+1}\right\Vert ^{2}}{2\eta_{T+1}}+\sum_{t=1}^{T}F(\bx_{t+1})-F_{\star} & \leq\frac{D^{2}}{\eta_{T+1}}+2\hc_{T}^{\cvx}\nonumber \\
\Rightarrow F(\bar{\bx}_{T+1}^{\cvx})-F_{\star} & \leq\frac{D^{2}}{\eta_{T+1}T}+\frac{2\hc_{T}^{\cvx}}{T},\label{eq:cvx-hp-dep-T-1}
\end{align}
where $\hc_{T}^{\cvx}$ is a constant in the order of
\begin{equation}
\O\left(\max_{t\in\left[T\right]}\eta_{t}\cm_{t}^{2}\ln^{2}\frac{3}{\delta}+\sum_{t=1}^{T}\sigma_{\lar}^{\p}\eta_{t}\cm_{t}^{2-\p}+\left(\sum_{t=1}^{T}\frac{\sigma_{\sma}\sigma_{\lar}^{\p-1}\sqrt{\eta_{t}}}{\cm_{t}^{\p-1}}+\frac{\sigma_{\lar}^{\p}G\sqrt{\eta_{t}}}{\alpha^{\p-1}\cm_{t}^{\p}}\right)^{2}+\sum_{t=1}^{T}G^{2}\eta_{t}\right).\label{eq:cvx-hp-dep-T-C}
\end{equation}

Our left task is to bound $\hc_{T}^{\cvx}$. When $\eta_{t}=\eta,\cm_{t}=\cm,\forall t\in\left[T\right]$
where $\eta>0$ and $\cm\geq\frac{G}{1-\alpha}$ (as required by Condition
\ref{enu:cvx-dep-res-2} in Lemma \ref{lem:cvx-dep-res}), we can
simplify (\ref{eq:cvx-hp-dep-T-C}) into
\begin{equation}
\hc_{T}^{\cvx}=\O\left(\eta\left(\cm^{2}\ln^{2}\frac{3}{\delta}+\sigma_{\lar}^{\p}\cm^{2-\p}T+\frac{\sigma_{\sma}^{2}\sigma_{\lar}^{2\p-2}}{\cm^{2\p-2}}T^{2}+\frac{\sigma_{\lar}^{2\p}G^{2}}{\alpha^{2\p-2}\cm^{2\p}}T^{2}+G^{2}T\right)\right).\label{eq:cvx-hp-dep-T-C-simplified-1}
\end{equation}
One more step, under changing $\cm$ to $\max\left\{ \frac{G}{1-\alpha},\cm\right\} $
(the second $\cm$ is only required to be nonnegative), we can further
write (\ref{eq:cvx-hp-dep-T-C-simplified-1}) into
\begin{align*}
\hc_{T}^{\cvx}= & \O\left(\eta\left(\frac{\sigma_{\lar}^{2\p}G^{2}T^{2}}{\min\left\{ \alpha^{2\p-2},(1-\alpha)^{2}\right\} \left(\max\left\{ \frac{G}{1-\alpha},\cm\right\} \right)^{2\p}}+\frac{G^{2}\ln^{2}\frac{3}{\delta}}{(1-\alpha)^{2}}+G^{2}T\right)\right.\\
 & \left.\quad+\eta\left(\cm^{2}\ln^{2}\frac{3}{\delta}+\sigma_{\lar}^{\p}\cm^{2-\p}T+\frac{\sigma_{\sma}^{2}\sigma_{\lar}^{2\p-2}}{\cm^{2\p-2}}T^{2}\right)\right),
\end{align*}
where we use
\begin{align*}
 & \sigma_{\lar}^{\p}\left(\max\left\{ \frac{G}{1-\alpha},\cm\right\} \right)^{2-\p}T\leq\frac{\sigma_{\lar}^{\p}G^{2}T}{(1-\alpha)^{2}\left(\max\left\{ \frac{G}{1-\alpha},\cm\right\} \right)^{\p}}+\sigma_{\lar}^{\p}\cm^{2-\p}T\\
\leq & \frac{\sigma_{\lar}^{2\p}G^{2}T^{2}}{(1-\alpha)^{2}\left(\max\left\{ \frac{G}{1-\alpha},\cm\right\} \right)^{2\p}}+\frac{G^{2}}{4(1-\alpha)^{2}}+\sigma_{\lar}^{\p}\cm^{2-\p}T\\
\leq & \frac{\sigma_{\lar}^{2\p}G^{2}T^{2}}{\min\left\{ \alpha^{2\p-2},(1-\alpha)^{2}\right\} \left(\max\left\{ \frac{G}{1-\alpha},\cm\right\} \right)^{2\p}}+\frac{G^{2}\ln^{2}\frac{3}{\delta}}{4(1-\alpha)^{2}}+\sigma_{\lar}^{\p}\cm^{2-\p}T.
\end{align*}
Furthermore, we replace the current $\cm$ with $\cm T^{\frac{1}{\p}}$
to obtain
\begin{align}
\hc_{T}^{\cvx}= & \O\left(\eta\left(\inf_{\beta\in\left[0,1/2\right]}\frac{(1-\alpha)^{2\beta\p}\sigma_{\lar}^{2\p}G^{2-2\beta\p}T^{2\beta}}{\min\left\{ \alpha^{2\p-2},(1-\alpha)^{2}\right\} \cm^{2(1-\beta)\p}}+\frac{G^{2}\ln^{2}\frac{3}{\delta}}{(1-\alpha)^{2}}+G^{2}T\right)\right.\nonumber \\
 & \left.\quad+\eta\left(\cm^{2}\ln^{2}\frac{3}{\delta}+\sigma_{\lar}^{\p}\cm^{2-\p}+\frac{\sigma_{\sma}^{2}\sigma_{\lar}^{2\p-2}}{\cm^{2\p-2}}\right)T^{\frac{2}{\p}}\right),\label{eq:cvx-hp-dep-T-C-simplified-2}
\end{align}
where the first term appears due to, for any $\beta\in\left[0,1/2\right]$,
\[
\frac{\sigma_{\lar}^{2\p}G^{2}T^{2}}{\left(\max\left\{ \frac{G}{1-\alpha},\cm T^{\frac{1}{\p}}\right\} \right)^{2\p}}\leq\frac{\sigma_{\lar}^{2\p}G^{2}T^{2}}{\left(\frac{G}{1-\alpha}\right)^{2\beta\p}\left(\cm T^{\frac{1}{\p}}\right)^{2(1-\beta)\p}}=\frac{(1-\alpha)^{2\beta\p}\sigma_{\lar}^{2\p}G^{2-2\beta\p}T^{2\beta}}{\cm^{2(1-\beta)\p}}.
\]

Now, we plug $\cm=\cm_{\star}$ (see (\ref{eq:hp-tau-star})) into
(\ref{eq:cvx-hp-dep-T-C-simplified-2}) to have, under $\cm_{t}=\max\left\{ \frac{G}{1-\alpha},\cm_{\star}T^{\frac{1}{\p}}\right\} ,\forall t\in\left[T\right]$,
\begin{align}
\hc_{T}^{\cvx}= & \O\left(\eta\left(\frac{G^{2}\varphi^{2}}{\min\left\{ \alpha^{2\p-2},(1-\alpha)^{2}\right\} }+\frac{G^{2}\ln^{2}\frac{3}{\delta}}{(1-\alpha)^{2}}\right)+\eta G^{2}T\right.\nonumber \\
 & \left.\quad+\eta\left(\sigma_{\sma}^{\frac{4}{\p}-2}\sigma_{\lar}^{4-\frac{4}{\p}}+\sigma_{\sma}^{\frac{2}{\p}}\sigma_{\lar}^{2-\frac{2}{\p}}\ln^{2-\frac{2}{\p}}\frac{3}{\delta}\right)T^{\frac{2}{\p}}\right),\label{eq:cvx-hp-dep-T-C-simplified-3}
\end{align}
where the first term is obtained by noticing
\begin{align*}
\frac{(1-\alpha)^{2\beta\p}\sigma_{\lar}^{2\p}G^{2-2\beta\p}T^{2\beta}}{\cm_{\star}^{2(1-\beta)\p}} & =G^{2}\cdot\left((1-\alpha)^{\beta\p}\frac{\sigma_{\lar}^{(1-\beta)\p}}{\cm_{\star}^{(1-\beta)\p}}\left(\frac{\sigma_{\lar}}{G}\right)^{\beta\p}T^{\beta}\right)^{2}\\
 & \overset{(\ref{eq:hp-varphi-tau-equation})}{=}G^{2}\cdot\left((1-\alpha)^{\beta\p}\varphi_{\star}^{1-\beta}\left(\frac{\sigma_{\lar}}{G}\right)^{\beta\p}T^{\beta}\right)^{2}\\
\Rightarrow\inf_{\beta\in\left[0,1/2\right]}\frac{(1-\alpha)^{2\beta\p}\sigma_{\lar}^{2\p}G^{2-2\beta\p}T^{2\beta}}{\cm_{\star}^{2(1-\beta)\p}} & \leq G^{2}\varphi^{2},
\end{align*}
in which
\begin{equation}
\varphi\defeq\inf_{\beta\in\left[0,1/2\right]}(1-\alpha)^{\beta\p}\varphi_{\star}^{1-\beta}\left(\frac{\sigma_{\lar}}{G}\right)^{\beta\p}T^{\beta}=\min\left\{ \varphi_{\star},\sqrt{(1-\alpha)^{\p}\varphi_{\star}\left(\frac{\sigma_{\lar}}{G}\right)^{\p}T}\right\} \leq\varphi_{\star}.\label{eq:varphi}
\end{equation}
Note that we have $\varphi=\varphi_{\star}$ when $T\geq\frac{G^{\p}\varphi_{\star}}{(1-\alpha)^{\p}\sigma_{\lar}^{\p}}=\Omega\left(\frac{G^{\p}}{\sigma_{\lar}^{\p}}\varphi_{\star}\right)$.

By (\ref{eq:cvx-hp-dep-T-1}), (\ref{eq:cvx-hp-dep-T-C-simplified-3})
and $\alpha=1/2$, we can find
\begin{align*}
F(\bar{\bx}_{T+1}^{\cvx})-F_{\star}\leq & \O\left(\frac{D^{2}}{\eta T}+\frac{\eta\left(\varphi^{2}+\ln^{2}\frac{3}{\delta}\right)G^{2}}{T}+\eta G^{2}\right.\\
 & \left.\quad+\eta\left(\sigma_{\sma}^{\frac{4}{\p}-2}\sigma_{\lar}^{4-\frac{4}{\p}}+\sigma_{\sma}^{\frac{2}{\p}}\sigma_{\lar}^{2-\frac{2}{\p}}\ln^{2-\frac{2}{\p}}\frac{3}{\delta}\right)T^{\frac{2}{\p}-1}\right).
\end{align*}
Plug in $\eta=\eta_{\star}$ (see (\ref{eq:cvx-hp-dep-T-eta-star}))
to conclude that $F(\bar{\bx}_{T+1}^{\cvx})-F_{\star}$ converges
at the rate of
\[
\O\left(\frac{\left(\varphi+\ln\frac{3}{\delta}\right)GD}{T}+\frac{GD}{\sqrt{T}}+\frac{\left(\sigma_{\sma}^{\frac{2}{\p}-1}\sigma_{\lar}^{2-\frac{2}{\p}}+\sigma_{\sma}^{\frac{1}{\p}}\sigma_{\lar}^{1-\frac{1}{\p}}\ln^{1-\frac{1}{\p}}\frac{3}{\delta}\right)D}{T^{1-\frac{1}{\p}}}\right).
\]
\end{proof}

\textbf{Recover the existing rate in }\citet{liu2023stochastic}\textbf{.}
Remarkably, our above analysis is essentially tighter than \citet{liu2023stochastic}.
To see this claim, we bound $\hc_{T}^{\cvx}$ in the following way
(take the same $\alpha=1/2$ as in \citet{liu2023stochastic} for
a fair comparison):
\begin{align*}
\hc_{T}^{\cvx} & \overset{(\ref{eq:cvx-hp-dep-T-C-simplified-1})}{=}\O\left(\eta\left(\cm^{2}\ln^{2}\frac{3}{\delta}+\sigma_{\lar}^{\p}\cm^{2-\p}T+\frac{\sigma_{\sma}^{2}\sigma_{\lar}^{2\p-2}}{\cm^{2\p-2}}T^{2}+\frac{\sigma_{\lar}^{2\p}G^{2}}{\cm^{2\p}}T^{2}+G^{2}T\right)\right)\\
 & \overset{(a)}{\leq}\O\left(\eta\left(\cm^{2}\ln^{2}\frac{3}{\delta}+\sigma_{\lar}^{\p}\cm^{2-\p}T+\frac{\sigma_{\sma}^{2}\sigma_{\lar}^{2\p-2}}{\cm^{2\p-2}}T^{2}+\frac{\sigma_{\lar}^{2\p}}{\cm^{2\p-2}}T^{2}+G^{2}T\right)\right)\\
 & \overset{(b)}{=}\O\left(\eta\left(\cm^{2}\ln^{2}\frac{3}{\delta}+\sigma_{\lar}^{\p}\cm^{2-\p}T+\frac{\sigma_{\lar}^{2\p}}{\cm^{2\p-2}}T^{2}+G^{2}T\right)\right),
\end{align*}
where $(a)$ is by $\cm\geq\frac{G}{1-\alpha}=2G$ and $(b)$ holds
due to $\sigma_{\sma}\leq\sigma_{\lar}$. Under the choice of $\eta=\min\left\{ \frac{\lambda}{G\sqrt{T}},\frac{\lambda}{\cm}\right\} $
used in Theorem 3 of \citet{liu2023stochastic}, we have
\[
\hc_{T}^{\cvx}\leq\O\left(\frac{\lambda^{2}}{\eta}\left(\ln^{2}\frac{3}{\delta}+\frac{\sigma_{\lar}^{\p}}{\cm^{\p}}T+\frac{\sigma_{\lar}^{2\p}}{\cm^{2\p}}T^{2}+1\right)\right)\leq\O\left(\frac{\lambda^{2}}{\eta}\left(\ln^{2}\frac{3}{\delta}+\frac{\sigma_{\lar}^{2\p}}{\cm^{2\p}}T^{2}\right)\right),
\]
where the second step is due to $\frac{2\sigma_{\lar}^{\p}}{\cm^{\p}}T\leq\frac{\sigma_{\lar}^{2\p}}{\cm^{2\p}}T^{2}+1$
(by AM-GM inequality) and $1\leq\ln^{2}\frac{3}{\delta}$. Lastly,
we replace $\cm$ with $\max\left\{ 2G,\cm T^{\frac{1}{\p}}\right\} $
given in Theorem 3 of \citet{liu2023stochastic} to obtain
\[
\hc_{T}^{\cvx}=\O\left(\frac{\lambda^{2}}{\eta}\left(\ln^{2}\frac{3}{\delta}+\frac{\sigma_{\lar}^{2\p}}{\cm^{2\p}}\right)\right).
\]
Combine with (\ref{eq:cvx-hp-dep-T-1}) to finally have
\[
F(\bar{\bx}_{T+1}^{\cvx})-F_{\star}\leq\O\left(\frac{D^{2}+\lambda^{2}\left(\ln^{2}\frac{3}{\delta}+\frac{\sigma_{\lar}^{2\p}}{\cm^{2\p}}\right)}{\eta T}\right),
\]
which is the same rate as given in \citet{liu2023stochastic} (see
their equation (7)), implying that our analysis is indeed more refined
than \citet{liu2023stochastic}.

\textbf{Unknown $T$.} We move to the case of unknown $T$. Theorem
\ref{thm:cvx-hp-dep-t} in the following gives the anytime high-probability
rate for Stabilized Clipped SGD.
\begin{thm}
\label{thm:cvx-hp-dep-t}Under Assumptions \ref{assu:minimizer},
\ref{assu:obj} (with $\mu=0$), \ref{assu:lip} and \ref{assu:oracle},
for any $T\in\N$ and $\delta\in\left(0,1\right]$, setting $\eta_{t}=\min\left\{ \gamma_{\star},\frac{\eta_{\star}}{\sqrt{t}},\frac{\lambda_{\star}}{\cm_{\star}t^{\frac{1}{\p}}}\right\} ,\cm_{t}=\max\left\{ \frac{G}{1-\alpha},\cm_{\star}t^{\frac{1}{\p}}\right\} ,\forall t\in\left[T\right]$
where $\alpha=1/2$,
\begin{eqnarray}
\gamma_{\star}=\frac{D/G}{\varphi_{\star}\psi_{\star}+\ln\frac{3}{\delta}}, & \eta_{\star}=D/G, & \lambda_{\star}=\frac{D}{\sqrt{\ln^{2}\frac{3}{\delta}+\frac{\sigma_{\lar}^{\p}}{\cm_{\star}^{\p}}+\frac{\sigma_{\sma}^{2}\sigma_{\lar}^{2\p-2}}{\cm_{\star}^{2\p}}}},\label{eq:cvx-hp-dep-t-choice}
\end{eqnarray}
and $\psi_{\star}\defeq1+\ln\varphi_{\star}$, then Stabilized Clipped
SGD (Algorithm \ref{alg:stabilized-clipped-SGD}) guarantees that
with probability at least $1-\delta$, $F(\bar{\bx}_{T+1}^{\cvx})-F_{\star}$
converges at the rate of 
\[
\O\left(\frac{\left(\varphi_{\star}\psi_{\star}+\ln\frac{3}{\delta}\right)GD}{T}+\frac{GD}{\sqrt{T}}+\frac{\left(\sigma_{\sma}^{\frac{2}{\p}-1}\sigma_{\lar}^{2-\frac{2}{\p}}+\sigma_{\sma}^{\frac{1}{\p}}\sigma_{\lar}^{1-\frac{1}{\p}}\ln^{1-\frac{1}{\p}}\frac{3}{\delta}\right)D}{T^{1-\frac{1}{\p}}}\right).
\]
\end{thm}
\begin{proof}
By the same argument for (\ref{eq:cvx-hp-dep-T-1}) in the proof of
Theorem \ref{thm:cvx-hp-dep-T}, we have with probability at least
$1-\delta$,
\begin{equation}
F(\bar{\bx}_{T+1}^{\cvx})-F_{\star}\leq\frac{D^{2}}{\eta_{T+1}T}+\frac{2\hc_{T}^{\cvx}}{T},\label{eq:cvx-hp-dep-t-1}
\end{equation}
where $\hc_{T}^{\cvx}$ is a constant in the order of
\begin{equation}
\O\left(\underbrace{\max_{t\in\left[T\right]}\eta_{t}\cm_{t}^{2}\ln^{2}\frac{3}{\delta}}_{\mathrm{I}}+\underbrace{\sum_{t=1}^{T}\sigma_{\lar}^{\p}\eta_{t}\cm_{t}^{2-\p}}_{\mathrm{II}}+\left(\underbrace{\sum_{t=1}^{T}\frac{\sigma_{\sma}\sigma_{\lar}^{\p-1}\sqrt{\eta_{t}}}{\cm_{t}^{\p-1}}}_{\mathrm{III}}+\underbrace{\sum_{t=1}^{T}\frac{\sigma_{\lar}^{\p}G\sqrt{\eta_{t}}}{\alpha^{\p-1}\cm_{t}^{\p}}}_{\mathrm{IV}}\right)^{2}+\underbrace{\sum_{t=1}^{T}G^{2}\eta_{t}}_{\mathrm{V}}\right).\label{eq:cvx-hp-dep-t-C}
\end{equation}
When $\eta_{t}=\min\left\{ \gamma,\frac{\eta}{\sqrt{t}},\frac{\lambda}{\cm t^{\frac{1}{\p}}}\right\} ,\cm_{t}=\max\left\{ \frac{G}{1-\alpha},\cm t^{\frac{1}{\p}}\right\} ,\forall t\in\left[T\right]$
for nonnegative $\gamma$, $\eta$, $\lambda$ and $\cm$, we can
bound the above five terms as follows.
\begin{itemize}
\item Term $\mathrm{I}$. We have
\begin{align}
 & \max_{t\in\left[T\right]}\eta_{t}\cm_{t}^{2}\ln^{2}\frac{3}{\delta}\leq\max_{t\in\left[T\right]}\left(\frac{\eta_{t}G^{2}}{(1-\alpha)^{2}}+\eta_{t}\left(\cm t^{\frac{1}{\p}}\right)^{2}\right)\ln^{2}\frac{3}{\delta}\nonumber \\
\leq & \max_{t\in\left[T\right]}\left(\frac{\gamma G^{2}}{(1-\alpha)^{2}}+\lambda\cm t^{\frac{1}{\p}}\right)\ln^{2}\frac{3}{\delta}=\O\left(\frac{\gamma G^{2}\ln^{2}\frac{3}{\delta}}{(1-\alpha)^{2}}+\lambda\cm\ln^{2}\left(\frac{3}{\delta}\right)T^{\frac{1}{\p}}\right).\label{eq:cvx-hp-dep-t-I}
\end{align}
\item Term $\mathrm{II}$. For any $t\in\left[T\right]$, we have
\[
\sigma_{\lar}^{\p}\eta_{t}\cm_{t}^{2-\p}\leq\frac{\sigma_{\lar}^{\p}G^{2}\eta_{t}}{(1-\alpha)^{2}\cm_{t}^{\p}}+\sigma_{\lar}^{\p}\left(\cm t^{\frac{1}{\p}}\right)^{2-\p}\eta_{t}\leq\frac{\sigma_{\lar}^{\p}G^{2}\sqrt{\gamma\eta_{t}}}{(1-\alpha)^{2}\cm_{t}^{\p}}+\frac{\sigma_{\lar}^{\p}\lambda}{\left(\cm t^{\frac{1}{\p}}\right)^{\p-1}},
\]
which implies that
\begin{align}
\sum_{t=1}^{T}\sigma_{\lar}^{\p}\eta_{t}\cm_{t}^{2-\p} & \leq\frac{\sqrt{\gamma}G}{1-\alpha}\left(\sum_{t=1}^{T}\frac{\sigma_{\lar}^{\p}G\sqrt{\eta_{t}}}{(1-\alpha)\cm_{t}^{\p}}\right)+\sum_{t=1}^{T}\frac{\sigma_{\lar}^{\p}\lambda}{\left(\cm t^{\frac{1}{\p}}\right)^{\p-1}}\nonumber \\
 & \leq\frac{\gamma G^{2}}{4(1-\alpha)^{2}}+\left(\sum_{t=1}^{T}\frac{\sigma_{\lar}^{\p}G\sqrt{\eta_{t}}}{(1-\alpha)\cm_{t}^{\p}}\right)^{2}+\sum_{t=1}^{T}\frac{\sigma_{\lar}^{\p}\lambda}{\left(\cm t^{\frac{1}{\p}}\right)^{\p-1}}\nonumber \\
 & \leq\O\left(\frac{\gamma G^{2}\ln^{2}\frac{3}{\delta}}{(1-\alpha)^{2}}+\left(\frac{\alpha^{\p-1}}{1-\alpha}\cdot\text{Term }\mathrm{IV}\right)^{2}+\frac{\lambda\sigma_{\lar}^{\p}}{\cm^{\p-1}}T^{\frac{1}{\p}}\right).\label{eq:cvx-hp-dep-t-II}
\end{align}
\item Term $\mathrm{III}$. For any $t\in\left[T\right]$, we have
\[
\frac{\sqrt{\eta_{t}}}{\cm_{t}^{\p-1}}\overset{\p\geq1}{\leq}\frac{\sqrt{\lambda/(\cm t^{\frac{1}{\p}})}}{(\cm t^{\frac{1}{\p}})^{\p-1}}=\frac{\sqrt{\lambda}}{(\cm t^{\frac{1}{\p}})^{\p-\frac{1}{2}}},
\]
which implies
\begin{equation}
\sum_{t=1}^{T}\frac{\sigma_{\sma}\sigma_{\lar}^{\p-1}\sqrt{\eta_{t}}}{\cm_{t}^{\p-1}}\leq\O\left(\frac{\sqrt{\lambda}\sigma_{\sma}\sigma_{\lar}^{\p-1}}{\cm^{\p-\frac{1}{2}}}T^{\frac{1}{2\p}}\right).\label{eq:cvx-hp-dep-t-III}
\end{equation}
\item Term $\mathrm{IV}$. For any $\beta\in\left[0,1\right]$, we have
\begin{equation}
\sum_{t=1}^{T}\frac{\sigma_{\lar}^{\p}G\sqrt{\eta_{t}}}{\alpha^{\p-1}\cm_{t}^{\p}}\leq\sum_{t=1}^{T}\frac{\sigma_{\lar}^{\p}G\gamma^{\frac{1-\beta}{2}}\eta^{\frac{\beta}{2}}}{\alpha^{\p-1}(\cm t^{\frac{1}{\p}})^{\p}t^{\frac{\beta}{4}}}=\O\left(\frac{\sqrt{\gamma}\sigma_{\lar}^{\p}G}{\alpha^{\p-1}\cm^{\p}}\left(\frac{\eta}{\gamma}\right)^{\frac{\beta}{2}}\psi(\beta,T)\right),\label{eq:cvx-hp-dep-t-IV}
\end{equation}
where 
\begin{equation}
\psi(\beta,T)\defeq\begin{cases}
1+\ln T & \beta=0\\
1+\frac{4}{\beta} & \beta\in\left(0,1\right]
\end{cases}.\label{eq:cvx-hp-dep-t-psi}
\end{equation}
\item Term $\mathrm{V}$. We have
\begin{equation}
\sum_{t=1}^{T}G^{2}\eta_{t}\leq\sum_{t=1}^{T}\frac{\eta G^{2}}{\sqrt{t}}=\O\left(\eta G^{2}\sqrt{T}\right).\label{eq:cvx-hp-dep-t-V}
\end{equation}
\end{itemize}
We plug (\ref{eq:cvx-hp-dep-t-I}), (\ref{eq:cvx-hp-dep-t-II}), (\ref{eq:cvx-hp-dep-t-III}),
(\ref{eq:cvx-hp-dep-t-IV}) and (\ref{eq:cvx-hp-dep-t-V}) back into
(\ref{eq:cvx-hp-dep-t-C}) to know
\begin{align*}
\hc_{T}^{\cvx}\leq & \O\left(\gamma\left(\frac{\sigma_{\lar}^{2\p}G^{2}(\eta/\gamma)^{\beta}\psi^{2}(\beta,T)}{\min\left\{ \alpha^{2\p-2},(1-\alpha)^{2}\right\} \cm^{2\p}}+\frac{G^{2}\ln^{2}\frac{3}{\delta}}{(1-\alpha)^{2}}\right)+\eta G^{2}\sqrt{T}\right.\\
 & \left.\quad+\lambda\left(\cm\ln^{2}\frac{3}{\delta}+\frac{\sigma_{\lar}^{\p}}{\cm^{\p-1}}+\frac{\sigma_{\sma}^{2}\sigma_{\lar}^{2\p-2}}{\cm^{2\p-1}}\right)T^{\frac{1}{\p}}\right),\forall\beta\in\left[0,1\right].
\end{align*}
Combine the above result with $\eta_{t}=\min\left\{ \gamma,\frac{\eta}{\sqrt{t}},\frac{\lambda}{\cm t^{\frac{1}{\p}}}\right\} $
and (\ref{eq:cvx-hp-dep-t-1}) to obtain
\begin{align}
F(\bar{\bx}_{T+1}^{\cvx})-F_{\star}\leq & \O\left(\frac{\frac{D^{2}}{\gamma}+\gamma\left(\frac{\sigma_{\lar}^{2\p}G^{2}(\eta/\gamma)^{\beta}\psi^{2}(\beta,T)}{\min\left\{ \alpha^{2\p-2},(1-\alpha)^{2}\right\} \cm^{2\p}}+\frac{G^{2}\ln^{2}\frac{3}{\delta}}{(1-\alpha)^{2}}\right)}{T}+\frac{\frac{D^{2}}{\eta}+\eta G^{2}}{\sqrt{T}}\right.\nonumber \\
 & \left.\quad+\frac{\frac{D^{2}\cm}{\lambda}+\lambda\left(\cm\ln^{2}\left(\frac{3}{\delta}\right)+\frac{\sigma_{\lar}^{\p}}{\cm^{\p-1}}+\frac{\sigma_{\sma}^{2}\sigma_{\lar}^{2\p-2}}{\cm^{2\p-1}}\right)}{T^{1-\frac{1}{\p}}}\right),\forall\beta\in\left[0,1\right].\label{eq:cvx-hp-dep-t-2}
\end{align}
Finally, we conclude after plugging in $\cm=\cm_{\star}$, $\gamma=\gamma_{\star}$,
$\eta=\eta_{\star}$, $\lambda=\lambda_{\star}$ (see (\ref{eq:hp-tau-star})
and (\ref{eq:cvx-hp-dep-t-choice})), $\alpha=1/2$, and the following
fact:
\begin{align*}
\inf_{\beta\in\left[0,1\right]}\gamma_{\star}\left(\frac{\eta_{\star}}{\gamma_{\star}}\right)^{\beta}\psi^{2}(\beta,T) & \overset{(\ref{eq:cvx-hp-dep-t-choice}),\beta\leq1}{\leq}\frac{D/G}{\varphi_{\star}\psi_{\star}}\inf_{\beta\in\left[0,1\right]}\left(\varphi_{\star}\psi_{\star}\right)^{\beta}\psi^{2}(\beta,T)\\
 & \leq\frac{D/G}{\varphi_{\star}\psi_{\star}}\left(\varphi_{\star}\psi_{\star}\right)^{\beta_{\star}}\psi^{2}(\beta_{\star},T)\quad\text{where}\quad\beta_{\star}=\frac{2}{\max\left\{ \ln\left(\varphi_{\star}\psi_{\star}\right),2\right\} }\\
 & \overset{(\ref{eq:cvx-hp-dep-t-psi})}{\leq}\frac{D/G}{\varphi_{\star}\psi_{\star}}\cdot e^{2}\cdot\left(1+2\max\left\{ \ln\left(\varphi_{\star}\psi_{\star}\right),2\right\} \right)^{2}\\
 & =\O\left(\frac{D/G}{\varphi_{\star}\psi_{\star}}\cdot\left(1+\ln^{2}\varphi_{\star}+\ln^{2}\psi_{\star}\right)\right)=\O\left(\frac{D/G}{\varphi_{\star}}\cdot\psi_{\star}\right),
\end{align*}
where the last step is by $\ln\psi_{\star}\leq2\sqrt{\psi_{\star}}$,
$1+\ln^{2}\varphi_{\star}\leq\psi_{\star}^{2}$ (since $\psi_{\star}=1+\ln\varphi_{\star}$
and $\varphi_{\star}\geq1$), and $\psi_{\star}\geq1$.
\end{proof}

We first compare Theorem \ref{thm:cvx-hp-dep-t} with our Theorem
\ref{thm:cvx-hp-dep-T}. As one can see, the only difference is the
term $\varphi$ versus the term $\varphi_{\star}\psi_{\star}$, the
former of which satisfies $\varphi\leq\varphi_{\star}$. This change
should be expected as the precise value of $\varphi$ depends on $T$
(see (\ref{eq:varphi})). Moreover, recall that $\varphi=\varphi_{\star}$
once $T$ exceeds $\Omega\left(\frac{G^{\p}}{\sigma_{\lar}^{\p}}\varphi_{\star}\right)$.
Hence, roughly speaking, the only loss in Theorem \ref{thm:cvx-hp-dep-t}
is an extra multiplicative term $\psi_{\star}$, which never grows
with $T$ and is in the order of
\[
1+\ln\varphi_{\star}\overset{(\ref{eq:hp-varphi-star})}{=}1+\ln\left(\max\left\{ \sqrt{d_{\eff}}\ln\frac{3}{\delta},d_{\eff}\1\left[\p<2\right]\right\} \right).
\]
This positive result, i.e., no extra $\poly(\ln T)$ term, is due
to the stabilization technique, as discussed in Appendix \ref{sec:stabilized}.

Without considering the extra stabilized step, following a similar
analysis given in Appendix \ref{sec:analysis} later, one can show
that for any general stepsize $\eta_{t}$ and any clipping threshold
$\cm_{t}\ge\frac{G}{1-\alpha}$, Clipped SGD guarantees with probability
at least $1-\delta$ (assuming that $\eta_{t}$ is nonincreasing for
simplicity),
\begin{equation}
F(\bar{\bx}_{T+1}^{\cvx})-F_{\star}\leq\left(\frac{D^{2}+\tilde{\hc}_{T}^{\cvx}}{\eta_{T}T}\right),\label{eq:clipped-SGD-rate}
\end{equation}
where $\tilde{\hc}_{T}^{\cvx}$ is in the order of
\begin{equation}
\O\left(\max_{t\in\left[T\right]}\eta_{t}^{2}\cm_{t}^{2}\ln^{2}\frac{3}{\delta}+\sum_{t=1}^{T}\sigma_{\lar}^{\p}\eta_{t}^{2}\cm_{t}^{2-\p}+\left(\sum_{t=1}^{T}\frac{\sigma_{\sma}\sigma_{\lar}^{\p-1}\eta_{t}}{\cm_{t}^{\p-1}}+\sum_{t=1}^{T}\frac{\sigma_{\lar}^{\p}G\eta_{t}}{\alpha^{\p-1}\cm_{t}^{\p}}\right)^{2}+\sum_{t=1}^{T}G^{2}\eta_{t}^{2}\right).\label{eq:clipped-SGD-C}
\end{equation}
As a sanity check, when $\eta_{t}=\eta,\cm_{t}=\cm,\forall t\in\left[T\right]$,
$\tilde{\hc}_{T}^{\cvx}/\eta$ coincides with $\hc_{T}^{\cvx}$ given
in (\ref{eq:cvx-hp-dep-T-C-simplified-1}). If $T$ is unknown, even
ignoring all other terms and only focusing on $\sum_{t=1}^{T}G^{2}\eta_{t}^{2}$
in (\ref{eq:clipped-SGD-C}), the final rate of Clipped SGD by (\ref{eq:clipped-SGD-rate})
will contain a term $\sum_{t=1}^{T}G^{2}\eta_{t}^{2}/(\eta_{T}T)$,
which is however well-known to give an extra $\poly(\ln T)$ factor
for a time-varying stepsize $\eta_{t}$.

Now let us compare Theorem \ref{thm:cvx-hp-dep-t} to Theorem 1 in
\citet{liu2023stochastic}. The latter gives the current best anytime
rate for Clipped SGD as follows (actually, this can be obtained by
(\ref{eq:clipped-SGD-rate}) and (\ref{eq:clipped-SGD-C}) above):
\[
F(\bar{\bx}_{T+1}^{\cvx})-F_{\star}\leq\O\left(\left(\ln\frac{1}{\delta}+\ln^{2}T\right)\left(\frac{GD}{\sqrt{T}}+\frac{\sigma_{\lar}D}{T^{1-\frac{1}{\p}}}\right)\right).
\]
Similar to our comparison when $T$ is known in Section \ref{sec:hp-rates},
one can see that our Theorem \ref{thm:cvx-hp-dep-t} is better (at
least in the case of large $T$).

\subsubsection{In-Expectation Convergence}

\textbf{Known $T$.} Now we consider the in-expectation convergence.
Theorem \ref{thm:cvx-ex-dep-T} gives the first rate $\O(\sigma_{\lar}d_{\eff}^{\frac{1}{2}-\frac{1}{\p}}DT^{\frac{1}{\p}-1})$
faster than the existing lower bound $\Omega(\sigma_{\lar}DT^{\frac{1}{\p}-1})$
\citep{nemirovskij1983problem,pmlr-v178-vural22a}.
\begin{thm}[Full statement of Theorem \ref{thm:main-cvx-ex-dep-T}]
\label{thm:cvx-ex-dep-T}Under Assumptions \ref{assu:minimizer},
\ref{assu:obj} (with $\mu=0$), \ref{assu:lip} and \ref{assu:oracle},
for any $T\in\N$, setting $\eta_{t}=\eta_{\star},\cm_{t}=\max\left\{ \frac{G}{1-\alpha},\widetilde{\cm}_{\star}T^{\frac{1}{\p}}\right\} ,\forall t\in\left[T\right]$
where $\alpha=1/2$,
\begin{equation}
\eta_{\star}=\min\left\{ \frac{D/G}{\widetilde{\varphi}},\frac{D/G}{\sqrt{T}},\frac{D}{\sigma_{\sma}^{\frac{2}{\p}-1}\sigma_{\lar}^{2-\frac{2}{\p}}T^{\frac{1}{\p}}}\right\} ,\label{eq:cvx-ex-dep-T-eta-star}
\end{equation}
and $\widetilde{\varphi}\leq\widetilde{\varphi}_{\star}$ is a constant
defined in (\ref{eq:varphi-tilde}) and equals $\widetilde{\varphi}_{\star}$
when $T=\Omega\left(\frac{G^{\p}}{\sigma_{\lar}^{\p}}\widetilde{\varphi}_{\star}\right)$,
then Clipped SGD (Algorithm \ref{alg:clipped-SGD}) guarantees that
$\E\left[F(\bar{\bx}_{T+1}^{\cvx})-F_{\star}\right]$ converges at
the rate of
\[
\O\left(\frac{\widetilde{\varphi}GD}{T}+\frac{GD}{\sqrt{T}}+\frac{\sigma_{\sma}^{\frac{2}{\p}-1}\sigma_{\lar}^{2-\frac{2}{\p}}D}{T^{1-\frac{1}{\p}}}\right).
\]
\end{thm}
\begin{proof}
By Lemmas \ref{lem:cvx-ex-anytime} and \ref{lem:cvx-dep-res}, we
can follow a similar argument until (\ref{eq:cvx-hp-dep-T-C-simplified-2})
in the proof of Theorem \ref{thm:cvx-hp-dep-T} to have
\[
\E\left[F(\bar{\bx}_{T+1}^{\cvx})-F_{\star}\right]\leq\frac{D^{2}}{\eta_{T+1}T}+\frac{2\ec_{T}^{\cvx}}{T},
\]
where, under $\eta_{t}=\eta,\cm_{t}=\max\left\{ \frac{G}{1-\alpha},\cm T^{\frac{1}{\p}}\right\} ,\forall t\in\left[T\right]$
for $\eta,\cm>0$,
\begin{align*}
\ec_{T}^{\cvx}\leq & \O\left(\eta\left(\inf_{\beta\in\left[0,1/2\right]}\frac{(1-\alpha)^{2\beta\p}\sigma_{\lar}^{2\p}G^{2-2\beta\p}T^{2\beta}}{\min\left\{ \alpha^{2\p-2},(1-\alpha)^{2}\right\} \cm^{2(1-\beta)\p}}+\frac{G^{2}}{(1-\alpha)^{2}}+G^{2}T\right)\right.\\
 & \left.\quad+\eta\left(\sigma_{\lar}^{\p}\cm^{2-\p}+\frac{\sigma_{\sma}^{2}\sigma_{\lar}^{2\p-2}}{\cm^{2\p-2}}\right)T^{\frac{2}{\p}}\right).
\end{align*}

Now, we plug $\cm=\widetilde{\cm}_{\star}$ (see (\ref{eq:ex-tau-star}))
into the above inequality to have under the choice of $\cm_{t}=\max\left\{ \frac{G}{1-\alpha},\widetilde{\cm}_{\star}T^{\frac{1}{\p}}\right\} ,\forall t\in\left[T\right]$,
\[
\ec_{T}^{\cvx}\leq\O\left(\eta\left(\frac{G^{2}\widetilde{\varphi}^{2}}{\min\left\{ \alpha^{2\p-2},(1-\alpha)^{2}\right\} }+\frac{G^{2}}{(1-\alpha)^{2}}+G^{2}T+\sigma_{\sma}^{\frac{4}{\p}-2}\sigma_{\lar}^{4-\frac{4}{\p}}T^{\frac{2}{\p}}\right)\right),
\]
where the first term is obtained by noticing
\begin{align*}
\frac{(1-\alpha)^{2\beta\p}\sigma_{\lar}^{2\p}G^{2-2\beta\p}T^{2\beta}}{\widetilde{\cm}_{\star}^{2(1-\beta)\p}} & =G^{2}\cdot\left((1-\alpha)^{\beta\p}\frac{\sigma_{\lar}^{(1-\beta)\p}}{\widetilde{\cm}_{\star}^{(1-\beta)\p}}\left(\frac{\sigma_{\lar}}{G}\right)^{\beta\p}T^{\beta}\right)^{2}\\
 & \overset{(\ref{eq:ex-varphi-tau-equation})}{=}G^{2}\cdot\left((1-\alpha)^{\beta\p}\widetilde{\varphi}_{\star}^{1-\beta}\left(\frac{\sigma_{\lar}}{G}\right)^{\beta\p}T^{\beta}\right)^{2}\\
\Rightarrow\inf_{\beta\in\left[0,1/2\right]}\frac{(1-\alpha)^{2\beta\p}\sigma_{\lar}^{2\p}G^{2-2\beta\p}T^{2\beta}}{\widetilde{\cm}_{\star}^{2(1-\beta)\p}} & \leq G^{2}\widetilde{\varphi}^{2},
\end{align*}
in which
\begin{equation}
\widetilde{\varphi}\defeq\inf_{\beta\in\left[0,1/2\right]}(1-\alpha)^{\beta\p}\widetilde{\varphi}_{\star}^{1-\beta}\left(\frac{\sigma_{\lar}}{G}\right)^{\beta\p}T^{\beta}=\min\left\{ \widetilde{\varphi}_{\star},\sqrt{(1-\alpha)^{\p}\widetilde{\varphi}_{\star}\left(\frac{\sigma_{\lar}}{G}\right)^{\p}T}\right\} \leq\widetilde{\varphi}_{\star}.\label{eq:varphi-tilde}
\end{equation}
Note that we have $\widetilde{\varphi}=\widetilde{\varphi}_{\star}$
when $T\geq\frac{G^{\p}\widetilde{\varphi}_{\star}}{(1-\alpha)^{\p}\sigma_{\lar}^{\p}}=\Omega\left(\frac{G^{\p}}{\sigma_{\lar}^{\p}}\widetilde{\varphi}_{\star}\right)$.

By the above results and $\alpha=1/2$, we find
\[
\E\left[F(\bar{\bx}_{T+1}^{\cvx})-F_{\star}\right]\leq\O\left(\frac{D^{2}}{\eta T}+\frac{\eta\widetilde{\varphi}^{2}G^{2}}{T}+\eta G^{2}+\eta\sigma_{\sma}^{\frac{4}{\p}-2}\sigma_{\lar}^{4-\frac{4}{\p}}T^{\frac{2}{\p}-1}\right).
\]
Plug in $\eta=\eta_{\star}$ (see (\ref{eq:cvx-ex-dep-T-eta-star}))
to conclude that $\E\left[F(\bar{\bx}_{T+1}^{\cvx})-F_{\star}\right]$
converges at the rate of
\[
\O\left(\frac{\widetilde{\varphi}GD}{T}+\frac{GD}{\sqrt{T}}+\frac{\sigma_{\sma}^{\frac{2}{\p}-1}\sigma_{\lar}^{2-\frac{2}{\p}}D}{T^{1-\frac{1}{\p}}}\right).
\]
\end{proof}

\textbf{Unknown $T$.} Next, we consider the in-expectation convergence
for Stabilized Clipped SGD. This anytime rate is also faster than
the lower bound $\Omega(\sigma_{\lar}DT^{\frac{1}{\p}-1})$.
\begin{thm}
\label{thm:cvx-ex-dep-t}Under Assumptions \ref{assu:minimizer},
\ref{assu:obj} (with $\mu=0$), \ref{assu:lip} and \ref{assu:oracle},
for any $T\in\N$, setting $\eta_{t}=\min\left\{ \gamma_{\star},\frac{\eta_{\star}}{\sqrt{t}},\frac{\lambda_{\star}}{\widetilde{\cm}_{\star}t^{\frac{1}{\p}}}\right\} ,\cm_{t}=\max\left\{ \frac{G}{1-\alpha},\widetilde{\cm}_{\star}t^{\frac{1}{\p}}\right\} ,\forall t\in\left[T\right]$
where $\alpha=1/2$,
\begin{eqnarray}
\gamma_{\star}=\frac{D/G}{\widetilde{\varphi}_{\star}\widetilde{\psi}_{\star}+1}, & \eta_{\star}=D/G, & \lambda_{\star}=\frac{D}{\sqrt{\frac{\sigma_{\lar}^{\p}}{\widetilde{\cm}_{\star}^{\p}}+\frac{\sigma_{\sma}^{2}\sigma_{\lar}^{2\p-2}}{\widetilde{\cm}_{\star}^{2\p}}}},\label{eq:cvx-ex-dep-t-choice}
\end{eqnarray}
and $\widetilde{\psi}_{\star}\defeq1+\ln\widetilde{\varphi}_{\star}$,
then Stabilized Clipped SGD (Algorithm \ref{alg:stabilized-clipped-SGD})
guarantees that $\E\left[F(\bar{\bx}_{T+1}^{\cvx})-F_{\star}\right]$
converges at the rate of 
\[
\O\left(\frac{\widetilde{\varphi}_{\star}\widetilde{\psi}_{\star}GD}{T}+\frac{GD}{\sqrt{T}}+\frac{\sigma_{\sma}^{\frac{2}{\p}-1}\sigma_{\lar}^{2-\frac{2}{\p}}D}{T^{1-\frac{1}{\p}}}\right).
\]
\end{thm}
\begin{proof}
By Lemmas \ref{lem:cvx-ex-anytime} and \ref{lem:cvx-dep-res}, we
can follow a similar argument until (\ref{eq:cvx-hp-dep-t-2}) in
the proof of Theorem \ref{thm:cvx-hp-dep-t} to have when $\eta_{t}=\min\left\{ \gamma,\frac{\eta}{\sqrt{t}},\frac{\lambda}{\cm t^{\frac{1}{\p}}}\right\} ,\cm_{t}=\max\left\{ \frac{G}{1-\alpha},\cm t^{\frac{1}{\p}}\right\} ,\forall t\in\left[T\right]$,
\begin{align*}
F(\bar{\bx}_{T+1}^{\cvx})-F_{\star}\leq & \O\left(\frac{\frac{D^{2}}{\gamma}+\gamma\left(\frac{\sigma_{\lar}^{2\p}G^{2}(\eta/\gamma)^{\beta}\psi^{2}(\beta,T)}{\min\left\{ \alpha^{2\p-2},(1-\alpha)^{2}\right\} \cm^{2\p}}+\frac{G^{2}}{(1-\alpha)^{2}}\right)}{T}+\frac{\frac{D^{2}}{\eta}+\eta G^{2}}{\sqrt{T}}\right.\\
 & \left.\quad+\frac{\frac{D^{2}\cm}{\lambda}+\lambda\left(\frac{\sigma_{\lar}^{\p}}{\cm^{\p-1}}+\frac{\sigma_{\sma}^{2}\sigma_{\lar}^{2\p-2}}{\cm^{2\p-1}}\right)}{T^{1-\frac{1}{\p}}}\right),\forall\beta\in\left[0,1\right],
\end{align*}
where $\psi(\beta,T)=\begin{cases}
1+\ln T & \beta=0\\
1+\frac{4}{\beta} & \beta\in\left(0,1\right]
\end{cases}$ is defined in (\ref{eq:cvx-hp-dep-t-psi}).

Finally, we conclude after plugging in $\cm=\widetilde{\cm}_{\star}$,
$\gamma=\gamma_{\star}$, $\eta=\eta_{\star}$, $\lambda=\lambda_{\star}$
(see (\ref{eq:ex-tau-star}) and (\ref{eq:cvx-ex-dep-t-choice})),
$\alpha=1/2$, and the following fact:
\begin{align*}
\inf_{\beta\in\left[0,1\right]}\gamma_{\star}\left(\frac{\eta_{\star}}{\gamma_{\star}}\right)^{\beta}\psi^{2}(\beta,T) & \overset{(\ref{eq:cvx-ex-dep-t-choice})}{\leq}\frac{D/G}{\widetilde{\varphi}_{\star}\widetilde{\psi}_{\star}}\inf_{\beta\in\left[0,1\right]}\left(\widetilde{\varphi}_{\star}\widetilde{\psi}_{\star}\right)^{\beta}\psi^{2}(\beta,T)\\
 & \leq\frac{D/G}{\widetilde{\varphi}_{\star}\widetilde{\psi}_{\star}}\left(\widetilde{\varphi}_{\star}\widetilde{\psi}_{\star}\right)^{\beta_{\star}}\psi^{2}(\beta_{\star},T)\quad\text{where}\quad\beta_{\star}=\frac{2}{\max\left\{ \ln\left(\widetilde{\varphi}_{\star}\widetilde{\psi}_{\star}\right),2\right\} }\\
 & \leq\frac{D/G}{\widetilde{\varphi}_{\star}\widetilde{\psi}_{\star}}\cdot e^{2}\cdot\left(1+2\max\left\{ \ln\left(\widetilde{\varphi}_{\star}\widetilde{\psi}_{\star}\right),2\right\} \right)^{2}\\
 & =\O\left(\frac{D/G}{\widetilde{\varphi}_{\star}\widetilde{\psi}_{\star}}\cdot\left(1+\ln^{2}\widetilde{\varphi}_{\star}+\ln^{2}\widetilde{\psi}_{\star}\right)\right)=\O\left(\frac{D/G}{\widetilde{\varphi}_{\star}}\cdot\widetilde{\psi}_{\star}\right),
\end{align*}
where the last step is by $\ln\widetilde{\psi}_{\star}\leq2\sqrt{\widetilde{\psi}_{\star}}$,
$1+\ln^{2}\widetilde{\varphi}_{\star}\leq\widetilde{\psi}_{\star}^{2}$
(since $\widetilde{\psi}_{\star}=1+\ln\widetilde{\varphi}_{\star}$
and $\widetilde{\varphi}_{\star}\geq1$), and $\widetilde{\psi}_{\star}\geq1$.
\end{proof}

Compared to Theorem \ref{thm:cvx-ex-dep-T}, we only incur an extra
multiplicative term $\widetilde{\psi}_{\star}=1+\ln\widetilde{\varphi}_{\star}=1+\ln\left(d_{\eff}\1\left[\p<2\right]\right)$
in the higher-order $\O(T^{-1})$ part.

\subsection{Strongly Convex Case}

We turn our attention to strongly convex objectives. In this setting,
we recall that $\bar{\bx}_{T+1}^{\str}$ denotes the following weighted
average iterate after $T$ steps:
\begin{equation}
\bar{\bx}_{T+1}^{\str}=\frac{\sum_{t=1}^{T}(t+4)(t+5)\bx_{t+1}}{\sum_{t=1}^{T}(t+4)(t+5)}.\label{eq:str-avg-x}
\end{equation}

\subsubsection{High-Probability Convergence}

Still, we first consider the high-probability convergence rate. Theorem
\ref{thm:str-hp-dep} gives the anytime high-probability rate of Clipped
SGD improving upon \citet{liu2023stochastic}.
\begin{thm}[Full statement of Theorem \ref{thm:main-str-hp-dep}]
\label{thm:str-hp-dep}Under Assumptions \ref{assu:minimizer}, \ref{assu:obj}
(with $\mu>0$), \ref{assu:lip} and \ref{assu:oracle}, for any $T\in\N$
and $\delta\in\left(0,1\right]$, setting $\eta_{t}=\frac{6}{\mu t},\cm_{t}=\max\left\{ \frac{G}{1-\alpha},\cm_{\star}t^{\frac{1}{\p}}\right\} ,\forall t\in\left[T\right]$
where $\alpha=1/2$, then Clipped SGD (Algorithm \ref{alg:clipped-SGD})
guarantees that with probability at least $1-\delta$, both $F(\bar{\bx}_{T+1}^{\str})-F_{\star}$
and $\mu\left\Vert \bx_{T+1}-\bx_{\star}\right\Vert ^{2}$ converge
at the rate of
\[
\O\left(\frac{\mu D^{2}}{T^{3}}+\frac{\left(\varphi^{2}+\ln^{2}\frac{3}{\delta}\right)G^{2}}{\mu T^{2}}+\frac{G^{2}}{\mu T}+\frac{\sigma_{\sma}^{\frac{4}{\p}-2}\sigma_{\lar}^{4-\frac{4}{\p}}+\sigma_{\sma}^{\frac{2}{\p}}\sigma_{\lar}^{2-\frac{2}{\p}}\ln^{2-\frac{2}{\p}}\frac{3}{\delta}}{\mu T^{2-\frac{2}{\p}}}\right),
\]
where $\varphi\leq\varphi_{\star}$ is a constant defined in (\ref{eq:varphi})
and equals $\varphi_{\star}$ when $T=\Omega\left(\frac{G^{\p}}{\sigma_{\lar}^{\p}}\varphi_{\star}\right)$.
\end{thm}
\begin{proof}
First, the choice of $\eta_{t}=\frac{6}{\mu t},\forall t\in\left[T\right]$
satisfies $\eta_{t}\leq\frac{\eta}{\mu},\forall t\in\left[T\right]$
for $\eta=6$, fulfilling the requirement of Lemma \ref{lem:str-hp-anytime}.
In addition, our choices of $\eta_{t}$ and $\cm_{t}$ also meet Conditions
\ref{enu:str-dep-res-1} and \ref{enu:str-dep-res-2} (with $\alpha=1/2$)
in Lemma \ref{lem:str-dep-res}. Therefore, given $T\in\N$ and $\delta\in\left(0,1\right]$,
Lemmas \ref{lem:str-hp-anytime} and \ref{lem:str-dep-res} together
yield that with probability at least $1-\delta$,
\begin{align}
\frac{\Gamma_{T+1}\left\Vert \bx_{\star}-\bx_{T+1}\right\Vert ^{2}}{2}+\sum_{t=1}^{T}\Gamma_{t}\eta_{t}\left(F(\bx_{t+1})-F_{\star}\right) & \leq4D^{2}+2\hc_{T}^{\str}\nonumber \\
\Rightarrow\frac{\Gamma_{T+1}\left\Vert \bx_{\star}-\bx_{T+1}\right\Vert ^{2}}{2\sum_{t=1}^{T}\Gamma_{t}\eta_{t}}+\frac{\sum_{t=1}^{T}\Gamma_{t}\eta_{t}\left(F(\bx_{t+1})-F_{\star}\right)}{\sum_{t=1}^{T}\Gamma_{t}\eta_{t}} & \leq\frac{4D^{2}+2\hc_{T}^{\str}}{\sum_{t=1}^{T}\Gamma_{t}\eta_{t}},\label{eq:str-hp-dep-1}
\end{align}
where $\Gamma_{t}=\prod_{s=2}^{t}\frac{1+\mu\eta_{s-1}}{1+\mu\eta_{s}/2}$
is introduced in (\ref{eq:str-Gamma}) and $\hc_{T}^{\str}$ is a
constant in the order of
\begin{align}
\O & \left(\max_{t\in\left[T\right]}\Gamma_{t}\eta_{t}^{2}\cm_{t}^{2}\ln^{2}\frac{3}{\delta}+\sum_{t=1}^{T}\sigma_{\lar}^{\p}\Gamma_{t}\eta_{t}^{2}\cm_{t}^{2-\p}+\sum_{t=1}^{T}\left(\sigma_{\sma}^{\p}\Gamma_{t}\eta_{t}^{2}\cm_{t}^{2-\p}+\frac{\sigma_{\lar}^{\p}G^{2}\Gamma_{t}\eta_{t}^{2}}{\alpha^{\p-1}\cm_{t}^{\p}}\right)\ln\frac{3}{\delta}\right.\nonumber \\
 & \left.\quad+\sum_{t=1}^{T}\left(\frac{\sigma_{\sma}^{2}\sigma_{\lar}^{2\p-2}\Gamma_{t}\eta_{t}}{\cm_{t}^{2\p-2}}+\frac{\sigma_{\lar}^{2\p}G^{2}\Gamma_{t}\eta_{t}}{\alpha^{2\p-2}\cm_{t}^{2\p}}\right)\frac{1}{\mu}+\sum_{t=1}^{T}G^{2}\Gamma_{t}\eta_{t}^{2}\right).\label{eq:str-hp-dep-C}
\end{align}

We use $\eta_{t}=\frac{6}{\mu t},\forall t\in\left[T\right]$ to compute
\begin{equation}
\Gamma_{t}=\prod_{s=2}^{t}\frac{1+\mu\eta_{s-1}}{1+\mu\eta_{s}/2}=\prod_{s=2}^{t}\frac{s}{s-1}\cdot\frac{s+5}{s+3}=\frac{t(t+4)(t+5)}{30},\forall t\in\left[T+1\right].\label{eq:str-hp-dep-Gamma}
\end{equation}
So for any $t\in\left[T\right]$,
\begin{eqnarray}
\Gamma_{t}\eta_{t}=\frac{(t+4)(t+5)}{5\mu}\leq\frac{6t^{2}}{\mu} & \text{and} & \Gamma_{t}\eta_{t}^{2}=\frac{6(t+4)(t+5)}{5\mu^{2}t}\leq\frac{36t}{\mu^{2}},\label{eq:str-hp-dep-Gamma-eta}
\end{eqnarray}
implying
\begin{equation}
\sum_{t=1}^{T}\Gamma_{t}\eta_{t}=\sum_{t=1}^{T}\frac{(t+4)(t+5)}{5\mu}=\frac{T(T^{2}+15T+74)}{15\mu}.\label{eq:str-hp-dep-Gamma-eta-sum}
\end{equation}

Lastly, let us bound (\ref{eq:str-hp-dep-1}). For the L.H.S. of (\ref{eq:str-hp-dep-1}),
we have
\[
\frac{\Gamma_{T+1}}{2\sum_{t=1}^{T}\Gamma_{t}\eta_{t}}\overset{(\ref{eq:str-hp-dep-Gamma}),(\ref{eq:str-hp-dep-Gamma-eta-sum})}{=}\frac{\mu(T+1)(T+5)(T+6)}{4T(T^{2}+15T+74)}\geq\min_{T\in\N}\frac{\mu(T+1)(T+5)(T+6)}{4T(T^{2}+15T+74)}=\frac{3\mu}{16}.
\]
In addition, we observe that
\[
\frac{\sum_{t=1}^{T}\Gamma_{t}\eta_{t}\bx_{t+1}}{\sum_{t=1}^{T}\Gamma_{t}\eta_{t}}\overset{(\ref{eq:str-hp-dep-Gamma-eta})}{=}\frac{\sum_{t=1}^{T}(t+4)(t+5)\bx_{t+1}}{\sum_{t=1}^{T}(t+4)(t+5)}\overset{(\ref{eq:str-avg-x})}{=}\bar{\bx}_{T+1}^{\str}.
\]
The above two results and the convexity of $F$ together lead us to
\begin{equation}
\frac{3\mu\left\Vert \bx_{\star}-\bx_{T+1}\right\Vert ^{2}}{16}+F(\bar{\bx}_{T+1}^{\str})-F_{\star}\leq\text{L.H.S. of }(\ref{eq:str-hp-dep-1}).\label{eq:str-hp-dep-lhs}
\end{equation}

For the R.H.S. of (\ref{eq:str-hp-dep-1}), we plug (\ref{eq:str-hp-dep-Gamma-eta-sum})
back into (\ref{eq:str-hp-dep-1}) to have
\begin{equation}
\text{R.H.S. of }(\ref{eq:str-hp-dep-1})\leq\O\left(\frac{\mu D^{2}+\mu\hc_{T}^{\str}}{T^{3}}\right).\label{eq:str-hp-dep-rhs-1}
\end{equation}
One more step, we use (\ref{eq:str-hp-dep-Gamma-eta}) to upper bound
(\ref{eq:str-hp-dep-C}) and obtain
\begin{align*}
\mu\hc_{T}^{\str}\leq & \frac{1}{\mu}\cdot\O\left(\max_{t\in\left[T\right]}\cm_{t}^{2}t\ln^{2}\frac{3}{\delta}+\sum_{t=1}^{T}\sigma_{\lar}^{\p}\cm_{t}^{2-\p}t+\sum_{t=1}^{T}\left(\sigma_{\sma}^{\p}\cm_{t}^{2-\p}t+\frac{\sigma_{\lar}^{\p}G^{2}t}{\alpha^{\p-1}\cm_{t}^{\p}}\right)\ln\frac{3}{\delta}\right.\\
 & \left.\qquad\quad+\sum_{t=1}^{T}\left(\frac{\sigma_{\sma}^{2}\sigma_{\lar}^{2\p-2}t^{2}}{\cm_{t}^{2\p-2}}+\frac{\sigma_{\lar}^{2\p}G^{2}t^{2}}{\alpha^{2\p-2}\cm_{t}^{2\p}}\right)+G^{2}T^{2}\right),
\end{align*}
When $\cm_{t}=\max\left\{ \frac{G}{1-\alpha},\cm t^{\frac{1}{\p}}\right\} ,\forall t\in\left[T\right]$,
we notice that for any $t\in\left[T\right]$,
\begin{align*}
\sigma_{\lar}^{\p}\cm_{t}^{2-\p}t & \leq\frac{\sigma_{\lar}^{\p}G^{2}t}{(1-\alpha)^{2}\cm_{t}^{\p}}+\sigma_{\lar}^{\p}\cm^{2-\p}t^{\frac{2}{\p}}\leq\frac{\sigma_{\lar}^{2\p}G^{2}t^{2}}{\min\left\{ \alpha^{2\p-2},(1-\alpha)^{2}\right\} \cm_{t}^{2\p}}+\frac{G^{2}}{4(1-\alpha)^{2}}+\sigma_{\lar}^{\p}\cm^{2-\p}t^{\frac{2}{\p}},\\
\sigma_{\sma}^{\p}\cm_{t}^{2-\p}t & \leq\frac{\sigma_{\sma}^{\p}G^{2}t}{(1-\alpha)^{2}\cm_{t}^{\p}}+\sigma_{\sma}^{\p}\cm^{2-\p}t^{\frac{2}{\p}}\leq\frac{\sigma_{\sma}^{2\p}G^{2}t^{2}}{\min\left\{ \alpha^{2\p-2},(1-\alpha)^{2}\right\} \cm_{t}^{2\p}\ln\frac{3}{\delta}}+\frac{G^{2}\ln\frac{3}{\delta}}{4(1-\alpha)^{2}}+\sigma_{\sma}^{\p}\cm^{2-\p}t^{\frac{2}{\p}},
\end{align*}
which implies that
\begin{align}
\mu\hc_{T}^{\str}\leq & \frac{1}{\mu}\cdot\O\left(\frac{G^{2}\ln^{2}\frac{3}{\delta}}{(1-\alpha)^{2}}T+\underbrace{\max_{t\in\left[T\right]}\cm_{t}^{2}t\ln^{2}\frac{3}{\delta}}_{\mathrm{I}}+\underbrace{\sum_{t=1}^{T}\sigma_{\lar}^{\p}\cm^{2-\p}t^{\frac{2}{\p}}}_{\mathrm{II}}+\underbrace{\sum_{t=1}^{T}\left(\sigma_{\sma}^{\p}\cm^{2-\p}t^{\frac{2}{\p}}+\frac{\sigma_{\lar}^{\p}G^{2}t}{\alpha^{\p-1}\cm_{t}^{\p}}\right)\ln\frac{3}{\delta}}_{\mathrm{III}}\right.\nonumber \\
 & \left.\qquad\quad+\underbrace{\sum_{t=1}^{T}\left(\frac{\sigma_{\sma}^{2}\sigma_{\lar}^{2\p-2}t^{2}}{\cm_{t}^{2\p-2}}+\frac{\sigma_{\lar}^{2\p}G^{2}t^{2}}{\min\left\{ \alpha^{2\p-2},(1-\alpha)^{2}\right\} \cm_{t}^{2\p}}\right)}_{\mathrm{IV}}+G^{2}T^{2}\right).\label{eq:str-hp-dep-C-simplified}
\end{align}
 We control the above four terms as follows.
\begin{itemize}
\item Term $\mathrm{I}$. We have
\begin{equation}
\max_{t\in\left[T\right]}\cm_{t}^{2}t\ln^{2}\frac{3}{\delta}=\cm_{T}^{2}T\ln^{2}\frac{3}{\delta}\leq\frac{G^{2}\ln^{2}\frac{3}{\delta}}{(1-\alpha)^{2}}T+\cm^{2}\ln^{2}\left(\frac{3}{\delta}\right)T^{1+\frac{2}{\p}}.\label{eq:str-hp-dep-I}
\end{equation}
\item Term $\mathrm{II}$. We have
\begin{equation}
\sum_{t=1}^{T}\sigma_{\lar}^{\p}\cm^{2-\p}t^{\frac{2}{\p}}\leq\sigma_{\lar}^{\p}\cm^{2-\p}T^{1+\frac{2}{\p}}.\label{eq:str-hp-dep-II}
\end{equation}
\item Term $\mathrm{III}$. We have
\[
\sum_{t=1}^{T}\sigma_{\sma}^{\p}\cm^{2-\p}t^{\frac{2}{\p}}\leq\sigma_{\sma}^{\p}\cm^{2-\p}T^{1+\frac{2}{\p}},
\]
and for any $\beta\in\left[0,1/2\right]$
\[
\sum_{t=1}^{T}\frac{\sigma_{\lar}^{\p}G^{2}t}{\cm_{t}^{\p}}\leq\sum_{t=1}^{T}\frac{\sigma_{\lar}^{\p}G^{2}t}{\left(\frac{G}{1-\alpha}\right)^{\beta\p}\left(\cm t^{\frac{1}{\p}}\right)^{(1-\beta)\p}}\leq\O\left(\frac{(1-\alpha)^{\beta\p}\sigma_{\lar}^{\p}G^{2-\beta\p}}{\cm^{(1-\beta)\p}}T^{1+\beta}\right).
\]
Thus, for any $\beta\in\left[0,1/2\right]$,
\begin{align}
 & \sum_{t=1}^{T}\left(\sigma_{\sma}^{\p}\cm_{t}^{2-\p}t+\frac{\sigma_{\lar}^{\p}G^{2}t}{\alpha^{\p-1}\cm_{t}^{\p}}\right)\ln\frac{3}{\delta}\nonumber \\
\leq & \O\left(\left(\sigma_{\sma}^{\p}\cm^{2-\p}T^{1+\frac{2}{\p}}+\frac{(1-\alpha)^{\beta\p}\sigma_{\lar}^{\p}G^{2-\beta\p}}{\alpha^{\p-1}\cm^{(1-\beta)\p}}T^{1+\beta}\right)\ln\frac{3}{\delta}\right).\label{eq:str-hp-dep-III}
\end{align}
\item Term $\mathrm{IV}$. We have
\[
\sum_{t=1}^{T}\frac{\sigma_{\sma}^{2}\sigma_{\lar}^{2\p-2}t^{2}}{\cm_{t}^{2\p-2}}\overset{\p\geq1}{\leq}\frac{\sigma_{\sma}^{2}\sigma_{\lar}^{2\p-2}}{\cm^{2\p-2}}T^{1+\frac{2}{\p}},
\]
and for any $\beta\in\left[0,1/2\right]$,
\[
\sum_{t=1}^{T}\frac{\sigma_{\lar}^{2\p}G^{2}t^{2}}{\cm_{t}^{2\p}}\leq\sum_{t=1}^{T}\frac{\sigma_{\lar}^{2\p}G^{2}t^{2}}{\left(\frac{G}{1-\alpha}\right)^{2\beta\p}\left(\cm t^{\frac{1}{\p}}\right)^{2(1-\beta)\p}}\leq\O\left(\frac{(1-\alpha)^{2\beta\p}\sigma_{\lar}^{2\p}G^{2-2\beta\p}}{\cm^{2(1-\beta)\p}}T^{1+2\beta}\right).
\]
Hence, for any $\beta\in\left[0,1/2\right]$,
\begin{align}
 & \sum_{t=1}^{T}\left(\frac{\sigma_{\sma}^{2}\sigma_{\lar}^{2\p-2}t^{2}}{\cm_{t}^{2\p-2}}+\frac{\sigma_{\lar}^{2\p}G^{2}t^{2}}{\min\left\{ \alpha^{2\p-2},(1-\alpha)^{2}\right\} \cm_{t}^{2\p}}\right)\nonumber \\
\leq & \O\left(\frac{\sigma_{\sma}^{2}\sigma_{\lar}^{2\p-2}}{\cm^{2\p-2}}T^{1+\frac{2}{\p}}+\frac{(1-\alpha)^{2\beta\p}\sigma_{\lar}^{2\p}G^{2-2\beta\p}}{\min\left\{ \alpha^{2\p-2},(1-\alpha)^{2}\right\} \cm^{2(1-\beta)\p}}T^{1+2\beta}\right).\label{eq:str-hp-dep-IV}
\end{align}
\end{itemize}
Next, for any fixed $\beta\in\left[0,1/2\right]$,
\begin{align*}
 & \text{R.H.S. of }(\ref{eq:str-hp-dep-I})+\text{R.H.S. of }(\ref{eq:str-hp-dep-IV})\\
= & \frac{G^{2}\ln^{2}\frac{3}{\delta}}{(1-\alpha)^{2}}T+\frac{(1-\alpha)^{2\beta\p}\sigma_{\lar}^{2\p}G^{2-2\beta\p}}{\min\left\{ \alpha^{2\p-2},(1-\alpha)^{2}\right\} \cm^{2(1-\beta)\p}}T^{1+2\beta}+\cm^{2}\ln^{2}\left(\frac{3}{\delta}\right)T^{1+\frac{2}{\p}}+\frac{\sigma_{\sma}^{2}\sigma_{\lar}^{2\p-2}}{\cm^{2\p-2}}T^{1+\frac{2}{\p}}\\
\overset{(a)}{\geq} & \frac{2(1-\alpha)^{\beta\p}\sigma_{\lar}^{\p}G^{2-\beta\p}\ln\frac{3}{\delta}}{(1-\alpha)\min\left\{ \alpha^{\p-1},(1-\alpha)\right\} \cm^{(1-\beta)\p}}T^{1+\beta}+2\sigma_{\sma}\sigma_{\lar}^{\p-1}\cm^{2-\p}\ln\left(\frac{3}{\delta}\right)T^{1+\frac{2}{\p}}\\
\overset{(b)}{\geq} & \frac{(1-\alpha)^{\beta\p}\sigma_{\lar}^{\p}G^{2-\beta\p}\ln\frac{3}{\delta}}{\alpha^{\p-1}\cm^{(1-\beta)\p}}T^{1+\beta}+\sigma_{\sma}^{\p}\cm^{2-\p}\ln\left(\frac{3}{\delta}\right)T^{1+\frac{2}{\p}}\\
= & \text{R.H.S. of }(\ref{eq:str-hp-dep-III}),
\end{align*}
where $(a)$ is by AM-GM inequality and $(b)$ is due to $\alpha<1$,
$\sigma_{\lar}\geq\sigma_{\sma}$ and $\p\geq1$. Therefore, after
plugging (\ref{eq:str-hp-dep-I}), (\ref{eq:str-hp-dep-II}), (\ref{eq:str-hp-dep-III}),
and (\ref{eq:str-hp-dep-IV}) back into (\ref{eq:str-hp-dep-C-simplified}),
we have for any $\beta\in\left[0,1/2\right]$,
\begin{align*}
\mu\hc_{T}^{\str}\leq & \frac{1}{\mu}\cdot\O\left(\left(\frac{(1-\alpha)^{2\beta\p}\sigma_{\lar}^{2\p}G^{2-2\beta\p}T^{2\beta}}{\min\left\{ \alpha^{2\p-2},(1-\alpha)^{2}\right\} \cm^{2(1-\beta)\p}}+\frac{G^{2}\ln^{2}\frac{3}{\delta}}{(1-\alpha)^{2}}\right)T+G^{2}T^{2}\right.\\
 & \left.\qquad\quad+\left(\cm^{2}\ln^{2}\frac{3}{\delta}+\sigma_{\lar}^{\p}\cm^{2-\p}+\frac{\sigma_{\sma}^{2}\sigma_{\lar}^{2\p-2}}{\cm^{2\p-2}}\right)T^{1+\frac{2}{\p}}\right).
\end{align*}
Combine the above bound on $\mu\hc_{T}^{\str}$ and (\ref{eq:str-hp-dep-rhs-1})
to have for any $\beta\in\left[0,1/2\right]$,
\begin{align}
\text{R.H.S. of }(\ref{eq:str-hp-dep-1})\leq & \O\left(\frac{\mu D^{2}}{T^{3}}+\frac{\frac{(1-\alpha)^{2\beta\p}\sigma_{\lar}^{2\p}G^{2-2\beta\p}T^{2\beta}}{\min\left\{ \alpha^{2\p-2},(1-\alpha)^{2}\right\} \cm^{2(1-\beta)\p}}+\frac{G^{2}\ln^{2}\frac{3}{\delta}}{(1-\alpha)^{2}}}{\mu T^{2}}\right.\nonumber \\
 & \left.\quad+\frac{G^{2}}{\mu T}+\frac{\cm^{2}\ln^{2}\frac{3}{\delta}+\sigma_{\lar}^{\p}\cm^{2-\p}+\frac{\sigma_{\sma}^{2}\sigma_{\lar}^{2\p-2}}{\cm^{2\p-2}}}{\mu T^{2-\frac{2}{\p}}}\right).\label{eq:str-hp-dep-rhs-2}
\end{align}

We put (\ref{eq:str-hp-dep-lhs}) and (\ref{eq:str-hp-dep-rhs-2})
together, then use $\alpha=1/2$ and $\cm=\cm_{\star}$ (see (\ref{eq:hp-tau-star})),
and follow the same argument of (\ref{eq:varphi}) to finally obtain
\begin{align*}
 & \frac{3\mu\left\Vert \bx_{\star}-\bx_{T+1}\right\Vert ^{2}}{16}+F(\bar{\bx}_{T+1}^{\str})-F_{\star}\\
\leq & \O\left(\frac{\mu D^{2}}{T^{3}}+\frac{\left(\varphi^{2}+\ln^{2}\frac{3}{\delta}\right)G^{2}}{\mu T^{2}}+\frac{G^{2}}{\mu T}+\frac{\sigma_{\sma}^{\frac{4}{\p}-2}\sigma_{\lar}^{4-\frac{4}{\p}}+\sigma_{\sma}^{\frac{2}{\p}}\sigma_{\lar}^{2-\frac{2}{\p}}\ln^{2-\frac{2}{\p}}\frac{3}{\delta}}{\mu T^{2-\frac{2}{\p}}}\right).
\end{align*}
\end{proof}

\subsubsection{In-Expectation Convergence}

Next, we consider the in-expectation convergence. Note that Theorem
\ref{thm:str-ex-dep} is also the first result that breaks the existing
lower bound $\Omega(\sigma_{\lar}^{2}T^{\frac{2}{\p}-2})$ \citep{NEURIPS2020_b05b57f6}.
\begin{thm}[Full statement of Theorem \ref{thm:main-str-ex-dep}]
\label{thm:str-ex-dep}Under Assumptions \ref{assu:minimizer}, \ref{assu:obj}
(with $\mu>0$), \ref{assu:lip} and \ref{assu:oracle}, for any $T\in\N$,
setting $\eta_{t}=\frac{6}{\mu t},\cm_{t}=\max\left\{ \frac{G}{1-\alpha},\widetilde{\cm}_{\star}t^{\frac{1}{\p}}\right\} ,\forall t\in\left[T\right]$
where $\alpha=1/2$, then Clipped SGD (Algorithm \ref{alg:clipped-SGD})
guarantees that both $\E\left[F(\bar{\bx}_{T+1}^{\str})-F_{\star}\right]$
and $\mu\E\left[\left\Vert \bx_{T+1}-\bx_{\star}\right\Vert ^{2}\right]$
converge at the rate of
\[
\O\left(\frac{\mu D^{2}}{T^{3}}+\frac{\widetilde{\varphi}^{2}G^{2}}{\mu T^{2}}+\frac{G^{2}}{\mu T}+\frac{\sigma_{\sma}^{\frac{4}{\p}-2}\sigma_{\lar}^{4-\frac{4}{\p}}}{\mu T^{2-\frac{2}{\p}}}\right),
\]
where $\widetilde{\varphi}\leq\widetilde{\varphi}_{\star}$ is a constant
defined in (\ref{eq:varphi-tilde}) and equals $\widetilde{\varphi}_{\star}$
when $T=\Omega\left(\frac{G^{\p}}{\sigma_{\lar}^{\p}}\widetilde{\varphi}_{\star}\right)$.
\end{thm}
\begin{proof}
By Lemmas \ref{lem:str-ex-anytime} and \ref{lem:str-dep-res}, we
can follow a similar argument until (\ref{eq:str-hp-dep-C-simplified})
in the proof of Theorem \ref{thm:str-hp-dep} to have
\begin{equation}
\frac{3\mu\E\left[\left\Vert \bx_{\star}-\bx_{T+1}\right\Vert ^{2}\right]}{16}+\E\left[F(\bar{\bx}_{T+1}^{\str})-F_{\star}\right]\leq\O\left(\frac{\mu D^{2}+\mu\ec_{T}^{\str}}{T^{3}}\right),\label{eq:str-ex-dep-1}
\end{equation}
where
\[
\mu\ec_{T}^{\str}\leq\frac{1}{\mu}\cdot\O\left(\sum_{t=1}^{T}\sigma_{\lar}^{\p}\cm_{t}^{2-\p}t+\sum_{t=1}^{T}\left(\frac{\sigma_{\sma}^{2}\sigma_{\lar}^{2\p-2}t^{2}}{\cm_{t}^{2\p-2}}+\frac{\sigma_{\lar}^{2\p}G^{2}t^{2}}{\alpha^{2\p-2}\cm_{t}^{2\p}}\right)+G^{2}T^{2}\right).
\]

When $\cm_{t}=\max\left\{ \frac{G}{1-\alpha},\cm t^{\frac{1}{\p}}\right\} ,\forall t\in\left[T\right]$,
we notice that for any $t\in\left[T\right]$,
\[
\sigma_{\lar}^{\p}\cm_{t}^{2-\p}t\leq\frac{\sigma_{\lar}^{\p}G^{2}t}{(1-\alpha)^{2}\cm_{t}^{\p}}+\sigma_{\lar}^{\p}\cm^{2-\p}t^{\frac{2}{\p}}\leq\frac{\sigma_{\lar}^{2\p}G^{2}t^{2}}{\min\left\{ \alpha^{2\p-2},(1-\alpha)^{2}\right\} \cm_{t}^{2\p}}+\frac{G^{2}}{4(1-\alpha)^{2}}+\sigma_{\lar}^{\p}\cm^{2-\p}t^{\frac{2}{\p}},
\]
which implies that
\[
\mu\ec_{T}^{\str}\leq\frac{1}{\mu}\cdot\O\left(\frac{G^{2}T}{(1-\alpha)^{2}}+\sum_{t=1}^{T}\sigma_{\lar}^{\p}\cm^{2-\p}t^{\frac{2}{\p}}+\sum_{t=1}^{T}\left(\frac{\sigma_{\sma}^{2}\sigma_{\lar}^{2\p-2}t^{2}}{\cm_{t}^{2\p-2}}+\frac{\sigma_{\lar}^{2\p}G^{2}t^{2}}{\min\left\{ \alpha^{2\p-2},(1-\alpha)^{2}\right\} \cm_{t}^{2\p}}\right)+G^{2}T^{2}\right).
\]
We know
\[
\sum_{t=1}^{T}\sigma_{\lar}^{\p}\cm^{2-\p}t^{\frac{2}{\p}}\overset{(\ref{eq:str-hp-dep-II})}{\leq}\sigma_{\lar}^{\p}\cm^{2-\p}T^{1+\frac{2}{\p}},
\]
and for any $\beta\in\left[0,1/2\right]$,
\begin{align*}
 & \sum_{t=1}^{T}\left(\frac{\sigma_{\sma}^{2}\sigma_{\lar}^{2\p-2}t^{2}}{\cm_{t}^{2\p-2}}+\frac{\sigma_{\lar}^{2\p}G^{2}t^{2}}{\min\left\{ \alpha^{2\p-2},(1-\alpha)^{2}\right\} \cm_{t}^{2\p}}\right)\\
\overset{(\ref{eq:str-hp-dep-IV})}{\leq} & \O\left(\frac{\sigma_{\sma}^{2}\sigma_{\lar}^{2\p-2}}{\cm^{2\p-2}}T^{1+\frac{2}{\p}}+\frac{(1-\alpha)^{2\beta\p}\sigma_{\lar}^{2\p}G^{2-2\beta\p}}{\min\left\{ \alpha^{2\p-2},(1-\alpha)^{2}\right\} \cm^{2(1-\beta)\p}}T^{1+2\beta}\right).
\end{align*}
Therefore, we can bound
\begin{align}
\mu\ec_{T}^{\str}\leq & \frac{1}{\mu}\cdot\O\left(\left(\frac{(1-\alpha)^{2\beta\p}\sigma_{\lar}^{2\p}G^{2-2\beta\p}T^{2\beta}}{\min\left\{ \alpha^{2\p-2},(1-\alpha)^{2}\right\} \cm^{2(1-\beta)\p}}+\frac{G^{2}}{(1-\alpha)^{2}}\right)T+G^{2}T^{2}\right.\nonumber \\
 & \left.\qquad\quad+\left(\sigma_{\lar}^{\p}\cm^{2-\p}+\frac{\sigma_{\sma}^{2}\sigma_{\lar}^{2\p-2}}{\cm^{2\p-2}}\right)T^{1+\frac{2}{\p}}\right),\forall\beta\in\left[0,1/2\right].\label{eq:str-ex-dep-C}
\end{align}

We put (\ref{eq:str-ex-dep-1}) and (\ref{eq:str-ex-dep-C}) together,
then use $\alpha=1/2$ and $\cm=\widetilde{\cm}_{\star}$ (see (\ref{eq:ex-tau-star})),
and follow the same argument of (\ref{eq:varphi-tilde}) to finally
obtain
\begin{align*}
 & \frac{3\mu\E\left[\left\Vert \bx_{\star}-\bx_{T+1}\right\Vert ^{2}\right]}{16}+\E\left[F(\bar{\bx}_{T+1}^{\str})-F_{\star}\right]\\
\leq & \O\left(\frac{\mu D^{2}}{T^{3}}+\frac{\widetilde{\varphi}^{2}G^{2}}{\mu T^{2}}+\frac{G^{2}}{\mu T}+\frac{\sigma_{\sma}^{\frac{4}{\p}-2}\sigma_{\lar}^{4-\frac{4}{\p}}}{\mu T^{2-\frac{2}{\p}}}\right).
\end{align*}
\end{proof}

\section{Theoretical Analysis\label{sec:analysis}}

This section provides the missing analysis for every lemma used in
the proof in Section \ref{sec:theorems}. As discussed in Section
\ref{sec:sketch}, our refined analysis has two core parts: better
application of Freedman's inequality and finer bounds for clipping
error.

Before starting, we summarize the frequently used notation in the
proof:
\begin{itemize}
\item $\bx_{\star}\in\X$, the optimal solution in the domain of the problem
$\X$.
\item $D=\left\Vert \bx_{\star}-\bx_{1}\right\Vert $, distance between
the optimal solution and the initial point.
\item $\F_{t}=\sigma(\xi_{1},\mydots,\xi_{t})$, the natural filtration
induced by i.i.d. samples $\xi_{1}$ to $\xi_{t}$ from $\dis$.
\item $\bg_{t}=\bg(\bx_{t},\xi_{t})$, the stochastic gradient accessed
at the $t$-th iteration for point $\bx_{t}$.
\item $\cm_{t}$, the clipping threshold used at the $t$-th iteration.
\item $\bg_{t}^{\rc}=\clip_{\cm_{t}}(\bg_{t})=\min\left\{ 1,\frac{\cm_{t}}{\left\Vert \bg_{t}\right\Vert }\right\} \bg_{t}$,
the clipped stochastic gradient.
\item $\bd_{t}^{\rc}=\bg_{t}^{\rc}-\nabla f(\bx_{t})$, difference between
the clipped stochastic gradient and the true gradient.
\item $\bd_{t}^{\ru}=\bg_{t}^{\rc}-\E\left[\bg_{t}^{\rc}\mid\F_{t-1}\right]$,
the unbiased part in $\bd_{t}^{\rc}$.
\item $\bd_{t}^{\rb}=\E\left[\bg_{t}^{\rc}\mid\F_{t-1}\right]-\nabla f(\bx_{t})$,
the biased part in $\bd_{t}^{\rc}$.
\end{itemize}

\subsection{General Lemmas}

We give two general lemmas in this subsection.

First, we apply Theorem \ref{thm:clip} to obtain the following error
bounds specialized for clipped gradient methods. As mentioned, the
technical condition required in Theorem \ref{thm:clip} automatically
holds for clipped gradient methods.
\begin{lem}[Full statement of Lemma \ref{lem:main-clip-ineq}]
\label{lem:clip-ineq}Under Assumption \ref{assu:oracle} and assuming
$0<\cm_{t}\in\F_{t-1}$, then for $\bd_{t}^{\ru}=\bg_{t}^{\rc}-\E\left[\bg_{t}^{\rc}\mid\F_{t-1}\right]$,
$\bd_{t}^{\rb}=\E\left[\bg_{t}^{\rc}\mid\F_{t-1}\right]-\nabla f(\bx_{t})$,
and $\chi_{t}(\alpha)=\1\left[(1-\alpha)\cm_{t}\geq\left\Vert \nabla f(\bx_{t})\right\Vert \right],\forall\alpha\in\left[0,1\right)$,
there are: 
\begin{enumerate}
\item \label{enu:clip-ineq-1}$\left\Vert \bd_{t}^{\ru}\right\Vert \leq2\cm_{t}$.
\item \label{enu:clip-ineq-2}$\E\left[\left\Vert \bd_{t}^{\ru}\right\Vert ^{2}\mid\F_{t-1}\right]\leq4\sigma_{\lar}^{\p}\cm_{t}^{2-\p}$.
\item \label{enu:clip-ineq-3}$\left\Vert \E\left[\bd_{t}^{\ru}\left(\bd_{t}^{\ru}\right)^{\top}\mid\F_{t-1}\right]\right\Vert \leq4\sigma_{\sma}^{\p}\cm_{t}^{2-\p}+4\left\Vert \nabla f(\bx_{t})\right\Vert ^{2}$.
\item \label{enu:clip-ineq-4}$\left\Vert \E\left[\bd_{t}^{\ru}\left(\bd_{t}^{\ru}\right)^{\top}\mid\F_{t-1}\right]\right\Vert \chi_{t}(\alpha)\leq4\sigma_{\sma}^{\p}\cm_{t}^{2-\p}+4\alpha^{1-\p}\sigma_{\lar}^{\p}\left\Vert \nabla f(\bx_{t})\right\Vert ^{2}\cm_{t}^{-\p}$.
\item \label{enu:clip-ineq-5}$\left\Vert \bd_{t}^{\rb}\right\Vert \leq\sqrt{2}\left(\sigma_{\lar}^{\p-1}+\left\Vert \nabla f(\bx_{t})\right\Vert ^{\p-1}\right)\sigma_{\sma}\cm_{t}^{1-\p}+2\left(\sigma_{\lar}^{\p}+\left\Vert \nabla f(\bx_{t})\right\Vert ^{\p}\right)\left\Vert \nabla f(\bx_{t})\right\Vert \cm_{t}^{-\p}$.
\item \label{enu:clip-ineq-6}$\left\Vert \bd_{t}^{\rb}\right\Vert \chi_{t}(\alpha)\leq\sigma_{\sma}\sigma_{\lar}^{\p-1}\cm_{t}^{1-\p}+\alpha^{1-\p}\sigma_{\lar}^{\p}\left\Vert \nabla f(\bx_{t})\right\Vert \cm_{t}^{-\p}$.
\end{enumerate}
\end{lem}
\begin{proof}
We invoke Theorem \ref{thm:clip} with $\F=\F_{t-1}$, $\bg=\bg_{t}$,
$\bff=\nabla f(\bx_{t})$, $\bar{\bg}=\bg(\bx_{t},\xi_{t+1})$, $\cm=\cm_{t}$,
$\bd^{\ru}=\bd_{t}^{\ru}$, $\bd^{\rb}=\bd_{t}^{\rb}$, and $\chi(\alpha)=\chi_{t}(\alpha)$
to conclude.
\end{proof}

Compared to Lemma \ref{lem:main-clip-ineq}, the clipping threshold
$\cm_{t}$ could be time-varying and random. Inequalities \ref{enu:clip-ineq-4}
and \ref{enu:clip-ineq-6} provide a further (though minor) generalization
by a new parameter $\alpha$, which might be useful in practice as
mentioned in Remark \ref{rem:alpha-dep}. Especially, setting $\alpha=1/2$
will recover Lemma \ref{lem:main-clip-ineq}. Moreover, as discussed
in Section \ref{sec:sketch}, Inequalities \ref{enu:clip-ineq-2},
\ref{enu:clip-ineq-4} and \ref{enu:clip-ineq-6} are all finer than
existing bounds for clipping error under heavy-tailed noise.

We then discuss Inequalities \ref{enu:clip-ineq-3} and \ref{enu:clip-ineq-5}
not provided in Lemma \ref{lem:main-clip-ineq}. As far as we know,
both of them are new in the literature. As one can see, we do not
require $\left\Vert \nabla f(\bx_{t})\right\Vert $ (which turns out
to be $G$ under Assumption \ref{assu:lip}) to set up $\cm_{t}$
now, which we believe could be useful for future work.

Next, we give two one-step descent inequalities for our algorithms.
The analysis is standard in the literature, which we reproduce here
for completeness.
\begin{lem}
\label{lem:descent}Under Assumptions \ref{assu:obj} and \ref{assu:lip},
for any $\by\in\X$ and $t\in\N$:
\begin{itemize}
\item Clipped SGD (Algorithm \ref{alg:clipped-SGD}) guarantees
\[
F(\bx_{t+1})-F(\by)\leq\frac{\left\Vert \by-\bx_{t}\right\Vert ^{2}}{2\eta_{t}}-\frac{(1+\mu\eta_{t})\left\Vert \by-\bx_{t+1}\right\Vert ^{2}}{2\eta_{t}}+\left\langle \bd_{t}^{\rc},\by-\bx_{t}\right\rangle +\eta_{t}\left\Vert \bd_{t}^{\rc}\right\Vert ^{2}+4\eta_{t}G^{2}.
\]
\item Stabilized Clipped SGD (Algorithm \ref{alg:stabilized-clipped-SGD})
guarantees, if $\eta_{t}$ is nonincreasing,
\begin{align*}
F(\bx_{t+1})-F(\by)\leq & \frac{\left\Vert \by-\bx_{t}\right\Vert ^{2}}{2\eta_{t}}-\frac{(1+\mu\eta_{t+1})\left\Vert \by-\bx_{t+1}\right\Vert ^{2}}{2\eta_{t+1}}+\left(\frac{1}{\eta_{t+1}}-\frac{1}{\eta_{t}}\right)\frac{\left\Vert \by-\bx_{1}\right\Vert ^{2}}{2}\\
 & +\left\langle \bd_{t}^{\rc},\by-\bx_{t}\right\rangle +\eta_{t}\left\Vert \bd_{t}^{\rc}\right\Vert ^{2}+4\eta_{t}G^{2}.
\end{align*}
\end{itemize}
\end{lem}
\begin{proof}
By the convexity of $f$,
\begin{align*}
 & f(\bx_{t+1})-f(\bx_{t})\leq\left\langle \nabla f(\bx_{t+1}),\bx_{t+1}-\bx_{t}\right\rangle \\
= & \left\langle \nabla f(\bx_{t+1})-\nabla f(\bx_{t}),\bx_{t+1}-\bx_{t}\right\rangle +\left\langle \nabla f(\bx_{t}),\bx_{t+1}-\bx_{t}\right\rangle .
\end{align*}
Recall that $\bd_{t}^{\rc}=\bg_{t}^{\rc}-\nabla f(\bx_{t})$, we hence
have for any $\by\in\X$,
\begin{align*}
\left\langle \nabla f(\bx_{t}),\bx_{t+1}-\bx_{t}\right\rangle  & =\left\langle \bd_{t}^{\rc},\bx_{t}-\bx_{t+1}\right\rangle +\left\langle \bg_{t}^{\rc},\bx_{t+1}-\by\right\rangle +\left\langle \bd_{t}^{\rc},\by-\bx_{t}\right\rangle +\left\langle \nabla f(\bx_{t}),\by-\bx_{t}\right\rangle \\
 & \leq\left\langle \bd_{t}^{\rc},\bx_{t}-\bx_{t+1}\right\rangle +\left\langle \bg_{t}^{\rc},\bx_{t+1}-\by\right\rangle +\left\langle \bd_{t}^{\rc},\by-\bx_{t}\right\rangle +f(\by)-f(\bx_{t}),
\end{align*}
where the inequality is, again, due to the convexity of $f$. Combine
the above two results to obtain
\begin{align}
f(\bx_{t+1})-f(\by)\leq & \underbrace{\left\langle \nabla f(\bx_{t+1})-\nabla f(\bx_{t}),\bx_{t+1}-\bx_{t}\right\rangle }_{\mathrm{I}}+\underbrace{\left\langle \bd_{t}^{\rc},\bx_{t}-\bx_{t+1}\right\rangle }_{\mathrm{II}}\nonumber \\
 & +\underbrace{\left\langle \bg_{t}^{\rc},\bx_{t+1}-\by\right\rangle }_{\mathrm{III}}+\left\langle \bd_{t}^{\rc},\by-\bx_{t}\right\rangle .\label{eq:descent-1}
\end{align}
Next, we bound these three terms separately.
\begin{itemize}
\item Term $\mathrm{I}$. By Cauchy-Schwarz inequality, $G$-Lipschitz property
of $f$, and AM-GM inequality, there is
\begin{align}
\left\langle \nabla f(\bx_{t+1})-\nabla f(\bx_{t}),\bx_{t+1}-\bx_{t}\right\rangle  & \leq\left\Vert \nabla f(\bx_{t+1})-\nabla f(\bx_{t})\right\Vert \left\Vert \bx_{t+1}-\bx_{t}\right\Vert \nonumber \\
 & \leq2G\left\Vert \bx_{t+1}-\bx_{t}\right\Vert \leq4\eta_{t}G^{2}+\frac{\left\Vert \bx_{t+1}-\bx_{t}\right\Vert ^{2}}{4\eta_{t}}.\label{eq:descent-I}
\end{align}
\item Term $\mathrm{II}$. By Cauchy-Schwarz inequality and AM-GM inequality,
we know
\begin{equation}
\left\langle \bd_{t}^{\rc},\bx_{t}-\bx_{t+1}\right\rangle \leq\left\Vert \bd_{t}^{\rc}\right\Vert \left\Vert \bx_{t+1}-\bx_{t}\right\Vert \leq\eta_{t}\left\Vert \bd_{t}^{\rc}\right\Vert ^{2}+\frac{\left\Vert \bx_{t+1}-\bx_{t}\right\Vert ^{2}}{4\eta_{t}}.\label{eq:descent-II}
\end{equation}
\item Term $\mathrm{III}$. For Clipped SGD, by the optimality condition
of the update rule, there exists $\nabla r(\bx_{t+1})\in\partial r(\bx_{t+1})$
such that
\[
\left\langle \nabla r(\bx_{t+1})+\bg_{t}^{\rc}+\frac{\bx_{t+1}-\bx_{t}}{\eta_{t}},\bx_{t+1}-\by\right\rangle \leq0,
\]
which implies
\begin{align}
 & \left\langle \bg_{t}^{\rc},\bx_{t+1}-\by\right\rangle \nonumber \\
\leq & \frac{1}{\eta_{t}}\left\langle \bx_{t}-\bx_{t+1},\bx_{t+1}-\by\right\rangle +\left\langle \nabla r(\bx_{t+1}),\by-\bx_{t+1}\right\rangle \nonumber \\
= & \frac{\left\Vert \by-\bx_{t}\right\Vert ^{2}-\left\Vert \by-\bx_{t+1}\right\Vert ^{2}-\left\Vert \bx_{t+1}-\bx_{t}\right\Vert ^{2}}{2\eta_{t}}+\left\langle \nabla r(\bx_{t+1}),\by-\bx_{t+1}\right\rangle \nonumber \\
\leq & \frac{\left\Vert \by-\bx_{t}\right\Vert ^{2}-\left\Vert \by-\bx_{t+1}\right\Vert ^{2}-\left\Vert \bx_{t+1}-\bx_{t}\right\Vert ^{2}}{2\eta_{t}}+r(\by)-r(\bx_{t+1})-\frac{\mu}{2}\left\Vert \by-\bx_{t+1}\right\Vert ^{2},\label{eq:descent-c-III}
\end{align}
where the last step is due to the $\mu$-strong convexity of $r$
(Assumption \ref{assu:obj}). For Stabilized Clipped SGD, a similar
argument yields that when $\eta_{t}\geq\eta_{t+1}$,
\begin{align}
\left\langle \bg_{t}^{\rc},\bx_{t+1}-\by\right\rangle \leq & \frac{\left\Vert \by-\bx_{t}\right\Vert ^{2}}{2\eta_{t}}-\frac{\left\Vert \by-\bx_{t+1}\right\Vert ^{2}}{2\eta_{t+1}}-\frac{\left\Vert \bx_{t+1}-\bx_{t}\right\Vert ^{2}}{2\eta_{t}}+\left(\frac{1}{\eta_{t+1}}-\frac{1}{\eta_{t}}\right)\frac{\left\Vert \by-\bx_{1}\right\Vert ^{2}}{2}\nonumber \\
 & +r(\by)-r(\bx_{t+1})-\frac{\mu}{2}\left\Vert \by-\bx_{t+1}\right\Vert ^{2}.\label{eq:descent-s-III}
\end{align}
\end{itemize}
We plug (\ref{eq:descent-I}), (\ref{eq:descent-II}), and (\ref{eq:descent-c-III})
(resp. (\ref{eq:descent-s-III})) back into (\ref{eq:descent-1})
and rearrange terms to obtain the desired result for Clipped SGD (resp.
Stabilized Clipped SGD).
\end{proof}

\subsection{Lemmas for General Convex Functions}

In this section, we focus on the general convex case, i.e., $\mu=0$
in Assumption \ref{assu:obj}. As mentioned before in Appendix \ref{sec:stabilized},
it is enough to only analyze the Stabilized Clipped SGD method since
it is the same as the original Clipped SGD when the stepsize is constant.

\subsubsection{Two Core Inequalities}

Before moving to the formal proof, we first introduce two quantities
that will be used in the analysis:
\begin{eqnarray}
R_{t}\defeq\max_{s\in\left[t\right]}\frac{\left\Vert \bx_{\star}-\bx_{s}\right\Vert }{\sqrt{\eta_{s}}},\forall t\in\left[T\right], & \text{and} & N_{t}\defeq\left\langle \sqrt{\eta_{t}}\bd_{t}^{\ru},\frac{\bx_{\star}-\bx_{t}}{R_{t}\sqrt{\eta_{t}}}\right\rangle ,\forall t\in\left[T\right].\label{eq:cvx-R-N}
\end{eqnarray}
Note that $R_{t}\in\F_{t-1}$ and $N_{t}\in\F_{t}$ by their definitions.
Importantly, $N_{t}$ is a real-valued MDS due to
\begin{equation}
\E\left[N_{t}\mid\F_{t-1}\right]=\left\langle \sqrt{\eta_{t}}\E\left[\bd_{t}^{\ru}\mid\F_{t-1}\right],\frac{\bx_{\star}-\bx_{t}}{R_{t}\sqrt{\eta_{t}}}\right\rangle =0,\forall t\in\left[T\right].\label{eq:cvx-N-MDS}
\end{equation}

Now we are ready to dive into the analysis. We first introduce the
following Lemma \ref{lem:cvx-hp-anytime}, which characterizes the
progress made by Stabilized Clipped SGD after $T$ iterations.
\begin{lem}
\label{lem:cvx-hp-anytime}Under Assumptions \ref{assu:minimizer},
\ref{assu:obj} (with $\mu=0$) and \ref{assu:lip}, if $\eta_{t}$
is nonincreasing, then for any $T\in\N$, Stabilized Clipped SGD (Algorithm
\ref{alg:stabilized-clipped-SGD}) guarantees
\[
\frac{\left\Vert \bx_{\star}-\bx_{T+1}\right\Vert ^{2}}{2\eta_{T+1}}+\sum_{t=1}^{T}F(\bx_{t+1})-F_{\star}\leq\frac{D^{2}}{\eta_{T+1}}+2\hres_{T}^{\cvx},
\]
where 
\[
\hres_{T}^{\cvx}\defeq8\max_{t\in\left[T\right]}\left(\sum_{s=1}^{t}N_{s}\right)^{2}+2\sum_{t=1}^{T}\eta_{t}\left\Vert \bd_{t}^{\ru}\right\Vert ^{2}+4\left(\sum_{t=1}^{T}\left\Vert \sqrt{\eta_{t}}\bd_{t}^{\rb}\right\Vert \right)^{2}+4G^{2}\sum_{t=1}^{T}\eta_{t}.
\]
\end{lem}
\begin{proof}
We invoke Lemma \ref{lem:descent} for Stabilized Clipped SGD with
$\mu=0$ and $\by=\bx_{\star}$, then replace the subscript $t$ with
$s$, and use $\left\Vert \bx_{\star}-\bx_{1}\right\Vert =D$ to have
\[
F(\bx_{s+1})-F_{\star}\leq\frac{\left\Vert \bx_{\star}-\bx_{s}\right\Vert ^{2}}{2\eta_{s}}-\frac{\left\Vert \bx_{\star}-\bx_{s+1}\right\Vert ^{2}}{2\eta_{s+1}}+\left(\frac{1}{\eta_{s+1}}-\frac{1}{\eta_{s}}\right)\frac{D^{2}}{2}+\left\langle \bd_{s}^{\rc},\bx_{\star}-\bx_{s}\right\rangle +\eta_{s}\left\Vert \bd_{s}^{\rc}\right\Vert ^{2}+4\eta_{s}G^{2},
\]
sum up which over $s$ from $1$ to $t\leq T$ to obtain
\begin{equation}
\frac{\left\Vert \bx_{\star}-\bx_{t+1}\right\Vert ^{2}}{2\eta_{t+1}}+\sum_{s=1}^{t}F(\bx_{s+1})-F_{\star}\leq\frac{D^{2}}{2\eta_{t+1}}+\sum_{s=1}^{t}\left\langle \bd_{s}^{\rc},\bx_{\star}-\bx_{s}\right\rangle +\sum_{s=1}^{t}\eta_{s}\left\Vert \bd_{s}^{\rc}\right\Vert ^{2}+4G^{2}\sum_{s=1}^{t}\eta_{s}.\label{eq:cvx-hp-anytime-1}
\end{equation}

We recall the decomposition $\bd_{s}^{\rc}=\bd_{s}^{\ru}+\bd_{s}^{\rb}$
to have
\[
\sum_{s=1}^{t}\left\langle \bd_{s}^{\rc},\bx_{\star}-\bx_{s}\right\rangle =\sum_{s=1}^{t}\left\langle \bd_{s}^{\ru},\bx_{\star}-\bx_{s}\right\rangle +\sum_{s=1}^{t}\left\langle \bd_{s}^{\rb},\bx_{\star}-\bx_{s}\right\rangle \overset{(\ref{eq:cvx-R-N})}{=}\sum_{s=1}^{t}R_{s}N_{s}+\sum_{s=1}^{t}\sqrt{\eta_{s}}\left\langle \bd_{s}^{\rb},\frac{\bx_{\star}-\bx_{s}}{\sqrt{\eta_{s}}}\right\rangle .
\]
We can bound
\[
\sum_{s=1}^{t}R_{s}N_{s}\leq\left|\sum_{s=1}^{t}R_{s}N_{s}\right|\overset{\text{Lemma }\ref{lem:Abel}}{\leq}2R_{t}\max_{S\in\left[t\right]}\left|\sum_{s=1}^{S}N_{s}\right|.
\]
In addition, Cauchy-Schwarz inequality gives us
\[
\sum_{s=1}^{t}\left\langle \sqrt{\eta_{s}}\bd_{s}^{\rb},\frac{\bx_{\star}-\bx_{s}}{\sqrt{\eta_{s}}}\right\rangle \leq\sum_{s=1}^{t}\left\Vert \sqrt{\eta_{s}}\bd_{s}^{\rb}\right\Vert \frac{\left\Vert \bx_{\star}-\bx_{s}\right\Vert }{\sqrt{\eta_{s}}}\overset{(\ref{eq:cvx-R-N})}{\leq}R_{t}\sum_{s=1}^{t}\left\Vert \sqrt{\eta_{s}}\bd_{s}^{\rb}\right\Vert .
\]
As such, we know
\begin{align}
\sum_{s=1}^{t}\left\langle \bd_{s}^{\rc},\bx_{\star}-\bx_{s}\right\rangle  & \leq2R_{t}\max_{S\in\left[t\right]}\left|\sum_{s=1}^{S}N_{s}\right|+R_{t}\sum_{s=1}^{t}\left\Vert \sqrt{\eta_{s}}\bd_{s}^{\rb}\right\Vert \nonumber \\
 & \le\frac{R_{t}^{2}}{4}+8\max_{S\in\left[t\right]}\left(\sum_{s=1}^{S}N_{s}\right)^{2}+2\left(\sum_{s=1}^{t}\left\Vert \sqrt{\eta_{s}}\bd_{s}^{\rb}\right\Vert \right)^{2},\label{eq:cvx-hp-anytime-2}
\end{align}
where the second inequality is by $R_{t}X\leq\frac{R_{t}^{2}}{8}+2X^{2}$
(due to AM-GM inequality) for $X=2\max_{S\in\left[t\right]}\left|\sum_{s=1}^{S}N_{s}\right|$
and $\sum_{s=1}^{t}\left\Vert \sqrt{\eta_{s}}\bd_{s}^{\rb}\right\Vert $,
respectively.

Plug (\ref{eq:cvx-hp-anytime-2}) back into (\ref{eq:cvx-hp-anytime-1})
to get
\begin{align}
 & \frac{\left\Vert \bx_{\star}-\bx_{t+1}\right\Vert ^{2}}{2\eta_{t+1}}+\sum_{s=1}^{t}F(\bx_{s+1})-F_{\star}\nonumber \\
\leq & \frac{R_{t}^{2}}{4}+\frac{D^{2}}{2\eta_{t+1}}+8\max_{S\in\left[t\right]}\left(\sum_{s=1}^{S}N_{s}\right)^{2}+2\left(\sum_{s=1}^{t}\left\Vert \sqrt{\eta_{s}}\bd_{s}^{\rb}\right\Vert \right)^{2}+\sum_{s=1}^{t}\eta_{s}\left\Vert \bd_{s}^{\rc}\right\Vert ^{2}+4G^{2}\sum_{s=1}^{t}\eta_{s}\nonumber \\
\leq & \frac{R_{t}^{2}}{4}+\frac{D^{2}}{2\eta_{t+1}}+\underbrace{8\max_{S\in\left[t\right]}\left(\sum_{s=1}^{S}N_{s}\right)^{2}+2\sum_{s=1}^{t}\eta_{s}\left\Vert \bd_{s}^{\ru}\right\Vert ^{2}+4\left(\sum_{s=1}^{t}\left\Vert \sqrt{\eta_{s}}\bd_{s}^{\rb}\right\Vert \right)^{2}+4G^{2}\sum_{s=1}^{t}\eta_{s}}_{\defeq\hres_{t}^{\cvx}},\label{eq:cvx-hp-anytime-3}
\end{align}
where the last step is by
\begin{align*}
\sum_{s=1}^{t}\eta_{s}\left\Vert \bd_{s}^{\rc}\right\Vert ^{2} & =\sum_{s=1}^{t}\eta_{s}\left\Vert \bd_{s}^{\ru}+\bd_{s}^{\rb}\right\Vert ^{2}\leq2\sum_{s=1}^{t}\eta_{s}\left\Vert \bd_{s}^{\ru}\right\Vert ^{2}+2\sum_{s=1}^{t}\eta_{s}\left\Vert \bd_{s}^{\rb}\right\Vert ^{2}\\
 & \leq2\sum_{s=1}^{t}\eta_{s}\left\Vert \bd_{s}^{\ru}\right\Vert ^{2}+2\left(\sum_{s=1}^{t}\left\Vert \sqrt{\eta_{s}}\bd_{s}^{\rb}\right\Vert \right)^{2}.
\end{align*}

Now we let $a_{t}\defeq\frac{\left\Vert \bx_{\star}-\bx_{t}\right\Vert ^{2}}{2\eta_{t}},\forall t\in\left[T+1\right]$,
$b_{t}\defeq\sum_{s=1}^{t}F(\bx_{s+1})-F_{\star},\forall t\in\left[T\right]$
and $c_{t}\defeq\frac{D^{2}}{2\eta_{t}}+\hres_{t-1}^{\cvx},\forall t\in\left[T+1\right]$
where $\hres_{0}^{\cvx}=0$. Note that $b_{t}$ is nonnegative, $c_{t}$
is nondecreasing as $\eta_{t}$ is nonincreasing, and
\[
a_{1}=\frac{\left\Vert \bx_{\star}-\bx_{1}\right\Vert ^{2}}{2\eta_{1}}=\frac{D^{2}}{2\eta_{1}}\leq\frac{D^{2}}{\eta_{1}}=2c_{1}.
\]
Moreover, (\ref{eq:cvx-hp-anytime-3}) is saying that
\[
a_{t+1}+b_{t}\leq\frac{\max_{s\in\left[t\right]}a_{s}}{2}+c_{t+1},\forall t\in\left[T\right].
\]
Thus, we can invoke Lemma \ref{lem:algebra} to obtain
\[
a_{T+1}+b_{T}\leq2c_{T+1},
\]
which means
\[
\frac{\left\Vert \bx_{\star}-\bx_{T+1}\right\Vert ^{2}}{2\eta_{T+1}}+\sum_{t=1}^{T}F(\bx_{t+1})-F_{\star}\leq\frac{D^{2}}{\eta_{T+1}}+2\hres_{T}^{\cvx}.
\]
\end{proof}

Equipped with Lemma \ref{lem:cvx-hp-anytime}, we prove the following
in-expectation convergence result for Stabilized Clipped SGD.
\begin{lem}
\label{lem:cvx-ex-anytime}Under the same setting in Lemma \ref{lem:cvx-hp-anytime},
Stabilized Clipped SGD (Algorithm \ref{alg:stabilized-clipped-SGD})
guarantees
\[
\frac{\E\left[\left\Vert \bx_{\star}-\bx_{T+1}\right\Vert ^{2}\right]}{2\eta_{T+1}}+\sum_{t=1}^{T}\E\left[F(\bx_{t+1})-F_{\star}\right]\leq\frac{D^{2}}{\eta_{T+1}}+2\eres_{T}^{\cvx},
\]
where 
\[
\eres_{T}^{\cvx}\defeq34\sum_{t=1}^{T}\eta_{t}\E\left[\left\Vert \bd_{t}^{\ru}\right\Vert ^{2}\right]+4\E\left[\left(\sum_{t=1}^{T}\left\Vert \sqrt{\eta_{t}}\bd_{t}^{\rb}\right\Vert \right)^{2}\right]+4G^{2}\sum_{t=1}^{T}\eta_{t}.
\]
\end{lem}
\begin{proof}
We invoke Lemma \ref{lem:cvx-hp-anytime} and take expectations to
obtain
\[
\frac{\E\left[\left\Vert \bx_{\star}-\bx_{T+1}\right\Vert ^{2}\right]}{2\eta_{T+1}}+\sum_{t=1}^{T}\E\left[F(\bx_{t+1})-F_{\star}\right]\leq\frac{D^{2}}{\eta_{T+1}}+2\E\left[\hres_{T}^{\cvx}\right],
\]
where, by the definition of $\hres_{T}^{\cvx}$,
\[
\E\left[\hres_{T}^{\cvx}\right]=8\E\left[\max_{t\in\left[T\right]}\left(\sum_{s=1}^{t}N_{s}\right)^{2}\right]+2\sum_{t=1}^{T}\eta_{t}\E\left[\left\Vert \bd_{t}^{\ru}\right\Vert ^{2}\right]+4\E\left[\left(\sum_{t=1}^{T}\left\Vert \sqrt{\eta_{t}}\bd_{t}^{\rb}\right\Vert \right)^{2}\right]+4G^{2}\sum_{t=1}^{T}\eta_{t}.
\]
Recall that $N_{t},\forall t\in\left[T\right]$ is a MDS (see (\ref{eq:cvx-N-MDS})).
Therefore, by Lemma \ref{lem:Doob}, there is
\[
\E\left[\max_{t\in\left[T\right]}\left(\sum_{s=1}^{t}N_{s}\right)^{2}\right]\leq4\sum_{t=1}^{T}\E\left[N_{s}^{2}\right]\overset{(\ref{eq:cvx-R-N})}{\leq}4\sum_{t=1}^{T}\eta_{t}\E\left[\left\Vert \bd_{t}^{\ru}\right\Vert ^{2}\right].
\]
Finally, we have
\[
\E\left[\hres_{T}^{\cvx}\right]\leq34\sum_{t=1}^{T}\eta_{t}\E\left[\left\Vert \bd_{t}^{\ru}\right\Vert ^{2}\right]+4\E\left[\left(\sum_{t=1}^{T}\left\Vert \sqrt{\eta_{t}}\bd_{t}^{\rb}\right\Vert \right)^{2}\right]+4G^{2}\sum_{t=1}^{T}\eta_{t}=\eres_{T}^{\cvx}.
\]
\end{proof}

\subsubsection{Bounding Residual Terms}

With Lemmas \ref{lem:cvx-hp-anytime} and \ref{lem:cvx-ex-anytime},
our next goal is naturally to bound the residual terms $\hres_{T}^{\cvx}$
and $\eres_{T}^{\cvx}$. Note that the $G^{2}\sum_{t=1}^{T}\eta_{t}$
part is standard in nonsmooth optimization. Hence, all important things
are to control the other terms left.

We now provide the bound in the following Lemma \ref{lem:cvx-dep-res},
a tighter estimation for the residual term compared to prior works
(e.g., \citet{liu2023stochastic}), which is achieved due to our finer
bounds for clipping error under heavy-tailed noise.
\begin{lem}
\label{lem:cvx-dep-res}Under Assumptions \ref{assu:lip}, \ref{assu:oracle}
and the following two conditions:
\begin{enumerate}
\item \label{enu:cvx-dep-res-1}$\eta_{t}$ and $\cm_{t}$ are deterministic
for all $t\in\left[T\right]$.
\item \label{enu:cvx-dep-res-2}$\cm_{t}\geq\frac{G}{1-\alpha}$ holds for
some constant $\alpha\in\left(0,1\right)$ and all $t\in\left[T\right]$.
\end{enumerate}
We have:
\begin{enumerate}
\item for any $\delta\in\left(0,1\right]$, with probability at least $1-\delta$,
$\hres_{T}^{\cvx}\leq\hc_{T}^{\cvx}$ where $\hres_{T}^{\cvx}$ is
defined in Lemma \ref{lem:cvx-hp-anytime} and $\hc_{T}^{\cvx}$ is
a constant in the order of
\[
\O\left(\max_{t\in\left[T\right]}\eta_{t}\cm_{t}^{2}\ln^{2}\frac{3}{\delta}+\sum_{t=1}^{T}\frac{\sigma_{\lar}^{\p}\eta_{t}}{\cm_{t}^{\p-2}}+\left(\sum_{t=1}^{T}\frac{\sigma_{\sma}\sigma_{\lar}^{\p-1}\sqrt{\eta_{t}}}{\cm_{t}^{\p-1}}+\frac{\sigma_{\lar}^{\p}G\sqrt{\eta_{t}}}{\alpha^{\p-1}\cm_{t}^{\p}}\right)^{2}+\sum_{t=1}^{T}G^{2}\eta_{t}\right).
\]
\item $\eres_{T}^{\cvx}\leq\ec_{T}^{\cvx}$ where $\eres_{T}^{\cvx}$ is
defined in Lemma \ref{lem:cvx-ex-anytime} and $\ec_{T}^{\cvx}$ is
a constant in the order of
\[
\O\left(\sum_{t=1}^{T}\frac{\sigma_{\lar}^{\p}\eta_{t}}{\cm_{t}^{\p-2}}+\left(\sum_{t=1}^{T}\frac{\sigma_{\sma}\sigma_{\lar}^{\p-1}\sqrt{\eta_{t}}}{\cm_{t}^{\p-1}}+\frac{\sigma_{\lar}^{\p}G\sqrt{\eta_{t}}}{\alpha^{\p-1}\cm_{t}^{\p}}\right)^{2}+\sum_{t=1}^{T}G^{2}\eta_{t}\right).
\]
\end{enumerate}
\end{lem}
\begin{proof}
We observe that for any $t\in\left[T\right]$, $\cm_{t}\geq\frac{G}{1-\alpha}\geq\frac{\left\Vert \nabla f(\bx_{t})\right\Vert }{1-\alpha}$
holds almost surely due to Condition \ref{enu:cvx-dep-res-2} and
Assumption \ref{assu:lip}, implying that $\chi_{t}(\alpha)$ in Lemma
\ref{lem:clip-ineq} equals $1$ for all $t\in\left[T\right]$. Then
Lemma \ref{lem:clip-ineq} and Assumption \ref{assu:lip} together
yield the following inequalities holding for any $t\in\left[T\right]$:
\begin{align}
\left\Vert \sqrt{\eta_{t}}\bd_{t}^{\ru}\right\Vert  & \overset{\text{Inequality }\ref{enu:clip-ineq-1}}{\leq}2\sqrt{\eta_{t}}\cm_{t}\leq2\max_{t\in\left[T\right]}\sqrt{\eta_{t}}\cm_{t},\label{eq:cvx-full-res-u-n}\\
\E\left[\left\Vert \sqrt{\eta_{t}}\bd_{t}^{\ru}\right\Vert ^{2}\mid\F_{t-1}\right] & \overset{\text{Inequality }\ref{enu:clip-ineq-2}}{\leq}\frac{4\sigma_{\lar}^{\p}\eta_{t}}{\cm_{t}^{\p-2}},\label{eq:cvx-full-res-u-v}\\
\left\Vert \E\left[\eta_{t}\bd_{t}^{\ru}\left(\bd_{t}^{\ru}\right)^{\top}\mid\F_{t-1}\right]\right\Vert  & \overset{\text{Inequality }\ref{enu:clip-ineq-4}}{\leq}\frac{4\sigma_{\sma}^{\p}\eta_{t}}{\cm_{t}^{\p-2}}+\frac{4\sigma_{\lar}^{\p}G^{2}\eta_{t}}{\alpha^{\p-1}\cm_{t}^{\p}},\label{eq:cvx-full-res-u-o}\\
\left\Vert \sqrt{\eta_{t}}\bd_{t}^{\rb}\right\Vert  & \overset{\text{Inequality }\ref{enu:clip-ineq-6}}{\leq}\frac{\sigma_{\sma}\sigma_{\lar}^{\p-1}\sqrt{\eta_{t}}}{\cm_{t}^{\p-1}}+\frac{\sigma_{\lar}^{\p}G\sqrt{\eta_{t}}}{\alpha^{\p-1}\cm_{t}^{\p}}.\label{eq:cvx-full-res-b-n}
\end{align}

We first bound $\hres_{T}^{\cvx}$ in high probability.
\begin{itemize}
\item Recall that $N_{t}=\left\langle \sqrt{\eta_{t}}\bd_{t}^{\ru},\frac{\bx_{\star}-\bx_{t}}{R_{t}\sqrt{\eta_{t}}}\right\rangle ,\forall t\in\left[T\right]$
is a real-valued MDS (see (\ref{eq:cvx-N-MDS})), whose absolute value
can be bounded by Cauchy-Schwarz inequality
\[
\left|N_{t}\right|\leq\left\Vert \sqrt{\eta_{t}}\bd_{t}^{\ru}\right\Vert \left\Vert \frac{\bx_{\star}-\bx_{t}}{R_{t}\sqrt{\eta_{t}}}\right\Vert \overset{(\ref{eq:cvx-R-N})}{\leq}\left\Vert \sqrt{\eta_{t}}\bd_{t}^{\ru}\right\Vert \overset{(\ref{eq:cvx-full-res-u-n})}{\leq}2\max_{t\in\left[T\right]}\sqrt{\eta_{t}}\cm_{t}.
\]
Moreover, its conditional variance can be controlled by
\begin{align*}
\E\left[N_{t}^{2}\mid\F_{t-1}\right] & =\left(\frac{\bx_{\star}-\bx_{t}}{R_{t}\sqrt{\eta_{t}}}\right)^{\top}\E\left[\eta_{t}\bd_{t}^{\ru}\left(\bd_{t}^{\ru}\right)^{\top}\mid\F_{t-1}\right]\frac{\bx_{\star}-\bx_{t}}{R_{t}\sqrt{\eta_{t}}}\\
 & \overset{(\ref{eq:cvx-R-N})}{\leq}\left\Vert \E\left[\eta_{t}\bd_{t}^{\ru}\left(\bd_{t}^{\ru}\right)^{\top}\mid\F_{t-1}\right]\right\Vert \overset{(\ref{eq:cvx-full-res-u-o})}{\leq}\frac{4\sigma_{\sma}^{\p}\eta_{t}}{\cm_{t}^{\p-2}}+\frac{4\sigma_{\lar}^{\p}G^{2}\eta_{t}}{\alpha^{\p-1}\cm_{t}^{\p}}.
\end{align*}
Therefore, Freedman's inequality (Lemma \ref{lem:Freedman}) gives
that with probability at least $1-2\delta/3$,
\[
\left|\sum_{s=1}^{t}N_{s}\right|\leq\frac{4}{3}\max_{t\in\left[T\right]}\sqrt{\eta_{t}}\cm_{t}\ln\frac{3}{\delta}+\sqrt{8\sum_{s=1}^{T}\left(\frac{\sigma_{\sma}^{\p}\eta_{s}}{\cm_{s}^{\p-2}}+\frac{\sigma_{\lar}^{\p}G^{2}\eta_{s}}{\alpha^{\p-1}\cm_{s}^{\p}}\right)\ln\frac{3}{\delta}},\forall t\in\left[T\right],
\]
which implies
\begin{equation}
\max_{t\in\left[T\right]}\left(\sum_{s=1}^{t}N_{s}\right)^{2}\leq\frac{2^{5}}{9}\max_{t\in\left[T\right]}\eta_{t}\cm_{t}^{2}\ln^{2}\frac{3}{\delta}+16\sum_{t=1}^{T}\left(\frac{\sigma_{\sma}^{\p}\eta_{t}}{\cm_{t}^{\p-2}}+\frac{\sigma_{\lar}^{\p}G^{2}\eta_{t}}{\alpha^{\p-1}\cm_{t}^{\p}}\right)\ln\frac{3}{\delta}.\label{eq:cvx-full-res-1}
\end{equation}
\item Note that $\left\Vert \sqrt{\eta_{t}}\bd_{t}^{\ru}\right\Vert ,\forall t\in\left[T\right]$
is a sequence of random variables satisfying
\begin{eqnarray*}
\left\Vert \sqrt{\eta_{t}}\bd_{t}^{\ru}\right\Vert \overset{(\ref{eq:cvx-full-res-u-n})}{\leq}2\max_{t\in\left[T\right]}\sqrt{\eta_{t}}\cm_{t} & \text{and} & \E\left[\left\Vert \sqrt{\eta_{t}}\bd_{t}^{\ru}\right\Vert ^{2}\mid\F_{t-1}\right]\overset{(\ref{eq:cvx-full-res-u-v})}{\leq}\frac{4\sigma_{\lar}^{\p}\eta_{t}}{\cm_{t}^{\p-2}}.
\end{eqnarray*}
Then by Lemma \ref{lem:Freedman-square}, we have with probability
at least $1-\delta/3$,
\begin{equation}
\sum_{t=1}^{T}\eta_{t}\left\Vert \bd_{t}^{\ru}\right\Vert ^{2}\leq\frac{14}{3}\max_{t\in\left[T\right]}\eta_{t}\cm_{t}^{2}\ln\frac{3}{\delta}+8\sum_{t=1}^{T}\frac{\sigma_{\lar}^{\p}\eta_{t}}{\cm_{t}^{\p-2}}\overset{\ln\frac{3}{\delta}\geq1}{\leq}\frac{14}{3}\max_{t\in\left[T\right]}\eta_{t}\cm_{t}^{2}\ln^{2}\frac{3}{\delta}+8\sum_{t=1}^{T}\frac{\sigma_{\lar}^{\p}\eta_{t}}{\cm_{t}^{\p-2}}.\label{eq:cvx-full-res-2}
\end{equation}
\item Lastly, there is
\begin{equation}
\sum_{t=1}^{T}\left\Vert \sqrt{\eta_{t}}\bd_{t}^{\rb}\right\Vert \overset{(\ref{eq:cvx-full-res-b-n})}{\leq}\sum_{t=1}^{T}\frac{\sigma_{\sma}\sigma_{\lar}^{\p-1}\sqrt{\eta_{t}}}{\cm_{t}^{\p-1}}+\frac{\sigma_{\lar}^{\p}G\sqrt{\eta_{t}}}{\alpha^{\p-1}\cm_{t}^{\p}}.\label{eq:cvx-full-res-3}
\end{equation}
\end{itemize}
Combine (\ref{eq:cvx-full-res-1}), (\ref{eq:cvx-full-res-2}) and
(\ref{eq:cvx-full-res-3}) to have with probability at least $1-\delta$,
\[
\hres_{T}^{\cvx}=8\max_{t\in\left[T\right]}\left(\sum_{s=1}^{t}N_{s}\right)^{2}+2\sum_{t=1}^{T}\eta_{t}\left\Vert \bd_{t}^{\ru}\right\Vert ^{2}+4\left(\sum_{t=1}^{T}\left\Vert \sqrt{\eta_{t}}\bd_{t}^{\rb}\right\Vert \right)^{2}+4G^{2}\sum_{t=1}^{T}\eta_{t}\leq\hc_{T}^{\cvx},
\]
where
\begin{align*}
\hc_{T}^{\cvx}\defeq & \left(\frac{2^{8}}{9}+\frac{28}{3}\right)\max_{t\in\left[T\right]}\eta_{t}\cm_{t}^{2}\ln^{2}\frac{3}{\delta}+16\sum_{t=1}^{T}\frac{\sigma_{\lar}^{\p}\eta_{t}}{\cm_{t}^{\p-2}}+128\sum_{t=1}^{T}\left(\frac{\sigma_{\sma}^{\p}\eta_{t}}{\cm_{t}^{\p-2}}+\frac{\sigma_{\lar}^{\p}G^{2}\eta_{t}}{\alpha^{\p-1}\cm_{t}^{\p}}\right)\ln\frac{3}{\delta}\\
 & +4\left(\sum_{t=1}^{T}\frac{\sigma_{\sma}\sigma_{\lar}^{\p-1}\sqrt{\eta_{t}}}{\cm_{t}^{\p-1}}+\frac{\sigma_{\lar}^{\p}G\sqrt{\eta_{t}}}{\alpha^{\p-1}\cm_{t}^{\p}}\right)^{2}+4G^{2}\sum_{t=1}^{T}\eta_{t}\\
= & \O\left(\max_{t\in\left[T\right]}\eta_{t}\cm_{t}^{2}\ln^{2}\frac{3}{\delta}+\sum_{t=1}^{T}\frac{\sigma_{\lar}^{\p}\eta_{t}}{\cm_{t}^{\p-2}}+\sum_{t=1}^{T}\left(\frac{\sigma_{\sma}^{\p}\eta_{t}}{\cm_{t}^{\p-2}}+\frac{\sigma_{\lar}^{\p}G^{2}\eta_{t}}{\alpha^{\p-1}\cm_{t}^{\p}}\right)\ln\frac{3}{\delta}\right.\\
 & \left.\quad+\left(\sum_{t=1}^{T}\frac{\sigma_{\sma}\sigma_{\lar}^{\p-1}\sqrt{\eta_{t}}}{\cm_{t}^{\p-1}}+\frac{\sigma_{\lar}^{\p}G\sqrt{\eta_{t}}}{\alpha^{\p-1}\cm_{t}^{\p}}\right)^{2}+\sum_{t=1}^{T}G^{2}\eta_{t}\right).
\end{align*}
Note that by AM-GM inequality
\begin{align*}
 & \max_{t\in\left[T\right]}\eta_{t}\cm_{t}^{2}\ln^{2}\frac{3}{\delta}+\left(\sum_{t=1}^{T}\frac{\sigma_{\sma}\sigma_{\lar}^{\p-1}\sqrt{\eta_{t}}}{\cm_{t}^{\p-1}}+\frac{\sigma_{\lar}^{\p}G\sqrt{\eta_{t}}}{\alpha^{\p-1}\cm_{t}^{\p}}\right)^{2}\\
\geq & 2\left(\max_{t\in\left[T\right]}\sqrt{\eta_{t}}\cm_{t}\right)\sum_{t=1}^{T}\left(\frac{\sigma_{\sma}\sigma_{\lar}^{\p-1}\sqrt{\eta_{t}}}{\cm_{t}^{\p-1}}+\frac{\sigma_{\lar}^{\p}G\sqrt{\eta_{t}}}{\alpha^{\p-1}\cm_{t}^{\p}}\right)\ln\frac{3}{\delta}\\
\overset{(a)}{\geq} & 2\sum_{t=1}^{T}\left(\frac{\sigma_{\sma}\sigma_{\lar}^{\p-1}\eta_{t}}{\cm_{t}^{\p-2}}+\frac{\sigma_{\lar}^{\p}G^{2}\eta_{t}}{(1-\alpha)\alpha^{\p-1}\cm_{t}^{\p}}\right)\ln\frac{3}{\delta}\overset{(b)}{\geq}2\sum_{t=1}^{T}\left(\frac{\sigma_{\sma}^{\p}\eta_{t}}{\cm_{t}^{\p-2}}+\frac{\sigma_{\lar}^{\p}G^{2}\eta_{t}}{\alpha^{\p-1}\cm_{t}^{\p}}\right)\ln\frac{3}{\delta},
\end{align*}
where $(a)$ is by $\cm_{t}\geq\frac{G}{1-\alpha}$ in Condition \ref{enu:cvx-dep-res-2}
and $(b)$ is due to $\sigma_{\lar}\geq\sigma_{\sma}$, $\p>1$ and
$\alpha\in\left(0,1\right)$. Hence, the order of $\hc_{T}^{\cvx}$
can be simplified into
\[
\O\left(\max_{t\in\left[T\right]}\eta_{t}\cm_{t}^{2}\ln^{2}\frac{3}{\delta}+\sum_{t=1}^{T}\frac{\sigma_{\lar}^{\p}\eta_{t}}{\cm_{t}^{\p-2}}+\left(\sum_{t=1}^{T}\frac{\sigma_{\sma}\sigma_{\lar}^{\p-1}\sqrt{\eta_{t}}}{\cm_{t}^{\p-1}}+\frac{\sigma_{\lar}^{\p}G\sqrt{\eta_{t}}}{\alpha^{\p-1}\cm_{t}^{\p}}\right)^{2}+\sum_{t=1}^{T}G^{2}\eta_{t}\right).
\]

Now let us bound $\eres_{T}^{\cvx}$. It can be done directly via
(\ref{eq:cvx-full-res-u-v}) and (\ref{eq:cvx-full-res-b-n}). Hence,
we omit the detail and claim $\eres_{T}^{\cvx}\leq\ec_{T}^{\cvx}$,
where $\ec_{T}^{\cvx}$ is a constant in the order of
\[
\O\left(\sum_{t=1}^{T}\frac{\sigma_{\lar}^{\p}\eta_{t}}{\cm_{t}^{\p-2}}+\left(\sum_{t=1}^{T}\frac{\sigma_{\sma}\sigma_{\lar}^{\p-1}\sqrt{\eta_{t}}}{\cm_{t}^{\p-1}}+\frac{\sigma_{\lar}^{\p}G\sqrt{\eta_{t}}}{\alpha^{\p-1}\cm_{t}^{\p}}\right)^{2}+\sum_{t=1}^{T}G^{2}\eta_{t}\right).
\]
\end{proof}

\subsection{Lemmas for Strongly Convex Functions}

In this section, we move to the strongly convex case, i.e., $\mu>0$
in Assumption \ref{assu:obj}. The algorithm that we study is Clipped
SGD.

\subsubsection{Two Core Inequalities}

We begin by introducing some notations that will be used later:
\begin{equation}
\Gamma_{t}\defeq\prod_{s=2}^{t}\frac{1+\mu\eta_{s-1}}{1+\mu\eta_{s}/2},\forall t\in\left[T+1\right],\label{eq:str-Gamma}
\end{equation}
which satisfies the equation
\begin{equation}
\Gamma_{t}(1+\mu\eta_{t})=\Gamma_{t+1}(1+\mu\eta_{t+1}/2),\forall t\in\left[T\right].\label{eq:str-Gamma-eq}
\end{equation}
Equipped with $\Gamma_{t}$, we redefine
\begin{align}
R_{t} & \defeq\max_{s\in\left[t\right]}\sqrt{\Gamma_{s}(1+\mu\eta_{s}/2)}\left\Vert \bx_{\star}-\bx_{s}\right\Vert ,\forall t\in\left[T\right],\label{eq:str-R}\\
N_{t} & \defeq\left\langle \sqrt{\frac{\Gamma_{t}}{1+\mu\eta_{t}/2}}\eta_{t}\bd_{t}^{\ru},\frac{\sqrt{\Gamma_{t}(1+\mu\eta_{t}/2)}(\bx_{\star}-\bx_{t})}{R_{t}}\right\rangle ,\forall t\in\left[T\right].\label{eq:str-N}
\end{align}
By their definitions, $R_{t}\in\F_{t-1}$ and $N_{t}\in\F_{t}$. Moreover,
$N_{t}$ is still a MDS due to
\begin{equation}
\E\left[N_{t}\mid\F_{t-1}\right]=\left\langle \sqrt{\frac{\Gamma_{t}}{1+\mu\eta_{t}/2}}\eta_{t}\E\left[\bd_{t}^{\ru}\mid\F_{t-1}\right],\frac{\sqrt{\Gamma_{t}(1+\mu\eta_{t}/2)}(\bx_{\star}-\bx_{t})}{R_{t}}\right\rangle =0,\forall t\in\left[T\right].\label{eq:str-N-MDS}
\end{equation}

Again, we first show the progress made by Clipped SGD after $T$ steps
in the following Lemma \ref{lem:str-hp-anytime}.
\begin{lem}
\label{lem:str-hp-anytime}Under Assumptions \ref{assu:minimizer},
\ref{assu:obj} (with $\mu>0$) and \ref{assu:lip}, if $\eta_{t}\leq\frac{\eta}{\mu}$
for some constant $\eta>0$, then for any $T\in\N$, Clipped SGD (Algorithm
\ref{alg:clipped-SGD}) guarantees
\[
\frac{\Gamma_{T+1}\left\Vert \bx_{\star}-\bx_{T+1}\right\Vert ^{2}}{2}+\sum_{t=1}^{T}\Gamma_{t}\eta_{t}\left(F(\bx_{t+1})-F_{\star}\right)\leq(1+\eta/2)D^{2}+2\hres_{T}^{\str},
\]
where
\[
\hres_{T}^{\str}\defeq4\max_{t\in\left[T\right]}\left(\sum_{s=1}^{t}N_{s}\right)^{2}+2\sum_{t=1}^{T}\Gamma_{t}\eta_{t}^{2}\left\Vert \bd_{t}^{\ru}\right\Vert ^{2}+\frac{2\eta+1}{\mu}\sum_{t=1}^{T}\Gamma_{t}\eta_{t}\left\Vert \bd_{t}^{\rb}\right\Vert ^{2}+4G^{2}\sum_{t=1}^{T}\Gamma_{t}\eta_{t}^{2}.
\]
\end{lem}
\begin{proof}
We invoke Lemma \ref{lem:descent} for Clipped SGD with $\mu>0$ and
$\by=\bx_{\star}$, then replace the subscript $t$ with $s$, and
multiply both sides by $\Gamma_{s}\eta_{s}$ to have
\begin{align*}
 & \Gamma_{s}\eta_{s}\left(F(\bx_{s+1})-F_{\star}\right)\\
\leq & \frac{\Gamma_{s}\left\Vert \bx_{\star}-\bx_{s}\right\Vert ^{2}}{2}-\frac{\Gamma_{s}(1+\mu\eta_{s})\left\Vert \bx_{\star}-\bx_{s+1}\right\Vert ^{2}}{2}+\left\langle \Gamma_{s}\eta_{s}\bd_{s}^{\rc},\bx_{\star}-\bx_{s}\right\rangle +\Gamma_{s}\eta_{s}^{2}\left\Vert \bd_{s}^{\rc}\right\Vert ^{2}+4\Gamma_{s}\eta_{s}^{2}G^{2}\\
\overset{(\ref{eq:str-Gamma-eq})}{=} & \frac{\Gamma_{s}\left\Vert \bx_{\star}-\bx_{s}\right\Vert ^{2}}{2}-\frac{\Gamma_{s+1}(1+\mu\eta_{s+1}/2)\left\Vert \bx_{\star}-\bx_{s+1}\right\Vert ^{2}}{2}+\left\langle \Gamma_{s}\eta_{s}\bd_{s}^{\rc},\bx_{\star}-\bx_{s}\right\rangle +\Gamma_{s}\eta_{s}^{2}\left\Vert \bd_{s}^{\rc}\right\Vert ^{2}+4\Gamma_{s}\eta_{s}^{2}G^{2},
\end{align*}
sum up which over $s$ from $1$ to $t\leq T$ to obtain
\begin{align}
 & \frac{\Gamma_{t+1}(1+\mu\eta_{t+1}/2)\left\Vert \bx_{\star}-\bx_{t+1}\right\Vert ^{2}}{2}+\sum_{s=1}^{t}\Gamma_{s}\eta_{s}\left(F(\bx_{s+1})-F_{\star}\right)\nonumber \\
\leq & \frac{\Gamma_{1}\left\Vert \bx_{\star}-\bx_{1}\right\Vert ^{2}}{2}-\frac{\mu}{4}\sum_{s=2}^{t}\Gamma_{s}\eta_{s}\left\Vert \bx_{\star}-\bx_{s}\right\Vert ^{2}+\sum_{s=1}^{t}\left\langle \Gamma_{s}\eta_{s}\bd_{s}^{\rc},\bx_{\star}-\bx_{s}\right\rangle +\sum_{s=1}^{t}\Gamma_{s}\eta_{s}^{2}\left\Vert \bd_{s}^{\rc}\right\Vert ^{2}+4G^{2}\sum_{s=1}^{t}\Gamma_{s}\eta_{s}^{2}\nonumber \\
= & \frac{D^{2}}{2}-\frac{\mu}{4}\sum_{s=2}^{t}\Gamma_{s}\eta_{s}\left\Vert \bx_{\star}-\bx_{s}\right\Vert ^{2}+\sum_{s=1}^{t}\left\langle \Gamma_{s}\eta_{s}\bd_{s}^{\rc},\bx_{\star}-\bx_{s}\right\rangle +\sum_{s=1}^{t}\Gamma_{s}\eta_{s}^{2}\left\Vert \bd_{s}^{\rc}\right\Vert ^{2}+4G^{2}\sum_{s=1}^{t}\Gamma_{s}\eta_{s}^{2},\label{eq:str-anytime-1}
\end{align}
where the last step holds by $\Gamma_{1}=1$ and $\left\Vert \bx_{\star}-\bx_{1}\right\Vert =D$.

We recall the decomposition $\bd_{s}^{\rc}=\bd_{s}^{\ru}+\bd_{s}^{\rb}$
to have
\begin{align*}
\sum_{s=1}^{t}\left\langle \Gamma_{s}\eta_{s}\bd_{s}^{\rc},\bx_{\star}-\bx_{s}\right\rangle  & =\sum_{s=1}^{t}\left\langle \Gamma_{s}\eta_{s}\bd_{s}^{\ru},\bx_{\star}-\bx_{s}\right\rangle +\sum_{s=1}^{t}\left\langle \Gamma_{s}\eta_{s}\bd_{s}^{\rb},\bx_{\star}-\bx_{s}\right\rangle \\
 & \overset{(\ref{eq:str-R}),(\ref{eq:str-N})}{=}\sum_{s=1}^{t}R_{s}N_{s}+\sum_{s=1}^{t}\left\langle \Gamma_{s}\eta_{s}\bd_{s}^{\rb},\bx_{\star}-\bx_{s}\right\rangle .
\end{align*}
By Lemma \ref{lem:Abel} and AM-GM inequality, there is
\[
\sum_{s=1}^{t}R_{s}N_{s}\leq2R_{t}\max_{S\in\left[t\right]}\left|\sum_{s=1}^{S}N_{s}\right|\leq\frac{R_{t}^{2}}{4}+4\max_{S\in\left[t\right]}\left(\sum_{s=1}^{S}N_{s}\right)^{2}.
\]
In addition, we use Cauchy-Schwarz inequality and AM-GM inequality
to bound
\[
\sum_{s=1}^{t}\left\langle \Gamma_{s}\eta_{s}\bd_{s}^{\rb},\bx_{\star}-\bx_{s}\right\rangle \leq\sum_{s=1}^{t}\Gamma_{s}\eta_{s}\left\Vert \bd_{s}^{\rb}\right\Vert \left\Vert \bx_{\star}-\bx_{s}\right\Vert \leq\sum_{s=1}^{t}\frac{\Gamma_{s}\eta_{s}\left\Vert \bd_{s}^{\rb}\right\Vert ^{2}}{\mu}+\frac{\mu\Gamma_{s}\eta_{s}\left\Vert \bx_{\star}-\bx_{s}\right\Vert ^{2}}{4}.
\]
As such, we obtain
\begin{equation}
\sum_{s=1}^{t}\left\langle \Gamma_{s}\eta_{s}\bd_{s}^{\rc},\bx_{\star}-\bx_{s}\right\rangle \leq\frac{R_{t}^{2}}{4}+4\max_{S\in\left[t\right]}\left(\sum_{s=1}^{S}N_{s}\right)^{2}+\sum_{s=1}^{t}\frac{\Gamma_{s}\eta_{s}\left\Vert \bd_{s}^{\rb}\right\Vert ^{2}}{\mu}+\frac{\mu\Gamma_{s}\eta_{s}\left\Vert \bx_{\star}-\bx_{s}\right\Vert ^{2}}{4}.\label{eq:str-anytime-2}
\end{equation}

Plug (\ref{eq:str-anytime-2}) back into (\ref{eq:str-anytime-1})
to get
\begin{align}
 & \frac{\Gamma_{t+1}(1+\mu\eta_{t+1}/2)\left\Vert \bx_{\star}-\bx_{t+1}\right\Vert ^{2}}{2}+\sum_{s=1}^{t}\Gamma_{s}\eta_{s}\left(F(\bx_{s+1})-F_{\star}\right)\nonumber \\
\leq & \frac{R_{t}^{2}}{4}+\frac{(1+\mu\eta_{1}/2)D^{2}}{2}+4\max_{S\in\left[t\right]}\left(\sum_{s=1}^{S}N_{s}\right)^{2}+\sum_{s=1}^{t}\frac{\Gamma_{s}\eta_{s}\left\Vert \bd_{s}^{\rb}\right\Vert ^{2}}{\mu}+\sum_{s=1}^{t}\Gamma_{s}\eta_{s}^{2}\left\Vert \bd_{s}^{\rc}\right\Vert ^{2}+4G^{2}\sum_{s=1}^{t}\Gamma_{s}\eta_{s}^{2}\nonumber \\
\leq & \frac{R_{t}^{2}}{4}+\frac{(1+\eta/2)D^{2}}{2}\nonumber \\
 & +\underbrace{4\max_{S\in\left[t\right]}\left(\sum_{s=1}^{S}N_{s}\right)^{2}+2\sum_{s=1}^{t}\Gamma_{s}\eta_{s}^{2}\left\Vert \bd_{s}^{\ru}\right\Vert ^{2}+\frac{2\eta+1}{\mu}\sum_{s=1}^{t}\Gamma_{s}\eta_{s}\left\Vert \bd_{s}^{\rb}\right\Vert ^{2}+4G^{2}\sum_{s=1}^{t}\Gamma_{s}\eta_{s}^{2}}_{\defeq\hres_{t}^{\str}},\label{eq:str-anytime-3}
\end{align}
where the last step is by $\eta_{1}\leq\eta/\mu$ and
\begin{align*}
\sum_{s=1}^{t}\Gamma_{s}\eta_{s}^{2}\left\Vert \bd_{s}^{\rc}\right\Vert ^{2} & =\sum_{s=1}^{t}\Gamma_{s}\eta_{s}^{2}\left\Vert \bd_{s}^{\ru}+\bd_{s}^{\rb}\right\Vert ^{2}\leq2\sum_{s=1}^{t}\Gamma_{s}\eta_{s}^{2}\left\Vert \bd_{s}^{\ru}\right\Vert ^{2}+2\sum_{s=1}^{t}\Gamma_{s}\eta_{s}^{2}\left\Vert \bd_{s}^{\rb}\right\Vert ^{2}\\
 & \overset{\eta_{s}\leq\eta/\mu,\forall s\in\left[T\right]}{\leq}2\sum_{s=1}^{t}\Gamma_{s}\eta_{s}^{2}\left\Vert \bd_{s}^{\ru}\right\Vert ^{2}+\frac{2\eta}{\mu}\sum_{s=1}^{t}\Gamma_{s}\eta_{s}\left\Vert \bd_{s}^{\rb}\right\Vert ^{2}.
\end{align*}

Now we let $a_{t}\defeq\frac{\Gamma_{t}(1+\mu\eta_{t}/2)\left\Vert \bx_{\star}-\bx_{t}\right\Vert ^{2}}{2},\forall t\in\left[T+1\right]$,
$b_{t}\defeq\sum_{s=1}^{t}\Gamma_{s}\eta_{s}\left(F(\bx_{s+1})-F_{\star}\right),\forall t\in\left[T\right]$
and $c_{t}\defeq\frac{(1+\eta/2)D^{2}}{2}+\hres_{t-1}^{\str},\forall t\in\left[T+1\right]$,
where $\hres_{0}^{\str}=0$. Note that $b_{t}$ is nonnegative, $c_{t}$
is nondecreasing, and 
\[
a_{1}=\frac{\Gamma_{1}(1+\mu\eta_{1}/2)\left\Vert \bx_{\star}-\bx_{1}\right\Vert ^{2}}{2}\leq\frac{(1+\eta/2)D^{2}}{2}\leq(1+\eta/2)D^{2}=2c_{1}.
\]
Moreover, (\ref{eq:str-anytime-3}) is saying that
\[
a_{t+1}+b_{t}\leq\frac{\max_{s\in\left[t\right]}a_{s}}{2}+c_{t+1},\forall t\in\left[T\right].
\]
Thus, we can invoke Lemma \ref{lem:algebra} to obtain
\[
a_{T+1}+b_{T}\leq2c_{T+1},
\]
which means
\[
\frac{\Gamma_{T+1}(1+\mu\eta_{T+1}/2)\left\Vert \bx_{\star}-\bx_{T+1}\right\Vert ^{2}}{2}+\sum_{t=1}^{T}\Gamma_{t}\eta_{t}\left(F(\bx_{t+1})-F_{\star}\right)\leq(1+\eta/2)D^{2}+2\hres_{T}^{\str}.
\]
Finally, we conclude from $\mu\eta_{T+1}\geq0$.
\end{proof}

Equipped with Lemma \ref{lem:str-hp-anytime}, we prove the following
in-expectation convergence result for Clipped SGD under strong convexity.
\begin{lem}
\label{lem:str-ex-anytime}Under the same setting in Lemma \ref{lem:str-hp-anytime},
Clipped SGD (Algorithm \ref{alg:clipped-SGD}) guarantees
\[
\frac{\Gamma_{T+1}\E\left[\left\Vert \bx_{\star}-\bx_{T+1}\right\Vert ^{2}\right]}{2}+\sum_{t=1}^{T}\Gamma_{t}\eta_{t}\E\left[F(\bx_{t+1})-F_{\star}\right]\leq(1+\eta/2)D^{2}+2\eres_{T}^{\str},
\]
where 
\[
\eres_{T}^{\str}\defeq18\sum_{t=1}^{T}\Gamma_{t}\eta_{t}^{2}\E\left[\left\Vert \bd_{t}^{\ru}\right\Vert ^{2}\right]+\frac{2\eta+1}{\mu}\sum_{t=1}^{T}\Gamma_{t}\eta_{t}\E\left[\left\Vert \bd_{t}^{\rb}\right\Vert ^{2}\right]+4G^{2}\sum_{t=1}^{T}\Gamma_{t}\eta_{t}^{2}.
\]
\end{lem}
\begin{proof}
Similar to the proof of Lemma \ref{lem:cvx-ex-anytime}, we take expectations
on both sides of Lemma \ref{lem:str-hp-anytime} and then invoke Lemma
\ref{lem:Doob}. The calculations are omitted here to save space.
\end{proof}

\subsubsection{Bounding Residual Terms}

Like previously, we need to upper bound $\hres_{T}^{\str}$ and $\eres_{T}^{\str}$,
which is done in the following lemma.
\begin{lem}
\label{lem:str-dep-res}Under Assumptions \ref{assu:lip}, \ref{assu:oracle}
and the following two conditions:
\begin{enumerate}
\item \label{enu:str-dep-res-1}$\eta_{t}$ and $\cm_{t}$ are deterministic
for all $t\in\left[T\right]$.
\item \label{enu:str-dep-res-2}$\cm_{t}\geq\frac{G}{1-\alpha}$ holds for
some constant $\alpha\in\left(0,1\right)$ and all $t\in\left[T\right]$.
\end{enumerate}
We have:
\begin{enumerate}
\item for any $\delta\in\left(0,1\right]$, with probability at least $1-\delta$,
$\hres_{T}^{\str}\leq\hc_{T}^{\str}$ where $\hres_{T}^{\str}$ is
defined in Lemma \ref{lem:str-hp-anytime} and $\hc_{T}^{\str}$ is
a constant in the order of
\begin{align*}
\O & \left(\max_{t\in\left[T\right]}\Gamma_{t}\eta_{t}^{2}\cm_{t}^{2}\ln^{2}\frac{3}{\delta}+\sum_{t=1}^{T}\frac{\sigma_{\lar}^{\p}\Gamma_{t}\eta_{t}^{2}}{\cm_{t}^{\p-2}}+\sum_{t=1}^{T}\left(\frac{\sigma_{\sma}^{\p}\Gamma_{t}\eta_{t}^{2}}{\cm_{t}^{\p-2}}+\frac{\sigma_{\lar}^{\p}G^{2}\Gamma_{t}\eta_{t}^{2}}{\alpha^{\p-1}\cm_{t}^{\p}}\right)\ln\frac{3}{\delta}\right.\\
 & \left.\quad+\sum_{t=1}^{T}\left(\frac{\sigma_{\sma}^{2}\sigma_{\lar}^{2\p-2}\Gamma_{t}\eta_{t}}{\cm_{t}^{2\p-2}}+\frac{\sigma_{\lar}^{2\p}G^{2}\Gamma_{t}\eta_{t}}{\alpha^{2\p-2}\cm_{t}^{2\p}}\right)\frac{2\eta+1}{\mu}+\sum_{t=1}^{T}G^{2}\Gamma_{t}\eta_{t}^{2}\right).
\end{align*}
\item $\eres_{T}^{\str}\leq\ec_{T}^{\str}$ where $\eres_{T}^{\str}$ is
defined in Lemma \ref{lem:str-ex-anytime} and $\ec_{T}^{\str}$ is
a constant in the order of
\[
\O\left(\sum_{t=1}^{T}\frac{\sigma_{\lar}^{\p}\Gamma_{t}\eta_{t}^{2}}{\cm_{t}^{\p-2}}+\sum_{t=1}^{T}\left(\frac{\sigma_{\sma}^{2}\sigma_{\lar}^{2\p-2}\Gamma_{t}\eta_{t}}{\cm_{t}^{2\p-2}}+\frac{\sigma_{\lar}^{2\p}G^{2}\Gamma_{t}\eta_{t}}{\alpha^{2\p-2}\cm_{t}^{2\p}}\right)\frac{2\eta+1}{\mu}+\sum_{t=1}^{T}G^{2}\Gamma_{t}\eta_{t}^{2}\right).
\]
\end{enumerate}
\end{lem}
\begin{proof}
We observe that for any $t\in\left[T\right]$, $\cm_{t}\geq\frac{G}{1-\alpha}\geq\frac{\left\Vert \nabla f(\bx_{t})\right\Vert }{1-\alpha}$
holds almost surely due to Condition \ref{enu:str-dep-res-2} and
Assumption \ref{assu:lip}, implying that $\chi_{t}(\alpha)$ in Lemma
\ref{lem:clip-ineq} equals $1$ for all $t\in\left[T\right]$. Then
Lemma \ref{lem:clip-ineq} and Assumption \ref{assu:lip} together
yield the following inequalities holding for any $t\in\left[T\right]$:
\begin{align}
\sqrt{\Gamma_{t}}\eta_{t}\left\Vert \bd_{t}^{\ru}\right\Vert  & \overset{\text{Inequality }\ref{enu:clip-ineq-1}}{\leq}2\sqrt{\Gamma_{t}}\eta_{t}\cm_{t}\leq2\max_{t\in\left[T\right]}\sqrt{\Gamma_{t}}\eta_{t}\cm_{t},\label{eq:str-full-res-u-n}\\
\E\left[\Gamma_{t}\eta_{t}^{2}\left\Vert \bd_{t}^{\ru}\right\Vert ^{2}\mid\F_{t-1}\right] & \overset{\text{Inequality }\ref{enu:clip-ineq-2}}{\leq}\frac{4\sigma_{\lar}^{\p}\Gamma_{t}\eta_{t}^{2}}{\cm_{t}^{\p-2}},\label{eq:str-full-res-u-v}\\
\left\Vert \E\left[\Gamma_{t}\eta_{t}^{2}\bd_{t}^{\ru}\left(\bd_{t}^{\ru}\right)^{\top}\mid\F_{t-1}\right]\right\Vert  & \overset{\text{Inequality }\ref{enu:clip-ineq-4}}{\leq}\frac{4\sigma_{\sma}^{\p}\Gamma_{t}\eta_{t}^{2}}{\cm_{t}^{\p-2}}+\frac{4\sigma_{\lar}^{\p}G^{2}\Gamma_{t}\eta_{t}^{2}}{\alpha^{\p-1}\cm_{t}^{\p}},\label{eq:str-full-res-u-o}\\
\sqrt{\Gamma_{t}\eta_{t}}\left\Vert \bd_{t}^{\rb}\right\Vert  & \overset{\text{Inequality }\ref{enu:clip-ineq-6}}{\leq}\frac{\sigma_{\sma}\sigma_{\lar}^{\p-1}\sqrt{\Gamma_{t}\eta_{t}}}{\cm_{t}^{\p-1}}+\frac{\sigma_{\lar}^{\p}G\sqrt{\Gamma_{t}\eta_{t}}}{\alpha^{\p-1}\cm_{t}^{\p}}.\label{eq:str-full-res-b-n}
\end{align}
\begin{itemize}
\item Similar to (\ref{eq:cvx-full-res-1}), we can prove now with probability
at least $1-2\delta/3$,
\begin{equation}
\max_{t\in\left[T\right]}\left(\sum_{s=1}^{t}N_{s}\right)^{2}\leq\frac{2^{5}}{9}\max_{t\in\left[T\right]}\Gamma_{t}\eta_{t}^{2}\cm_{t}^{2}\ln^{2}\frac{3}{\delta}+16\sum_{t=1}^{T}\left(\frac{\sigma_{\sma}^{\p}\Gamma_{t}\eta_{t}^{2}}{\cm_{t}^{\p-2}}+\frac{\sigma_{\lar}^{\p}G^{2}\Gamma_{t}\eta_{t}^{2}}{\alpha^{\p-1}\cm_{t}^{\p}}\right)\ln\frac{3}{\delta}.\label{eq:str-full-res-1}
\end{equation}
\item Similar to (\ref{eq:cvx-full-res-2}), we can prove now with probability
at least $1-\delta/3$,
\begin{equation}
\sum_{t=1}^{T}\Gamma_{t}\eta_{t}^{2}\left\Vert \bd_{t}^{\ru}\right\Vert ^{2}\leq\frac{14}{3}\max_{t\in\left[T\right]}\Gamma_{t}\eta_{t}^{2}\cm_{t}^{2}\ln^{2}\frac{3}{\delta}+8\sum_{t=1}^{T}\frac{\sigma_{\lar}^{\p}\Gamma_{t}\eta_{t}^{2}}{\cm_{t}^{\p-2}}.\label{eq:str-full-res-2}
\end{equation}
\item Lastly, there is
\begin{align}
\sum_{t=1}^{T}\Gamma_{t}\eta_{t}\left\Vert \bd_{t}^{\rb}\right\Vert ^{2} & \overset{(\ref{eq:str-full-res-b-n})}{\leq}\sum_{t=1}^{T}\left(\frac{\sigma_{\sma}\sigma_{\lar}^{\p-1}\sqrt{\Gamma_{t}\eta_{t}}}{\cm_{t}^{\p-1}}+\frac{\sigma_{\lar}^{\p}G\sqrt{\Gamma_{t}\eta_{t}}}{\alpha^{\p-1}\cm_{t}^{\p}}\right)^{2}\nonumber \\
 & \leq\sum_{t=1}^{T}\left(\frac{2\sigma_{\sma}^{2}\sigma_{\lar}^{2\p-2}\Gamma_{t}\eta_{t}}{\cm_{t}^{2\p-2}}+\frac{2\sigma_{\lar}^{2\p}G^{2}\Gamma_{t}\eta_{t}}{\alpha^{2\p-2}\cm_{t}^{2\p}}\right).\label{eq:str-full-res-3}
\end{align}
\end{itemize}
We combine (\ref{eq:str-full-res-1}), (\ref{eq:str-full-res-2})
and (\ref{eq:str-full-res-3}) to have with probability at least $1-\delta$,
\[
\hres_{T}^{\str}=4\max_{t\in\left[T\right]}\left(\sum_{s=1}^{t}N_{s}\right)^{2}+2\sum_{t=1}^{T}\Gamma_{t}\eta_{t}^{2}\left\Vert \bd_{t}^{\ru}\right\Vert ^{2}+\frac{2\eta+1}{\mu}\sum_{t=1}^{T}\Gamma_{t}\eta_{t}\left\Vert \bd_{t}^{\rb}\right\Vert ^{2}+4G^{2}\sum_{t=1}^{T}\Gamma_{t}\eta_{t}^{2}\leq\hc_{T}^{\str},
\]
where
\begin{align*}
\hc_{T}^{\str}\defeq & \left(\frac{2^{7}}{9}+\frac{28}{3}\right)\max_{t\in\left[T\right]}\Gamma_{t}\eta_{t}^{2}\cm_{t}^{2}\ln^{2}\frac{3}{\delta}+16\sum_{t=1}^{T}\frac{\sigma_{\lar}^{\p}\Gamma_{t}\eta_{t}^{2}}{\cm_{t}^{\p-2}}+64\sum_{t=1}^{T}\left(\frac{\sigma_{\sma}^{\p}\Gamma_{t}\eta_{t}^{2}}{\cm_{t}^{\p-2}}+\frac{\sigma_{\lar}^{\p}G^{2}\Gamma_{t}\eta_{t}^{2}}{\alpha^{\p-1}\cm_{t}^{\p}}\right)\ln\frac{3}{\delta}\\
 & +\sum_{t=1}^{T}\left(\frac{2\sigma_{\sma}^{2}\sigma_{\lar}^{2\p-2}\Gamma_{t}\eta_{t}}{\cm_{t}^{2\p-2}}+\frac{2\sigma_{\lar}^{2\p}G^{2}\Gamma_{t}\eta_{t}}{\alpha^{2\p-2}\cm_{t}^{2\p}}\right)\frac{2\eta+1}{\mu}+4\sum_{t=1}^{T}G^{2}\Gamma_{t}\eta_{t}^{2}\\
= & \O\left(\max_{t\in\left[T\right]}\Gamma_{t}\eta_{t}^{2}\cm_{t}^{2}\ln^{2}\frac{3}{\delta}+\sum_{t=1}^{T}\frac{\sigma_{\lar}^{\p}\Gamma_{t}\eta_{t}^{2}}{\cm_{t}^{\p-2}}+\sum_{t=1}^{T}\left(\frac{\sigma_{\sma}^{\p}\Gamma_{t}\eta_{t}^{2}}{\cm_{t}^{\p-2}}+\frac{\sigma_{\lar}^{\p}G^{2}\Gamma_{t}\eta_{t}^{2}}{\alpha^{\p-1}\cm_{t}^{\p}}\right)\ln\frac{3}{\delta}\right.\\
 & \left.\quad+\sum_{t=1}^{T}\left(\frac{\sigma_{\sma}^{2}\sigma_{\lar}^{2\p-2}\Gamma_{t}\eta_{t}}{\cm_{t}^{2\p-2}}+\frac{\sigma_{\lar}^{2\p}G^{2}\Gamma_{t}\eta_{t}}{\alpha^{2\p-2}\cm_{t}^{2\p}}\right)\frac{2\eta+1}{\mu}+\sum_{t=1}^{T}G^{2}\Gamma_{t}\eta_{t}^{2}\right).
\end{align*}

Now let us bound $\eres_{T}^{\str}$. It can be done directly via
(\ref{eq:str-full-res-u-v}) and (\ref{eq:str-full-res-b-n}). Hence,
we omit the detail and claim $\eres_{T}^{\str}\leq\ec_{T}^{\str}$,
where $\ec_{T}^{\str}$ is a constant in the order of
\[
\O\left(\sum_{t=1}^{T}\frac{\sigma_{\lar}^{\p}\Gamma_{t}\eta_{t}^{2}}{\cm_{t}^{\p-2}}+\sum_{t=1}^{T}\left(\frac{\sigma_{\sma}^{2}\sigma_{\lar}^{2\p-2}\Gamma_{t}\eta_{t}}{\cm_{t}^{2\p-2}}+\frac{\sigma_{\lar}^{2\p}G^{2}\Gamma_{t}\eta_{t}}{\alpha^{2\p-2}\cm_{t}^{2\p}}\right)\frac{2\eta+1}{\mu}+\sum_{t=1}^{T}G^{2}\Gamma_{t}\eta_{t}^{2}\right).
\]
\end{proof}

\subsection{Existing Technical Results}

This section contains some technical results existing (or implicitly
used) in prior works.

First, Lemma \ref{lem:Freedman} is the famous Freedman's inequality,
a useful tool to bound a real-valued MDS.
\begin{lem}[Freedman's inequality \citep{10.1214/aop/1176996452}]
\label{lem:Freedman}Suppose $X_{t}\in\R,\forall t\in\left[T\right]$
is a real-valued MDS adapted to the filtration $\F_{t},\forall t\in\left\{ 0\right\} \cup\left[T\right]$
satisfying for any $t\in\left[T\right]$, $X_{t}\leq b$ and $\E\left[X_{t}^{2}\mid\F_{t-1}\right]\leq\sigma_{t}^{2}$
almost surely, where $b\geq0$ and $\sigma_{t}^{2}$ are both constant,
then for any $\delta\in\left(0,1\right]$, there is
\[
\Pr\left[\sum_{s=1}^{t}X_{s}\leq\frac{2b}{3}\ln\frac{1}{\delta}+\sqrt{2\sum_{s=1}^{T}\sigma_{s}^{2}\ln\frac{1}{\delta}},\forall t\in\left[T\right]\right]\geq1-\delta.
\]
\end{lem}
Next, Lemma \ref{lem:Freedman-square} is another concentration inequality.
This is not a new result, and similar ideas were used before in, e.g.,
\citet{NEURIPS2021_26901deb,NEURIPS2022_349956de,liu2023stochastic}.
We provide a proof here to make the work self-contained.
\begin{lem}
\label{lem:Freedman-square}Suppose $X_{t}\in\R,\forall t\in\left[T\right]$
is a sequence of random variables adapted to the filtration $\F_{t},\forall t\in\left\{ 0\right\} \cup\left[T\right]$
satisfying for any $t\in\left[T\right]$, $\left|X_{t}\right|\leq b$
and $\E\left[X_{t}^{2}\mid\F_{t-1}\right]\leq\sigma_{t}^{2}$ almost
surely, where $b\geq0$ and $\sigma_{t}^{2}$ are both constant, then
for any $\delta\in\left(0,1\right]$, there is
\[
\Pr\left[\sum_{t=1}^{T}X_{t}^{2}\leq\frac{7b^{2}}{6}\ln\frac{1}{\delta}+2\sum_{t=1}^{T}\sigma_{t}^{2}\right]\geq1-\delta.
\]
\end{lem}
\begin{proof}
Note that we can bound
\[
\sum_{t=1}^{T}X_{t}^{2}=\sum_{t=1}^{T}\underbrace{X_{t}^{2}-\E\left[X_{t}^{2}\mid\F_{t-1}\right]}_{\defeq Y_{t}}+\sum_{t=1}^{T}\E\left[X_{t}^{2}\mid\F_{t-1}\right]\leq\sum_{t=1}^{T}Y_{t}+\sum_{t=1}^{T}\sigma_{t}^{2}.
\]
Observe that $Y_{t},\forall t\in\left[T\right]$ is a real-valued
MDS adapted to the filtration $\F_{t},\forall t\in\left\{ 0\right\} \cup\left[T\right]$
satisfying
\begin{eqnarray*}
Y_{t}\leq X_{t}^{2}\leq b^{2} & \text{and} & \E\left[Y_{t}^{2}\mid\F_{t-1}\right]\leq\E\left[X_{t}^{4}\mid\F_{t-1}\right]\leq b^{2}\sigma_{t}^{2}.
\end{eqnarray*}
Then Lemma \ref{lem:Freedman} yields that, for any $\delta\in\left(0,1\right]$,
we have with probability at least $1-\delta$,
\[
\sum_{s=1}^{t}Y_{s}\leq\frac{2b^{2}}{3}\ln\frac{1}{\delta}+\sqrt{2\sum_{s=1}^{T}b^{2}\sigma_{s}^{2}\ln\frac{1}{\delta}},\forall t\in\left[T\right],
\]
which implies
\[
\sum_{t=1}^{T}Y_{t}\leq\frac{2b^{2}}{3}\ln\frac{1}{\delta}+\sqrt{2\sum_{t=1}^{T}b^{2}\sigma_{t}^{2}\ln\frac{1}{\delta}}\leq\frac{7b^{2}}{6}\ln\frac{1}{\delta}+\sum_{t=1}^{T}\sigma_{t}^{2}
\]
where the last step is by $\sqrt{2\sum_{t=1}^{T}b^{2}\sigma_{t}^{2}\ln\frac{1}{\delta}}\leq\frac{b^{2}}{2}\ln\frac{1}{\delta}+\sum_{t=1}^{T}\sigma_{t}^{2}$
due to AM-GM inequality. Hence, it follows that
\[
\Pr\left[\sum_{t=1}^{T}X_{t}^{2}\leq\frac{7b^{2}}{6}\ln\frac{1}{\delta}+2\sum_{t=1}^{T}\sigma_{t}^{2}\right]\geq1-\delta.
\]
\end{proof}

The following Lemma \ref{lem:Doob} is the famous Doob's $L^{2}$
maximum inequality. For its proof, see, e.g., Theorem 4.4.4 in \citet{Durrett_2019}.
\begin{lem}[Doob's $L^{2}$ maximum inequality]
\label{lem:Doob}Suppose $X_{t}\in\R,\forall t\in\left[T\right]$
is a real-valued MDS, then there is
\[
\E\left[\max_{t\in\left[T\right]}\left(\sum_{s=1}^{t}X_{s}\right)^{2}\right]\leq4\sum_{t=1}^{T}\E\left[X_{t}^{2}\right].
\]
\end{lem}
In addition, we need the following algebraic fact in our analysis.
\begin{lem}[Lemma C.2 in \citet{pmlr-v202-ivgi23a}]
\label{lem:Abel}Let $a_{1},\mydots,a_{T}$ and $b_{1},\mydots,b_{T}$
be two sequences in $\R$ such that $a_{t}$ is nonnegative and nondecreasing,
then there is
\[
\left|\sum_{s=1}^{t}a_{s}b_{s}\right|\leq2a_{t}\max_{S\in\left[t\right]}\left|\sum_{s=1}^{S}b_{s}\right|,\forall t\in\left[T\right].
\]
\end{lem}
Lastly, we introduce another algebraic inequality, the idea behind
which can also be found in previous works like \citet{pmlr-v202-ivgi23a,liu2023stochastic}.
For completeness, we produce a proof here.
\begin{lem}
\label{lem:algebra}Let $a_{1},\mydots,a_{T+1}$, $b_{1},\mydots,b_{T}$
and $c_{1},\mydots,c_{T+1}$ be three sequences in $\R$ such that
$b_{t}$ is nonnegative and $c_{t}$ is nondecreasing, if $a_{1}\leq2c_{1}$
and
\[
a_{t+1}+b_{t}\leq\frac{\max_{s\in\left[t\right]}a_{s}}{2}+c_{t+1},\forall t\in\left[T\right],
\]
then there is
\[
a_{T+1}+b_{T}\leq2c_{T+1}.
\]
\end{lem}
\begin{proof}
We first use induction to show
\begin{equation}
a_{t}\leq2c_{t},\forall t\in\left[T\right].\label{eq:algebra-hypothesis}
\end{equation}
For the base case $t=1$, we know $a_{1}\leq2c_{1}$ by the assumption.
Suppose (\ref{eq:algebra-hypothesis}) holds for all time not greater
than $t$ for some $t\in\left[T-1\right]$. Then for time $t+1$,
we know
\[
a_{t+1}\overset{b_{t}\geq0}{\leq}a_{t+1}+b_{t}\leq\frac{\max_{s\in\left[t\right]}a_{s}}{2}+c_{t+1}\overset{(\ref{eq:algebra-hypothesis})}{\leq}\frac{\max_{s\in\left[t\right]}2c_{s}}{2}+c_{t+1}\leq2c_{t+1},
\]
where the last inequality holds because $c_{t}$ is nondecreasing.
Therefore, (\ref{eq:algebra-hypothesis}) is true by induction. Hence,
we know
\[
a_{T+1}+b_{T}\leq\frac{\max_{s\in\left[T\right]}a_{s}}{2}+c_{T+1}\overset{(\ref{eq:algebra-hypothesis})}{\leq}\frac{\max_{s\in\left[T\right]}2c_{s}}{2}+c_{T+1}\leq2c_{T+1},
\]
where the last step is also because $c_{t}$ is nondecreasing.
\end{proof}

\section{Full Theorems for Lower Bounds and Proofs\label{sec:formal-lb}}

This section aims to prove the lower bounds stated in Section \ref{sec:informal-lb}.

\subsection{Basic Background and Problem Formulation}

In this subsection, we provide the basic background and problem formulation
for proving lower bounds.

To begin with, given two parameters $0<\sigma_{\sma}\leq\sigma_{\lar}$,
it is only reasonable to consider the case $d\geq d_{\eff}$ as discussed
in Section \ref{sec:preliminary}. Moreover, for any $d\geq d_{\eff}$,
we stick to $\X=\R^{d}$.

\paragraph{Function class.}

For any given $d\geq d_{\eff}$, we introduce the following function
class,
\begin{equation}
\mathfrak{f}_{G}^{\cvx}\defeq\left\{ f:\R^{d}\to\R:\text{\ensuremath{f} is covnex and \ensuremath{G}-Lipschitz on \ensuremath{\R^{d}}}\right\} .\label{eq:lb-f-class}
\end{equation}

\begin{itemize}
\item For the general convex case, i.e., $\mu=0$, we consider
\begin{equation}
\mathfrak{F}_{D,G}^{\cvx}\defeq\left\{ F\in\mathfrak{f}_{G}^{\cvx}:\inf_{\bx_{\star}\in\argmin_{\bx\in\R^{d}}F(\bx)}\left\Vert \bx_{\star}\right\Vert \leq D\right\} .\label{eq:lb-cvx-class}
\end{equation}
In other words, for convex problems, we simply let $r(\bx)=0$.
\item For the strongly convex case, given $\mu>0$, we consider
\begin{equation}
\mathfrak{F}_{D,G}^{\str}\defeq\left\{ F\in\mathfrak{f}_{G}^{\cvx}+\frac{\mu}{2}\left\Vert \bx\right\Vert ^{2}:\inf_{\bx_{\star}\in\argmin_{\bx\in\R^{d}}F(\bx)}\left\Vert \bx_{\star}\right\Vert \leq D\right\} .\label{eq:lb-str-class}
\end{equation}
In other words, for strongly convex problems, we set $r(\bx)=\frac{\mu}{2}\left\Vert \bx\right\Vert ^{2}$.
\end{itemize}
The above specification of $r(\bx)$ does not hurt the generality
of our results, since we are proving lower bounds.

\paragraph{Stochastic first-order oracle.}

Following the literature, we define the class of stochastic first-order
oracle as following,
\[
\G_{\sigma_{\sma},\sigma_{\lar}}^{\p}\defeq\left\{ \bg:\R^{d}\times\mathfrak{f}_{G}^{\cvx}\to\R^{d}:\substack{\E\left[\bg(\bx,f)\mid\bx,f\right]=\nabla f(\bx)\in\partial f(\bx)\\
\E\left[\left|\left\langle \be,\bg(\bx,f)-\nabla f(\bx)\right\rangle \right|^{\p}\mid\bx,f\right]\leq\sigma_{\sma}^{\p},\forall\be\in\S^{d-1}\\
\E\left[\left\Vert \bg(\bx,f)-\nabla f(\bx)\right\Vert ^{\p}\mid\bx,f\right]\leq\sigma_{\lar}^{\p}
}
,\forall\bx\in\R^{d},f\in\mathfrak{f}_{G}^{\cvx}\right\} ,
\]
where $\bg$ is a random map.

\paragraph{Optimization algorithm.}

We define the algorithm set $\A_{T}$, the class of all possible optimization
methods that have $T$ iterations, as follows,
\[
\A_{T}\defeq\left\{ \left\{ {\bf A}_{0},\mydots,{\bf A}_{T}\right\} :{\bf A}_{t}:\left\{ r\right\} \times\left(\R^{d}\right)^{t}\to\R^{d},\forall t\in\left\{ 0\right\} \cup\left[T\right]\right\} ,
\]
where ${\bf A}_{t}$ is any measurable map and $r$ is assumed to
be fully revealed to the algorithm.
\begin{rem}
For simplicity, we only consider the class of deterministic algorithms;
extending the proof to randomized algorithms is straightforward.
\end{rem}

\paragraph{Optimization protocol.}

With the above definitions, the whole procedure for optimizing an
$F=f+r\in\mathfrak{F}_{D,G}^{\mathrm{type}}$ (where $\mathrm{\type}\in\left\{ \mathrm{cvx},\mathrm{str}\right\} $)
by an algorithm ${\bf A}_{0:T}\in\A_{T}$ (where ${\bf A}_{0:T}$
is the shorthand for the series of functions) interacting with a stochastic
first-order $\bg\in\G_{\sigma_{\sma},\sigma_{\lar}}^{\p}$ can be
described as below:
\begin{enumerate}
\item At the beginning, the algorithm chooses $\bx_{1}={\bf A}_{0}(r(\cdot))$.
\item At the $t$-th iteration for $t\in\left[T\right]$, the algorithm
queries and observes the stochastic gradient $\bg_{t}=\bg(\bx_{t},f)$
and sets $\bx_{t+1}={\bf A}_{t}\left(r(\cdot),\bg_{1},\mydots,\bg_{t}\right)$.
\end{enumerate}

\paragraph{Minimax lower bound.}

Under the above protocol, our goal is to lower bound the following
two quantities for $\mathrm{\type}\in\left\{ \mathrm{cvx},\mathrm{str}\right\} $
and $\delta\in\left(0,1\right]$ (where we recall $F_{\star}=\inf_{\bx\in\R^{d}}F(\bx)$),
\begin{align}
R_{\star}^{\mathrm{\type}} & \defeq\sup_{\bg\in\G_{\sigma_{\sma},\sigma_{\lar}}^{\p}}\inf_{{\bf A}_{0:T}\in\A_{T}}\sup_{F\in\mathfrak{F}_{D,G}^{\mathrm{\type}}}\E\left[F(\bx_{T+1})-F_{\star}\right],\label{eq:ex-lb-minimax}\\
R_{\star}^{{\rm \type}}(\delta) & \defeq\sup_{\bg\in\G_{\sigma_{\sma},\sigma_{\lar}}^{\p}}\inf_{{\bf A}_{0:T}\in\A_{T}}\sup_{F\in\mathfrak{F}_{D,G}^{\mathrm{\type}}}\inf\left\{ \epsilon\geq0:\Pr\left[F(\bx_{T+1})-F_{\star}>\epsilon\right]\leq\delta\right\} .\label{eq:hp-lb-minimax}
\end{align}

\paragraph{Other notation.}

Given two probability distributions $\mathbb{P}$ and $\mathbb{Q}$,
$\TV(\mathbb{P},\mathbb{Q})\defeq\frac{1}{2}\int\left|\d\mathbb{P}-\d\mathbb{Q}\right|$
denotes the $\TV$ distance, and $\KL(\mathbb{P}\Vert\mathbb{Q})\defeq\begin{cases}
\int\ln\left(\frac{\d\mathbb{P}}{\d\mathbb{Q}}\right)\d\mathbb{P} & \mathbb{P}\ll\mathbb{Q}\\
+\infty & \text{otherwise}
\end{cases}$ is the $\KL$ divergence. Given two random variables $X$ and $Y$,
$\I(X;Y)\defeq\KL(\mathbb{P}_{X,Y}\Vert\mathbb{P}_{X}\mathbb{P}_{Y})$
is the mutual information. For two vectors $u$ and $v$ with the
same length of $d$, $u\odot v$ is the vector obtained by coordinate-wise
production, i.e., $(u\odot v)_{i}\defeq u_{i}v_{i},\forall i\in\left[d\right]$
and $\Delta_{\H}(u,v)\defeq\sum_{i=1}^{d}\1\left[u_{i}\neq v_{i}\right]$
denotes the Hamming distance.

\subsection{Hard Function and Oracle}

We introduce the hard function and stochastic first-order oracle that
will be used in the later proof. In the following, let $d\geq d_{\eff}$
be fixed. Moreover, we write
\begin{eqnarray*}
\V\defeq\left\{ \pm1\right\} ^{d} & \text{and} & \Xi\defeq\left\{ -1,0,1\right\} ^{d}.
\end{eqnarray*}

\paragraph{A useful distribution.}

Inspired by \citet{NIPS2013_2812e5cf}, given $v\in\V$, let $\dis_{v}$
be a probability distribution on $\Xi$, such that all coordinates
of $\xi\sim\dis_{v}$ are mutually independent, i.e., $\dis_{v}=\prod_{i=1}^{d}\dis_{v,i}$.
For any $i\in\left[d\right]$, the marginal probability distribution
$\dis_{v,i}$ satisfies

\begin{eqnarray}
\dis_{v,i}\left[\xi_{i}=0\right]=1-q_{i}, & \dis_{v,i}\left[\xi_{i}=1\right]=\frac{1+v_{i}\theta_{i}}{2}q_{i}, & \dis_{v,i}\left[\xi_{i}=-1\right]=\frac{1-v_{i}\theta_{i}}{2}q_{i},\label{eq:lb-dis}
\end{eqnarray}
where $q_{i}\in\left[0,1\right]$ and $\theta_{i}\in\left[0,1\right]$
will be picked in the proof.

\subsubsection{General Convex Case}

We define the convex function $f:\R^{d}\times\Xi\to\R$ as
\begin{equation}
f(\bx,\xi)\defeq\sum_{i=1}^{d}M_{i}\left|\xi_{i}\right|\left|\bx_{i}-\xi_{i}\by_{i}\right|,\label{eq:cvx-lb-stoc-f}
\end{equation}
where $\bx_{i}$ denotes the $i$-th coordinate of $\bx$\footnote{Here, we slightly abuse the notation to use $\bx_{i}$ to denote the
$i$-th coordinate of $\bx$ and $\bx_{t}$ to denote the optimization
trajectory at the $t$-th iteration. Similarly, we also have $\xi_{i}$
and $\xi_{t}$. In the proof, the subscripts $i$ and $t$ will not
be used simultaneously.}, $M_{i}\geq0,\forall i\in\left[d\right]$ and $\by\in\R^{d}$ will
be determined later in the proof. Equipped with $f(\bx,\xi)$, we
introduce the following function $f_{v}:\R^{d}\to\R$ labelled by
$v\in\V$,
\begin{equation}
f_{v}(\bx)\defeq\E_{\dis_{v}}\left[f(\bx,\xi)\right]=\sum_{i=1}^{d}M_{i}q_{i}\left(\frac{1+v_{i}\theta_{i}}{2}\left|\bx_{i}-\by_{i}\right|+\frac{1-v_{i}\theta_{i}}{2}\left|\bx_{i}+\by_{i}\right|\right).\label{eq:cvx-lb-f}
\end{equation}
With the above definitions, we have
\begin{align}
\nabla f(\bx,\xi) & =\sum_{i=1}^{d}M_{i}\left|\xi_{i}\right|\sgn\left(\bx_{i}-\xi_{i}\by_{i}\right)\be_{i},\label{eq:cvx-lb-stoc-grad}\\
\nabla f_{v}(\bx) & =\sum_{i=1}^{d}M_{i}q_{i}\left(\frac{1+v_{i}\theta_{i}}{2}\sgn\left(\bx_{i}-\by_{i}\right)+\frac{1-v_{i}\theta_{i}}{2}\sgn\left(\bx_{i}+\by_{i}\right)\right)\be_{i},\label{eq:cvx-lb-grad}
\end{align}
where $\nabla$ is taken w.r.t. $\bx$, and $\be_{i}$ denotes the
all-zero vector except for the $i$-th coordinate, which is one.

In the following Lemma \ref{lem:cvx-lb-prop}, we list some useful
properties of the constructed hard function.
\begin{lem}
\label{lem:cvx-lb-prop}For $\dis_{v}$ in (\ref{eq:lb-dis}), $f(\bx,\xi)$
in (\ref{eq:cvx-lb-stoc-f}), and $f_{v}(\bx)$ in (\ref{eq:cvx-lb-f}),
the following properties hold
\begin{enumerate}
\item \label{enu:cvx-lb-prop-1}$\argmin_{\bx\in\R^{d}}f_{v}(\bx)=v\odot\by$.
\item \label{enu:cvx-lb-prop-2}$f_{u}(\bx)-f_{u,\star}+f_{v}(\bx)-f_{v,\star}\geq\sum_{i=1}^{d}2\theta_{i}q_{i}M_{i}\left|\by_{i}\right|\1\left[u_{i}\neq v_{i}\right],\forall\bx\in\R^{d},\forall u,v\in\V$.
\item \label{enu:cvx-lb-prop-3}$\left\Vert \nabla f_{v}(\bx)\right\Vert \leq\sqrt{\sum_{i=1}^{d}M_{i}^{2}q_{i}^{2}},\forall\bx\in\R^{d}$.
\item \label{enu:cvx-lb-prop-4}$\E_{\dis_{v}}\left[\left|\left\langle \be,\nabla f(\bx,\xi)-\nabla f_{v}(\bx)\right\rangle \right|^{\p}\right]\leq\begin{cases}
4\left(\sum_{i=1}^{d}M_{i}^{\frac{2\p}{2-\p}}q_{i}^{\frac{2}{2-\p}}\right)^{\frac{2-\p}{2}} & \p\in\left(1,2\right)\\
4\max_{i\in\left[d\right]}M_{i}^{\p}q_{i} & \p=2
\end{cases},\forall\bx\in\R^{d},\be\in\S^{d-1}$.
\item \label{enu:cvx-lb-prop-5}$\E_{\dis_{v}}\left[\left\Vert \nabla f(\bx,\xi)-\nabla f_{v}(\bx)\right\Vert ^{\p}\right]\leq4\sum_{i=1}^{d}M_{i}^{\p}q_{i},\forall\bx\in\R^{d}$.
\end{enumerate}
\end{lem}
\begin{proof}
For the \ref{enu:cvx-lb-prop-1}st property, observe that
\[
\min_{\bx_{i}\in\R}\left(\frac{1+v_{i}\theta_{i}}{2}\left|\bx_{i}-\by_{i}\right|+\frac{1-v_{i}\theta_{i}}{2}\left|\bx_{i}+\by_{i}\right|\right)\overset{\text{ when }\bx_{i}=v_{i}\by_{i}}{=}(1-\theta_{i})\left|\by_{i}\right|.
\]
Therefore, we know
\begin{eqnarray}
\argmin_{\bx\in\R^{d}}f_{v}(\bx)=v\odot\by & \text{and} & f_{v,\star}=f_{v}(v\odot\by)=\sum_{i=1}^{d}(1-\theta_{i})q_{i}M_{i}\left|\by_{i}\right|.\label{eq:cvx-lb-prop-1}
\end{eqnarray}

For the \ref{enu:cvx-lb-prop-2}nd property, note that for any $\bx\in\R^{d}$,
$u,v\in\V$,
\begin{align*}
f_{u}(\bx)+f_{v}(\bx) & =\sum_{i=1}^{d}M_{i}q_{i}\left(\frac{2+(u_{i}+v_{i})\theta_{i}}{2}\left|\bx_{i}-\by_{i}\right|+\frac{2-(u_{i}+v_{i})\theta_{i}}{2}\left|\bx_{i}+\by_{i}\right|\right)\\
 & \geq\sum_{i=1}^{d}M_{i}q_{i}\left(2-2\theta_{i}\1\left[u_{i}=v_{i}\right]\right)\left|\by_{i}\right|\\
 & \overset{(\ref{eq:cvx-lb-prop-1})}{=}f_{u,\star}+f_{v,\star}+\sum_{i=1}^{d}2\theta_{i}q_{i}M_{i}\left|\by_{i}\right|\1\left[u_{i}\neq v_{i}\right]\\
\Rightarrow f_{u}(\bx)-f_{u,\star}+f_{v}(\bx)-f_{v,\star} & \geq\sum_{i=1}^{d}2\theta_{i}q_{i}M_{i}\left|\by_{i}\right|\1\left[u_{i}\neq v_{i}\right].
\end{align*}

For the \ref{enu:cvx-lb-prop-3}rd property, we have for any $\bx\in\R^{d}$,
\begin{align*}
 & \left\Vert \nabla f_{v}(\bx)\right\Vert \overset{(\ref{eq:cvx-lb-grad})}{=}\left\Vert \sum_{i=1}^{d}M_{i}q_{i}\left(\frac{1+v_{i}\theta_{i}}{2}\sgn\left(\bx_{i}-\by_{i}\right)+\frac{1-v_{i}\theta_{i}}{2}\sgn\left(\bx_{i}+\by_{i}\right)\right)\be_{i}\right\Vert \\
= & \sqrt{\sum_{i=1}^{d}M_{i}^{2}q_{i}^{2}\left(\frac{1+v_{i}\theta_{i}}{2}\sgn\left(\bx_{i}-\by_{i}\right)+\frac{1-v_{i}\theta_{i}}{2}\sgn\left(\bx_{i}+\by_{i}\right)\right)^{2}}\leq\sqrt{\sum_{i=1}^{d}M_{i}^{2}q_{i}^{2}}.
\end{align*}

For the last two properties, we write $Z_{i}(\bx)\defeq\left|\xi_{i}\right|\sgn\left(\bx_{i}-\xi_{i}\by_{i}\right)$.
Under this notation, we have
\[
\nabla f(\bx,\xi)-\nabla f_{v}(\bx)\overset{(\ref{eq:cvx-lb-stoc-grad}),(\ref{eq:cvx-lb-grad})}{=}\sum_{i=1}^{d}M_{i}\left(Z_{i}(\bx)-\E_{\dis_{v}}\left[Z_{i}(\bx)\right]\right)\be_{i}.
\]
Moreover, we can find
\begin{align}
\left|Z_{i}(\bx)-\E_{\dis_{v}}\left[Z_{i}(\bx)\right]\right|^{\p} & \leq2^{\p-1}\left(\left|Z_{i}(\bx)\right|^{\p}+\left|\E_{\dis_{v}}\left[Z_{i}(\bx)\right]\right|^{\p}\right)\leq2^{\p-1}\left(\left|Z_{i}(\bx)\right|^{\p}+\E_{\dis_{v}}\left[\left|Z_{i}(\bx)\right|^{\p}\right]\right)\nonumber \\
\Rightarrow\E\left[\left|Z_{i}(\bx)-\E_{\dis_{v}}\left[Z_{i}(\bx)\right]\right|^{\p}\right] & \leq2^{\p}\E_{\dis_{v}}\left[\left|Z_{i}(\bx)\right|^{\p}\right]\overset{(\ref{eq:lb-dis})}{\leq}2^{\p}q_{i}.\label{eq:cvx-lb-prop-2}
\end{align}
So, for any $\be=\sum_{i=1}^{d}\lambda_{i}\be_{i}$ where $\sum_{i=1}^{d}\lambda_{i}^{2}=1$,
there is
\begin{align}
 & \E_{\dis_{v}}\left[\left|\left\langle \be,\nabla f(\bx,\xi)-\nabla f_{v}(\bx)\right\rangle \right|^{\p}\right]=\E_{\dis_{v}}\left[\left|\sum_{i=1}^{d}\lambda_{i}M_{i}\left(Z_{i}(\bx)-\E_{\dis_{v}}\left[Z_{i}(\bx)\right]\right)\right|^{\p}\right]\nonumber \\
\leq & 2^{2-\p}\sum_{i=1}^{d}\left|\lambda_{i}\right|^{\p}M_{i}^{\p}\E_{\dis_{v}}\left[\left|Z_{i}(\bx)-\E_{\dis_{v}}\left[Z_{i}(\bx)\right]\right|^{\p}\right]\overset{(\ref{eq:cvx-lb-prop-2})}{\leq}4\sum_{i=1}^{d}\left|\lambda_{i}\right|^{\p}M_{i}^{\p}q_{i},\label{eq:cvx-lb-prop-3}
\end{align}
where the first inequality holds due to $\left|a+b\right|^{\p}\leq\left|a\right|^{\p}+\p\left|a\right|^{\p-1}\sgn(a)b+2^{2-\p}\left|b\right|^{\p}$
(see Proposition 18 of \citet{pmlr-v178-vural22a}) and the mutual
independence of $\xi_{i}$. Therefore, by H\"{o}lder's inequality
and (\ref{eq:cvx-lb-prop-3}), we obtain
\[
\sup_{\be\in\S^{d-1}}\E_{\dis_{v}}\left[\left|\left\langle \be,\nabla f(\bx,\xi)-\nabla f_{v}(\bx)\right\rangle \right|^{\p}\right]\leq\begin{cases}
4\left(\sum_{i=1}^{d}M_{i}^{\frac{2\p}{2-\p}}q_{i}^{\frac{2}{2-\p}}\right)^{\frac{2-\p}{2}} & \p\in\left(1,2\right)\\
4\max_{i\in\left[d\right]}M_{i}^{\p}q_{i} & \p=2
\end{cases}.
\]
Lastly, we observe that
\begin{align*}
 & \E_{\dis_{v}}\left[\left\Vert \nabla f(\bx,\xi)-\nabla f_{v}(\bx)\right\Vert ^{\p}\right]=\E_{\dis_{v}}\left[\left\Vert \sum_{i=1}^{d}M_{i}\left(Z_{i}(\bx)-\E_{\dis_{v}}\left[Z_{i}(\bx)\right]\right)\be_{i}\right\Vert ^{\p}\right]\\
= & \E_{\dis_{v}}\left[\left(\sum_{i=1}^{d}M_{i}^{2}\left(Z_{i}(\bx)-\E_{\dis_{v}}\left[Z_{i}(\bx)\right]\right)^{2}\right)^{\frac{\p}{2}}\right]\leq\sum_{i=1}^{d}M_{i}^{\p}\E_{\dis_{v}}\left[\left|Z_{i}(\bx)-\E_{\dis_{v}}\left[Z_{i}(\bx)\right]\right|^{\p}\right]\overset{(\ref{eq:cvx-lb-prop-2}),\p\leq2}{\leq}4\sum_{i=1}^{d}M_{i}^{\p}q_{i},
\end{align*}
where the first inequality is due to $\left(a+b\right)^{\p/2}\leq a^{\p/2}+b^{\p/2}$
for $a,b\geq0$ when $0<\p\leq2$.
\end{proof}

\subsubsection{Strongly Convex Case}

Given $\mu>0$, we define the convex function $f:\R^{d}\times\Xi\to\R$
as
\begin{equation}
f(\bx,\xi)\defeq-\mu\left\langle \bx,M\odot\xi\right\rangle ,\label{eq:str-lb-stoc-f}
\end{equation}
where $M_{i}\geq0,\forall i\in\left[d\right]$ will be determined
later in the proof. Equipped with $f(\bx,\xi)$, we introduce the
following function $f_{v}:\R^{d}\to\R$ labelled by $v\in\V$,
\begin{equation}
f_{v}(\bx)\defeq\E_{\dis_{v}}\left[f(\bx,\xi)\right]=-\mu\left\langle \bx,\E_{\dis_{v}}\left[M\odot\xi\right]\right\rangle .\label{eq:str-lb-f}
\end{equation}
With the above definitions, we have
\begin{eqnarray}
\nabla f(\bx,\xi)=-\mu M\odot\xi & \text{and} & \nabla f_{v}(\bx)=-\mu\E_{\dis_{v}}\left[M\odot\xi\right],\label{eq:str-lb-grad}
\end{eqnarray}
where $\nabla$ is taken w.r.t. $\bx$.

In the following Lemma \ref{lem:str-lb-prop}, we list some useful
properties of the constructed hard function.
\begin{lem}
\label{lem:str-lb-prop}For $\dis_{v}$ in (\ref{eq:lb-dis}), $f(\bx,\xi)$
in (\ref{eq:str-lb-stoc-f}), and $f_{v}(\bx)$ in (\ref{eq:str-lb-f}),
let $F_{v}\defeq f_{v}+\frac{\mu}{2}\left\Vert \cdot\right\Vert ^{2}$,
the following properties hold
\begin{enumerate}
\item \label{enu:str-lb-prop-1}$\argmin_{\bx\in\R^{d}}F_{v}(\bx)=\E_{\dis_{v}}\left[M\odot\xi\right]$.
\item \label{enu:str-lb-prop-2}$F_{u}(\bx)-F_{u,\star}+F_{v}(\bx)-F_{v,\star}\geq\sum_{i=1}^{d}\mu\theta_{i}^{2}q_{i}^{2}M_{i}^{2}\1\left[u_{i}\neq v_{i}\right],\forall\bx\in\R^{d},\forall u,v\in\V$.
\item \label{enu:str-lb-prop-3}$\left\Vert \nabla f_{v}(\bx)\right\Vert =\mu\left\Vert \E_{\dis_{v}}\left[M\odot\xi\right]\right\Vert =\mu\sqrt{\sum_{i=1}^{d}M_{i}^{2}q_{i}^{2}\theta_{i}^{2}},\forall\bx\in\R^{d}$.
\item \label{enu:str-lb-prop-4}$\E_{\dis_{v}}\left[\left|\left\langle \be,\nabla f(\bx,\xi)-\nabla f_{v}(\bx)\right\rangle \right|^{\p}\right]\leq\begin{cases}
4\mu^{\p}\left(\sum_{i=1}^{d}M_{i}^{\frac{2\p}{2-\p}}q_{i}^{\frac{2}{2-\p}}\right)^{\frac{2-\p}{2}} & \p\in\left(1,2\right)\\
4\mu^{\p}\max_{i\in\left[d\right]}M_{i}^{\p}q_{i} & \p=2
\end{cases},\forall\bx\in\R^{d},\be\in\S^{d-1}$.
\item \label{enu:str-lb-prop-5}$\E_{\dis_{v}}\left[\left\Vert \nabla f(\bx,\xi)-\nabla f_{v}(\bx)\right\Vert ^{\p}\right]\leq4\mu^{\p}\sum_{i=1}^{d}M_{i}^{\p}q_{i},\forall\bx\in\R^{d}$.
\end{enumerate}
\end{lem}
\begin{proof}
First of all, we can find that
\begin{equation}
F_{v}(\bx)=f_{v}(\bx)+\frac{\mu}{2}\left\Vert \bx\right\Vert ^{2}\overset{(\ref{eq:str-lb-f})}{=}\frac{\mu}{2}\left\Vert \bx-\E_{\dis_{v}}\left[M\odot\xi\right]\right\Vert ^{2}-\frac{\mu}{2}\left\Vert \E_{\dis_{v}}\left[M\odot\xi\right]\right\Vert ^{2}.\label{eq:str-lb-prop-1}
\end{equation}

For the \ref{enu:str-lb-prop-1}st property, it holds trivially due
to (\ref{eq:str-lb-prop-1}).

For the \ref{enu:str-lb-prop-2}nd property, note that for any $\bx\in\R^{d}$,
$u,v\in\V$,
\begin{align*}
F_{u}(\bx)-F_{u,\star}+F_{v}(\bx)-F_{v,\star} & \overset{(\ref{eq:str-lb-prop-1})}{=}\frac{\mu}{2}\left(\left\Vert \bx-\E_{\dis_{u}}\left[M\odot\xi\right]\right\Vert ^{2}+\left\Vert \bx-\E_{\dis_{v}}\left[M\odot\xi\right]\right\Vert ^{2}\right)\\
 & \geq\frac{\mu}{4}\left\Vert \E_{\dis_{u}}\left[M\odot\xi\right]-\E_{\dis_{v}}\left[M\odot\xi\right]\right\Vert ^{2}=\frac{\mu}{4}\sum_{i=1}^{d}M_{i}^{2}\left(\E_{\dis_{u}}\left[\xi_{i}\right]-\E_{\dis_{v}}\left[\xi_{i}\right]\right)^{2}\\
 & \overset{(\ref{eq:lb-dis})}{=}\frac{\mu}{4}\sum_{i=1}^{d}\theta_{i}^{2}q_{i}^{2}M_{i}^{2}\left(u_{i}-v_{i}\right)^{2}=\sum_{i=1}^{d}\mu\theta_{i}^{2}q_{i}^{2}M_{i}^{2}\1\left[u_{i}\neq v_{i}\right].
\end{align*}

For the \ref{enu:str-lb-prop-3}rd property, we have for any $\bx\in\R^{d}$,
\[
\left\Vert \nabla f_{v}(\bx)\right\Vert \overset{(\ref{eq:str-lb-grad})}{=}\mu\left\Vert \E_{\dis_{v}}\left[M\odot\xi\right]\right\Vert \overset{(\ref{eq:lb-dis})}{=}\mu\sqrt{\sum_{i=1}^{d}M_{i}^{2}q_{i}^{2}\theta_{i}^{2}}.
\]

For the last two properties, we have
\[
\nabla f(\bx,\xi)-\nabla f_{v}(\bx)\overset{(\ref{eq:str-lb-grad})}{=}-\mu\sum_{i=1}^{d}M_{i}\left(\xi_{i}-\E_{\dis_{v}}\left[\xi_{i}\right]\right)\be_{i}.
\]
Moreover, we can find
\begin{align}
\left|\xi_{i}-\E_{\dis_{v}}\left[\xi_{i}\right]\right|^{\p} & \leq2^{\p-1}\left(\left|\xi_{i}\right|^{\p}+\left|\E_{\dis_{v}}\left[\xi_{i}\right]\right|^{\p}\right)\leq2^{\p-1}\left(\left|\xi_{i}\right|^{\p}+\E_{\dis_{v}}\left[\left|\xi_{i}\right|^{\p}\right]\right)\nonumber \\
\Rightarrow\E_{\dis_{v}}\left[\left|\xi_{i}-\E_{\dis_{v}}\left[\xi_{i}\right]\right|^{\p}\right] & \leq2^{\p}\E_{\dis_{v}}\left[\left|\xi_{i}\right|^{\p}\right]\leq2^{\p}q_{i}.\label{eq:str-lb-prop-2}
\end{align}
So, for any $\be=\sum_{i=1}^{d}\lambda_{i}\be_{i}$ where $\sum_{i=1}^{d}\lambda_{i}^{2}=1$,
there is
\begin{align}
 & \E_{\dis_{v}}\left[\left|\left\langle \be,\nabla f(\bx,\xi)-\nabla f_{v}(\bx)\right\rangle \right|^{\p}\right]=\mu^{\p}\E_{\dis_{v}}\left[\left|\sum_{i=1}^{d}\lambda_{i}M_{i}\left(\xi_{i}-\E_{\dis_{v}}\left[\xi_{i}\right]\right)\right|^{\p}\right]\nonumber \\
\leq & 2^{2-\p}\mu^{\p}\sum_{i=1}^{d}\left|\lambda_{i}\right|^{\p}M_{i}^{\p}\E_{\dis_{v}}\left[\left|\xi_{i}-\E_{\dis_{v}}\left[\xi_{i}\right]\right|^{\p}\right]\overset{(\ref{eq:str-lb-prop-2})}{\leq}4\mu^{\p}\sum_{i=1}^{d}\left|\lambda_{i}\right|^{\p}M_{i}^{\p}q_{i},\label{eq:str-lb-prop-3}
\end{align}
where the first inequality holds due to $\left|a+b\right|^{\p}\leq\left|a\right|^{\p}+\p\left|a\right|^{\p-1}\sgn(a)b+2^{2-\p}\left|b\right|^{\p}$
(see Proposition 18 of \citet{pmlr-v178-vural22a}) and the mutual
independence of $\xi_{i}$. Therefore, by H\"{o}lder's inequality
and (\ref{eq:str-lb-prop-3}),
\[
\sup_{\be\in\S^{d-1}}\E_{\dis_{v}}\left[\left|\left\langle \be,\nabla f(\bx,\xi)-\nabla f_{v}(\bx)\right\rangle \right|^{\p}\right]\leq\begin{cases}
4\mu^{\p}\left(\sum_{i=1}^{d}M_{i}^{\frac{2\p}{2-\p}}q_{i}^{\frac{2}{2-\p}}\right)^{\frac{2-\p}{2}} & \p\in\left(1,2\right)\\
4\mu^{\p}\max_{i\in\left[d\right]}M_{i}^{\p}q_{i} & \p=2
\end{cases}.
\]
Lastly, we observe that
\begin{align*}
 & \E_{\dis_{v}}\left[\left\Vert \nabla f(\bx,\xi)-\nabla f_{v}(\bx)\right\Vert ^{\p}\right]=\mu^{\p}\E_{\dis_{v}}\left[\left\Vert \sum_{i=1}^{d}M_{i}\left(\xi_{i}-\E_{\dis_{v}}\left[\xi_{i}\right]\right)\be_{i}\right\Vert ^{\p}\right]\\
= & \mu^{\p}\E_{\dis_{v}}\left[\left(\sum_{i=1}^{d}M_{i}^{2}\left(\xi_{i}-\E_{\dis_{v}}\left[\xi_{i}\right]\right)^{2}\right)^{\frac{\p}{2}}\right]\leq\mu^{\p}\sum_{i=1}^{d}M_{i}^{\p}\E_{\dis_{v}}\left[\left|\xi_{i}-\E_{\dis_{v}}\left[\xi_{i}\right]\right|^{\p}\right]\overset{(\ref{eq:str-lb-prop-2}),\p\leq2}{\leq}4\mu^{\p}\sum_{i=1}^{d}M_{i}^{\p}q_{i},
\end{align*}
where the first inequality is due to $\left(a+b\right)^{\p/2}\leq a^{\p/2}+b^{\p/2}$
for $a,b\geq0$ when $0<\p\leq2$.
\end{proof}

\subsubsection{The Oracle}

Given a subset $\W\subseteq\V$, assume for any $v\in\W$, there are
\begin{itemize}
\item $F_{v}=f_{v}+\frac{\mu}{2}\left\Vert \cdot\right\Vert ^{2}\in\mathfrak{F}_{D,G}^{\type}$,
where $\type=\cvx$ and $f_{v}$ follows (\ref{eq:cvx-lb-f}) if $\mu=0$;
$\type=\str$ and $f_{v}$ follows (\ref{eq:str-lb-f}) if $\mu>0$.
\item $\E_{\dis_{v}}\left[\left|\left\langle \be,\nabla f(\bx,\xi)-\nabla f_{v}(\bx)\right\rangle \right|^{\p}\right]\leq\sigma_{\sma}^{\p}$
and $\E_{\dis_{v}}\left[\left\Vert \nabla f(\bx,\xi)-\nabla f_{v}(\bx)\right\Vert ^{\p}\right]\leq\sigma_{\lar}^{\p}$
for all $\bx\in\R^{d}$ and $\be\in\S^{d-1}$.
\end{itemize}
We construct $\bg\in\G_{\sigma_{\sma},\sigma_{\lar}}^{\p}$ in the
following form, for any $\bx\in\R^{d}$ and $f\in\mathfrak{f}_{G}^{\cvx}$,
\begin{equation}
\bg(\bx,f)=\begin{cases}
\nabla f(\bx,\xi)\text{ for \ensuremath{\xi\sim\dis_{v}}} & f\in\left\{ f_{v}:v\in\W\right\} \\
\nabla f(\bx) & f\notin\left\{ f_{v}:v\in\W\right\} 
\end{cases}.\label{eq:lb-g}
\end{equation}
In other words, if $f=f_{v}$ for some $v\in\W$, then for the $t$-th
query, $\bg(\bx_{t},f)=\nabla f(\bx_{t},\xi_{t})$ where $\xi_{t}\sim\dis_{v}$
is independent from the history; if $f\neq f_{v}$ for any $v\in\W$,
$\bg(\bx_{t},f)=\nabla f(\bx_{t})$ is the true (sub)gradient.

\subsection{High-Probability Lower Bounds}

In this subsection, we give the high-probability lower bounds. Before
presenting the proofs, we state the following simple but useful lemma,
which is inspired by Theorem 4 of \citet{ma2024high}.
\begin{lem}
\label{lem:hp-lb-reformulation}For any $\mathrm{\type}\in\left\{ \mathrm{cvx},\mathrm{str}\right\} $
and $\delta\in\left(0,1\right]$, we have $R_{\star}^{\mathrm{\type}}(\delta)\geq\bar{R}_{\star}^{\mathrm{\type}}(\delta)$,
where
\begin{equation}
\bar{R}_{\star}^{\mathrm{\type}}(\delta)\defeq\sup_{\bg\in\G_{\sigma_{\sma},\sigma_{\lar}}^{\p}}\inf\left\{ \epsilon\geq0:\inf_{{\bf A}_{0:T}\in\A_{T}}\sup_{F\in\mathfrak{F}_{D,G}^{\mathrm{\type}}}\Pr\left[F(\bx_{T+1})-F_{\star}>\epsilon\right]\leq\delta\right\} .\label{eq:hp-lb-minimax-equiv}
\end{equation}
\end{lem}
\begin{proof}
Given $\bg\in\G_{\sigma_{\sma},\sigma_{\lar}}^{\p}$, for any ${\bf A}_{0:T}\in\A_{T}$
and $F\in\mathfrak{F}_{D,G}^{\mathrm{\type}}$, we note that $\Pr\left[F(\bx_{T+1})-F_{\star}>\epsilon\right]$
is nonincreasing in $\epsilon\geq0$. Therefore, Lemma \ref{lem:minimax}
gives us, for any $\delta\in\left(0,1\right]$,
\begin{align*}
 & \inf_{{\bf A}_{0:T}\in\A_{T}}\sup_{F\in\mathfrak{F}_{D,G}^{\mathrm{\type}}}\inf\left\{ \epsilon\geq0:\Pr\left[F(\bx_{T+1})-F_{\star}>\epsilon\right]\leq\delta\right\} \\
\geq & \inf\left\{ \epsilon\geq0:\inf_{{\bf A}_{0:T}\in\A_{T}}\sup_{F\in\mathfrak{F}_{D,G}^{\mathrm{\type}}}\Pr\left[F(\bx_{T+1})-F_{\star}>\epsilon\right]\leq\delta\right\} ,
\end{align*}
which further implies that $R_{\star}^{\mathrm{\type}}(\delta)\geq\bar{R}_{\star}^{\mathrm{\type}}(\delta)$.
\end{proof}

With Lemma \ref{lem:hp-lb-reformulation}, we only need to lower bound
$\bar{R}_{\star}^{\mathrm{\type}}(\delta)$.

\subsubsection{General Convex Case}
\begin{thm}[Formal version of Theorem \ref{thm:main-cvx-hp-lb}]
\label{thm:cvx-hp-lb}Given $D>0$, $G>0$, $\p\in\left(1,2\right]$,
and $0<\sigma_{\sma}\leq\sigma_{\lar}$, for any $d\geq d_{\eff}$,
we have 
\[
R_{\star}^{\cvx}(\delta)\geq\Omega\left(\min\left\{ GD,\sigma_{\sma}^{\frac{2}{\p}-1}\sigma_{\lar}^{2-\frac{2}{\p}}D,\frac{\left(\sigma_{\sma}^{\frac{2}{\p}-1}\sigma_{\lar}^{2-\frac{2}{\p}}+\sigma_{\sma}\ln^{1-\frac{1}{\p}}\frac{1}{8\delta}\right)D}{T^{1-\frac{1}{\p}}}\right\} \right),\forall\delta\in\left(0,\frac{1}{10}\right),
\]
 where $R_{\star}^{\cvx}(\delta)$ is defined in (\ref{eq:hp-lb-minimax}).
\end{thm}
\begin{proof}
In the following, let $d\geq d_{\eff}\Leftrightarrow d\geq\left\lceil d_{\eff}\right\rceil $
be fixed and $d_{\star}\defeq\left\lceil d_{\eff}\right\rceil $ for
convenience.

\paragraph{First bound.}

We first prove
\begin{equation}
R_{\star}^{\cvx}(\delta)\overset{\text{Lemma \ref{lem:hp-lb-reformulation}}}{\geq}\bar{R}_{\star}^{\cvx}(\delta)\geq\Omega\left(\min\left\{ GD,\frac{\sigma_{\sma}^{\frac{2}{\p}-1}\sigma_{\lar}^{2-\frac{2}{\p}}D}{T^{1-\frac{1}{\p}}}\right\} \right),\forall\delta\in\left(0,\frac{1}{10}\right).\label{eq:cvx-hp-lb-bound-1}
\end{equation}
We will split the proof into two cases: $d_{\star}>32\ln2$ and $d_{\star}\in\left[1,32\ln2\right]$.

\subparagraph{The case $d_{\star}>32\ln2$.}

By the Gilbert-Varshamov bound \citep{6773017,varshamov1957evaluation},
there exists a subset $\bar{\W}\subseteq\left\{ -1,1\right\} ^{d_{\star}}$
such that $\Delta_{\H}(u,v)\geq\frac{d_{\star}}{4},\forall u,v\in\bar{\W}$
and $\left|\bar{\W}\right|\geq\exp\left(\frac{d_{\star}}{8}\right)$.
As such, we can construct 
\begin{equation}
\W\defeq\left\{ (v^{\top},1,\mydots,1)^{\top}:v\in\bar{\W}\right\} \subseteq\left\{ -1,1\right\} ^{d}=\V,\label{eq:cvx-hp-lb-W-1}
\end{equation}
satisfying
\begin{eqnarray}
\Delta_{\H}(u,v)\geq\frac{d_{\star}}{4},\forall u,v\in\W & \text{and} & \left|\W\right|\geq\exp\left(\frac{d_{\star}}{8}\right).\label{eq:cvx-hp-lb-W-prop}
\end{eqnarray}

For any $i\in\left[d\right]$, we pick
\begin{equation}
\begin{array}{ccc}
q_{i}=q\defeq\frac{1}{T} & \text{and} & \theta_{i}=\theta\defeq\frac{1}{10},\\
M_{i}=M\1\left[i\leq d_{\star}\right] & \text{and} & \by_{i}=y\1\left[i\leq d_{\star}\right],\\
M\defeq\min\left\{ \frac{G}{q\sqrt{d_{\star}}},\frac{\sigma_{\lar}}{(4qd_{\star})^{\frac{1}{\p}}}\right\}  & \text{and} & y\defeq\frac{D}{\sqrt{d_{\star}}}.
\end{array}\label{eq:cvx-hp-lb-parameter-1}
\end{equation}
Then by Lemma \ref{lem:cvx-lb-prop}, for any $v\in\W$,
\begin{eqnarray*}
\left\Vert \argmin_{\bx\in\R^{d}}f_{v}(\bx)\right\Vert =\left\Vert v\odot\by\right\Vert =\left\Vert \by\right\Vert =D & \text{and} & \left\Vert \nabla f_{v}(\bx)\right\Vert \leq Mq\sqrt{d_{\star}}\leq G,
\end{eqnarray*}
which implies $F_{v}=f_{v}\in\mathfrak{F}_{D,G}^{\cvx}$. Still by
Lemma \ref{lem:cvx-lb-prop}, for any $\bx\in\R^{d}$,
\begin{align*}
\E_{\dis_{v}}\left[\left|\left\langle \be,\nabla f(\bx,\xi)-\nabla f_{v}(\bx)\right\rangle \right|^{\p}\right] & \leq4M^{\p}qd_{\star}^{\frac{2-\p}{2}}\leq\sigma_{\lar}^{\p}/d_{\star}^{\frac{\p}{2}}\leq\sigma_{\sma}^{\p},\forall\be\in\S^{d-1},\\
\E_{\dis_{v}}\left[\left\Vert \nabla f(\bx,\xi)-\nabla f_{v}(\bx)\right\Vert ^{\p}\right] & \leq4M^{\p}qd_{\star}\leq\sigma_{\lar}^{\p}.
\end{align*}
Therefore, the oracle $\bg$ constructed in (\ref{eq:lb-g}) satisfies
$\bg\in\G_{\sigma_{\sma},\sigma_{\lar}}^{\p}$.

Now let us consider the optimization procedure for any algorithm ${\bf A}_{0:T}\in\A_{T}$
interacting with the oracle $\bg$ in (\ref{eq:lb-g}). We define
\begin{equation}
\epsilon_{\star}\defeq\frac{\theta qMyd_{\star}}{8}.\label{eq:cvx-hp-lb-epsilon-1}
\end{equation}
Moreover, let $V$ be uniformly distributed on $\W$ and $W\defeq\argmin_{v\in\W}F_{v}(\bx_{T+1})-F_{v,\star}$,
where $F_{v,\star}\defeq\inf_{\bx\in\R^{d}}F_{v}(\bx)$. Note that
\begin{align}
\sup_{F\in\mathfrak{F}_{D,G}^{\cvx}}\Pr\left[F(\bx_{T+1})-F_{\star}>\epsilon_{\star}\right] & \geq\frac{1}{\left|\W\right|}\sum_{v\in\W}\Pr\left[F_{v}(\bx_{T+1})-F_{v,\star}>\epsilon_{\star}\right]\nonumber \\
 & \geq\frac{1}{\left|\W\right|}\sum_{v\in\W}\Pr\left[\frac{F_{W}(\bx_{T+1})-F_{W,\star}+F_{v}(\bx_{T+1})-F_{v,\star}}{2}>\epsilon_{\star}\right]\nonumber \\
 & \overset{(a)}{\geq}\frac{1}{\left|\W\right|}\sum_{v\in\W}\Pr\left[\Delta_{\H}(W,v)>\frac{d_{\star}}{8}\right]\overset{(\ref{eq:cvx-hp-lb-W-prop})}{=}\frac{1}{\left|\W\right|}\sum_{v\in\W}\Pr\left[W\neq v\right]\nonumber \\
 & \overset{(b)}{\geq}1-\frac{\I(W;V)+\ln2}{\ln\left|\W\right|}\overset{(\ref{eq:cvx-hp-lb-W-prop}),d_{\star}>32\ln2}{>}\frac{3}{4}-\frac{8\I(W;V)}{d_{\star}},\label{eq:cvx-hp-lb-1}
\end{align}
where $(a)$ is by
\begin{align*}
\frac{F_{W}(\bx_{T+1})-F_{W,\star}+F_{v}(\bx_{T+1})-F_{v,\star}}{2} & \overset{\text{Lemma }\ref{lem:cvx-lb-prop}}{\geq}\sum_{i=1}^{d}\theta_{i}q_{i}M_{i}\left|\by_{i}\right|\1\left[W_{i}\neq v_{i}\right]\\
 & \overset{(\ref{eq:cvx-hp-lb-parameter-1})}{=}\theta qMy\sum_{i=1}^{d}\1\left[W_{i}\neq v_{i}\right]\overset{(\ref{eq:cvx-hp-lb-epsilon-1})}{=}\frac{8\epsilon_{\star}}{d_{\star}}\Delta_{\H}(W,v),
\end{align*}
and $(b)$ is due to Fano's inequality.

Note that $V\to\bg_{1:T}\to W$ forms a Markov chain (where $\bg_{s:t}$
is the shorthand for $(\bg(\bx_{s},f_{V}),\mydots,\bg(\bx_{t},f_{V}))$
given $1\leq s\leq t\leq T$), by the Data Processing Inequality (DPI)
for mutual information,
\begin{equation}
\I(W;V)\leq\I(\bg_{1:T};V)=\frac{1}{\left|\W\right|}\sum_{v\in\W}\KL(\bg_{1:T}\mid V=v\Vert\bg_{1:T})\leq\frac{1}{\left|\W\right|^{2}}\sum_{u,v\in\W}\KL(\dis_{v}^{\bg_{1:T}}\Vert\dis_{u}^{\bg_{1:T}}),\label{eq:cvx-hp-lb-2}
\end{equation}
where the last step is by $\bg_{1:T}\mid V=v\sim\dis_{v}^{\bg_{1:T}}$,
$\bg_{1:T}\sim\frac{1}{\left|\W\right|}\sum_{v\in\W}\dis_{v}^{\bg_{1:T}}$,
and the convexity of the $\KL$ divergence, in which $\dis_{v}^{\bg_{s:t}}$
is the joint probability distribution of $\bg_{s:t}$ given $V=v$.
Next, observe that for any $u,v\in\W$,
\begin{align}
\KL(\dis_{v}^{\bg_{1:T}}\Vert\dis_{u}^{\bg_{1:T}}) & \overset{(c)}{=}\sum_{t=1}^{T}\E_{\bg_{1:t-1}\sim\dis_{v}^{\bg_{1:t-1}}}\left[\KL(\bg_{t}\sim\dis_{v}^{\bg_{t}}\mid\bg_{1:t-1}\Vert\bg_{t}\sim\dis_{u}^{\bg_{t}}\mid\bg_{1:t-1})\right]\nonumber \\
 & \overset{(d)}{\leq}T\KL\left(\dis_{v}\Vert\dis_{u}\right)\overset{(e)}{=}T\sum_{i=1}^{d}\KL\left(\dis_{v,i}\Vert\dis_{u,i}\right)\overset{(\ref{eq:cvx-hp-lb-W-1})}{=}T\sum_{i=1}^{d_{\star}}\KL\left(\dis_{v,i}\Vert\dis_{u,i}\right)\overset{(f)}{<}\frac{d_{\star}}{32},\label{eq:cvx-hp-lb-3}
\end{align}
where $(c)$ and $(e)$ are by the chain rule of the $\KL$ divergence
(we also use the fact that $\dis_{v}=\prod_{i=1}^{d}\dis_{v,i}$ in
$(e)$), $(d)$ is true by noticing that $\bx_{t}$ is fixed given
$\bg_{1:t-1}$, meaning that $\bg_{t}=\nabla f(\bx_{t},\xi_{t})$
is a function of $\xi_{t}$, which further implies that
\[
\KL(\bg_{t}\sim\dis_{v}^{\bg_{t}}\mid\bg_{1:t-1}\Vert\bg_{t}\sim\dis_{u}^{\bg_{t}}\mid\bg_{1:t-1})\leq\KL\left(\xi_{t}\sim\dis_{v}\mid\bg_{1:t-1}\Vert\xi_{t}\sim\dis_{u}\mid\bg_{1:t-1}\right)=\KL\left(\dis_{v}\Vert\dis_{u}\right)
\]
holds almost surely by DPI and the independence of $\xi_{t}$ from
the history, and $(f)$ holds due to for any $i\in\left[d_{\star}\right]$,
\begin{align*}
\KL\left(\dis_{v,i}\Vert\dis_{u,i}\right) & \overset{(\ref{eq:lb-dis}),(\ref{eq:cvx-hp-lb-parameter-1})}{=}\frac{1+v_{i}\theta}{2}q\ln\frac{\frac{1+v_{i}\theta}{2}q}{\frac{1+u_{i}\theta}{2}q}+\frac{1-v_{i}\theta}{2}q\ln\frac{\frac{1-v_{i}\theta}{2}q}{\frac{1-u_{i}\theta}{2}q}=\1\left[u_{i}\neq v_{i}\right]\theta q\ln\frac{1+\theta}{1-\theta}\\
 & \leq\theta q\ln\frac{1+\theta}{1-\theta}\overset{(\ref{eq:cvx-hp-lb-parameter-1})}{=}\frac{\ln\frac{1.1}{0.9}}{10T}<\frac{1}{32T}.
\end{align*}
Combine (\ref{eq:cvx-hp-lb-2}) and (\ref{eq:cvx-hp-lb-3}) to obtain
$\I(W;V)\leq\frac{d_{\star}}{32}$, which further implies that, by
(\ref{eq:cvx-hp-lb-1}),
\[
\sup_{F\in\mathfrak{F}_{D,G}^{\cvx}}\Pr\left[F(\bx_{T+1})-F_{\star}>\epsilon_{\star}\right]>\frac{1}{2}.
\]

Since ${\bf A}_{0:T}\in\A_{T}$ is arbitrarily chosen, we finally
have for $\bg$ given in (\ref{eq:lb-g})
\begin{align*}
 & \inf_{{\bf A}_{0:T}\in\A_{T}}\sup_{F\in\mathfrak{F}_{D,G}^{\cvx}}\Pr\left[F(\bx_{T+1})-F_{\star}>\epsilon_{\star}\right]\geq\frac{1}{2}\overset{\delta<1/10}{>}\delta\\
\Rightarrow & \inf\left\{ \epsilon\geq0:\inf_{{\bf A}_{0:T}\in\A_{T}}\sup_{F\in\mathfrak{F}_{D,G}^{\mathrm{\cvx}}}\Pr\left[F(\bx_{T+1})-F_{\star}>\epsilon\right]\leq\delta\right\} \geq\epsilon_{\star},
\end{align*}
which implies that
\[
\bar{R}_{\star}^{\cvx}(\delta)\geq\epsilon_{\star}\overset{(\ref{eq:cvx-hp-lb-epsilon-1})}{=}\frac{\theta qMyd_{\star}}{8}\overset{(\ref{eq:cvx-hp-lb-parameter-1})}{=}\Omega\left(\min\left\{ GD,\frac{\sigma_{\sma}^{\frac{2}{\p}-1}\sigma_{\lar}^{2-\frac{2}{\p}}D}{T^{1-\frac{1}{\p}}}\right\} \right).
\]

\subparagraph{The case $d_{\star}\in\left[1,32\ln2\right]$.}

For this case, it is enough to show $\bar{R}_{\star}^{\cvx}(\delta)\geq\Omega\left(\min\left\{ GD,\frac{\sigma_{\sma}D}{T^{1-\frac{1}{\p}}}\right\} \right)$,
since $\sigma_{\sma}=\frac{\sigma_{\sma}^{\frac{2}{\p}-1}\sigma_{\lar}^{2-\frac{2}{\p}}}{d_{\eff}^{1-\frac{1}{\p}}}=\Theta(\sigma_{\sma}^{\frac{2}{\p}-1}\sigma_{\lar}^{2-\frac{2}{\p}})$
when $d_{\star}=\left\lceil d_{\eff}\right\rceil \in\left[1,32\ln2\right]$.
By (\ref{eq:cvx-hp-lb-bound-2}), we have for any $\delta\in\left(0,\frac{1}{10}\right)$,
\begin{align*}
\bar{R}_{\star}^{\cvx}(\delta) & \geq\Omega\left(\min\left\{ GD,\sigma_{\sma}^{\frac{2}{\p}-1}\sigma_{\lar}^{2-\frac{2}{\p}}D,\frac{\sigma_{\sma}\ln^{1-\frac{1}{\p}}(\frac{5}{4})D}{T^{1-\frac{1}{\p}}}\right\} \right)\\
 & \geq\Omega\left(\min\left\{ GD,\sigma_{\sma}D,\frac{\sigma_{\sma}D}{T^{1-\frac{1}{\p}}}\right\} \right)=\Omega\left(\min\left\{ GD,\frac{\sigma_{\sma}D}{T^{1-\frac{1}{\p}}}\right\} \right).
\end{align*}

\paragraph{Second bound.}

For the second bound, we will show
\begin{equation}
R_{\star}^{\cvx}(\delta)\overset{\text{Lemma \ref{lem:hp-lb-reformulation}}}{\geq}\bar{R}_{\star}^{\cvx}(\delta)\geq\Omega\left(\min\left\{ GD,\sigma_{\sma}^{\frac{2}{\p}-1}\sigma_{\lar}^{2-\frac{2}{\p}}D,\frac{\sigma_{\sma}\ln^{1-\frac{1}{\p}}(\frac{1}{8\delta})D}{T^{1-\frac{1}{\p}}}\right\} \right),\forall\delta\in\left(0,\frac{1}{8}\right).\label{eq:cvx-hp-lb-bound-2}
\end{equation}

In this setting, we set
\begin{equation}
\W\defeq\left\{ v^{+}\defeq(1,\mydots,1)^{\top},v^{-}\defeq(\underbrace{-1,\mydots,-1}_{d_{\star}},1,\mydots,1)^{\top}\right\} \subseteq\left\{ -1,1\right\} ^{d}=\V.\label{eq:cvx-hp-lb-W-2}
\end{equation}
For any $i\in\left[d\right]$, we pick
\begin{equation}
\begin{array}{ccc}
q_{i}=q\defeq\min\left\{ \frac{\ln\frac{1}{8\delta}}{Td_{\star}\theta\ln\frac{1+\theta}{1-\theta}},1\right\}  & \text{and} & \theta_{i}=\theta\defeq\frac{1}{2},\\
M_{i}=M\1\left[i\leq d_{\star}\right] & \text{and} & \by_{i}=y\1\left[i\leq d_{\star}\right],\\
M\defeq\min\left\{ \frac{G}{q\sqrt{d_{\star}}},\frac{\sigma_{\lar}}{(4qd_{\star})^{\frac{1}{\p}}}\right\}  & \text{and} & y\defeq\frac{D}{\sqrt{d_{\star}}}.
\end{array}\label{eq:cvx-hp-lb-parameter-2}
\end{equation}
Similar to before, one can use Lemma \ref{lem:cvx-lb-prop} to verify
$F_{v}=f_{v}\in\mathfrak{F}_{D,G}^{\cvx}$, and check that the oracle
$\bg$ constructed in (\ref{eq:lb-g}) satisfies $\bg\in\G_{\sigma_{\sma},\sigma_{\lar}}^{\p}$.

Now let us consider the optimization procedure for any algorithm ${\bf A}_{0:T}\in\A_{T}$
interacting with the oracle $\bg$ in (\ref{eq:lb-g}) and define
\begin{equation}
\epsilon_{\star}\defeq\frac{\theta qMyd_{\star}}{2}.\label{eq:cvx-hp-lb-epsilon-2}
\end{equation}
Note that
\begin{align*}
\sup_{F\in\mathfrak{F}_{D,G}^{\cvx}}\Pr\left[F(\bx_{T+1})-F_{\star}>\epsilon_{\star}\right] & \geq\frac{1}{2}\sum_{v\in\W}\Pr\left[F_{v}(\bx_{T+1})-F_{v,\star}>\epsilon_{\star}\right]\overset{(g)}{\geq}\frac{1}{2}\left(1-\TV\left(\dis_{v^{+}}^{\bg_{1:T}},\dis_{v^{-}}^{\bg_{1:T}}\right)\right)\\
 & \overset{(h)}{\geq}\frac{1}{4}\exp\left(-\KL\left(\dis_{v^{+}}^{\bg_{1:T}}\Vert\dis_{v^{-}}^{\bg_{1:T}}\right)\right)\overset{(i)}{\geq}\frac{1}{4}\exp\left(-Td_{\star}q\theta\ln\frac{1+\theta}{1-\theta}\right)\overset{(\ref{eq:cvx-hp-lb-parameter-2})}{\geq}2\delta,
\end{align*}
where $(g)$ holds by Neyman-Pearson Lemma and the fact $\1\left[F_{v^{+}}(\bx_{T+1})-F_{v^{+},\star}>\epsilon_{\star}\right]+\1\left[F_{v^{-}}(\bx_{T+1})-F_{v^{-},\star}>\epsilon_{\star}\right]\geq1$,
since
\[
F_{v^{+}}(\bx_{T+1})-F_{v^{+},\star}+F_{v^{-}}(\bx_{T+1})-F_{v^{-},\star}\overset{\text{Lemma }\ref{lem:cvx-lb-prop}}{\geq}\sum_{i=1}^{d}2\theta_{i}q_{i}M_{i}\left|\by_{i}\right|\1\left[v_{i}^{+}\neq v_{i}^{-}\right]\overset{(\ref{eq:cvx-hp-lb-W-2}),(\ref{eq:cvx-hp-lb-parameter-2})}{=}2\theta qMyd_{\star}=4\epsilon_{\star},
\]
$(h)$ is due to Bretagnolle--Huber inequality \citep{Bretagnolle1979},
and $(i)$ follows a similar analysis of proving (\ref{eq:cvx-hp-lb-3}).

Since ${\bf A}_{0:T}\in\A_{T}$ is arbitrarily chosen, we finally
have, under $\bg$ given in (\ref{eq:lb-g}),
\begin{align*}
 & \inf_{{\bf A}_{0:T}\in\A_{T}}\sup_{F\in\mathfrak{F}_{D,G}^{\cvx}}\Pr\left[F(\bx_{T+1})-F_{\star}>\epsilon_{\star}\right]\geq2\delta\\
\Rightarrow & \inf\left\{ \epsilon\geq0:\inf_{{\bf A}_{0:T}\in\A_{T}}\sup_{F\in\mathfrak{F}_{D,G}^{\mathrm{\cvx}}}\Pr\left[F(\bx_{T+1})-F_{\star}>\epsilon\right]\leq\delta\right\} \geq\epsilon_{\star},
\end{align*}
which implies that
\[
\bar{R}_{\star}^{\cvx}(\delta)\geq\epsilon_{\star}\overset{(\ref{eq:cvx-hp-lb-epsilon-2})}{=}\frac{\theta qMyd_{\star}}{2}\overset{(\ref{eq:cvx-hp-lb-parameter-2})}{=}\Omega\left(\min\left\{ GD,\sigma_{\sma}^{\frac{2}{\p}-1}\sigma_{\lar}^{2-\frac{2}{\p}}D,\frac{\sigma_{\sma}\ln^{1-\frac{1}{\p}}(\frac{1}{8\delta})D}{T^{1-\frac{1}{\p}}}\right\} \right).
\]

\paragraph{Final bound.}

\end{proof}
Finally, we combine (\ref{eq:cvx-hp-lb-bound-1}) and (\ref{eq:cvx-hp-lb-bound-2})
to conclude
\begin{align*}
R_{\star}^{\cvx}(\delta) & \ge\Omega\left(\min\left\{ GD,\frac{\sigma_{\sma}^{\frac{2}{\p}-1}\sigma_{\lar}^{2-\frac{2}{\p}}D}{T^{1-\frac{1}{\p}}}\right\} +\min\left\{ GD,\sigma_{\sma}^{\frac{2}{\p}-1}\sigma_{\lar}^{2-\frac{2}{\p}}D,\frac{\sigma_{\sma}\ln^{1-\frac{1}{\p}}(\frac{1}{8\delta})D}{T^{1-\frac{1}{\p}}}\right\} \right)\\
 & =\Omega\left(\min\left\{ GD,\sigma_{\sma}^{\frac{2}{\p}-1}\sigma_{\lar}^{2-\frac{2}{\p}}D,\frac{\left(\sigma_{\sma}^{\frac{2}{\p}-1}\sigma_{\lar}^{2-\frac{2}{\p}}+\sigma_{\sma}\ln^{1-\frac{1}{\p}}\frac{1}{8\delta}\right)D}{T^{1-\frac{1}{\p}}}\right\} \right),\forall\delta\in\left(0,\frac{1}{10}\right).
\end{align*}

\subsubsection{Strongly Convex Case}
\begin{thm}[Formal version of Theorem \ref{thm:main-str-hp-lb}]
\label{thm:str-hp-lb}Given $D>0$, $G>0$, $\p\in\left(1,2\right]$,
and $0<\sigma_{\sma}\leq\sigma_{\lar}$, for any $d\geq d_{\eff}$,
we have 
\[
R_{\star}^{\str}(\delta)\geq\Omega\left(\min\left\{ \mu D^{2},\frac{G^{2}}{\mu},\frac{\sigma_{\sma}^{\frac{4}{\p}-2}\sigma_{\lar}^{4-\frac{4}{\p}}}{\mu},\frac{\sigma_{\sma}^{\frac{4}{\p}-2}\sigma_{\lar}^{4-\frac{4}{\p}}+\sigma_{\sma}^{2}\ln^{2-\frac{2}{\p}}\frac{1}{8\delta}}{\mu T^{2-\frac{2}{\p}}}\right\} \right),\forall\delta\in\left(0,\frac{1}{10}\right),
\]
 where $R_{\star}^{\str}(\delta)$ is defined in (\ref{eq:hp-lb-minimax}).
\end{thm}
\begin{proof}
In the following, let $d\geq d_{\eff}\Leftrightarrow d\geq\left\lceil d_{\eff}\right\rceil $
be fixed and $d_{\star}\defeq\left\lceil d_{\eff}\right\rceil $ for
convenience.

\paragraph{First bound.}

We first prove 
\begin{equation}
R_{\star}^{\str}(\delta)\overset{\text{Lemma \ref{lem:hp-lb-reformulation}}}{\geq}\bar{R}_{\star}^{\str}(\delta)\geq\Omega\left(\min\left\{ \mu D^{2},\frac{G^{2}}{\mu},\frac{\sigma_{\sma}^{\frac{4}{\p}-2}\sigma_{\lar}^{4-\frac{4}{\p}}}{\mu T^{2-\frac{2}{\p}}}\right\} \right),\forall\delta\in\left(0,\frac{1}{10}\right).\label{eq:str-hp-lb-bound-1}
\end{equation}
We will split the proof into two cases: $d_{\star}>32\ln2$ and $d_{\star}\in\left[1,32\ln2\right]$.

\subparagraph{The case $d_{\star}>32\ln2$.}

Again, by the Gilbert-Varshamov bound \citep{6773017,varshamov1957evaluation},
there exists a subset $\bar{\W}\subseteq\left\{ -1,1\right\} ^{d_{\star}}$
such that $\Delta_{\H}(u,v)\geq\frac{d_{\star}}{4},\forall u,v\in\bar{\W}$
and $\left|\bar{\W}\right|\geq\exp\left(\frac{d_{\star}}{8}\right)$.
As such, we can construct 
\begin{equation}
\W\defeq\left\{ (v^{\top},1,\mydots,1)^{\top}:v\in\bar{\W}\right\} \subseteq\left\{ -1,1\right\} ^{d}=\V,\label{eq:str-hp-lb-W-1}
\end{equation}
satisfying
\begin{eqnarray}
\Delta_{\H}(u,v)\geq\frac{d_{\star}}{4},\forall u,v\in\W & \text{and} & \left|\W\right|\geq\exp\left(\frac{d_{\star}}{8}\right).\label{eq:str-hp-lb-W-prop}
\end{eqnarray}

For any $i\in\left[d\right]$, we pick
\begin{equation}
\begin{array}{ccc}
q_{i}=q\defeq\frac{1}{T} & \text{and} & \theta_{i}=\theta\defeq\frac{1}{10},\\
M_{i}=M\1\left[i\leq d_{\star}\right] & \text{and} & M\defeq\min\left\{ \frac{D}{\theta q\sqrt{d_{\star}}},\frac{G}{\mu\theta q\sqrt{d_{\star}}},\frac{\sigma_{\lar}}{\mu(4qd_{\star})^{\frac{1}{\p}}}\right\} .
\end{array}\label{eq:str-hp-lb-parameter-1}
\end{equation}
Then by Lemma \ref{lem:str-lb-prop}, for any $v\in\W$,
\begin{eqnarray*}
\left\Vert \argmin_{\bx\in\R^{d}}F_{v}(\bx)\right\Vert =\left\Vert \E_{\dis_{v}}\left[M\odot\xi\right]\right\Vert \overset{(\ref{eq:lb-dis})}{=}Mq\theta\sqrt{d_{\star}}\leq D & \text{and} & \left\Vert \nabla f_{v}(\bx)\right\Vert =\mu Mq\theta\sqrt{d_{\star}}\leq G,
\end{eqnarray*}
which implies that $f_{v}\in\mathfrak{f}_{G}^{\cvx}$ and $F_{v}=f_{v}+\frac{\mu}{2}\left\Vert \cdot\right\Vert ^{2}\in\mathfrak{F}_{D,G}^{\str}$.
Still by Lemma \ref{lem:str-lb-prop}, for any $\bx\in\R^{d}$,
\begin{align*}
\E_{\dis_{v}}\left[\left|\left\langle \be,\nabla f(\bx,\xi)-\nabla f_{v}(\bx)\right\rangle \right|^{\p}\right] & \leq4\mu^{\p}M^{\p}qd_{\star}^{\frac{2-\p}{2}}\leq\sigma_{\lar}^{\p}/d_{\star}^{\frac{\p}{2}}\leq\sigma_{\sma}^{\p},\forall\be\in\S^{d-1},\\
\E_{\dis_{v}}\left[\left\Vert \nabla f(\bx,\xi)-\nabla f_{v}(\bx)\right\Vert ^{\p}\right] & \leq4\mu^{\p}M^{\p}qd_{\star}\leq\sigma_{\lar}^{\p}.
\end{align*}
Therefore, the oracle $\bg$ constructed in (\ref{eq:lb-g}) satisfies
$\bg\in\G_{\sigma_{\sma},\sigma_{\lar}}^{\p}$.

Now let us consider the optimization procedure for any algorithm ${\bf A}_{0:T}\in\A_{T}$
interacting with the oracle $\bg$ in (\ref{eq:lb-g}). We define
\begin{equation}
\epsilon_{\star}\defeq\frac{\mu\theta^{2}q^{2}M^{2}d_{\star}}{16}.\label{eq:str-hp-lb-epsilon-1}
\end{equation}
Moreover, let $V$ be uniformly distributed on $\W$ and $W\defeq\argmin_{v\in\W}F_{v}(\bx_{T+1})-F_{v,\star}$,
where $F_{v,\star}\defeq\inf_{\bx\in\R^{d}}F_{v}(\bx)$. Note that
\begin{align}
\sup_{F\in\mathfrak{F}_{D,G}^{\str}}\Pr\left[F(\bx_{T+1})-F_{\star}>\epsilon_{\star}\right] & \geq\frac{1}{\left|\W\right|}\sum_{v\in\W}\Pr\left[F_{v}(\bx_{T+1})-F_{v,\star}>\epsilon_{\star}\right]\nonumber \\
 & \geq\frac{1}{\left|\W\right|}\sum_{v\in\W}\Pr\left[\frac{F_{W}(\bx_{T+1})-F_{W,\star}+F_{v}(\bx_{T+1})-F_{v,\star}}{2}>\epsilon_{\star}\right]\nonumber \\
 & \overset{(a)}{\geq}\frac{1}{\left|\W\right|}\sum_{v\in\W}\Pr\left[\Delta_{\H}(W,v)>\frac{d_{\star}}{8}\right]\overset{(\ref{eq:str-hp-lb-W-prop})}{=}\frac{1}{\left|\W\right|}\sum_{v\in\W}\Pr\left[W\neq v\right]\nonumber \\
 & \overset{(b)}{\geq}1-\frac{\I(W;V)+\ln2}{\ln\left|\W\right|}\overset{(\ref{eq:str-hp-lb-W-prop}),d_{\star}>32\ln2}{>}\frac{3}{4}-\frac{8\I(W;V)}{d_{\star}},\label{eq:str-hp-lb-1}
\end{align}
where $(a)$ is by
\begin{align*}
\frac{F_{W}(\bx_{T+1})-F_{W,\star}+F_{v}(\bx_{T+1})-F_{v,\star}}{2} & \overset{\text{Lemma }\ref{lem:str-lb-prop}}{\geq}\frac{\sum_{i=1}^{d}\mu\theta_{i}^{2}q_{i}^{2}M_{i}^{2}\1\left[W_{i}\neq v_{i}\right]}{2}\\
 & \overset{(\ref{eq:str-hp-lb-parameter-1})}{=}\frac{\mu\theta^{2}q^{2}M^{2}\sum_{i=1}^{d}\1\left[W_{i}\neq v_{i}\right]}{2}\overset{(\ref{eq:str-hp-lb-epsilon-1})}{=}\frac{8\epsilon_{\star}}{d_{\star}}\Delta_{\H}(W,v),
\end{align*}
and $(b)$ is due to Fano's inequality.

Under the same argument used in the proof of Theorem \ref{thm:cvx-hp-lb},
one can still obtain $\I(W;V)\leq\frac{d_{\star}}{32}$, which further
implies that, by (\ref{eq:str-hp-lb-1}),
\[
\sup_{F\in\mathfrak{F}_{D,G}^{\str}}\Pr\left[F(\bx_{T+1})-F_{\star}>\epsilon_{\star}\right]>\frac{1}{2}.
\]

Since ${\bf A}_{0:T}\in\A_{T}$ is arbitrarily chosen, we finally
have for $\bg$ given in (\ref{eq:lb-g})
\begin{align*}
 & \inf_{{\bf A}_{0:T}\in\A_{T}}\sup_{F\in\mathfrak{F}_{D,G}^{\str}}\Pr\left[F(\bx_{T+1})-F_{\star}>\epsilon_{\star}\right]\geq\frac{1}{2}\overset{\delta<1/10}{>}\delta\\
\Rightarrow & \inf\left\{ \epsilon\geq0:\inf_{{\bf A}_{0:T}\in\A_{T}}\sup_{F\in\mathfrak{F}_{D,G}^{\mathrm{\str}}}\Pr\left[F(\bx_{T+1})-F_{\star}>\epsilon\right]\leq\delta\right\} \geq\epsilon_{\star},
\end{align*}
which implies that
\[
\bar{R}_{\star}^{\str}(\delta)\geq\epsilon_{\star}\overset{(\ref{eq:str-hp-lb-epsilon-1})}{=}\frac{\mu\theta^{2}q^{2}M^{2}d_{\star}}{16}\overset{(\ref{eq:str-hp-lb-parameter-1})}{=}\Omega\left(\min\left\{ \mu D^{2},\frac{G^{2}}{\mu},\frac{\sigma_{\sma}^{\frac{4}{\p}-2}\sigma_{\lar}^{4-\frac{4}{\p}}}{\mu T^{2-\frac{2}{\p}}}\right\} \right).
\]

\subparagraph{The case $d_{\star}\in\left[1,32\ln2\right]$.}

For this case, it is enough to show $\bar{R}_{\star}^{\str}(\delta)\geq\Omega\left(\min\left\{ \mu D^{2},\frac{G^{2}}{\mu},\frac{\sigma_{\sma}^{2}}{\mu T^{2-\frac{2}{\p}}}\right\} \right)$,
since $\sigma_{\sma}^{2}=\frac{\sigma_{\sma}^{\frac{4}{\p}-2}\sigma_{\lar}^{4-\frac{4}{\p}}}{d_{\eff}^{2-\frac{2}{\p}}}=\Theta(\sigma_{\sma}^{\frac{4}{\p}-2}\sigma_{\lar}^{4-\frac{4}{\p}})$
when $d_{\star}=\left\lceil d_{\eff}\right\rceil \in\left[1,32\ln2\right]$.
By (\ref{eq:str-hp-lb-bound-2}), we have for any $\delta\in\left(0,\frac{1}{10}\right)$,
\begin{align*}
\bar{R}_{\star}^{\str}(\delta) & \geq\Omega\left(\min\left\{ \mu D^{2},\frac{G^{2}}{\mu},\frac{\sigma_{\sma}^{\frac{4}{\p}-2}\sigma_{\lar}^{4-\frac{4}{\p}}}{\mu},\frac{\sigma_{\sma}^{2}\ln^{2-\frac{2}{\p}}\frac{5}{4}}{\mu T^{2-\frac{2}{\p}}}\right\} \right)\\
 & \geq\Omega\left(\min\left\{ \mu D^{2},\frac{G^{2}}{\mu},\frac{\sigma_{\sma}^{2}}{\mu},\frac{\sigma_{\sma}^{2}}{\mu T^{2-\frac{2}{\p}}}\right\} \right)=\Omega\left(\min\left\{ \mu D^{2},\frac{G^{2}}{\mu},\frac{\sigma_{\sma}^{2}}{\mu T^{2-\frac{2}{\p}}}\right\} \right).
\end{align*}

\paragraph{Second bound.}

For the second bound, we will show
\begin{equation}
R_{\star}^{\str}(\delta)\overset{\text{Lemma \ref{lem:hp-lb-reformulation}}}{\geq}\bar{R}_{\star}^{\str}(\delta)\geq\Omega\left(\min\left\{ \mu D^{2},\frac{G^{2}}{\mu},\frac{\sigma_{\sma}^{\frac{4}{\p}-2}\sigma_{\lar}^{4-\frac{4}{\p}}}{\mu},\frac{\sigma_{\sma}^{2}\ln^{2-\frac{2}{\p}}\frac{1}{8\delta}}{\mu T^{2-\frac{2}{\p}}}\right\} \right),\forall\delta\in\left(0,\frac{1}{8}\right).\label{eq:str-hp-lb-bound-2}
\end{equation}

In this setting, we set
\begin{equation}
\W\defeq\left\{ v^{+}\defeq(1,\mydots,1)^{\top},v^{-}\defeq(\underbrace{-1,\mydots,-1}_{d_{\star}},1,\mydots,1)^{\top}\right\} \subseteq\left\{ -1,1\right\} ^{d}=\V.\label{eq:str-hp-lb-W-2}
\end{equation}
For any $i\in\left[d\right]$, we pick
\begin{equation}
\begin{array}{ccc}
q_{i}=q\defeq\min\left\{ \frac{\ln\frac{1}{8\delta}}{Td_{\star}\theta\ln\frac{1+\theta}{1-\theta}},1\right\}  & \text{and} & \theta_{i}=\theta\defeq\frac{1}{2},\\
M_{i}=M\1\left[i\leq d_{\star}\right] & \text{and} & M\defeq\min\left\{ \frac{D}{\theta q\sqrt{d_{\star}}},\frac{G}{\mu\theta q\sqrt{d_{\star}}},\frac{\sigma_{\lar}}{\mu(4qd_{\star})^{\frac{1}{\p}}}\right\} .
\end{array}\label{eq:str-hp-lb-parameter-2}
\end{equation}
Similar to before, one can use Lemma \ref{lem:str-lb-prop} to verify
$f_{v}\in\mathfrak{f}_{G}^{\cvx}$ and $F_{v}=f_{v}+\frac{\mu}{2}\left\Vert \cdot\right\Vert ^{2}\in\mathfrak{F}_{D,G}^{\str}$,
and check that the oracle $\bg$ constructed in (\ref{eq:lb-g}) satisfies
$\bg\in\G_{\sigma_{\sma},\sigma_{\lar}}^{\p}$.

Now let us consider the optimization procedure for any algorithm ${\bf A}_{0:T}\in\A_{T}$
interacting with the oracle $\bg$ in (\ref{eq:lb-g}) and define
\begin{equation}
\epsilon_{\star}\defeq\frac{\mu\theta^{2}q^{2}M^{2}d_{\star}}{4}.\label{eq:str-hp-lb-epsilon-2}
\end{equation}
Under almost the same argument used in the proof of Theorem \ref{thm:cvx-hp-lb},
one can still obtain
\[
\sup_{F\in\mathfrak{F}_{D,G}^{\str}}\Pr\left[F(\bx_{T+1})-F_{\star}>\epsilon_{\star}\right]\geq2\delta.
\]

Since ${\bf A}_{0:T}\in\A_{T}$ is arbitrarily chosen, we finally
have, under $\bg$ given in (\ref{eq:lb-g}),
\begin{align*}
 & \inf_{{\bf A}_{0:T}\in\A_{T}}\sup_{F\in\mathfrak{F}_{D,G}^{\str}}\Pr\left[F(\bx_{T+1})-F_{\star}>\epsilon_{\star}\right]\geq2\delta\\
\Rightarrow & \inf\left\{ \epsilon\geq0:\inf_{{\bf A}_{0:T}\in\A_{T}}\sup_{F\in\mathfrak{F}_{D,G}^{\mathrm{\str}}}\Pr\left[F(\bx_{T+1})-F_{\star}>\epsilon\right]\leq\delta\right\} \geq\epsilon_{\star},
\end{align*}
which implies that
\[
\bar{R}_{\star}^{\str}(\delta)\geq\epsilon_{\star}\overset{(\ref{eq:str-hp-lb-epsilon-2})}{=}\frac{\mu\theta^{2}q^{2}M^{2}d_{\star}}{4}\overset{(\ref{eq:str-hp-lb-parameter-2})}{=}\Omega\left(\min\left\{ \mu D^{2},\frac{G^{2}}{\mu},\frac{\sigma_{\sma}^{\frac{4}{\p}-2}\sigma_{\lar}^{4-\frac{4}{\p}}}{\mu},\frac{\sigma_{\sma}^{2}\ln^{2-\frac{2}{\p}}\frac{1}{8\delta}}{\mu T^{2-\frac{2}{\p}}}\right\} \right).
\]

\paragraph{Final bound.}

\end{proof}
Finally, we combine (\ref{eq:str-hp-lb-bound-1}) and (\ref{eq:str-hp-lb-bound-2})
to conclude
\begin{align*}
R_{\star}^{\str}(\delta) & \ge\Omega\left(\min\left\{ \mu D^{2},\frac{G^{2}}{\mu},\frac{\sigma_{\sma}^{\frac{4}{\p}-2}\sigma_{\lar}^{4-\frac{4}{\p}}}{\mu T^{2-\frac{2}{\p}}}\right\} +\min\left\{ \mu D^{2},\frac{G^{2}}{\mu},\frac{\sigma_{\sma}^{\frac{4}{\p}-2}\sigma_{\lar}^{4-\frac{4}{\p}}}{\mu},\frac{\sigma_{\sma}^{2}\ln^{2-\frac{2}{\p}}\frac{1}{8\delta}}{\mu T^{2-\frac{2}{\p}}}\right\} \right)\\
 & =\Omega\left(\min\left\{ \mu D^{2},\frac{G^{2}}{\mu},\frac{\sigma_{\sma}^{\frac{4}{\p}-2}\sigma_{\lar}^{4-\frac{4}{\p}}}{\mu},\frac{\sigma_{\sma}^{\frac{4}{\p}-2}\sigma_{\lar}^{4-\frac{4}{\p}}+\sigma_{\sma}^{2}\ln^{2-\frac{2}{\p}}\frac{1}{8\delta}}{\mu T^{2-\frac{2}{\p}}}\right\} \right),\forall\delta\in\left(0,\frac{1}{10}\right).
\end{align*}

\subsection{In-Expectation Lower Bounds}

In this part, we provide the in-expectation lower bounds. The proof
is based on the following lemma, which reduces any valid high-probability
lower bound to an in-expectation lower bound. Essentially, Lemma \ref{lem:ex-lb-reduction}
is Proposition 2 of \citet{ma2024high}.
\begin{lem}
\label{lem:ex-lb-reduction}For any $\mathrm{\type}\in\left\{ \mathrm{cvx},\mathrm{str}\right\} $
and $\delta\in\left(0,1\right]$, we have
\[
R_{\star}^{\mathrm{\type}}\geq\delta R_{\star}^{\mathrm{\type}}(\delta).
\]
\end{lem}
\begin{proof}
Given $\delta\in\left(0,1\right]$, $\bg\in\G_{\sigma_{\sma},\sigma_{\lar}}^{\p}$,
${\bf A}_{0:T}\in\A_{T}$ and $F\in\mathfrak{F}_{D,G}^{\mathrm{\type}}$,
there exists $\epsilon(\delta)\geq0$ such that
\[
\left\{ \epsilon\geq0:\Pr\left[F(\bx_{T+1})-F_{\star}>\epsilon\right]\leq\delta\right\} =\left[\epsilon(\delta),+\infty\right),
\]
since $\Pr\left[F(\bx_{T+1})-F_{\star}>\epsilon\right]$ is upper
semicontinuous and nonincreasing in $\epsilon>0$. Thus, for any $\epsilon\in\left[0,\epsilon(\delta)\right)$,
\[
\E\left[F(\bx_{T+1})-F_{\star}\right]\geq\epsilon\Pr\left[F(\bx_{T+1})-F_{\star}>\epsilon\right]>\epsilon\delta,
\]
implying that
\[
\E\left[F(\bx_{T+1})-F_{\star}\right]\geq\delta\epsilon(\delta)=\delta\inf\left\{ \epsilon\geq0:\Pr\left[F(\bx_{T+1})-F_{\star}>\epsilon\right]\leq\delta\right\} .
\]
Finally, we obtain
\begin{align*}
R_{\star}^{\mathrm{\type}} & \overset{(\ref{eq:ex-lb-minimax})}{=}\sup_{\bg\in\G_{\sigma_{\sma},\sigma_{\lar}}^{\p}}\inf_{{\bf A}_{0:T}\in\A_{T}}\sup_{F\in\mathfrak{F}_{D,G}^{\mathrm{\type}}}\E\left[F(\bx_{T+1})-F_{\star}\right]\\
 & \geq\sup_{\bg\in\G_{\sigma_{\sma},\sigma_{\lar}}^{\p}}\inf_{{\bf A}_{0:T}\in\A_{T}}\sup_{F\in\mathfrak{F}_{D,G}^{\mathrm{\type}}}\delta\inf\left\{ \epsilon\geq0:\Pr\left[F(\bx_{T+1})-F_{\star}>\epsilon\right]\leq\delta\right\} \overset{(\ref{eq:hp-lb-minimax})}{=}\delta R_{\star}^{\mathrm{\type}}(\delta).
\end{align*}
\end{proof}

\subsubsection{General Convex Case}
\begin{thm}[Formal version of Theorem \ref{thm:main-cvx-ex-lb}]
\label{thm:cvx-ex-lb}Given $D>0$, $G>0$, $\p\in\left(1,2\right]$,
and $0<\sigma_{\sma}\leq\sigma_{\lar}$, for any $d\geq d_{\eff}$,
we have 
\[
R_{\star}^{\cvx}\geq\Omega\left(\min\left\{ GD,\frac{\sigma_{\sma}^{\frac{2}{\p}-1}\sigma_{\lar}^{2-\frac{2}{\p}}D}{T^{1-\frac{1}{\p}}}\right\} \right),
\]
 where $R_{\star}^{\cvx}$ is defined in (\ref{eq:ex-lb-minimax}).
\end{thm}
\begin{proof}
We apply Lemma \ref{lem:ex-lb-reduction} with $\type=\cvx$ and $\delta=\frac{1}{20}$
to conclude
\[
R_{\star}^{\cvx}\geq\frac{R_{\star}^{\mathrm{\cvx}}(1/20)}{20}\overset{(\ref{eq:cvx-hp-lb-bound-1})}{\geq}\Omega\left(\min\left\{ GD,\frac{\sigma_{\sma}^{\frac{2}{\p}-1}\sigma_{\lar}^{2-\frac{2}{\p}}D}{T^{1-\frac{1}{\p}}}\right\} \right).
\]
\end{proof}

\subsubsection{Strongly Convex Case}
\begin{thm}[Formal version of Theorem \ref{thm:main-str-ex-lb}]
\label{thm:str-ex-lb}Given $D>0$, $\mu>0$, $G>0$, $\p\in\left(1,2\right]$,
and $0<\sigma_{\sma}\leq\sigma_{\lar}$, for any $d\geq d_{\eff}$,
we have
\[
R_{\star}^{\str}\geq\Omega\left(\min\left\{ \mu D^{2},\frac{G^{2}}{\mu},\frac{\sigma_{\sma}^{\frac{4}{\p}-2}\sigma_{\lar}^{4-\frac{4}{\p}}}{\mu T^{2-\frac{2}{\p}}}\right\} \right),
\]
 where $R_{\star}^{\str}$ is defined in (\ref{eq:ex-lb-minimax}).
\end{thm}
\begin{proof}
We apply Lemma \ref{lem:ex-lb-reduction} with $\type=\str$ and $\delta=\frac{1}{20}$
to conclude
\[
R_{\star}^{\str}\geq\frac{R_{\star}^{\str}(1/20)}{20}\overset{(\ref{eq:str-hp-lb-bound-1})}{\geq}\Omega\left(\min\left\{ \mu D^{2},\frac{G^{2}}{\mu},\frac{\sigma_{\sma}^{\frac{4}{\p}-2}\sigma_{\lar}^{4-\frac{4}{\p}}}{\mu T^{2-\frac{2}{\p}}}\right\} \right).
\]
\end{proof}

\subsection{A Helpful Lemma}
\begin{lem}
\label{lem:minimax}Given two arbitrary sets $\mathbb{I}$ and $\mathbb{J}$,
suppose the function $h:\mathbb{I}\times\mathbb{J}\times\left[0,+\infty\right)\to\left[0,1\right],(i,j,\epsilon)\mapsto h(i,j,\epsilon)$
is nonincreasing in $\epsilon$ for any $\left(i,j\right)\in\mathbb{I}\times\mathbb{J}$,
then we have
\[
\inf_{i\in\mathbb{I}}\sup_{j\in\mathbb{J}}\inf\left\{ \epsilon\geq0:h(i,j,\epsilon)\leq\delta\right\} \geq\inf\left\{ \epsilon\geq0:\inf_{i\in\mathbb{I}}\sup_{j\in\mathbb{J}}h(i,j,\epsilon)\leq\delta\right\} ,\forall\delta\in\left[0,1\right].
\]
\end{lem}
\begin{proof}
We fix $\delta\in\left[0,1\right]$ in the following proof. 

First, we show that, for any given $i\in\mathbb{I}$,
\begin{equation}
\sup_{j\in\mathbb{J}}\inf\left\{ \epsilon\geq0:h(i,j,\epsilon)\leq\delta\right\} =\inf\left\{ \epsilon\geq0:\sup_{j\in\mathbb{J}}h(i,j,\epsilon)\leq\delta\right\} .\label{lem:minimax-1}
\end{equation}
On the one hand, we have
\begin{align}
\left\{ \epsilon\geq0:h(i,j,\epsilon)\leq\delta\right\}  & \supseteq\left\{ \epsilon\geq0:\sup_{j\in\mathbb{J}}h(i,j,\epsilon)\leq\delta\right\} ,\forall j\in\mathbb{J}\nonumber \\
\Rightarrow\inf\left\{ \epsilon\geq0:h(i,j,\epsilon)\leq\delta\right\}  & \leq\inf\left\{ \epsilon\geq0:\sup_{j\in\mathbb{J}}h(i,j,\epsilon)\leq\delta\right\} ,\forall j\in\mathbb{J}\nonumber \\
\Rightarrow\sup_{j\in\mathbb{J}}\inf\left\{ \epsilon\geq0:h(i,j,\epsilon)\leq\delta\right\}  & \leq\inf\left\{ \epsilon\geq0:\sup_{j\in\mathbb{J}}h(i,j,\epsilon)\leq\delta\right\} .\label{lem:minimax-2}
\end{align}
On the other hand, for any $j\in\mathbb{J}$ and $\zeta>0$, there
exists $0\leq\epsilon(j,\zeta)\leq\inf\left\{ \epsilon\geq0:h(i,j,\epsilon)\leq\delta\right\} +\zeta$
such that $h(i,j,\epsilon(j,\zeta))\leq\delta$. Since $h(i,j,\epsilon)$
is nonincreasing in $\epsilon$, we know
\[
h(i,j,\sup_{j\in\mathbb{J}}\epsilon(j,\zeta))\leq\delta,\forall j\in\mathbb{J}\Rightarrow\sup_{j\in\mathbb{J}}h(i,j,\sup_{j\in\mathbb{J}}\epsilon(j,\zeta))\leq\delta,
\]
implying that, for any $\zeta>0$,
\begin{align}
\sup_{j\in\mathbb{J}}\inf\left\{ \epsilon\geq0:h(i,j,\epsilon)\leq\delta\right\} +\zeta & \geq\sup_{j\in\mathbb{J}}\epsilon(j,\zeta)\geq\inf\left\{ \epsilon\geq0:\sup_{j\in\mathbb{J}}h(i,j,\epsilon)\leq\delta\right\} \nonumber \\
\Rightarrow\sup_{j\in\mathbb{J}}\inf\left\{ \epsilon\geq0:h(i,j,\epsilon)\leq\delta\right\}  & \geq\inf\left\{ \epsilon\geq0:\sup_{j\in\mathbb{J}}h(i,j,\epsilon)\leq\delta\right\} .\label{lem:minimax-3}
\end{align}
Combining (\ref{lem:minimax-2}) and (\ref{lem:minimax-3}) yields
(\ref{lem:minimax-1}).

Next, due to (\ref{lem:minimax-1}), it suffices to show
\[
\inf_{i\in\mathbb{I}}\inf\left\{ \epsilon\geq0:\sup_{j\in\mathbb{J}}h(i,j,\epsilon)\leq\delta\right\} \geq\inf\left\{ \epsilon\geq0:\inf_{i\in\mathbb{I}}\sup_{j\in\mathbb{J}}h(i,j,\epsilon)\leq\delta\right\} ,
\]
which is true, since
\[
\left\{ \epsilon\geq0:\sup_{j\in\mathbb{J}}h(i,j,\epsilon)\leq\delta\right\} \subseteq\left\{ \epsilon\geq0:\inf_{i\in\mathbb{I}}\sup_{j\in\mathbb{J}}h(i,j,\epsilon)\leq\delta\right\} ,\forall i\in\mathbb{I}.
\]
\end{proof}

\section{Numerical Simulations}

In this section, we provide some numerical simulations to support
our theory. We limit our attention to the additive noise model, i.e.,
$\bg(\bx,\xi)=\nabla f(\bx)+\xi$, where all coordinates $\xi_{i}$
are assumed to be i.i.d. Moreover, we denote by $\sigma\defeq\left(\E\left[\left|\xi_{1}\right|^{\p}\right]\right)^{\frac{1}{\p}}$.

\textbf{Objective.} We pick $\X=\R^{d}$, $f(\bx)=\left\Vert \bx-\by\right\Vert _{1}$
for some $\by\in\R^{d}$, and $r(\bx)=0$. Therefore, we know $F=f$,
$\argmin_{\bx\in\R^{d}}F(\bx)=\by$ and $F_{\star}=0$. Moreover,
we have $\mu=0$ and $G=\sqrt{d}$.

\textbf{Noise.} We choose $\xi_{i}\sim\epsilon Z$ i.i.d. for all
$i\in\left[d\right]$, where $\epsilon$ and $Z$ are independent
and satisfy that $\Pr\left[\epsilon=2\right]=\frac{1}{3}$ and $\Pr\left[\epsilon=-1\right]=\frac{2}{3}$,
and $Z$ follows the Pareto distribution with the scale parameter
$\frac{\alpha-1}{\alpha}$ and the shape parameter $\alpha=\p+0.001$,
i.e., $\Pr\left[Z>z\right]=\left(\frac{\alpha-1}{\alpha z}\right)^{\alpha}\1\left[z\geq\frac{\alpha-1}{\alpha}\right]+\1\left[z<\frac{\alpha-1}{\alpha}\right]$.
Note that we have $\E\left[\epsilon Z\right]=0$, $\E\left[\left|\epsilon\right|^{\p}\right]=\frac{2^{\p}+2}{3}$
and $\E\left[Z^{\p}\right]=\frac{\alpha}{\alpha-\p}\left(\frac{\alpha-1}{\alpha}\right)^{\p}$,
implying that $\E\left[\xi_{i}\right]=0$ and $\sigma=\left(\E\left[\left|\xi_{1}\right|^{\p}\right]\right)^{\frac{1}{\p}}=\left(\frac{2^{\p}+2}{3}\right)^{\frac{1}{\p}}\left(\frac{\alpha}{\alpha-\p}\right)^{\frac{1}{\p}}\frac{\alpha-1}{\alpha}$.

\textbf{Algorithms.} We consider \citet{liu2023stochastic} as the
baseline, since it is closest to our setting, and choose the stepsize
$\eta_{t}$ and the clipping threshold $\cm_{t}$ as follows:
\begin{itemize}
\item Adopted from Theorem 4 in \citet{liu2023stochastic}: $\eta_{t}=\frac{\eta}{\sigma_{\lar}t^{1/\p}}$
and $\cm_{t}=\max\left\{ 2\sqrt{d},\sigma_{\lar}t^{1/\p}\right\} $,
where $\eta=\left\Vert \bx_{1}-\by\right\Vert $ and $\bx_{1}$ is
the initial point.
\item Adopted from our Theorem \ref{thm:main-cvx-ex-dep-T}: $\eta_{t}=\frac{\eta}{\sigma_{\sma}^{2/\p-1}\sigma_{\lar}^{2-2/\p}t^{1/\p}}$
and $\cm_{t}=\max\left\{ 2\sqrt{d},\frac{\sigma_{\lar}}{d_{\eff}^{1/\p}}t^{1/\p}\right\} $,
where $\eta=\left\Vert \bx_{1}-\by\right\Vert $ and $\bx_{1}$ is
the initial point.
\end{itemize}
\begin{rem}
For both $\eta_{t}$, we only keep the dominant term in the order
of $\O(1/t^{1/\p})$ for simplicity. We pick $\eta=\left\Vert \bx_{1}-\by\right\Vert $
to match the optimal choice in theory. Moreover, $\eta_{t}$ is set
in an anytime fashion, i.e., depending on $t$ instead of $T$. $\sigma_{\lar}=\sqrt{d}\sigma$
and $\sigma_{\sma}=2^{\frac{2}{\p}-1}d^{\frac{1}{\p}-\frac{1}{2}}\sigma$
are set based on their bounds given in (\ref{eq:sigma-l-general})
and (\ref{eq:sigma-s-general}), respectively. $d_{\eff}$ is set
as its lower bound $\frac{d^{2-\frac{2}{\p}}}{2^{\frac{4}{\p}-2}}$
established in (\ref{eq:d-eff-iid}).
\end{rem}
\textbf{Parameter values. }In experiments, we fix $d=50$, set $\by_{i}=\begin{cases}
2i/d & i\leq d/2\\
-2i/d & i>d/2
\end{cases}$, initialize $\bx_{1}=\bzero$, and let $T=10000$. For two kinds
of $(\eta_{t},\cm_{t})$, we run 10 trials for each and plot the mean
($\pm$ standard error) of the trajectory $F(\bar{\bx}_{t+1}^{\cvx})-F_{\star}=F(\bar{\bx}_{t+1}^{\cvx})$,
as used in the convergence theory, where we recall $\bar{\bx}_{t+1}^{\cvx}=\frac{1}{t}\sum_{s=1}^{t}\bx_{s+1}$.
We test $\p\in\left\{ 1.2,1.4,1.6,1.8\right\} $ and report the results
in Figure \ref{fig:simulation-asym}.

\begin{figure}[h]
\centering\includegraphics[scale=0.275]{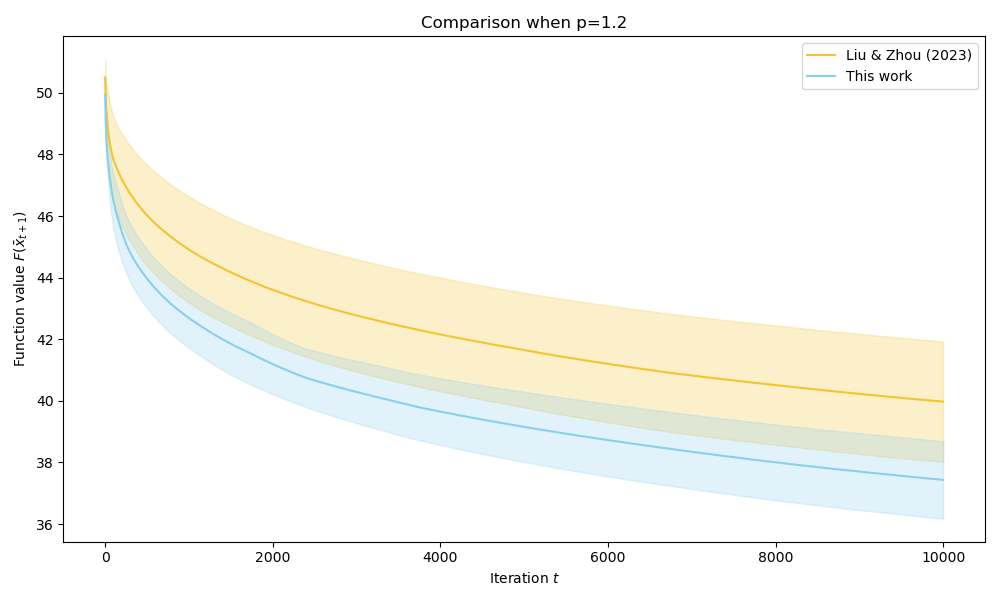}\includegraphics[scale=0.275]{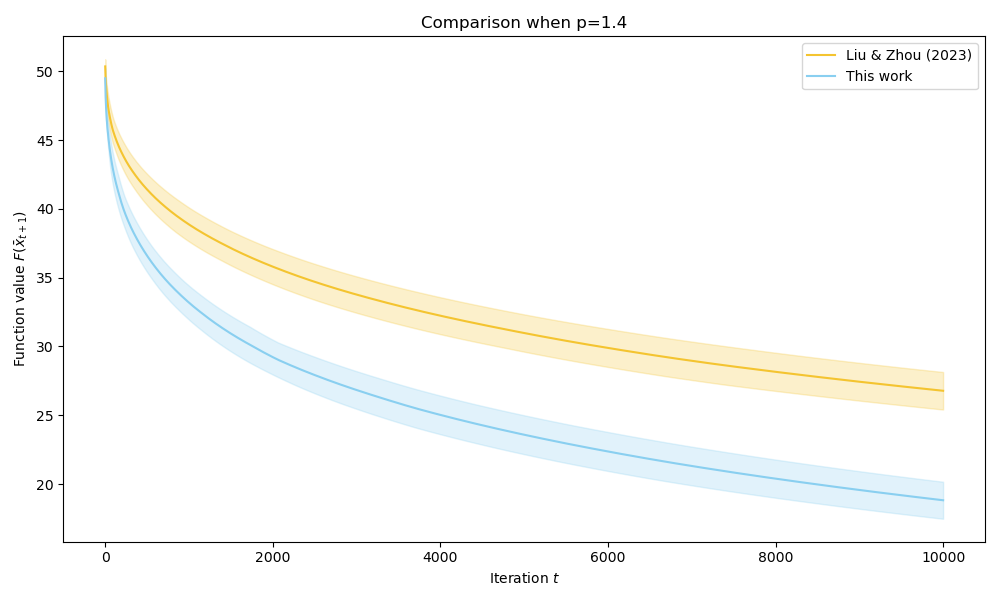}

\centering\includegraphics[scale=0.275]{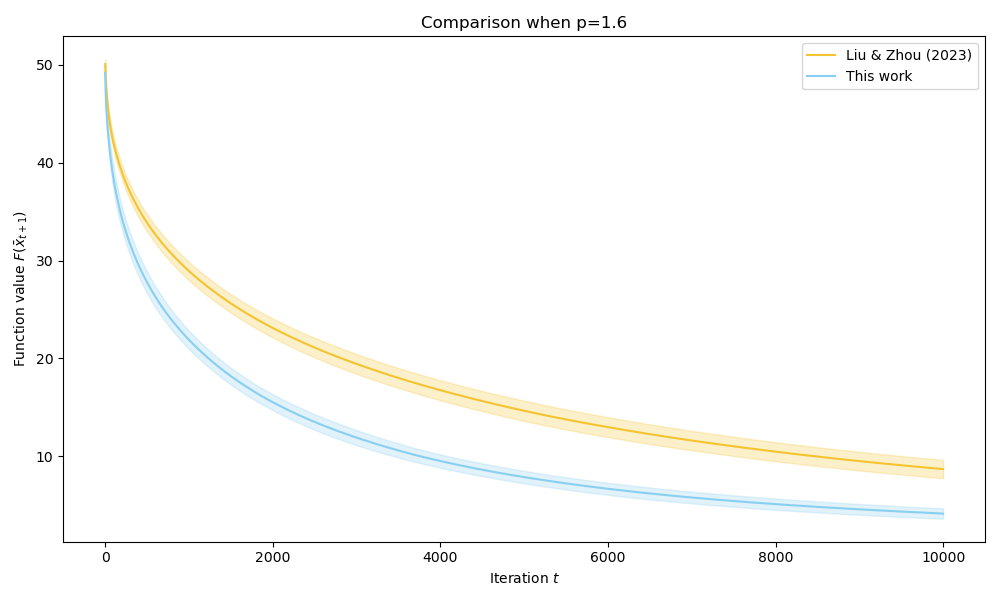}\includegraphics[scale=0.275]{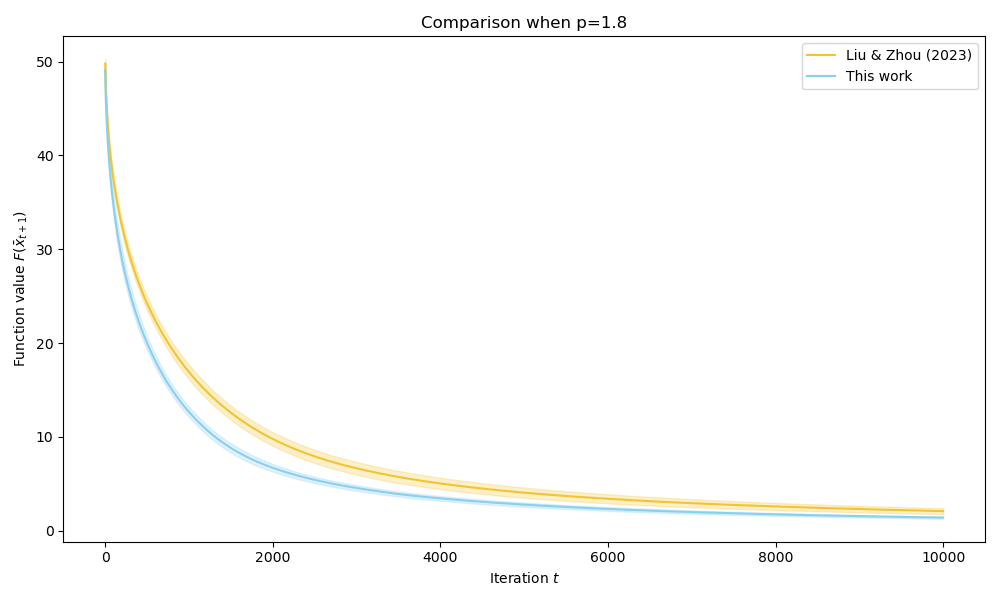}

\caption{\label{fig:simulation-asym}Comparison between \citet{liu2023stochastic}
and this work when $\protect\p=1.2$ (top left), $\protect\p=1.4$
(top right), $\protect\p=1.6$ (bottom left), $\protect\p=1.8$ (bottom
right).}
\end{figure}

\textbf{Observation and Conclusion. }In all cases, the $(\eta_{t},\cm_{t})$
pair chosen based on our work is faster, matching the new theoretical
finding when $\sigma_{\sma}\neq\sigma_{\lar}$. As $\p$ approaches
$2$, the difference becomes minor, which should be expected, since
the improvement predicted by our theory is in the order of $\Theta(1/d_{\eff}^{\frac{2-\p}{2\p}})$
(see discussion under Theorem \ref{thm:main-cvx-ex-dep-T}), which
will vanish if $\p$ is close to $2$.

\end{document}